\pdfoutput=1 
\documentclass[a4paper,12pt]{article}
\usepackage{amsmath,amsthm,amssymb}
\usepackage[colorlinks=true,citecolor=black,linkcolor=black,urlcolor=blue]{hyperref}

\usepackage{rotating,graphicx}
\usepackage{tikz}
\usetikzlibrary{shapes}
\usepackage[format=plain,indention=0cm]{caption}
\usepackage{float}
\usepackage{mathdots,enumitem}

\hypersetup{
pdftitle={Perfect colorings of the infinite square grid: coverings and twin colors},
pdfauthor={Denis Krotov},
pdfkeywords={perfect coloring, equitable partition, square grid, rectangular grid, twin colors, graph covering}
}

\date{}

\theoremstyle{plain}
\newtheorem{theorem}{Theorem}
\newtheorem{lemma}[theorem]{Lemma}
\newtheorem{corollary}[theorem]{Corollary}
\newtheorem{proposition}[theorem]{Proposition}
\newtheorem{claim}[theorem]{Claim}

\theoremstyle{definition}

\theoremstyle{remark}
\newtheorem{remark}[theorem]{\bf Remark}

\renewcommand{\thefigure}{\Alph{figure}}

\renewcommand{\thesubsection}{\Roman{subsection}}

\makeatletter
\renewcommand{\paragraph}[2][\scndArg]{{%
\def\scndArg{#2} 
\@startsection{paragraph}{4}{0ex}%
   {-3.25ex plus -1ex minus -0.2ex}%
   {1.5ex plus 0.2ex}%
   {\normalfont\normalsize\bfseries}[#1]{#2}}}
\renewcommand{\theparagraph}{\thesubsubsection.\@arabic\c@paragraph}

\renewcommand{\subparagraph}[2][\scndArg]{{%
\def\scndArg{#2} 
\@startsection{subparagraph}{5}{0ex}%
   {-1.25ex plus -0.5ex minus -0.2ex}%
   {0.5ex plus 0.1ex}%
   {\normalfont\normalsize\bfseries\em}[#1]{#2}}}
\renewcommand{\thesubparagraph}{\theparagraph.\@roman\c@subparagraph}
\makeatother

\theoremstyle{plain}
\newtheorem*{claim*}{Claim}
\theoremstyle{remark}
\newtheorem*{remark*}{Remark}

\newenvironment{tikztab}{\begin{tabular}{c}\begin{tikzpicture}}{\end{tikzpicture}\end{tabular}}
\def\ZZ{\mathbb{Z}}
\def\FfF{{\mathcal F}}
\def\GgG{{\mathcal G}}
\def\HhH{{\mathcal H}}

\def\Lft{([}
\def\Rgt{])}

\newlength{\dk} \dk=1.3em
\newlength{\mk} \mk=0.7em
\newlength{\DK} \DK=1.41421356\dk
\tikzset{x=1.0\dk,y=1.0\dk,
 cell/.style={rectangle,draw=black!15,very thin, minimum size=1\dk},
 mell/.style={rectangle,draw=black!15,very thin, minimum size=1\mk},
 oell/.style={rectangle,very thin, minimum size=1\mk},
 cell/.style={rectangle,draw=black!15,very thin, minimum size=1\dk},
 dell/.style={diamond,draw=black!15,very thin, minimum size=\DK},
 zell/.style={diamond,draw=white,very thin,  minimum size=\DK},
 diag/.style={draw=black!30!white, thin, dash pattern=on 0.28\DK off 0.44\DK on 0.28\DK off 0em}
 }
\def\dfant{ ++(1\DK,0) node [zell] {} }
\def\dcolor#1#2{ ++(1\DK,0) node [dell,fill=#1] {\raisebox{-0.25em}[0.2em][0em]{\makebox[0mm][c]{#2}}} }
\def\Dcolor#1#2#3{ ++(1\DK,0) node [dell,inner color=#1,outer color=#2] {\raisebox{-0.25em}[0.2em][0em]{\makebox[0mm][c]{#3}}} }
\def\DDolor#1#2#3#4{ ++(1\DK,0) node [dell,shading = axis,left color=#1,right color=#2,shading angle=#3] {\raisebox{-0.25em}[0.2em][0em]{\makebox[0mm][c]{#4}}} }
\def\ccolor#1#2{ ++(1,0) node [cell,fill=#1] {\raisebox{-0.25em}[0.2em][0em]{\makebox[0mm][c]{#2}}} }
\def\CColor#1#2#3#4{ ++(1,0) node [cell,shading = axis,left color=#1,right color=#2,shading angle=#3] {\raisebox{-0.25em}[0.2em][0em]{\makebox[0mm][c]{#4}}} }
\def\Ccolor#1#2#3{ ++(1,0) node [cell,inner color=#1,outer color=#2] {\raisebox{-0.25em}[0.2em][0em]{\makebox[0mm][c]{#3}}} }

\def\cocolor#1#2{ ++(1,0) node [oell,fill=#1,] {\raisebox{-0.25em}[0.2em][0em]{\makebox[0mm][c]{$ #2$}}} }

\def\t{ ++(1,0) node  {\raisebox{-0.25em}[0.2em][0em]{\makebox[0mm][c]{$\cdots$}}} }
\def\td{ ++(1,0) node  {\makebox[0mm][c]{\begin{turn}{-45}$\cdots$\end{turn}}} }
\def\dt{ ++(1,0) node  {\makebox[0mm][c]{\begin{turn}{45}$\cdots$\end{turn}}} }
\def\o{\ccolor{white}{$\ $}}
\def\dfc#1#2{\definecolor{#1}{HTML}{#2}}
\definecolor{clrTWO}{rgb}{0.78, 0.95, 0.70}
\def\colONE{blue!30!green!25!white}
\definecolor{yell}{rgb}{0.95,0.9,0.45}
\def\colTWO{clrTWO!90!white}
\def\1{\ccolor{\colONE}{$1$}}
\def\2{\ccolor{\colTWO}{$2$}}
\def\3{\ccolor{blue!70!green!25!white}{$3$}}
\def\4{\ccolor{blue!70!red!25!white}{$4$}}
\def\0{++(1,0)}

\newcommand\cll[2][0.7em]{\protect\framebox{$\smash{#2}\protect\vphantom{1}$}}
\newcommand\ccll[1]{\cll{#1}}
\newcommand\Cll[1]{\cll{#1}}
\def\scll#1{#1}
\def\dll#1#2{\cll{#1}\cll{#2}}

\def\DDll#1#2{\ccll{#1}\ccll{#2}}
\def\Ddll#1#2{\ccll{#1}\cll{#2}}

\title{Perfect colorings of the infinite square grid:\\ coverings and twin colors%
\thanks{The study was carried out within the framework of the state contract
of the Sobolev Institute of Mathematics (Project  FWNF-2022-0017) and supported by the Russian Foundation for Basic Research (Grant 20-51-53023).}\thanks{%
This is the author's final version of the manuscript published in The Electronic Journal of Combinatorics 30(2), 2023,
\url{https://doi.org/10.37236/10005}
}
}

\author{Denis S. Krotov\thanks{Sobolev Institute of Mathematics, Novosibirsk, Russia. \tt krotov@math.nsc.ru}}
\begin{document}

\maketitle

\begin{abstract}
A perfect coloring (equivalent concepts are equitable partition and partition design) of a graph $G$ is a function $f$ from the set of vertices onto some finite set (of colors) such that every node of color $i$ has exactly $S(i,j)$ neighbors of color $j$, where  $S(i,j)$ are constants, forming the matrix $S$ called quotient. If $S$ is an adjacency matrix of some simple graph $T$ on the set of colors, then $f$ is called a covering of the target graph $T$ by the cover graph $G$. We characterize all coverings by the infinite square grid, proving that every such coloring is either orbit (that is, corresponds to the orbit partition under the action of some group of graph automorphisms) or has twin colors (that is, two colors such that unifying them keeps the coloring perfect). The case of twin colors is separately classified.

\end{abstract}

\vspace{1.5em}
\section{Introduction}

This work is focused on the problem of classification of
perfect colorings of the infinite square grid $G(\ZZ^2)$.
We classify an important subcase of perfect colorings
when there are no two nodes of the same color at distance one
or two from each other. Separately, we characterize bipartite
perfect colorings in the case when two colors can be merged
keeping the perfectness of the coloring.

Perfect colorings
(equivalent concepts are equitable partitions,
regular partitions,
partition designs)
of graphs are usually considered as objects
of algebraic combinatorics, see, e.g.,~\cite[9.3]{GoRo},
because of their connection with
the eigenspaces of graphs, which determine many of their properties.
While perfect colorings can be considered for any graph or even multigraph,
the most natural classes of graphs to study from this point of view are distance-regular
graphs~\cite{Brouwer} and vertex-transitive graphs.
Perfect colorings of vertex-transitive infinite grids can be considered as
combinatorial analogs of crystals.

One of motivations of the study of perfect colorings is that
they naturally generalize
 more special classes of objects
with similar regular properties.
The list of concepts that can be
defined as perfect colorings with
special parameters
includes
perfect codes,
MDS codes
(maximum distance separable codes \cite[Ch.\,11]{MWS}, not necessarily linear) and
MRD (maximum rank distance) codes \cite{Gabidulin:2021}
with distance~$2$ or~$3$,
latin squares \cite{DK:74} and latin hypercubes \cite{MK-W:small},
Steiner triple and quadruple systems \cite{ColMat:Steiner},
some other types of combinatorial designs
and their subspace analogues \cite{BKW:2018subspace},
Cameron--Liebler line classes \cite{CL:82},
intriguing sets \cite{DeBruSuz:2010}
(which are essentially perfect $2$-colorings).
For some subclasses of classical objects,
their equivalence to perfect colorings
with corresponding
parameters are less straightforward and
the study of such connections can be considered
as a notable direction in discrete mathematics.
For example, orthogonal arrays \cite{HSS:OA} and error-correcting
codes \cite{MWS} whose parameters attain certain bounds
are equivalent to perfect colorings
with certain parameters
(see e.g.  \cite{FDF:CorrImmBound},
\cite{Pot:2012:color},
\cite[Th.1]{Kro:OA13},
\cite{Kro:2m-4},
\cite{BRZ:CR}),
Boolean bent functions are equivalent to perfect $4$-colorings
of Grassmann graphs \cite[Th.\,2]{PotAvg:20},
some optimal edge cuts in a regular graphs
correspond to perfect $2$-colorings
\cite[Th.\,2.4]{Golubev:2020}.

Perfect colorings of $G(\ZZ^2)$ were studied in the following preceding papers.
The most important result was proved by Puzynina,
who showed in~\cite{Puz2004.en} that every such coloring is either periodic (in two independent directions),
or has so-called binary diagonals and can be made periodic by shifting some of them.
In~\cite{Axenovich:2003,Puz2005.en,Kro:9colors},
perfect colorings of~$G(\ZZ^2)$ up to, respectively, $2$, $3$, and $9$ colors are listed
(the last result is computational).
In~\cite{AVS:2012en}, there is a characterization of
\emph{completely regular codes} in $G(\ZZ^2)$, i.e., sets $C$ of vertices such that
the distance coloring with respect to $C$ is perfect.

Similar results on perfect colorings of other notable infinite
grids are worth mentioning. A partial analogs of the result of~\cite{Puz2004.en}
was obtained in \cite{Puz:2011en}, where it was proved that for every
perfect coloring of the hexagonal or triangular grid (of degree $3$ and $6$, respectively), there is a periodic (in two directions)
perfect coloring with the same quotient matrix. Completely regular codes
in the hexagonal or triangular grids were described in~\cite{AKV:hex}
and~\cite{Vas:Maltsev14}, \cite{Vas:grid:22} respectively. In~\cite{Plotnikov:2005},
perfect colorings of the infinite hexagonal grid with three colors were described.
Some results on completely regular codes in $n$-dimensional square grid
were obtained in~\cite{AvgVas:2022}. A systematic study of perfect colorings of infinite one-dimensional periodic graphs (finite-degree periodic graphs on $\ZZ$: infinite circulant graphs,
infinite multipath graphs, etc.) is carried out by Avgustinovich, Lisitsyna,
and Parshina, see one of the latest works~\cite{LAP:2020} and references there. That study is also related to perfect colorings of infinite grids
because some of those graphs are quotient graphs of one of the grids mentioned above.


\smallskip

Consider the set $\ZZ^2$ of pairs $[x,y]$ of integers $x$, $y$.
For convenience, the elements of~$\ZZ^2$ will be called \emph{nodes} and pictured as squares.
Pictures that show a fragment of~$\ZZ^2$ with colored nodes are very frequent in this paper
and formally can be considered as mathematical expressions, which can be a part of a sentence.
For example, the phrase ``a~coloring $\FfF$ contains the fragment
\begin{tikzpicture}
\dfc{Ub}{ffd800}
\dfc{Ua}{0057b8}
\def\1{\ccolor{Ua!60!white}{$1$}}
\def\2{\ccolor{Ub!70!white}{$2$}}
\draw (0,-0) \1\0\0\0\2;
\draw (0,-1) \0\2\0\1\0;
\end{tikzpicture}
'' means ``there are integers $x$ and $y$ such that
$\FfF\Lft x,y+1\Rgt =\FfF\Lft x+3,y\Rgt =\cll 1$ and
$\FfF\Lft x+1,y\Rgt =\FfF\Lft x+4,y+1\Rgt =\cll 2$''.
In such pictures the first coordinate $x$  grows in the right
(sometimes, right-down) direction, while the second, $y$,
grows in the upward (sometimes, up-right) direction.
If a picture is implied to show a coloring of the whole grid
(or of some of the nodes, for example only the even nodes), see, e.g.,
Fig.~\ref{AA}, \ref{BB}, \ldots, then the whole coloring is reconstructed from
the shown fragment by translations with periods
that are obvious from the picture.

A node $[x,y]$ is \emph{even} (\emph{odd}) iff the number $x+y$ is even (odd).
We define the \emph{distance}
between two nodes ${\bar v}=[x,y]$ and ${\bar v}'=[x',y']$
as
$d({\bar v},{\bar v}')\triangleq |x'-x|+|y'-y|$.
Two nodes at distance one from each other
are \emph{adjacent} or \emph{neighbors}.
In such a manner, an infinite graph $G(\ZZ^2)$ is defined, called the \emph{infinite square grid} (sometimes, infinite rectangular grid).
The \emph{neighborhood} of a node is the set of all (four) its neighbors.
We say that two nodes $[x,y]$ and $[x',y']$
are placed \emph{diagonally} from each other iff $|x'-x|=|y'-y|=1$.
For a node $[x,y]$ the set $\{[x+i,y+i] \mid i\in\ZZ\}$ is called the
\emph{R-diagonal}
through $[x,y]$;
the set $\{[x+i,y-i] \mid i\in\ZZ\}$ is called the
\emph{L-diagonal}
through $[x,y]$.
By definition, a \emph{diagonal} is
an R-diagonal (R is for ``right'')
or
an L-diagonal (L is for ``left'').
Two
R-diagonals
(L-diagonals)
through neighbor nodes are also called \emph{neighbor}.

Let $G=(V(G),E(G))$ be a simple graph;
let $ \Sigma=(\cll{1},\cll{2},\ldots,\cll{n})$ be some fixed collection
of distinct elements,
which will be referred to as \emph{colors};
and let $S=(S_{ij})_{ij}$ be an $n \times n$ nonnegative integer matrix.
A surjective function $\FfF:V(G)\to\Sigma$
is called a \emph{perfect coloring} (of~$G$) with (quotient) matrix~$S$
iff for every~$i$ and~$j$ from~$1$ to~$n$ every node $v\in V(G)$ such that
$\FfF(v)=\cll{i}$ has exactly~$S_{ij}$ neighbors colored with~$\cll{j}$.

The convenient term \emph{semicoloring}, suggested in~\cite{LAP:2020},
refers to any function from a part of a bipartite graph
to a finite set of colors.

Two colorings  $\FfF,\GgG:\ZZ^2\to\Sigma$ (or semicolorings)
are called \emph{equivalent}, $\FfF \sim \GgG$, iff
$\FfF(\cdot)\equiv \pi \GgG(\tau(\cdot))$
where $\pi$ is a permutation of the colors,
$\tau$ is an adjacency-preserving transformation of $\ZZ^2$
(an \emph{automorphism} of $G(\ZZ^2)$),
i.e., translation,
rotation, reflection, or a sliding symmetry
(the composition of a reflection and a translation).

Next, we define three important classes of
perfect colorings: bipartite perfect colorings, coverings,
and perfect colorings with twin colors.

A perfect coloring is called \emph{bipartite}
iff the set of nodes of each color consists of
nodes of the same parity (even or odd).
Every perfect coloring $\FfF$
(of $G(\ZZ^2)$ or any other bipartite graph)
can be treated as a bipartite perfect coloring:
if $\FfF$ is not bipartite itself, then each color can be split into the even subcolor and the odd subcolor,
resulting in a bipartite perfect coloring $\overline{\FfF}$.
Note that
 $\overline{\FfF} \sim \overline{\GgG}$ neither
 means that $\FfF$ and $\GgG$ are equivalent, nor that
 they have the same quotient matrix (up to permutation of colors).
 On the other hand, some bipartite perfect colorings do not correspond to
 any non-bipartite ones.

If the quotient matrix of a perfect coloring $\FfF$ of $G$
is a $\{0,1\}$-matrix with zero diagonal,
then it is the adjacency matrix of some graph $T$ on the set of colors.
In this case, $\FfF$ is known as a \emph{covering}
of the \emph{target} graph $T$ by the \emph{cover} graph $G$.

Assume that we have a perfect coloring with an $n$-by-$n$ quotient matrix~$S$.
Two different colors~$\cll{a}$ and~$\cll{b}$ are called \emph{twin}
(twin colors or just \emph{twins}) iff
identifying them results in a perfect coloring~$\FfF'$ with $n-1$ colors.
Equivalently, two different colors~$\cll{a}$ and~$\cll{b}$ are twin
if and only if $S_{aj}=S_{bj}$ for all $j\neq a,b$
(however, $S_{ja}$ and~$S_{jb}$ might be distinct,
and the ``densities'' of two twin colors in a perfect coloring can be different).

\begin{remark*}
In preceding papers, twin colors were called ``equivalent''.
However, equivalence is a general mathematical concept:
any reflexive, symmetric,
and transitive relation is an equivalence.
In this paper, a specific term ``twin'' is suggested for that reason.
\end{remark*}

In this paper, we (1) characterize coverings by~$G(\ZZ^2)$,
and (2) classify all bipartite perfect colorings of~$G(\ZZ^2)$
with twin colors.

In the next section,
some additional concepts and two main theorems are introduced.
Sections~\ref{s:p1} and~\ref{s:p2} contain proofs of the theorems.
In Section~\ref{s:no}, we briefly discuss the existence of non-orbit perfect colorings.
Appendix~A completes Theorem~\ref{th:01} by describing the subgroups of the automorphism
group of the infinite square grid whose orbit coloring is
a covering of a simple graph.
Appendix~B contains a catalogue of perfect colorings that,
together with Theorem~\ref{th:main},
form the classification of perfect colorings
of the infinite square grid with twin colors.


\section{Main results}

Assume that some perfect coloring $\FfF$ with $n \times n$ matrix
$S=(S_{ij})_{ij}$ is fixed.
Let $\cll{a}$ and~$\cll{b}$ be some colors, $C$ some set of colors.
We say that ${\bar v}\in \ZZ^2$ is an \emph{$\cll{a}$-node}
(or ${\bar v}$ is colored with \cll{a}, or ${\bar v}$ has color \cll{a})
iff $\FfF({\bar v})=\cll{a}$.
We say that ${\bar v}\in \ZZ^2$ is a \emph{$C$-node} iff $\FfF({\bar v})\in C$.

A diagonal (R-diagonal, L-diaginal) is called an $\dll ab$-diagonal
($\dll ab$-R-diagonal or $\dll ab$-L-diaginal, depending on its direction) iff
it consists of $\{\cll{a},\cll{b}\}$-nodes and, moreover,
the colors alternate on the diagonal, i.e.,
all its nodes of type $[2i,\cdot]$ ($i\in \ZZ$) are $\cll{a}$-nodes
and all its nodes of type $[2i+1,\cdot]$ ($i\in \ZZ$) are $\cll{b}$-nodes,
or vice versa.
An \dll{a}{b}-diagonal is called a
\emph{binary diagonal} iff $\cll{a}\neq\cll{b}$.
The following simple fact is crucial for the rest of the article,
and it will be used without explicit references.
\begin{claim*}[{\cite[Proposition~6]{Puz2004.en}}]
If there is an $\dll ab$-diagonal with $\cll a \ne \cll b$,
then $\cll a$ and $\cll b$ are twins.
\end{claim*}
\begin{proof}
 Assume without loss of generality that
 $$\FfF\Lft 0,0\Rgt =\FfF\Lft 2,2\Rgt =\FfF\Lft 4,4\Rgt =\cll a
 \quad\mbox{and}\quad
 \FfF\Lft 1,1\Rgt =\FfF\Lft 3,3\Rgt =\FfF\Lft 5,5\Rgt =\cll b.$$
 Denote by~$N_a$ and~$N_b$ the union of the neighborhoods of
 $[0,0]$, $[2,2]$, $[4,4]$ and
 $[1,1]$, $[3,3]$, $[5,5]$, respectively.
 For a color $\cll i$, the number of $\cll i$-nodes
 in $N_a$ is $3S_{ai}$,
 and the number of $\cll i$-nodes
 in $N_b$ is $3S_{bi}$. Since
 $|N_a\backslash N_b|=|N_b\backslash N_a|=2$,
 the difference between $3S_{ai}$ and $3S_{bi}$ cannot
 be larger than $2$. Hence, $S_{ai}=S_{bi}$.
\end{proof}

By \emph{shifting} a binary \dll{a}{b}-diagonal
we mean swapping the colors \cll{a} and \cll{b} on that diagonal.
Such operation does not change the perfectness of the coloring
of the whole grid. Two colorings or semicolorings are \emph{shifting equivalent}
iff one of them is equivalent to a coloring (respectively, semicoloring)
obtained from the other by shifting binary diagonals.

A coloring of a graph $G$ is called \emph{orbit}, if there is a subgroup $\Phi$
of the automorphism group $\mathrm{Aut}(G)$ of $G$
such that
\begin{itemize}
 \item
any two nodes ${\bar x}$ and ${\bar y}$ have the same color
if and only if ${\bar x}=\varphi({\bar y})$ for some $\varphi$ in $\Phi$
(i.e., ${\bar x}$ and ${\bar y}$ belong to the same orbit under the action of $\Phi$).
\end{itemize}
It is straightforward and well known that every orbit coloring is a perfect coloring.

We are now ready to state the main results of the paper.

\begin{theorem}\label{th:01}
 Every perfect coloring of $G(\ZZ^2)$
 with quotient $\{0,1\}$-matrix with zero diagonal
 (that is, the coloring is a covering of a simple graph)
 is either orbit
 or has a binary diagonal (and, hence, twin colors).
\end{theorem}

Two corollaries (\ref{c:bipartite01} and \ref{c:nonbipartite01}) of the next theorem, of independent interest, complete Theorem~\ref{th:01}
by characterizing coverings with binary diagonals.

\begin{theorem}\label{th:main}
Every bipartite perfect coloring of $G(\ZZ^2)$
with at least one pair of twin colors is either
shifting equivalent
to one of the colorings listed in Fig.~\ref{AA}--\ref{WW}
with the corresponding matrices (see Appendix~B) or can be obtained
from such a coloring by merging groups
of (two, three, or four) mutually twin colors.
\end{theorem}
A special case in the claim of Theorem~\ref{th:main}
should be highlighted because of the risk to be forgotten.
The only case that satisfies the hypothesis of the theorem
with twin colors of different parity is the
\emph{chessboard coloring},
obtained from the coloring in Fig.~\ref{II}
by merging two groups (even and odd)
from four mutually twin colors.

\begin{remark}
Each of the infinite families of colorings shown in Fig.~\ref{KK}--\ref{WW}
has one or two small-parameter cases that are included
in other, finite, families.
In each of such cases, the coloring has binary diagonals of both directions,
in contrast to the other colorings of the corresponding infinite family.
For example the coloring in Fig.~\ref{MM}($n=6$) is equivalent to a coloring
in Fig.~\ref{II} with shifted L-diagonals and some merged colors.
That coloring has an additional pair (\cll 3, \cll 4) of twin colors
and binary \dll34-L-diagonals, which are not illustrated
in the general diagram (Fig.~\ref{MM}).
By this reason, one can find natural to exclude such special cases
from the infinite series (which keeps Theorem~\ref{th:main} complete).
To make this easy, such special cases
 are explicitly indicated by the reference
to a finite family in the parenthesis after the corresponding
number of the colors, see the description to each of Fig.~\ref{KK}--\ref{WW}.
\end{remark}

\begin{remark}\label{r:intersect}
The coloring shown in Fig.~\ref{AA} has intersecting binary
diagonals,
which cannot be shifted independently.
Actually, any shifting
R-diagonals,
except shifting \emph{all}
R-diagonals
(which has the same effect as renaming colors $\cll4\leftrightarrow\cll5$)
results in a coloring without binary
L-diagonals.
Since the original coloring has an automorphism
(horizontal reflection plus swapping colors $\cll5\leftrightarrow\cll6$)
that changes the roles of
R-diagonals
and
L-diagonals,
it is safe to say that we can use only
R-diagonal
shifts
and ordinary equivalence transformations
(automorphisms of the grid and renaming colors)
to exhaust the shifting equivalence class of the coloring shown in Fig.~\ref{AA}.

A similar situation is with Fig.~\ref{II}. However,
in that case there are
L-diagonals
and
R-diagonals
of both parities.
Shifting some (but not all)
\dll12-diagonals or some (but not all)
\dll34-diagonals breaks all
L-diagonals
of the same parity,
but does not affect the diagonals of the opposite
parity.
Similar ``nontrivial'' shifting
\dll13- and \dll24-diagonals breaks all
R-diagonals
of the same parity.
Similarly, for \dll57- and \dll68-diagonals
and for  \dll56- and \dll78-diagonals.
It is not difficult to conclude that the shifting equivalence class
is the union of two families.
One family is obtained by shifting only diagonals of the same direction
(say, R-diagonals);
the second family is obtained by shifting only even diagonal
of the same direction
(say, R-diagonals)
and odd diagonals
of the opposite direction
(respectively, L-diagonals).
The two families intersect in the colorings equivalent
to the one shown in Fig.~\ref{II}.

If Fig.~\ref{JJ}, both
binary R-diagonals
and
binary L-diagonals
can be found, but they do not intersect and can
be shifted independently.
\end{remark}

\begin{corollary}\label{c:bipartite01}
Every bipartite perfect colorings of $G(\ZZ^2)$
with quotient $\{0,1\}$-matrix and twin colors
is shifting equivalent to a coloring shown
in Fig.~\ref{II}, Fig.~\ref{KK}, or  Fig.~\ref{TT}.
All such colorings have binary diagonals.
\end{corollary}

Any non-bipartite perfect coloring $\FfF$ of $G(\ZZ^2)$
with equal rows in the quotient matrix
corresponds to a bipartite perfect coloring $\overline{\FfF}$
(with the number of colors twice larger than in $\FfF$)
with twin colors. Using this fact and Theorem~\ref{th:main},
we can derive the following.
\begin{corollary}\label{c:nonbipartite01}
Every non-bipartite perfect colorings of $G(\ZZ^2)$
whose quotient matrix is a $\{0,1\}$-matrix
with zero diagonal and two equal rows
is shown in Fig.~\ref{KK},
where $n=6,10,14,\ldots$ in the non-bipartite case.
All such colorings have
binary diagonals.
\end{corollary}

\begin{corollary}
There are two quotient matrices admitting perfect colorings
of $G(\ZZ^2)$ both with binary diagonals and without binary diagonals.
These matrices are shown in
Fig.~\ref{CC} (Fig.~\ref{JJ})
and
Fig.~\ref{EE} (Fig.~\ref{WW}, $n=10$).
\end{corollary}

\section{Proof of Theorem~\ref{th:main}: cases and subcases}\label{s:p1}

Let us consider a bipartite perfect coloring $\FfF$ with twin colors $\cll{1}$ and $\cll{2}$.
We start the proof with a very simple observation.
\begin{claim}\label{cl:check}
 Either $\FfF$ is the {chessboard coloring}
 (i.e., obtained from the coloring in Fig.~\ref{II}
 by merging the groups $\{\cll1,\cll2,\cll3,\cll4\}$
 and
  $\{\cll5,\cll6,\cll7,\cll8\}$ of twin colors)
  or all the $\{\cll 1,\cll2\}$-nodes are of the same parity.
\end{claim}
\begin{proof}
 Since $\FfF$ is bipartite,
 there are no two neighbor nodes of the same color.
 Therefore, either \cll1-nodes have only \cll2-neighbors
 and \cll2-nodes have only \cll1-neighbors
 (hence the coloring is chessboard),
 or both \cll1-nodes and \cll2-nodes have neighbors
 of some third color, say \cll3.
 If the \cll3-nodes are even then
 the $\{\cll 1,\cll2\}$-nodes are odd and vice versa.
\end{proof}

Since the chessboard coloring satisfies the conclusion
of Theorem~\ref{th:main},
we further assume w.l.o.g.\
that the $\{\cll{1},\cll{2}\}$-nodes are even.
We will say that a node is of type $k{:}l$ iff it is adjacent to exactly
$k$ $\cll{1}$-nodes and $l$ $\cll{2}$-nodes.
Clearly, nodes of the same color are of the same type.
Moreover, nodes of the same color are adjacent to the same number of nodes of some fixed type.

From the definition of twin colors we directly get the following fact.
A node is adjacent to a $\cll{1}$-node if and only if it is adjacent to a $\cll{2}$-node.
To put it differently, there are no nodes of type $0{:}k$ or $k{:}0$, where $k>0$.
On the other hand, the number of $\cll{1}$-nodes and the number of $\cll{2}$-nodes
that are adjacent to a given node
can be different.

We divide the situation into the following cases.

\begin{enumerate}
\renewcommand{\theenumi}{\Roman{enumi}}
    \item
There exists a node of type $3{:}1$ or $1{:}3$.
   W.l.o.g., we can consider only the first case.
    \item
There exists a node of type $2{:}1$ or $1{:}2$.
   W.l.o.g., we can consider only the first case.
    \item
All the neighbors of a $\{\cll{1},\cll{2}\}$-node are of type $2{:}2$.
    \item
Every $\{\cll{1},\cll{2}\}$-node has two neighbors of type $2{:}2$
and two neighbors of type $1{:}1$.
    \item
Every $\{\cll{1},\cll{2}\}$-node has one neighbor of type $2{:}2$ and three neighbors of type $1{:}1$.
    \item
All the neighbors of a $\{\cll{1},\cll{2}\}$-node are of type $1{:}1$.
\end{enumerate}

It is easy to see that if a node has three neighbors of type $2{:}2$ then
the fourth neighbor can not be of type $1{:}1$.
So, the completeness of the case list is obvious.
Before starting to prove Theorem~\ref{th:main} for each of the cases,
we introduce one convenient tool related with the number of paths of special color types
in a perfect coloring.

\def\dd{d}
Assume that $\FfF$ is a perfect coloring of $G(\ZZ^2)$ (or any finite-degree
graph $G$) with quotient $n\times n$ matrix
$S=(S_{ij})_{i,j\in\{1,\ldots,n\}}$.
For a node ${\bar v}_0$ of color $\cll{b}$, a positive integer $k$,
and colors $\Cll{b_1},\ldots,\Cll{b_k}$
denote by $\dd_{{\bar v}_0}(\cll{b},\Cll{b_1},\ldots,\Cll{b_k})$ the number of $k$-tuples
$({\bar v}_1,\ldots,{\bar v}_k)$ such that ${\bar v}_{i-1}$ and ${\bar v}_i$ are adjacent in $G$ and
$\FfF({\bar v}_i)=\Cll{b_i}$ for every $i\in\{1,\ldots,k\}$;
denote by $\dd_{{\bar v}_0}^k(\cll{b},\Cll{b_k})$ the number of $k$-tuples
$({\bar v}_1,\ldots,{\bar v}_k)$ such that ${\bar v}_{i-1}$ and ${\bar v}_i$ are adjacent in $G$
for every $i\in\{1,\ldots,k\}$ and $\FfF({\bar v}_k)=\Cll{b_k}$.
The following straightforward statement is well known.

\begin{lemma}\label{l:dist}
The value
$\dd_{\bar v}^k(\cll{b},\Cll{b_k})$
does not depend on the choice
of the $\cll{b}$-node
 ${\bar v}$.
\end{lemma}
\begin{proof}
It is clear that
$$\displaystyle
\dd_{\bar v}^k(\cll{b},\Cll{b_k})=\sum_{b_1=1}^n\ldots\sum_{b_{k-1}=1}^n \dd_{\bar v}(\cll{b},\Cll{b_1},\ldots,\Cll{b_k}).
$$
By induction, $\dd_{\bar v}(\cll{b},\Cll{b_1},\ldots,\Cll{b_k})$ equals
$S_{\scll{b},\scll{b_1}}S_{\scll{b_1},\scll{b_2}}\dots S_{\scll{b_{k-1}},\scll{b_k}}$
and does not depend on the choice of the $\cll{b}$-node ${\bar v}$.
(The induction base and step are straightforward from the definition
of perfect coloring.)
\end{proof}

As $\dd_{\bar v}^k(\cll{b},\Cll{b_k})$ does not depend on ${\bar v}$,
we will use the notation $\dd^k(\cll{b},\Cll{b_k})$ instead.

We will also need the following fact.
\begin{lemma}[{\cite[Claim~9]{Puz2005.en}}]\label{p:freq}\label{l:freq}\label{l:friq}
  There are positive rational numbers
  $P_{\scll{1}}$, \ldots, $P_{\scll{n}}$
  (the ``densities'' of the corresponding colors)
  such that
  $\sum_{\scll{i}=\scll{1}}^{\scll{n}}P_{\scll{i}}=1$ and
  $S_{\scll{i}\scll{j}}P_{\scll{i}} = S_{\scll{j}\scll{i}}P_{\scll{j}}$
  for all $ i, j\in\{ 1,\ldots, n\}$.
\end{lemma}

\begin{remark}\label{r:density}
 The density $P_{\scll{i}}$
 equals the limit of the portion
 of the $\cll i$-nodes in a square of growing
 size.
\end{remark}


\subsection[There exists a node of type 3:1]{There exists a node of type $3{:}1$}\label{case1}
We assume that a node of type $3{:}1$ has color $\cll{3}$.
\begin{claim}\label{cl1} Every node is of type $0{:}0$ or of type $3{:}1$.
\end{claim}
\begin{proof}
By Lemma~\ref{p:freq},
for any color $\cll{i}$ we have
\[S_{1,3}S_{3,2}S_{2,i}S_{i,1}=S_{3,1}S_{2,3}S_{i,2}S_{1,i}.\]
Since $S_{1,3}=S_{2,3}$, $S_{2,i}=S_{1,i}$, $S_{3,1}=3$, and $S_{3,2}=1$,
we get $S_{i,1}=3S_{i,2}$. Hence,
either $S_{i,1}=S_{i,2}=0$, or $S_{i,1}=3$ and $S_{i,2}=1$.
\end{proof}

It is easy to see that all even nodes are $\{\cll{1},\cll{2}\}$-nodes.
Note that there are only $\cll{1}$-nodes at the distance $2$ from any $\cll{2}$-node.
Consider two possibilities.
\begin{enumerate}
 \item [\ref{case1}.a)]
 $\cll{2}$-nodes are placed periodically
                 with periods $[2,2]$ and $[2,-2]$.
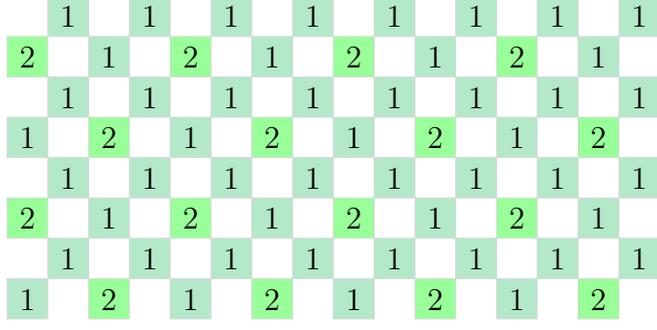
\begin{center}
\begin{tikzpicture}
\dfc{Ub}{ffd800}\def\clUb{Ub!50!white}
\dfc{Ua}{0057b8}\def\clUa{Ua!45!white}
\def\2{\ccolor{\clUa}{$2$}}
\def\1{\ccolor{\clUb}{$1$}}
\draw (0,0) \1\0\2\0\1\0\2\0\1\0\2\0\1\0\2\0 ;
\draw (0,1) \0\1\0\1\0\1\0\1\0\1\0\1\0\1\0\1 ;
\draw (0,2) \2\0\1\0\2\0\1\0\2\0\1\0\2\0\1\0 ;
\draw (0,3) \0\1\0\1\0\1\0\1\0\1\0\1\0\1\0\1 ;
\draw (0,4) \1\0\2\0\1\0\2\0\1\0\2\0\1\0\2\0 ;
\draw (0,5) \0\1\0\1\0\1\0\1\0\1\0\1\0\1\0\1 ;
\draw (0,6) \2\0\1\0\2\0\1\0\2\0\1\0\2\0\1\0 ;
\draw (0,7) \0\1\0\1\0\1\0\1\0\1\0\1\0\1\0\1 ;
\end{tikzpicture}
\end{center}

\item[\ref{case1}.b)]
In some $\cll{2}$-node ${\bar v}$ one of the periods is broken,
i.e., $\FfF({\bar v}+{\bar d})=\cll{1}$, where
\begin{enumerate}
\item[\ref{case1}.b')]
${\bar d}=[2,2]$, or
${\bar d}=[-2,-2]$, or
\item[\ref{case1}.b'')]
${\bar d}=[2,-2]$, or
${\bar d}=[-2,2]$.
\end{enumerate}
\end{enumerate}
Without loss of generality, we assume ${\bar v}=[0,0]$ and ${\bar d}=[2,2]$.
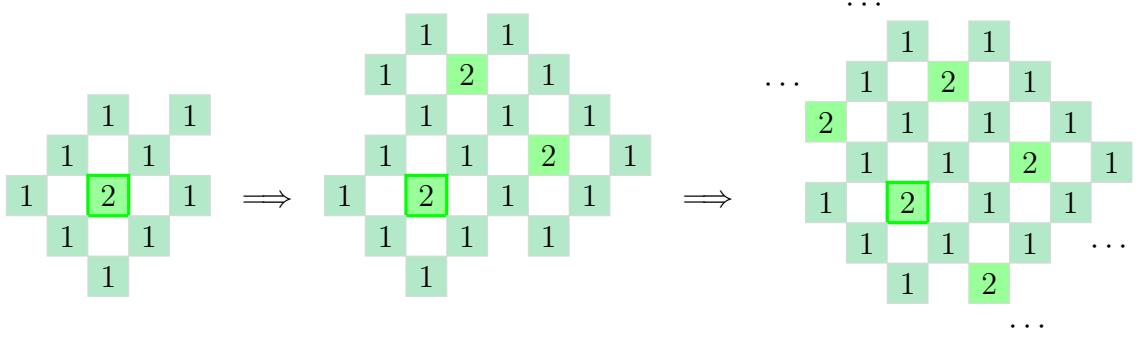
\begin{center}
\dfc{Ub}{ffd800}\def\clUb{Ub!50!white}
\dfc{Ua}{0057b8}\def\clUa{Ua!45!white}
\def\2{\ccolor{\clUa}{$2$}}
\def\1{\ccolor{\clUb}{$1$}}
\def\B{ ++(1,0) node [cell,fill=\clUa,draw=blue!50!white,very thick]{\raisebox{-0.25em}[0.2em][0em]{\makebox[0mm][c]{$2$}}} }
\begin{tikzpicture}
\draw (0,-0) \0\0\1\0\1 ;
\draw (0,-1) \0\1\0\1 ;
\draw (0,-2) \1\0\B\0\1 ;
\draw (0,-3) \0\1\0\1 ;
\draw (0,-4) \0\0\1 ;
\phantom{\draw (0,-5) \1 ;}
\end{tikzpicture}
\ \ \raisebox{4.2em}{$\Longrightarrow$}\
\begin{tikzpicture}
\draw (0, 2) \0\0\1\0\1 ;
\draw (0, 1) \0\1\0\2\0\1 ;
\draw (0,-0) \0\0\1\0\1\0\1 ;
\draw (0,-1) \0\1\0\1\0\2\0\1 ;
\draw (0,-2) \1\0\B\0\1\0\1 ;
\draw (0,-3) \0\1\0\1\0\1 ;
\draw (0,-4) \0\0\1 ;
\phantom{\draw (0,-5) \1 ;}
\end{tikzpicture}
\ \ \raisebox{4.2em}{$\Longrightarrow$}
\begin{tikzpicture}
\draw (0, 3) \0\0\td\0\0\0 ;
\draw (0, 2) \0\0\0\1\0\1 ;
\draw (0, 1) \td\0\1\0\2\0\1 ;
\draw (0,-0) \0\2\0\1\0\1\0\1 ;
\draw (0,-1) \0\0\1\0\1\0\2\0\1 ;
\draw (0,-2) \0\1\0\B\0\1\0\1 ;
\draw (0,-3) \0\0\1\0\1\0\1\0\td ;
\draw (0,-4) \0\0\0\1\0\2 ;
\draw (0,-5) \0\0\0\0\0\0\td ;
\end{tikzpicture}
\end{center}
From Claim~\ref{cl1} we get that $[1,3]$ and $[3,1]$ are $2$-nodes; then,
$[-2,2]$ and $[2,-2]$ are $\cll{2}$-nodes, and, by induction,
$k[-2,2]$ and $k[-2,2]+[3,1]$ are $\cll{2}$-nodes for any integer $k$.
Further, if $[4,4]$ is a $\cll{2}$-node, then in a similar manner
we have that $k[-2,2]+[4,4]$ is a $\cll{2}$-node for any $k$;
otherwise the nodes $k[-2,2]+[4,4]+[-1,1]$ are $\cll{2}$-nodes.
In both cases we have one more
$\dll{1}{2}$-L-diagonal.
In a similar manner, by induction, we get that for each integer $n$
either the nodes $k[-2,2]+n[2,2]$
or the nodes $k[-2,2]+[-1,1]+n[2,2]$
are $\cll{2}$-nodes.
So, we can conclude the following.

\begin{proposition}\label{p:1}
A coloring corresponding to
subcase~{\rm\ref{case1}.b}
can be obtained from a coloring
corresponding to subcase~{\rm\ref{case1}.a}
by shifting
 $\dll{1}{2}$-L-diagonals
or
by shifting
 $\dll{1}{2}$-R-diagonals {(but not both)}.
\end{proposition}
So, the colors of the even nodes are described.
Every odd node is adjacent to a $\cll 1$-node,
and hence,
the quotient matrix $S$ is completely determined
by the colors in the neighborhood of a $\cll 1$-node.
Up to renaming the colors, $S$ is one of the following:
{\footnotesize
$${
\left(%
\begin{array}{ccc}
  0 & 0 & 4 \\
  0 & 0 & 4 \\
  3 & 1 & 0 \\
\end{array}%
\right),
\left(%
\begin{array}{cccc}
  0 & 0 & 3 & 1 \\
  0 & 0 & 3 & 1 \\
  3 & 1 & 0 & 0 \\
  3 & 1 & 0 & 0 \\
\end{array}%
\right),
\left(%
\begin{array}{cccc}
  0 & 0 & 2 & 2 \\
  0 & 0 & 2 & 2 \\
  3 & 1 & 0 & 0 \\
  3 & 1 & 0 & 0 \\
\end{array}%
\right),
\left(%
\begin{array}{ccccc}
  0 & 0 & 2 & 1 & 1 \\
  0 & 0 & 2 & 1 & 1 \\
  3 & 1 & 0 & 0 & 0 \\
  3 & 1 & 0 & 0 & 0 \\
  3 & 1 & 0 & 0 & 0 \\
\end{array}%
\right),
\left(%
\begin{array}{cccccc}
  0 & 0 & 1 & 1 & 1 & 1 \\
  0 & 0 & 1 & 1 & 1 & 1 \\
  3 & 1 & 0 & 0 & 0 & 0 \\
  3 & 1 & 0 & 0 & 0 & 0 \\
  3 & 1 & 0 & 0 & 0 & 0 \\
  3 & 1 & 0 & 0 & 0 & 0 \\
\end{array}%
\right).
}
$$}
If $S$ is the first matrix,
then the odd nodes are colored with one color.
If the second one,
then we can change the roles of the even and odd nodes
to see that the colors of the odd nodes are also described by
Proposition~\ref{p:1}
(with colors $\cll 3$ and $\cll 4$ instead of $\cll 1$ and $\cll 2$).
For the last three matrices, we also see that every two odd colors
are twin, but after changing the roles
of the even and odd nodes we get one of cases that will be considered
later, namely, in cases II, III, and VI.
To help structuring the proof, we will state the following lemma
as the summary of these cases.
The corollary of the lemma finalizes the proof
of Theorem~\ref{th:main} for
cases~I, II, III, and~VI.a.

\begin{lemma}\label{l:even}
Assume that every two nodes of the same parity are colored with the same or twin colors.
Then the coloring of the even (similarly, odd) nodes is
 equivalent
 to a semicoloring obtained from the following semicoloring $\HhH$
 by shifting
 R-diagonals
 and/or
 merging colors (two, three, or all four colors, or two groups of two colors):
$$
\definecolor{yell}{rgb}{0.95,0.9,0.45}
\definecolor{cSIX}{rgb}{0.92, 0.99, 0.70}
\definecolor{pSIX}{rgb}{0.88, 0.94, 0.60}
\begin{tikzpicture}
\clip (-0.51,0.51) -- ++(12.02,0) -- ++(0,-8.02) -- ++(-12.02,0) -- cycle;
 diag/.style={draw=black!30!white, thin, dash pattern=on 0.5em off 0.8385em on 0.5em off 0em}
\def\1{\ccolor{\colONE!60!white}{$1$}}
\def\2{\ccolor{\colTWO!60!white}{$2$}}
\def\3{\ccolor{yell!60!white          }{$3$}}
\def\4{\ccolor{cSIX }{$4$}}
\def\a{++(1,0)}
\draw (-1,-0) \3\a\2\a\3\a\2\a\3\a\2\a ;
\draw (-1,-1) \a\1\a\4\a\1\a\4\a\1\a\4 ;
\draw (-1,-2) \2\a\3\a\2\a\3\a\2\a\3\a ;
\draw (-1,-3) \a\4\a\1\a\4\a\1\a\4\a\1 ;
\draw (-1,-4) \3\a\2\a\3\a\2\a\3\a\2\a ;
\draw (-1,-5) \a\1\a\4\a\1\a\4\a\1\a\4 ;
\draw (-1,-6) \2\a\3\a\2\a\3\a\2\a\3\a ;
\draw (-1,-7) \a\4\a\1\a\4\a\1\a\4\a\1 ;
\draw [rotate=-45, draw=black!30!white, thin, step=\DK,
      dash pattern=on 0\DK off 0.28\DK on 0.44\DK off 0.28\DK]
(-1.4140,-5.6562) grid (13.5,11);
\end{tikzpicture}
\ \ \raisebox{4.6em}{$\stackrel{\displaystyle\circlearrowright}{\longleftrightarrow }$} \!\!\!\!\!\!\!
\def\0{} 
\def\o{\dcolor{white}{$\ $}}
\def\1{\Dcolor{\colONE}{white}{$1$}}
\def\2{\Dcolor{\colTWO}{white}{$2$}}
\def\a{\Dcolor{yell        }{white}{$3$}}
\def\b{\Dcolor{pSIX}{cSIX!30!white}{$4$}}
\def\c{\0}
\def\d{\0}
\def\e{\0}
\def\f{\0}
\def\t{} 
\begin{tikzpicture}
\draw (-0\DK, 4\DK)      \t\a\b\a\b\a\b\a\b\a;
\draw (-0.5\DK, 3.5\DK) \0\c\d\c\d\c\d\c\d\c\d;
\draw (-0\DK, 3\DK)      \t\1\2\1\2\1\2\1\2\1;
\draw (-0.5\DK, 2.5\DK) \0\e\f\e\f\e\f\e\f\e\f;
\draw (-0\DK, 2\DK)      \t\a\b\a\b\a\b\a\b\a;
\draw (-0.5\DK, 1.5\DK) \0\c\d\c\d\c\d\c\d\c\d;
\draw (-0\DK, 1\DK)      \t\1\2\1\2\1\2\1\2\1;
\draw (-0.5\DK, 0.5\DK) \0\e\f\e\f\e\f\e\f\e\f;
\draw (-0\DK, 0\DK)      \t\a\b\a\b\a\b\a\b\a;
\draw (-0.5\DK, -.5\DK) \0\c\d\c\d\c\d\c\d\c\d;
\draw (-0\DK,-1\DK)      \t\1\2\1\2\1\2\1\2\1;
\begin{scope}
\clip [](0.5\DK,-1.5\DK) rectangle +(9.0\DK,6\DK);
\end{scope}
\end{tikzpicture}
$$
\end{lemma}

\begin{proof}
If there is only one even color, then the claim is trivial.
If there are at least two even colors, $\cll 1$ and $\cll 2$,
then they are twin and all odd nodes are of the same type
by the hypothesis of the lemma.
Depending on the type, $3{:}1$, $2{:}1$, $2{:}2$, or $1{:}1$,
the situation falls into the scope
of cases~\ref{case1}, \ref{case2}, \ref{case3}, \ref{case6}, respectively.

The claim of the lemma is proved in
Proposition~\ref{p:1} for case~\ref{case1},
Proposition~\ref{p:2}(a)
for case~\ref{case2},
Proposition~\ref{p:3}
for case~\ref{case3},
Propositions~\ref{p:6a}(ii,iii) and~\ref{p:6d1}(a)
for subcases~\ref{case6a} and~\ref{case6d} of case~\ref{case6},
while in the other subcases,
the claim of Theorem~\ref{th:main} is
proved directly, which also means that either
the claim of the lemma is true (Fig.~\ref{II})
or its hypothesis is not satisfied
(Fig.~\ref{AA}--\ref{HH},\,\ref{JJ}--\ref{WW}).
\end{proof}
Applying Lemma~\ref{l:even} to both even and odd nodes, we get
\begin{corollary}\label{c:evenodd}
Under the hypothesis of Lemma~\ref{l:even},
 the conclusion of Theorem~\ref{th:main} holds
  with the corresponding Fig.~\ref{II}.
\end{corollary}

\subsection[There exists a node of type 2:1]{There exists a node of type $2{:}1$}\label{case2}
Similarly to Claim~\ref{cl1}, we have
\begin{claim}\label{cl2}
Every node has one of the types $0{:}0$, $2{:}1$.
\end{claim}
It is not difficult to see that
the type $0{:}0$ nodes are the even nodes and all the odd nodes have the type $2{:}1$ (indeed, if, seeking a contradiction, we assume that an odd node~${\bar u}$ is of type $0{:}0$, then ${\bar u}+[\pm1,\pm1]$ are also type $0{:}0$;
by induction, all odd nodes are type $0{:}0$).
Consider a new coloring~$\GgG$ with
\begin{itemize}
 \item $\GgG({\bar v})=\FfF({\bar v}) \mbox{ if } \FfF({\bar v})\in\{\cll{1},\cll{2}\},$
 \item
$\GgG({\bar v})=\cll{4} \mbox{ if } {\bar v} \mbox{ is odd,}$
\item
$\GgG({\bar v})=\cll{3} \mbox{ otherwise.}$
\end{itemize}
By the definition, $\GgG$ is a bipartite perfect coloring with the matrix
\[
\left(%
\begin{array}{cccc}
  0 & 0 & 0 & 4 \\
  0 & 0 & 0 & 4 \\
  0 & 0 & 0 & 4 \\
  2 & 1 & 1 & 0 \\
\end{array}%
\right).
\]
Joining together twin colors $\cll{1}$ and $\cll{3}$
(or $\cll{1}$ and $\cll{2}$),
we get the situation of  case~{\rm\ref{case1}}.
Hence, the location of the $\cll{2}$-nodes
(respectively, $\cll{3}$-nodes) is described by subcases~{\rm\ref{case1}.a} and~{\rm\ref{case1}.b}.
\begin{claim}\label{cl:b'b''}
The coloring $\GgG$ has a period $(2,2)$ or $(2,-2)$.
\end{claim}
\begin{proof}
 If the set of $\cll{2}$-nodes
 corresponds to subcases~{\rm\ref{case1}.a},
 then it has both periods  $(2,2)$ and $(2,-2)$;
 if it corresponds to subcases~{\rm\ref{case1}.b'},
 then it has period $(2,-2)$;
 if it corresponds to subcases~{\rm\ref{case1}.b''},
 then it has period $(2,2)$.
 The same can be stated for the set of $\cll{3}$-nodes.
 It remains to show that
 subcases~{\rm\ref{case1}.b'} and~{\rm\ref{case1}.b''}
 cannot happen simultaneously
 for the $\cll{2}$-nodes and the $\cll{3}$-nodes
 (or, similarly, for the $\cll{3}$-nodes
 and the $\cll{2}$-nodes) respectively.
 Seeking a contradiction, assume that
 the $\cll{2}$-nodes correspond
 to subcases~{\rm\ref{case1}.b'}
 and
 the $\cll{3}$-nodes correspond
 to subcases~{\rm\ref{case1}.b''}.
 This means that for some even $\bar v=[x,y]$,
 \begin{equation}\label{eq:ii}
 \GgG\Lft x+2i,y-2i\Rgt
 = \GgG\Lft x+3+2i,y+1-2i\Rgt
 = \cll{2},\quad \mbox{for all $i\in\ZZ$},
 \end{equation}
 and, similarly, for some even $[x',y']$,
 \begin{equation}\label{eq:jj}
\GgG\Lft x'+2j,y'+2j\Rgt
= \GgG\Lft x'+3+2j,y'-1+2j\Rgt
= \cll{3},\quad \mbox{for all $j\in\ZZ$}.
 \end{equation}
Now, at least one of the four pairs
$\left(\frac{y-x-y'+x'  }4, \frac{y+x-y'-x'  }4 \right)$,
$\left(\frac{y-x-y'+x'-2}4, \frac{y+x-y'-x'+4}4 \right)$,
$\left(\frac{y-x-y'+x'+4}4, \frac{y+x-y'-x'-2}4 \right)$,
$\left(\frac{y-x-y'+x'+2}4, \frac{y+x-y'-x'+2}4 \right)$
is integer; denote it $(i_0,j_0)$.
Substituting $i=i_0$ in \eqref{eq:ii}
and $j=j_0$ in \eqref{eq:jj},
we find that the same node has color $\cll 2$ and $\cll 3$
(for example, if $(i,j)$ is equal to the first pair,
then $[x+2i,y-2i]=[x'+2j,y'+2j]$).
The contradiction obtained proves the claim.
\end{proof}

{%
See Proposition~\ref{p:1} for
the placement of nodes with colors~\cll2 and~\cll3.
Moreover, if we assume that $\GgG$ has period $[2,-2]$,
then we find that up to equivalence,
the coloring~$\GgG$ can be obtained
by shifting L-diagonals
from one of the following two colorings:
}
\begin{center}
\begin{tikzpicture}
\def\4{\ccolor{white}{$4$}}
\draw (0,-0) \3\4\2\4\3\4\2\4\3\4\2\4 ++(2,0) \1\4\2\4\1\4\2\4\1\4\2\4;
\draw (0,-1) \4\1\4\1\4\1\4\1\4\1\4\1 ++(2,0) \4\1\4\3\4\1\4\3\4\1\4\3;
\draw (0,-2) \2\4\3\4\2\4\3\4\2\4\3\4 ++(2,0) \2\4\1\4\2\4\1\4\2\4\1\4;
\draw (0,-3) \4\1\4\1\4\1\4\1\4\1\4\1 ++(2,0) \4\3\4\1\4\3\4\1\4\3\4\1;
\draw (0,-4) \3\4\2\4\3\4\2\4\3\4\2\4 ++(2,0) \1\4\2\4\1\4\2\4\1\4\2\4;
\draw (0,-5) \4\1\4\1\4\1\4\1\4\1\4\1 ++(2,0) \4\1\4\3\4\1\4\3\4\1\4\3;
\draw (0,-6) \2\4\3\4\2\4\3\4\2\4\3\4 ++(2,0) \2\4\1\4\2\4\1\4\2\4\1\4;
\draw (0,-7) \4\1\4\1\4\1\4\1\4\1\4\1 ++(2,0) \4\3\4\1\4\3\4\1\4\3\4\1;
\end{tikzpicture}
\end{center}

\begin{claim}\label{cl:3ab}
One of the following assertions is true.\\
{\rm a)} The color $\cll{3}$ of the coloring $\GgG$ corresponds to one color of the coloring $\FfF$.\\
{\rm b)} The color $\cll{3}$ of the coloring $\GgG$ corresponds to two colors of the coloring $\FfF$
and the corresponding nodes are placed periodically with periods $[0,4]$ and $[4,0]$.
\end{claim}
\begin{proof}
1.
Assume the placement of $\cll{3}$-nodes corresponds
to case~{\rm\ref{case1}.b}.
Then the coloring $\GgG$ contains (up to reflection) the fragment
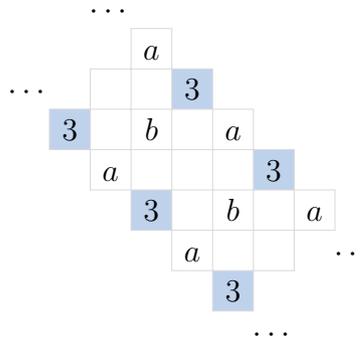
\begin{center}
\begin{tikzpicture}
\def\a{\ccolor{white}{$a$}}
\def\b{\ccolor{white}{$b$}}
\def\0{\ccolor{white}{$\ $}}
\draw (2, 3) \td ;
\draw (3, 2) \a ;
\draw (0, 1) \td ++(1,0) \0 ++(1,0) \3 ;
\draw (1,-0) \3 ++(1,0) \b ++(1,0) \a ;
\draw (2,-1) \a ++(1,0) \0 ++(1,0) \3 ;
\draw (3,-2) \3 ++(1,0) \b ++(1,0) \a ;
\draw (4,-3) \a ++(1,0) \0 ++(1,0) \td ;
\draw (5,-4) \3 ;
\draw (6,-5) \td ;
\end{tikzpicture}
\end{center}

1.1. Consider the values $\dd^2(\cll{a},\cll{x})$
for the coloring $\FfF$ and each color $\cll{x}$
that corresponds to the color $\cll{3}$ of $\GgG$.
Using Lemma~\ref{l:dist} we get that
there are not more than two such colors (say, $\cll{3}$ and $\cll{3'}$),
moreover, their nodes alternate with period $[4,-4]$.
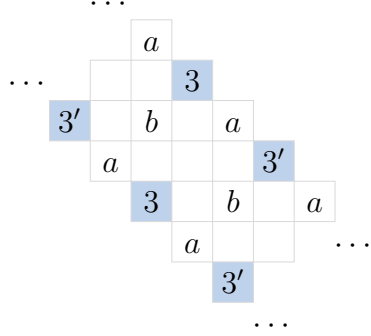
\begin{center}
\begin{tikzpicture}
\def\a{\ccolor{white}{$a$}}
\def\b{\ccolor{white}{$b$}}
\def\0{\ccolor{white}{$\ $}}
\def\9{\ccolor{blue!70!green!25!white}{$3'$}}
\draw (2, 3) \td;
\draw (3, 2) \a ;
\draw (0, 1) \td++(1,0) \0 ++(1,0) \3 ;
\draw (1,-0) \9 ++(1,0) \b ++(1,0) \a ;
\draw (2,-1) \a ++(1,0) \0 ++(1,0) \9 ;
\draw (3,-2) \3 ++(1,0) \b ++(1,0) \a ;
\draw (4,-3) \a ++(1,0) \0 ++(1,0) \td;
\draw (5,-4) \9 ;
\draw (6,-5) \td;
\end{tikzpicture}
\end{center}

1.2. Considering the values $\dd^2(\cll{b},\cll{3'})$
for two $\cll{b}$-nodes we get that $\cll{3}$ and $\cll{3'}$
are the same color.

2. Assume the placement of $\cll{3}$-nodes corresponds to the case {\rm\ref{case1}.a}.
Then by arguments similar to p.\,1.1 we get that one of the statements a), b) holds.
\end{proof}

Now we can conclude
{(with a similar argument than the one given
in the proof of Claim~\ref{cl:3ab})}
that the following proposition is true.
\begin{proposition}\label{p:2}
In case {\rm\ref{case2}}, $\FfF$ is equivalent to a coloring
obtained by shifting
L-diagonals
from a colorings of the following
three types, where the colors of any two ${a}$-nodes (${b}$-nodes)
are either twin or the same:
$$
{\rm (a')}\!\!\!\!\!\!\!\!\!\!
\raisebox{-38mm}{
\begin{tikzpicture}
\clip (-0.51,0.51) -- ++(9.02,0) -- ++(0,-8.02) -- ++(-9.02,0) -- cycle;
 diag/.style={draw=black!30!white, thin, dash pattern=on 0.5em off 0.8385em on 0.5em off 0em}
\def\a{\Ccolor{blue!20!green!10!white}{white}{$a$}}
\draw ( 1,-0)     \2\a\3\a\2\a\3\a ; 
\draw ( 1,-1)     \a\1\a\1\a\1\a\1 ; 
\draw (-1,-2) \2\a\3\a\2\a\3\a\2\a ; 
\draw (-1,-3) \a\1\a\1\a\1\a\1\a\1 ; 
\draw (-1,-4) \3\a\2\a\3\a\2\a\3\a ; 
\draw (-1,-5) \a\1\a\1\a\1\a\1\a\1 ; 
\draw (-1,-6) \2\a\3\a\2\a\3\a\2\a ; 
\draw (-1,-7) \a\1\a\1\a\1\a\1\a\1 ; 
\draw [rotate=-45, draw=black!30!white, thin, step=\DK,
      dash pattern=on 0\DK off 0.28\DK on 0.44\DK off 0.28\DK]
(-1.4140,-5.6562) grid (13.5,11);
\end{tikzpicture}
  }
\ \ {\rm (a'')}\!\!\!\!\!\!\!\!\!\!
\raisebox{-38mm}{\begin{tikzpicture}
\clip (-0.51,0.51) -- ++(9.02,0) -- ++(0,-8.02) -- ++(-9.02,0) -- cycle;
 diag/.style={draw=black!30!white, thin, dash pattern=on 0.5em off 0.8385em on 0.5em off 0em}
\def\a{\Ccolor{blue!20!green!10!white}{white}{$a$}}
\draw ( 1,-0)     \2\a\1\a\2\a\1\a ; 
\draw ( 1,-1)     \a\3\a\1\a\3\a\1 ; 
\draw (-1,-2) \2\a\1\a\2\a\1\a\2\a ; 
\draw (-1,-3) \a\3\a\1\a\3\a\1\a\3 ; 
\draw (-1,-4) \1\a\2\a\1\a\2\a\1\a ; 
\draw (-1,-5) \a\1\a\3\a\1\a\3\a\1 ; 
\draw (-1,-6) \2\a\1\a\2\a\1\a\2\a ; 
\draw (-1,-7) \a\3\a\1\a\3\a\1\a\3 ; 
\draw [rotate=-45, draw=black!30!white, thin, step=\DK,
      dash pattern=on 0\DK off 0.28\DK on 0.44\DK off 0.28\DK]
(-1.4140,-5.6562) grid (13.5,11);
\end{tikzpicture}}
\ \ {\rm (b)}\!\!\!\!\!\!\!\!\!\!
\raisebox{-38mm}{
\begin{tikzpicture}
\clip (-0.51,0.51) -- ++(9.02,0) -- ++(0,-8.02) -- ++(-9.02,0) -- cycle;
\def\a{\Ccolor{blue!20!white}{white}{$a$}}
\def\b{\Ccolor{red!20!white }{white}{$b$}}
\draw ( 1,-0)     \2\a\3\a\2\a\3\a ; 
\draw ( 1,-1)     \b\1\a\1\b\1\a\1 ; 
\draw (-1,-2) \2\b\4\b\2\b\4\b\2\b ; 
\draw (-1,-3) \a\1\b\1\a\1\b\1\a\1 ; 
\draw (-1,-4) \3\a\2\a\3\a\2\a\3\a ; 
\draw (-1,-5) \a\1\b\1\a\1\b\1\a\1 ; 
\draw (-1,-6) \2\b\4\b\2\b\4\b\2\b ; 
\draw (-1,-7) \a\1\b\1\a\1\b\1\a\1 ; 
\draw [xshift=2\dk,rotate=-45, draw=black!30!white, thin, step=2\DK,
      dash pattern=on 0\DK off 0.28\DK on 0.44\DK off 0.28\DK]
(-1.4140,-7.070) grid (13.5,11);
\end{tikzpicture}}
$$
\end{proposition}

Proposition~\ref{p:2}(a',a'') confirms
the claim of Lemma~\ref{l:even} with the following correspondence
of colors:
$$
\begin{array}{r|c|c|c}
 \rm (a'): & \cll 1 & \cll 2 &  \cll 3 \\ \hline
 \HhH: & \cll{1} \cup \cll{4}   & \cll 2 &  \cll{3}
 \\
\end{array},\qquad
\begin{array}{r|c|c|c}
 \rm (a''): & \cll 1 & \cll 2 &  \cll 3 \\ \hline
 \HhH: & \cll{1} \cup \cll{2}  & \cll 3 &  \cll{4}
 \\
\end{array}.
$$
Corollary~\ref{c:evenodd} finalizes the proof
{of Theorem~\ref{th:main},
case~\ref{case2}
in subcases~(a') and~(a'') (according to the conclusion of Proposition~\ref{p:2}).
In subcase (b), the claim of Theorem~\ref{th:main}}
holds corresponding to Fig.~\ref{AA}
in accordance to the following table:
$$
\begin{array}{r|c|c|c|c|c|c}
 \mbox{(b)}: & \cll 1 & \cll 2 &  \cll 3 &  \cll 4 &  a &  b \\  \hline
 \mbox{Fig.~\ref{AA}}: & \cll{5}\cup\cll{6}   & \cll 4 &  \cll{1} &\cll{9} & \cll{2}\cup\cll{3}   & \cll{7}\cup\cll{8}
 \\
\end{array}.
$$

\subsection[All neighbors of a \{1,2\}-node are of type 2:2]{All the neighbors of a $\{\cll{1},\cll{2}\}$-node are of type $2{:}2$}\label{case3}
Consider two subcases.
\begin{enumerate}
 \item[\ref{case3}.a)] The coloring $\FfF$ is of the following type,
 where the odd nodes are colored with not more than four pairwise twin colors:
\begin{center}
\begin{tikzpicture}
\clip (-0.51,0.51) -- ++(16.02,0) -- ++(0,-8.02) -- ++(-16.02,0) -- cycle;
\def\a{\ccolor{white}{$\ $}}
\draw (-1,-0) \1\a\1\a\1\a\1\a\1\a\1\a\1\a\1\a ;
\draw (-1,-1) \a\2\a\2\a\2\a\2\a\2\a\2\a\2\a\2 ;
\draw (-1,-2) \1\a\1\a\1\a\1\a\1\a\1\a\1\a\1\a ;
\draw (-1,-3) \a\2\a\2\a\2\a\2\a\2\a\2\a\2\a\2 ;
\draw (-1,-4) \1\a\1\a\1\a\1\a\1\a\1\a\1\a\1\a ;
\draw (-1,-5) \a\2\a\2\a\2\a\2\a\2\a\2\a\2\a\2 ;
\draw (-1,-6) \1\a\1\a\1\a\1\a\1\a\1\a\1\a\1\a ;
\draw (-1,-7) \a\2\a\2\a\2\a\2\a\2\a\2\a\2\a\2 ;
\draw [rotate=-45, draw=black!30!white, thin, step=\DK,
      dash pattern=on 0\DK off 0.28\DK on 0.44\DK off 0.28\DK]
(-1.4140,-5.6562) grid (16.5,11);
\end{tikzpicture}
\end{center}

\item[\ref{case3}.b)] The coloring $\FfF$ contains up to reflection
the following fragment:
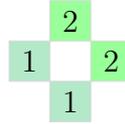
$$
\begin{tikzpicture}
\draw (1,0)   \2 ;
\draw (0,-1) \1 (2,-1) \2 ;
\draw (1,-2)   \1 ;
\end{tikzpicture}
$$
\end{enumerate}
In the last subcase,
by the hypothesis of case~{\rm\ref{case3}} we get
\begin{itemize}
 \item the fragment is uniquely extended to two
  $\dll{1}{2}$-R-diagonals,
\item the set of all even nodes consists of
 $\dll{1}{2}$-R-diagonals.
\end{itemize}
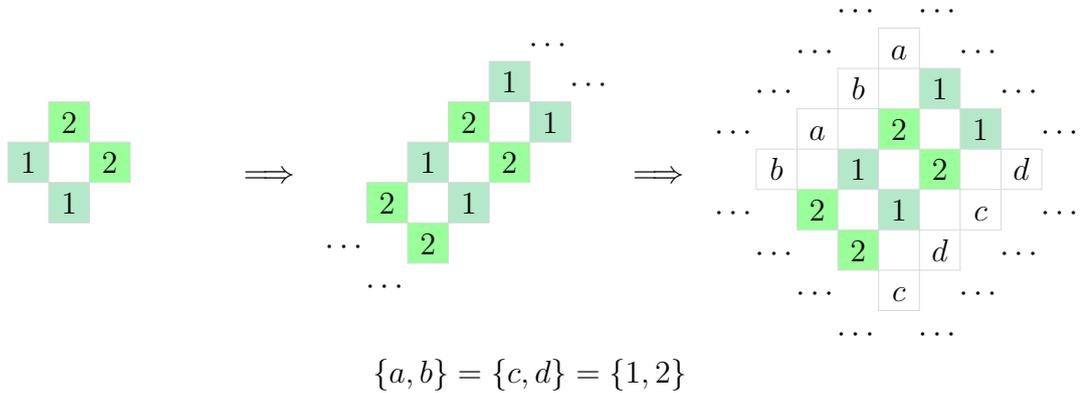
\begin{center}
\begin{tikzpicture}
\draw (1,0)   \2 ;
\draw (0,-1) \1 (2,-1) \2 ;
\draw (1,-2)   \1 ;
\phantom{\draw (4,-5) \1 ;}
\end{tikzpicture}
\ \ \raisebox{5.2em}{$\Longrightarrow$}
\begin{tikzpicture}
\def\0{ ++(1,0)}
\draw (6, 2)             \dt ;
\draw (5, 1)           \1\0\dt ;
\draw (4,-0)         \2\0\1 ;
\draw (3,-1)       \1\0\2 ;
\draw (2,-2)     \2\0\1 ;
\draw (1,-3)  \dt\0\2 ;
\draw (2,-4)     \dt ;
\phantom{\draw (3,-5) \1 ;}
\end{tikzpicture}
\ \ \raisebox{5.2em}{$\Longrightarrow$}
\begin{tikzpicture}
\def\a{\ccolor{white}{$a$}}
\def\b{\ccolor{white}{$b$}}
\def\c{\ccolor{white}{$c$}}
\def\f{\ccolor{white}{$d$}}
\def\0{ ++(1,0)}
\draw (3, 3)       \td\0\dt ;
\draw (2, 2)     \td\0\a\0\dt ;
\draw (1, 1)   \td\0\b\0\1\0\dt ;
\draw (0,-0) \td\0\a\0\2\0\1\0\dt ;
\draw (1,-1)    \b\0\1\0\2\0\f ;
\draw (0,-2) \dt\0\2\0\1\0\c\0\td ;
\draw (1,-3)   \dt\0\2\0\f\0\td ;
\draw (2,-4)     \dt\0\c\0\td ;
\draw (3,-5)       \dt\0\td ;
\end{tikzpicture}
$\{a,b\}=\{c,d\}=\{1,2\}$
\end{center}
So, we have the following fact.
\begin{proposition}\label{p:3}
A coloring corresponding to subcase
{\rm\ref{case3}.b} can be obtained from a coloring
corresponding to subcase {\rm\ref{case3}.a}
by shifting
$\dll{1}{2}$-R-diagonals
or
by shifting
$\dll{1}{2}$-L-diagonals {(but not both)}.
\end{proposition}

Again, we see that the claim of
Lemma~\ref{l:even} is confirmed by Proposition~\ref{p:3}
in the considered case, where the colors $\cll 1$ and $\cll 2$
in Proposition~\ref{p:3} correspond to the merged  $\cll 1$ and $\cll 3$
and merged  $\cll 2$ and $\cll 4$ in Lemma~\ref{l:even}.
Corollary~\ref{c:evenodd} finalizes the proof
of Theorem~\ref{th:main} in this case.

\subsection%
[Every \{1,2\}-node has two type-2:2
and two type-1:1 neighbors]
{Every $\{\cll{1},\cll{2}\}$-node has two neighbors of type $2{:}2$
and two neighbors of type $1{:}1$}
\label{case4}

Consider any $\cll{1}$-node. Let its neighbors of type $2{:}2$
have colors $\cll{a}$ and $\cll{b}$
(it is possible that $\cll{a}=\cll{b}$).
Note that these neighbors must be placed diagonally from each other, otherwise the
other two neighbors cannot be of type~$1{:}1$.
By the same reason, a fragment like
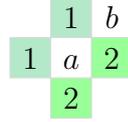
\begin{center}
\begin{tikzpicture}[yscale=-1]
\def\a{\ccolor{white}{$a$}}
\def\b{\ccolor{white}{$b$}}
\draw (1,0)   \2 ;
\draw (0,-1) \1 \a \2 ;
\draw (1,-2)   \1 \b ;
\end{tikzpicture}
\end{center}
is impossible. Hence we have two possibilities up to rotation and reflection.
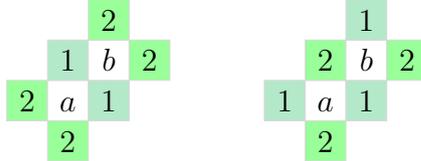
\begin{center}
\begin{tikzpicture}[yscale=-1]
\def\a{\ccolor{white}{$a$}}
\def\b{\ccolor{white}{$b$}}
\draw (1,0)    \2 ;
\draw (0,-1) \2\a\1 ;
\draw (1,-2)   \1\b\2 ;
\draw (2,-3)     \2 ;
\end{tikzpicture}
\qquad
\begin{tikzpicture}[yscale=-1]
\def\a{\ccolor{white}{$a$}}
\def\b{\ccolor{white}{$b$}}
\draw (1,0)    \2 ;
\draw (0,-1) \1\a\1 ;
\draw (1,-2)   \2\b\2 ;
\draw (2,-3)     \1 ;
\end{tikzpicture}
\end{center}
In the both cases the fragment is uniquely extended
to two $\cll{1}\cll{2}$-diagonals and $\cll{a}\cll{b}$-diagonal.

The rest of the proof of the conclusion of
Theorem~\ref{th:main} in Case~\ref{case4}
is contained in the following proposition,
which reflects a more general situation
and will be repeatedly used in Section~\ref{case6}.

\begin{proposition}\label{p:twobin}
If a bipartite perfect coloring $\FfF$  of $G(\ZZ^2)$
has neighbor $\dll ab$ and $\dll ef$ diagonals, where
$\cll a$, $\cll b$, $\cll e$, $\cll f$ are some colors,
not necessarily distinct, then $\FfF$ is equivalent
to a coloring obtained from a coloring in Fig~\ref{II},
Fig~\ref{KK}, or  Fig~\ref{LL} by shifting
R-diagonals
and merging some groups of twin colors.
\end{proposition}
\begin{proof}
Without loss of generality, we assume that $\FfF$
is a bipartite perfect coloring
with two neighbor binary or one-color R-diagonals.


\begin{claim}\label{cl:allD}
Every
R-diagonal
is either binary or one-color.
\end{claim}
Indeed, consider the nodes $[1,0]$ and $[0,-1]$.
Their neighbors have colors
$\cll a$, $\cll b$, $\FfF\Lft 1,-1\Rgt $, $\FfF\Lft 2,0\Rgt $
and
$\cll a$, $\cll b$, $\FfF\Lft 1,-1\Rgt $, $\FfF\Lft 0,-2\Rgt $,
respectively.
Since $[1,0]$ and $[0,-1]$ are of the same or twin colors, we conclude
$\FfF\Lft 2,0\Rgt  = \FfF\Lft 0,-2\Rgt $.
Similarly, by induction on $|i|$,
we get
$\FfF\Lft i,j\Rgt =\FfF\Lft i+2,j+2\Rgt $ for every $i,j\in\ZZ$,
which proves Claim~\ref{cl:allD}.

We now consider two subcases.

(i) For every two
R-diagonals,
the sets of their colors coincide
or do not intersect. In this case, for
an R-diagonal,
its set of colors uniquely determines
the sets of colors of the two neighbor diagonals.
In other words, we have a perfect coloring of
the infinite path graph~$G(\ZZ)$,
where the node~$[j]$ is colored
with the set of colors of the diagonal through $[j,0]$
in the coloring $\FfF$.
It is easy to find (see, e.g., \cite[Lemma~4]{LAP:2020})
that there are only two classes
of bipartite perfect colorings of~$G(\ZZ)$:
\begin{itemize}
 \item[(a)] \emph{cyclic} colorings
 $\cdots\ccll{c_1}\ccll{c_2}\cdots\ccll{c_s}\ccll{c_1}\ccll{c_2}\cdots\ccll{c_s}\ccll{c_1}\ccll{c_2}\cdots$
 {(the number~$s$ of colors is even because the coloring is bipartite)};
 \item[(b)] \emph{mirror} colorings
$\cdots\ccll{c_2}\ccll{c_1}\ccll{c_2}\cdots\cll[2em]{c_{s{-}1}}\ccll{c_s}\cll[2em]{c_{s{-}1}}\cdots\ccll{c_2}\ccll{c_1}\ccll{c_2}\cdots$ (any number $s$ of colors).
 \end{itemize}
Once we fix the coloring of $G(\ZZ)$, the colors of the one-color
R-diagonals
of $\FfF$ are determined,
and the colors of the binary
R-diagonals
of $\FfF$ are determined up to shifting.
If there are only binary
R-diagonals,
then $\FfF$ is obtained by shifting from the coloring
in Fig.~\ref{KK}
(the case $n=8$ is shown separately in Fig.~\ref{II},
and the case $n=4$ is obtained from Fig.~\ref{II}
by merging twins $\cll 1$ and $\cll 3$, $\cll 2$ and $\cll 4$,
 $\cll 5$ and $\cll 6$, $\cll 7$ and $\cll 8$)
 or in Fig.~\ref{LL}
 (again the case $n=4$ is obtained from Fig.~\ref{II}
by merging twin colors as above, and the case $n=6$
is obtained from Fig.~\ref{II}
by merging $\cll 5$ and $\cll 6$, $\cll 7$ and $\cll 8$),
for subcases (a) and (b) respectively.

(ii) There are two
R-diagonals,
say $\dll{a}{b}$- and $\dll{a}{c}$-diagonals,
whose sets of colors are different but with non-zero intersection,
i.e., $\cll b\ne\cll c$.
Clearly, the colors from $\{\cll a, \cll b, \cll c\}$
are pairwise twin.
We assume w.l.o.g. that $\cll c\ne\cll a$,
$\FfF\Lft 0,0\Rgt = \cll a$,
the R-diagonal
through $[0,0]$ is an \dll{a}{b}-diagonal,
and the R-diagonal
through $[1,0]$ is an \dll{e}{f}-diagonal.
Denote
$\cll g=\FfF\Lft -1,0\Rgt $, $\cll h=\FfF\Lft 0,1\Rgt $,
(hence, the whole diagonal through the last two nodes is a $\dll gh$-diagonal).
Since \cll{a} and \cll{c} are twins
{(this is the only place where we use the existence of an $\dll{a}{c}$-diagonal)},
every $\cll{e}$-node
has a $\cll c$-neighbor. Therefore, the diagonal through $[2,0]$
is a $\dll cd$-diagonal for some $\cll d$.
Similarly, the diagonal through $[-2,0]$
is a $\dll c{d'}$-diagonal for some $\cll{d'}$.
The colors from $\{\cll{a},\cll{b},\cll{c},\cll{d},\cll{d'}\}$
are mutually twin, and we see that the diagonal through $[0,3]$
is a $\dll gh$-diagonal and
the diagonal through $[0,-3]$
is an $\dll ef$-diagonal.
Similarly (by induction on $|i|$),
for every $i\in \ZZ$
the colors of the diagonal through $[2i,0]$
are equal or twin to $\cll a$ and the diagonal
through $[2i+1,0]$ is an  $\dll ef$- or $\dll gh$- diagonal,
depending on the parity of $i$.
In particular, every odd node is an
$\{\cll{e},\cll{f},\cll{g},\cll{h}\}$-node, where
\cll{e} and \cll{f} are equal or twin,
\cll{g} and \cll{h} are equal or twin.
By shifting some of
R-diagonals,
we can make a
$\dll{e}{g}$-L-diagonal.
Therefore, any two different colors from
$\{\cll{e},\cll{f},\cll{g},\cll{h}\}$ are twins.
It follows that $\cll d = \cll{d'}$ and for every even (odd)
$i$ the
R-diagonal
through $[2i,0]$
is an \dll{a}{b}-diagonal (a \dll{c}{d}-diagonal, respectively).
We now see that if all colors
$\cll a$,
$\cll b$,
$\cll c$,
$\cll d$,
$\cll e$,
$\cll f$,
$\cll g$,
$\cll h$
are pairwise different, then $\FfF$ is shifting equivalent
to the coloring in Fig.~\ref{II}.
If there are some equalities between them,
then $\FfF$ is shifting equivalent
to a coloring obtained from the one in Fig.~\ref{II} by merging twin colors.
\end{proof}

\subsection%
[Each \{1,2\}-node has one type-2:2
and three type-1:1 neighbors]
{Each $\{\cll{1},\cll{2}\}$-node has one neighbor of type $2{:}2$
and three neighbors of type $1{:}1$}
\label{case5}

Denote by $\cll{3}$ the color of the type $2{:}2$ nodes.
Up to rotation the coloring contains one of the following fragments.
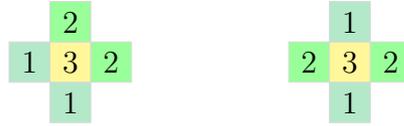
\begin{center}
\def\3{\ccolor{yellow!35!white}{$3$}}
\begin{tikzpicture}
\draw (1,0)   \1 ;
\draw (0,1) \1\3\2 ;
\draw (1,2)   \2 ;
\end{tikzpicture}
\qquad\qquad
\begin{tikzpicture}
\draw (1,0)   \1 ;
\draw (0,1) \2\3\2 ;
\draw (1,2)   \1 ;
\end{tikzpicture}
\end{center}
It is easy to see that the first fragment contradicts the condition
of case~{\rm\ref{case5}}.
Directly from this condition we get that the second
fragment uniquely determines the placement of the $\{\cll{1},\cll{2},\cll{3}\}$-nodes:
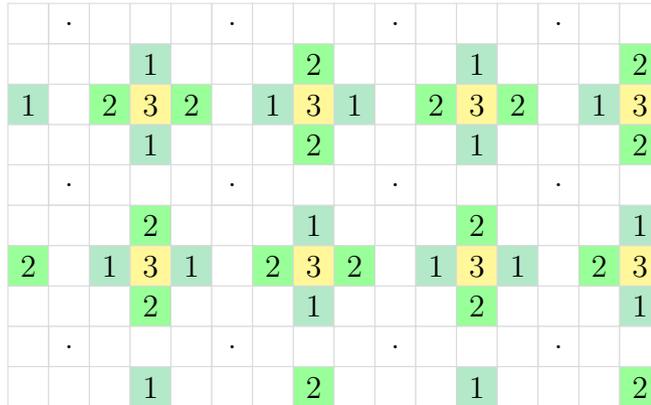
\begin{center}
\begin{tikzpicture}
\def\ccolor#1#2{ ++(1,0) node [cell,fill=#1] {\raisebox{-0.25em}[0.2em][0em]{\makebox[0mm][c]{#2}}} }
\def\3{\ccolor{yellow!35!white}{$3$}}
\def\0{\ccolor{white}{$\ $}}
\def\d{ ++(1,0) node [cell] {.} }
\draw (0,0) \0\0\0\1\0\0\0\2\0\0\0\1\0\0\0\2 ;
\draw (0,1) \0\d\0\0\0\d\0\0\0\d\0\0\0\d\0\0 ;
\draw (0,2) \0\0\0\2\0\0\0\1\0\0\0\2\0\0\0\1 ;
\draw (0,3) \2\0\1\3\1\0\2\3\2\0\1\3\1\0\2\3 ;
\draw (0,4) \0\0\0\2\0\0\0\1\0\0\0\2\0\0\0\1 ;
\draw (0,5) \0\d\0\0\0\d\0\0\0\d\0\0\0\d\0\0 ;
\draw (0,6) \0\0\0\1\0\0\0\2\0\0\0\1\0\0\0\2 ;
\draw (0,7) \1\0\2\3\2\0\1\3\1\0\2\3\2\0\1\3 ;
\draw (0,8) \0\0\0\1\0\0\0\2\0\0\0\1\0\0\0\2 ;
\draw (0,9) \0\d\0\0\0\d\0\0\0\d\0\0\0\d\0\0 ;
\end{tikzpicture}
\end{center}

Denote by $E$ the set of the colors
of the even nodes excluding $\cll{1}$ and $\cll{2}$.

\begin{claim}\label{cl:E2}
$|E|\leq 2$. If $|E|= 2$, then $\dd^3(\cll{3},\cll{e})=12$
for each $\cll{e}$ in $E$.
\end{claim}

\begin{proof}
First we note that
$$
\foreach \xX in {1,2} {\sum_{e\in E}\dd^2(\cll{\xX},\cll{e})=6, \qquad }
$$ and
\begin{equation}\label{eq:d21e}
\dd^2(\cll{1},\cll{e})\geq 2,
\qquad \dd^2(\cll{2},\cll{e})\geq 2
\qquad \mbox{ for all }\cll{e}\in E
\end{equation}
(the last follows from the fact that every $E$-node is at $[\pm 1, \pm 1]$ from some $\cll{1}$-node).
Then,
\begin{equation}\label{eq:0md3}
\dd^3(\cll{3},\cll{e}) = 2 \dd^2(\cll{1},\cll{e}) + 2 \dd^2(\cll{2},\cll{e}) \equiv 0\bmod 3
\end{equation}
(indeed, there are exactly $3$ paths of length $3$ from every $\cll{3}$-node to every nearest $E$-node,
each path containing exactly one $\{\cll{1},\cll{2}\}$-node); hence, $\dd^3(\cll{3},\cll{e})$ is even, not less than $8$,
and divisible by $3$. Since
\begin{equation}\label{eq:d33e}
\sum_{e\in E}\dd^3(\cll{3},\cll{e})=24.
\end{equation}
the only possibilities for $\dd^3(\cll{3},\cll{e})$
are $12$ (with $|E|=2$) and $24$ (with $|E|=1$).
\end{proof}

\begin{proposition}\label{p:krest}
In case {\rm\ref{case5}}, the coloring $\FfF$ either

{\rm(a)} is equivalent to
the coloring shown in Fig.~\ref{BB}, or

{\rm(b)} $\FfF$ can be obtained from the coloring
shown in Fig.~\ref{AA}
by shifting $\dll{4}{5}$- or $\dll{4}{6}$-diagonals {(but not both)},
unifying the colors $\cll{7}$ and $\cll{8}$,
unifying two or three colors from $\{\cll{4},\cll{5},\cll{6}\}$,
and/or renaming colors.
\end{proposition}

\begin{proof} (a)
 We first consider the subcase when $|E|=2$, say
 $E=\{\cll{7},\cll{8}\}$
 (in agree with Fig.~\ref{BB}),
 and an odd node of type $0{:}0$
 has two non-opposite neighbors of the same color.
\begin{center}
\begin{tikzpicture}
\def\ccolor#1#2{ ++(1,0) node [cell,fill=#1] {\raisebox{-0.25em}[0.2em][0em]{\makebox[0mm][c]{#2}}} }
\def\3{\ccolor{yellow!35!white}{$3$}}
\def\7{\ccolor{red!15!white}{$7$}}
\def\8{\ccolor{blue!15!white}{$8$}}
\def\d{ ++(1,0) node [cell] {.} }
\draw (0,-2) \0\0\0\2\0\0;
\draw (0,-3) \2\0\1\3\1\0\2;
\draw (0,-4) \0\0\0\2\0\7;
\draw (0,-5) \0\0\0\0\7\d;
\end{tikzpicture}
\end{center}

{%
We see that
$\dd^2(\cll{1},\cll{7})\geq 3$
and
$\dd^2(\cll{2},\cll{7})\geq 3$.
From this, $\dd^3(\cll{3},\cll{7})=12$ (Claim~\ref{cl:E2}),
and $\dd^3(\cll{3},\cll{7})=4\cdot\dd^2(\cll{1},\cll{7})$,
we find
\begin{equation}
\dd^2(\cll{1},\cll{7}) = \dd^2(\cll{2},\cll{7}) = 3,
\quad\mbox{and}\quad
\dd^2(\cll{1},\cll{8}) = \dd^2(\cll{2},\cll{8}) = 3.
\end{equation}
}
With the last equations, the colors of all nodes at distance~$3$
from the $\cll{3}$-node are uniquely reconstructed.
\def\7{\ccolor{red!15!white}{$7$}}
\def\8{\ccolor{red!15!blue!15!white}{$8$}}
\def\3{\ccolor{yellow!35!white}{$3$}}
\begin{center}
\begin{tikzpicture}
\draw (0,-0) \0\0\0\1\0\0;
\draw (0,-1) \0\0\0\0\0\0;
\draw (0,-2) \0\0\0\2\0\0;
\draw (0,-3) \2\0\1\3\1\0\2;
\draw (0,-4) \0\0\0\2\0\7;
\draw (0,-5) \0\0\0\0\7\0;
\draw (0,-6) \0\0\0\1\0\0;
\end{tikzpicture}
\ \raisebox{17mm}{$\Rightarrow$}
\begin{tikzpicture}
\draw (0,-0) \0\0\0\1\0\0;
\draw (0,-1) \0\0\0\0\8\0;
\draw (0,-2) \0\0\0\2\0\8;
\draw (0,-3) \2\0\1\3\1\0\2;
\draw (0,-4) \0\8\0\2\0\7;
\draw (0,-5) \0\0\8\0\7\0;
\draw (0,-6) \0\0\0\1\0\0;
\end{tikzpicture}
\ \raisebox{17mm}{$\Rightarrow$}
\begin{tikzpicture}
\draw (0,-0) \0\0\0\1\0\0;
\draw (0,-1) \0\0\7\0\8\0;
\draw (0,-2) \0\7\0\2\0\8;
\draw (0,-3) \2\0\1\3\1\0\2;
\draw (0,-4) \0\8\0\2\0\7;
\draw (0,-5) \0\0\8\0\7\0;
\draw (0,-6) \0\0\0\1\0\0;
\end{tikzpicture}
\end{center}
Next, we see that the cells at distance $2$ from the $\cll{3}$-cell
are uniquely colored with three new colors.
From this point, the fragment considered is continued
in a straightforward manner,
resulting in the coloring in Fig.~\ref{BB}.

(b') Next, we consider the subcase when $|E|=2$,
say
 $E=\{\cll{7},\cll{8}\}$
 and
 any two non-opposite neighbors of each odd node of type $0{:}0$
 have different colors.
We have a fragment like the following:
\definecolor{yell}{rgb}{0.95, 0.9, 0.45}
\def\3{\Ccolor{white}{yell}{$3$}}
\def\4{\Ccolor{white}{cyan!30!blue!15!white}{$9$}}
\def\5{\ccolor{red!90!blue!20!white}{$7$}}
\def\P{\ccolor{red!90!blue!20!white}{$\underline{{7}}$}}
\def\6{\ccolor{red!30!blue!20!white}{$8$}}
\def\S{\ccolor{red!30!blue!20!white}{$\underline{{8}}$}}
\def\x{\ccolor{white}{$\bar x$}}
\def\y{\ccolor{white}{$\bar y$}}
\def\d{ ++(1,0) node  {.} }
\def\a{\ccolor{white}{$a$}}
\def\b{\ccolor{white}{$b$}}
\def\c{\ccolor{white}{$c$}}
\def\A{\ccolor{\colTWO}{$\underline{{2}}$}}
\def\O{\ccolor{\colONE}{$\underline{{1}}$}}
\def\7{\Ccolor{blue!20!  }{white}{$4$}}
\def\8{\Ccolor{purple!20!}{white}{$5$}}
\def\9{\Ccolor{orange!30!}{white}{$6$}}
\def\D{ ++(1,0) node  [cell]{.} }
$$
\begin{tikzpicture}
\draw (0,-2) \0\0\0\2\0\y;
\draw (0,-3) \0\0\1\3\1\0\2;
\draw (0,-4) \0\0\0\2\0\6;
\draw (0,-5) \0\0\x\0\5\4\5;
\draw (0,-6) \0\0\0\1\0\6;
\end{tikzpicture}
$$
We claim that $\FfF(\bar x) = \cll{8}$ and similarly
$\FfF(\bar y) = \cll{7}$.
Indeed, if $\FfF(\bar x) = \cll{7}$,
then $\dd^2(\cll{1},\cll{7})\geq 4$ and
$\dd^2(\cll{2},\cll{7})\geq 4$,
which means $\dd^3(\cll{3},\cll{7})\geq 16$
and contradicts Claim~\ref{cl:E2}.

In a similar manner, we can now determine the colors
of all nodes except those at distance $2$ from the
$\cll{3}$-nodes. In the picture below,
these nodes are marked by the labels $a$, $b$, and $c$.
$$
\begin{tikzpicture}
\draw (0,0) \b\6\c\1\b\5\c\2\b\6\c\1\b\5\c\2 ;
\draw (0,1) \5\4\5\a\6\4\6\a\5\4\5\a\6\4\6\a ;
\draw (0,2) \c\6\b\2\c\5\b\1\c\6\b\2\c\5\b\1 ;
\draw (0,3) \2\a\1\3\1\a\2\3\2\a\1\3\1\a\2\3 ;
\draw (0,4) \b\5\c\2\b\6\c\1\b\5\c\2\b\6\c\1 ;
\draw (0,5) \6\4\6\a\5\4\5\a\6\4\6\a\5\4\5\a ;
\draw (0,6) \c\5\b\1\c\6\b\2\c\5\b\1\c\6\b\2 ;
\draw (0,7) \1\a\2\3\2\a\1\3\1\a\2\3\2\a\1\3 ;
\end{tikzpicture}
$$
Clearly, the remaining nodes are colored with one, two, or three
mutually twin colors.
Moreover, it is easy to see that $S_{1,i}=S_{2,i}=S_{7,i}=S_{8,i}$
for every color $\cll{i}$ of a node marked by $a$, $b$, or $c$.
Hence, $\cll{7}$ and $\cll{8}$ are twins.

If the nodes marked by $a$ are of the same color,
then it is not difficult to find that the same is true for
the nodes marked by $b$ and for
the nodes marked by $c$. Then p.\,(b) of the proposition
takes place.

Otherwise, there are two $a$-nodes $\bar z$ and $\bar z'$
of different colors and with difference
$\bar z' - \bar z = [2,\pm 2]$.
W.l.o.g., assume
$\bar z=[0,0]$, $\bar z=[2,-2]$,
$ \FfF( \bar z ) = \cll{4} $,
$ \FfF( \bar z ) = \cll{5} $.

\begin{center}
\begin{tikzpicture}
\draw (0,6) \0\0\0\0\o\6\o\2\o ;
\draw (0,5) \0\0\0\7\5\4\5\o\0 ;
\draw (0,4) \0\0\D\A\o\6\o\0\0 ;
\draw (0,3) \0\o\1\3\O\8\0\0\0 ;
\draw (0,2) \o\6\o\2\D\0\0\0\0 ;
\end{tikzpicture}
\!\!\!\!\!\!\raisebox{12mm}{$\Rightarrow$}\!\!\!\!\!\!
\begin{tikzpicture}
\draw (0,6) \0\0\0\0\D\6\o\2\o ;
\draw (0,5) \0\0\0\7\P\4\5\o\0 ;
\draw (0,4) \0\0\8\2\o\S\D\0\0 ;
\draw (0,3) \0\o\1\3\1\8\0\0\0 ;
\draw (0,2) \o\6\o\2\7\0\0\0\0 ;
\end{tikzpicture}
\!\!\!\!\!\!\raisebox{12mm}{$\Rightarrow$}\!\!\!\!\!\!
\begin{tikzpicture}
\draw (0,6) \0\0\0\0\8\6\o\2\7 ;
\draw (0,5) \0\0\0\7\5\4\5\8\0 ;
\draw (0,4) \0\0\8\2\o\6\7\0\0 ;
\draw (0,3) \0\7\1\3\1\8\0\0\0 ;
\draw (0,2) \8\6\o\2\7\0\0\0\0 ;
\draw [diag] (0.5,1.5)--(5.5,6.5);
\draw [diag] (4.5,1.5)--(9.5,6.5);
\end{tikzpicture}
\end{center}
Since $\cll{1}$ and  $\cll{2}$ are twins,
the neighbors of each of the nodes $[0,-1]$, $[1,-2]$
(underlined in the picture above)
have colors $\cll{3}$, $\cll{4}$, $\cll{5}$, and
$\FfF\Lft 1,-1\Rgt $, independently on the value
$\FfF\Lft 1,-1\Rgt $ (it can be $\cll{4}$, $\cll{5}$, or a new color, say~$\cll{6}$).
We see that $\FfF\Lft -1,-1\Rgt =\cll{5}$ and $\FfF\Lft 1,-3\Rgt =\cll{4}$.
Similarly, considering the neighborhoods
of the $\{\cll{7},\cll{8}\}$-nodes $[1,0]$ and $[2,-1]$,
we see $\FfF\Lft 1,1\Rgt =\cll{5}$ and $\FfF\Lft 3,-1\Rgt =\cll{4}$.
Processing in a similar way,
we find that the both
R-diagonals
through $[0,0]$ and $[2,-2]$ are $\dll{4}{5}$-diagonals.

The same argument show that the
R-diagonal
through $[4,-4]$ is a $\dll{4}{5}$-diagonal,
independently on the value $\FfF\Lft 4,-4\Rgt $,
which can be $\cll{4}$ or $\cll{5}$.
Similarly, by induction, for every $i$ the
R-diagonal
through $[2i,-2i]$ is a $\dll{4}{5}$-diagonal.
It remains to note that all nodes marked by $b$,
i.e., of form $[2+2i+2j,2i-2j]$ are colored with the same color,
$\cll{4}$, $\cll{5}$, or a new color $\cll{6}$.

(b'') The subcase $|E|=1$ is considered similarly to (b').
\end{proof}

\subsection[All the neighbors of a \{1,2\}-node are of type 1:1]{All the neighbors of a $\{\cll{1},\cll{2}\}$-node are of type $1{:}1$}\label{case6}
We divide the case into the following four subcases.
\begin{claim}\label{cl:VI}
One of the following four assertions takes place.
\begin{enumerate}
 \item[\rm\ref{case6a}.] The $\{1,2\}$-nodes are all the nodes with even coordinates or all the nodes with odd coordinates
 (without loss of generality, we consider the even subcase).
 \item[\rm\ref{case6b}.]
 The coloring $\FfF$ contains the following fragment,
up to rotation and reflection:
$$
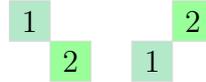
\begin{tikzpicture}
\draw (0,-0) \1\0\0\0\2;
\draw (0,-1) \0\2\0\1\0;
\end{tikzpicture}$$
\item[\rm\ref{case6c}.]
The coloring $\FfF$ contains the following fragment,
up to rotation and renaming the colors $\cll{1}$ and $\cll{2}$:
$$
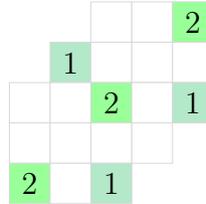
\begin{tikzpicture}
\draw (2, 4)     \o\o\2;
\draw (1, 3)   \1\o\o\o;
\draw (0, 2) \o\o\2\o\1;
\draw (0, 1) \o\o\o\o  ;
\draw (0, 0) \2\o\1    ;
\end{tikzpicture}$$
\item[\rm\ref{case6d}.]
The coloring $\FfF$ contains
a $\dll{1}{2}$-diagonal.
\end{enumerate}
\end{claim}

\begin{proof}
At first, if every type $1{:}1$ node has opposite
$\cll{1}$- and $\cll{2}$-neighbors, then we obviously
have subcase~\ref{case6a}.

Thus, we can assume that there are two
$\{\cll{1},\cll{2}\}$-neighbors
of a type $1{:}1$ node that are placed diagonally from each other.
W.l.o.g., assume $\FfF\Lft 0,0\Rgt =2$ and $\FfF\Lft -1,1\Rgt =1$.
Let us consider the diagonal through these two nodes.

If it is  a $\dll{1}{2}$-diagonal, then we have case~\ref{case6d}.

Otherwise, there is an integer $i$ with the minimum absolute value
such that the node $\bar x=[i,-i]$ is not a $\{\cll{1},\cll{2}\}$-node.
Depending on the sign and the parity of $i$, we have one of the four situations:
\begin{center}
\def\3{\ccolor{yellow!35!white}{$\bar x$}}
\def\4{\ccolor{white}{$\bar y$}}
\def\9{\ccolor{white}{$*$}}
\begin{tikzpicture}
\draw (0,2) \1\0\9;
\draw (0,1) \0\2\4;
\draw (0,0) \0\0\3;
\end{tikzpicture}
\quad
\begin{tikzpicture}
\draw (0,2) \2\0;
\draw (0,1) \0\1;
\draw (0,0) \0\0\3;
\end{tikzpicture}
\quad
\begin{tikzpicture}
\draw (0,2) \3\0;
\draw (0,1) \0\2;
\draw (0,0) \0\0\1;
\end{tikzpicture}
\quad
\begin{tikzpicture}
\draw (0,2) \3\0;
\draw (0,1) \0\1;
\draw (0,0) \0\0\2;
\end{tikzpicture}
\end{center}
We consider only the first one, as the other three are similar.
We observe that the node $[i,-i+2]$,
marked by $*$ in the picture, is not a $\{\cll{1},\cll{2}\}$-node,
by the condition of case~\ref{case6}.
The node $\bar y = [i,-i+1]$ is of type $1{:}1$, and the only place for its
$\cll{1}$-neighbor is $[i+1,-i+1]$:
\begin{center}
\def\3{\ccolor{yellow!35!white}{$\bar x$}}
\def\4{\ccolor{white}{$\bar z$}}
\def\9{\ccolor{white}{$*$}}
\begin{tikzpicture}
\draw (0,2) \1\0\9\4;
\draw (0,1) \0\2\0\1;
\draw (0,0) \0\0\3;
\end{tikzpicture}
\end{center}
The node $\bar z = [i+1,-i+2]$
is also of type $1{:}1$,
and there are two possibilities for its
$\cll{2}$-neighbor:
\begin{center}
\def\3{\ccolor{yellow!35!white}{$\bar x$}}
\def\4{\ccolor{white}{$\bar z$}}
\def\9{\ccolor{white}{$*$}}
\begin{tikzpicture}
\draw (0,2) \1\0\9\4\2;
\draw (0,1) \0\2\0\1;
\draw (0,0) \0\0\3;
\end{tikzpicture}
\ \ \ \mbox{or}\ \ \
\begin{tikzpicture}
\draw (0,3) \0\0\0\2;
\draw (0,2) \1\0\9\4;
\draw (0,1) \0\2\0\1;
\draw (0,0) \0\0\3;
\end{tikzpicture}
\end{center}
The first possibility leads
to subcase~\ref{case6b}. In the remaining case,
considering similarly the type $1{:}1$ nodes
$\bar y' = [i-1,-i]$ and $\bar z' = [i-2,-i-1]$,
we find that subcase~\ref{case6b}
or subcase~\ref{case6c} takes place:
\begin{center}
\def\3{\ccolor{yellow!35!white}{$\bar x$}}
\def\4{\ccolor{white}{$\bar y'$}}
\def\5{\ccolor{white}{$\bar z'$}}
\begin{tikzpicture}
\draw (0,3)  \0\0\0\2;
\draw (0,2)  \1\0\0\0;
\draw (0,1)  \0\2\0\1;
\draw (0,0)  \0\4\3;
\draw (0,-1) \5\1\0;
\draw (0,-2) \2\0\0;
\end{tikzpicture}
\ \ \ \mbox{or}\ \ \
\begin{tikzpicture}
\draw (0,3)    \0\0\0\2;
\draw (0,2)    \1\0\0\0;
\draw (0,1)    \0\2\0\1;
\draw (0,0)    \0\4\3;
\draw (-1,-1)\2\5\1\0;
\end{tikzpicture}
\end{center}

\end{proof}

In the following four subsections,
we show that
the claim of Theorem~\ref{th:main} holds in each of the four subcases.

\renewcommand{\thesubsubsection}{\thesubsection.\alph{subsubsection}}
\subsubsection[All \{1,2\}-nodes are all nodes with even coordinates]{All the $\{\cll{1},\cll{2}\}$-nodes are all the nodes with even coordinates}\label{case6a}
\begin{center}
\def\x{++(+1,0) node {.}}
\begin{tikzpicture}
\draw (-1,-0) \1\o\2\o\1\o\2\o\1\o\2\o ;
\draw (-1,-1) \o\x\o\x\o\o\o\o\o\o\o\o ;
\draw (-1,-2) \2\o\1\o\2\o\1\o\2\o\1\o ;
\draw (-1,-3) \o\x\o\x\o\o\o\o\o\o\o\o ;
\draw (-1,-4) \1\o\2\o\1\o\2\o\1\o\2\o ;
\draw (-1,-5) \o\o\o\o\o\o\o\o\o\o\o\o ;
\draw (-1,-6) \2\o\1\o\2\o\1\o\2\o\1\o ;
\end{tikzpicture}
\end{center}

\begin{proposition}\label{p:6a}
In case {\rm \ref{case6a}},
either the coloring $\FfF$ is equivalent
to one of the colorings in Fig.~\ref{AA}, \ref{CC}, \ref{DD}, \ref{EE}
where some twin colors can be merged,
or the coloring of the even vertices is
equivalent to one of the following two
semicolorings:
$$
\raisebox{9em}{\rm (i)}\,\!\!\!
\begin{tikzpicture}
\def\3{\ccolor{yellow!40!white}{$3$}}
\draw (0,-0) \1\0\2\0\1\0\2\0;
\draw (0,-1) \0\3\0\3\0\3\0\3;
\draw (0,-2) \2\0\1\0\2\0\1\0;
\draw (0,-3) \0\3\0\3\0\3\0\3;
\draw (0,-4) \1\0\2\0\1\0\2\0;
\draw (0,-5) \0\3\0\3\0\3\0\3;
\draw (0,-6) \2\0\1\0\2\0\1\0;
\draw (0,-7) \0\3\0\3\0\3\0\3;
\end{tikzpicture}
\ \
\raisebox{9em}{\rm (ii)}\,\!\!\!
\begin{tikzpicture}
\def\3{\ccolor{yellow!40!white}{$3$}}
\def\4{\ccolor{orange!30!white}{$4$}}
\draw (0,-0) \1\0\2\0\1\0\2\0;
\draw (0,-1) \0\3\0\4\0\3\0\4;
\draw (0,-2) \2\0\1\0\2\0\1\0;
\draw (0,-3) \0\4\0\3\0\4\0\3;
\draw (0,-4) \1\0\2\0\1\0\2\0;
\draw (0,-5) \0\3\0\4\0\3\0\4;
\draw (0,-6) \2\0\1\0\2\0\1\0;
\draw (0,-7) \0\4\0\3\0\4\0\3;
\end{tikzpicture}
$$
\end{proposition}
\begin{proof}
Up to equivalence, there are seven ways to color
the four nodes placed diagonally from
a $\cll{1}$-node:
$$
\def\3{\ccolor{yellow!40!white}{$3$}}
\def\4{\ccolor{orange!30!white}{$4$}}
\def\5{\ccolor{red!25!white}{$5$}}
\begin{array}{lllllll}
\mbox{ (i)}&\mbox{ (ii)}&\mbox{ (iii)}&\mbox{ (iv)}&\mbox{ (v)}&\mbox{ (vi)}&\mbox{ (vii)}
\\
\begin{tikzpicture}
\draw (0,-1) \3\0\3;
\draw (0,-2) \0\1\0;
\draw (0,-3) \3\0\3;
\end{tikzpicture}
&
 \begin{tikzpicture}
 \draw (0,-1) \3\0\4;
 \draw (0,-2) \0\1\0;
 \draw (0,-3) \4\0\3;
 \end{tikzpicture}
&
 \begin{tikzpicture}
 \draw (0,-1) \3\0\3;
 \draw (0,-2) \0\1\0;
 \draw (0,-3) \4\0\4;
 \end{tikzpicture}
&
 \begin{tikzpicture}
 \draw (0,-1) \3\0\3;
 \draw (0,-2) \0\1\0;
 \draw (0,-3) \4\0\3;
 \end{tikzpicture}
&
 \begin{tikzpicture}
 \def\4{\ccolor{orange!30!white}{$4$}}
 \def\5{\ccolor{red!25!white}{$5$}}
 \draw (0,-1) \3\0\4;
 \draw (0,-2) \0\1\0;
 \draw (0,-3) \5\0\3;
 \end{tikzpicture}
&
 \begin{tikzpicture}
 \def\4{\ccolor{cyan!30!blue!25!white}{$4$}}
 \def\5{\ccolor{purple!80!blue!25!white}{$5$}}
 \draw (0,-1) \3\0\3;
 \draw (0,-2) \0\1\0;
 \draw (0,-3) \5\0\4;
 \end{tikzpicture}
&
 \begin{tikzpicture}
 \def\4{\ccolor{cyan!70!blue!30!white}{$4$}}
 \def\5{\ccolor{purple!80!blue!30!white}{$5$}}
 \def\6{\ccolor{red!80!yellow!30!white}{$6$}}
 \draw (0,-1) \3\0\4;
 \draw (0,-2) \0\1\0;
 \draw (0,-3) \6\0\5;
 \end{tikzpicture}
 \end{array}
$$

In case (i), all non-$\{\cll 1,\cll 2\}$ even nodes have color $\cll 3$,
corresponding to p.\,(i) of the claim.

In case (ii)
 each neighbor of a $\cll 1$-node has
 odd node has
 one $\cll 1$-,
  one $\cll 2$-,
 one $\cll 3$-, and
  one $\cll 4$-neighbor.
  And this is true for every odd node.
  It is not difficult to see that the coloring of the even nodes is uniquely determined.
  Since every odd node is a neighbor of a $\cll 1$-node,

Now consider case (iii).
We denote the colors of the $1$-node neighbors at the picture as
follows, where each of $\cll a$, $\Cll{a'}$ is different
from $\cll 5$ and $\cll 6$ but it is possible that
$\cll a =\Cll{a'}$:
$$
\def\4{\ccolor{orange!30!white}{$4$}}
\def\3{\ccolor{yellow!40!white}{$3$}}
\def\5{\ccolor{white}{$5$}}
\def\6{\ccolor{white}{$6$}}
\def\7{\ccolor{white}{$a$}}
\def\8{\ccolor{white}{$a'$}}
\begin{array}{c}
\begin{tikzpicture}
\draw (0,-1) \o\3\5\3\o\0;
\draw (0,-2) \2\7\1\8\2\o;
\draw (0,-3) \o\4\6\4\o\0;
\end{tikzpicture}
\end{array}
\Rightarrow
\begin{array}{c}
\begin{tikzpicture}
\draw (0,-1) \o\3\5\3\5\o\0;
\draw (0,-2) \2\7\1\8\2\7\1;
\draw (0,-3) \o\4\6\4\6\o\0;
\end{tikzpicture}
\end{array}
\Rightarrow
\begin{array}{c}
\begin{tikzpicture}
\draw (0,-1) \o\3\5\3\5\3\o;
\draw (0,-2) \2\7\1\8\2\7\1;
\draw (0,-3) \o\4\6\4\6\4\o;
\end{tikzpicture}
\end{array}
\Rightarrow\
\cdots
$$
Since \cll1 and \cll2 are twins, the rightest \cll{2}-node
in the picture
has neighbors of colors \cll 5, \cll6, \cll{a}, and \cll{a'}.
The only places for the \cll5- and \cll6-neighbors are above and below
the considered \cll2-node, respectively.
Now, this last \cll5-node must have a second \cll3-neighbor,
and the only place is at the right (because of the \cll1- and \cll2-neighbors in the vertical direction).
Similarly, all nodes in the three horizontal
rows are uniquely colored.
The whole coloring is now uniquely determined, whether
$\cll a =\Cll{a'}$ or $\cll a \ne \Cll{a'}$,
see Fig.~\ref{CC}.

Case~(iv) leads to a contradiction
when trying to reconstruct the coloring,
which is straightforward but not too short.
We will use more intuitive
arguments. It follows from Lemma~\ref{l:dist}
that there are exactly three $\cll3$-nodes
at distance $2$ from every $\cll 1$-node and exactly two $\cll1$-nodes
at distance $2$ from every $\cll 3$-node.
Considering a sufficiently large square,
we find (Remark~\ref{r:density}) that
the densities $P_1$ and $P_3$ of the colors \cll1 and \cll3
are related as $P_1:P_3 = 2:3$.
On the other hand,
it is easy to see that
there are no nodes that have more than~$2$
neighbors of the same color. It follows
from Lemma~\ref{l:friq} that the quotient of the
densities of any two colors is a power of two, a contradiction.

Now consider case (v).
We denote the colors of the $1$-node neighbors at the picture as
follows, where   $\{\cll a, \Cll{a'}\}$ and
$\{\cll b, \Cll{b'}\}$ are disjoint
but it is possible that
$\cll a =\Cll{a'}$ or $\cll b =\Cll{b'}$:
$$
\def\3{\ccolor{yellow!40!white}{$3$}}
\def\4{\ccolor{orange!30!white}{$4$}}
\def\5{\ccolor{red!20!white}{$5$}}
\def\6{\ccolor{white}{$a$}}
\def\7{\ccolor{white}{$a'$}}
\def\8{\ccolor{white}{$b$}}
\def\9{\ccolor{white}{$b'$}}
\begin{array}{c}
\begin{tikzpicture}
\draw (0,-1) \o\3\6\4\o\0;
\draw (0,-2) \2\9\1\7\2\o;
\draw (0,-3) \o\5\8\3\o\0;
\end{tikzpicture}
\end{array}
\Longrightarrow
\begin{array}{c}
\begin{tikzpicture}
\draw (0,-1) \o\3\6\4\6\o\0;
\draw (0,-2) \2\9\1\7\2\8\1;
\draw (0,-3) \o\5\8\3\9\o\0;
\end{tikzpicture}
\end{array}
\Rightarrow
\begin{array}{c}
\begin{tikzpicture}
\draw (0,-1) \o\3\6\4\6\3\o;
\draw (0,-2) \2\9\1\7\2\8\1;
\draw (0,-3) \o\5\8\3\9\5\o;
\end{tikzpicture}
\end{array}
\Rightarrow\
\cdots
$$
Let us determine the colors of the neighbors of the rightest
\cll2-node in the picture.
The neighbor above has color \cll{a}
because $\{\cll b,\Cll{b'}\}$-nodes have no \cll4-neighbors.
The neighbor below has color \Cll{b'}
because by Lemma~\ref{l:friq} $S_{1b}=S_{1b'}$
implies
$S_{3b}=S_{3b'}$.
So, arguing as in case (iii),
we reconstruct the colors
in the three horizontal rows and then
in the whole grid, obtaining the coloring
corresponding to Fig~\ref{AA} as follows:
$$
{
\begin{array}{r|c|c|c|c|c|c|c|c|c}
 \FfF: & \cll 1 & \cll 2 & \cll{3} & \cll{4}  & \cll{5} & \cll{a} & \Cll{a'} &  \cll{b} & \Cll{b'}  \\ \hline
 \mbox{Fig.~\ref{AA}:} & \cll 6 & \cll 5 &\cll 4  & \cll 9 & \cll 1 &   \cll{7} & \cll{8}  &  \cll{2} & \cll{3} \\
\end{array}.
}
$$

In each of cases (vi), (vii), the \cll1-node has neighbors of four distinct colors,
and with the same strategy as in cases (iii) and (v) the coloring is uniquely reconstructed,
corresponding to Fig.~\ref{DD} or Fig.~\ref{EE}, respectively.
\end{proof}

We see that in subcases (i) and (ii)
in Proposition~\ref{p:6a}
the even colors are pairwise twin
(we remember that twin colors might have different densities
like~\cll1 and~\cll3 in subcase~(i)).
The same is true for the odd colors
(if there are more than one such colors)
because the neighborhoods of all odd nodes are colored
with the same colors.
Proposition~\ref{p:6a}
confirms the claim of Lemma~\ref{l:even}
and
Corollary~\ref{c:evenodd} finalizes the proof
of Theorem~\ref{th:main} in the considered case \ref{case6a}.

\subsubsection[The coloring F contains fragment 12..12]{The coloring $\FfF$ contains the fragment}\label{case6b}
\begin{center}
\begin{tikzpicture}
\draw (0,-0) \1\0\0\0\2;
\draw (0,-1) \0\2\0\1\0;
\end{tikzpicture}
\end{center}
Without loss of generality, we assume that the $\cll{1}$-nodes of the fragment are
$[-2,2]$ and $[1,1]$ and the $\cll{2}$-nodes are $[-1,1]$ and $[2,2]$.
By the condition of the case \ref{case6} the colors of the nodes
$[0,0]$, $[-2,0]$, $[2,0]$, $[3,1]$, $[0,2]$, $[-1,3]$, $[1,3]$, $[0,4]$,
$[-3,1]$
differ from $\cll{1}$ and $\cll{2}$.
Denote by $\cll{3}$ and $\cll{4}$ the colors of $[0,0]$ and $[0,3]$ respectively.
At the pictures, an even node marked by~$*$ is not a $\{\cll{1},\cll{2}\}$-node.
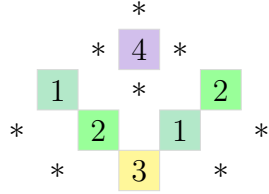
\begin{center}
\begin{tikzpicture}
\def\3{\ccolor{yellow!40!white}{$3$}}
\def\s{++(1,0) node {$*$}}
\draw (0,4) \0\0\0\s\0\0;
\draw (0,3) \0\0\s\4\s\0;
\draw (0,2) \0\1\0\s\0\2;
\draw (0,1) \s\0\2\0\1\0\s;
\draw (0,0) \0\s\0\3\0\s;
\end{tikzpicture}
\end{center}
Using the condition of case \ref{case6}, we see that
$[-1,-1]$ is a $\cll{1}$-node, $[1,-1]$ is a $\cll{2}$-node,
$[0,-2]$ is not a $\{\cll{1},\cll{2}\}$-node.
\begin{center}
\begin{tikzpicture}
\def\3{\ccolor{yellow!40!white}{$3$}}
\def\s{++(1,0)  node[cell] {$*$}}
\def\n{\ccolor{white}{\makebox[0mm][c]{$3$}\makebox[0mm][c]{$/$}}}
\draw (0,6) \0\0\0\s\0;
\draw (0,5) \0\0\s\4\s\0;
\draw (0,4) \0\1\0\s\0\2;
\draw (0,3) \s\0\2\0\1\0\s;
\draw (0,2) \0\s\0\3\0\s;
\draw (0,1) \0\0\1\0\2\0;
\draw (0,0) \0\0\0\s\0\0;
\end{tikzpicture}
\end{center}
From the picture above, we see
 \begin{equation}
  \label{eq:3243}
 d^3(\cll{2},\cll{4}) > 0.
 \end{equation}
 If $[2,-2]$ is not a $\cll{1}$-node,
 then $[3,-1]$ and $[1,-3]$ are $\cll{1}$-nodes, by
 the condition of case~\ref{case6}.
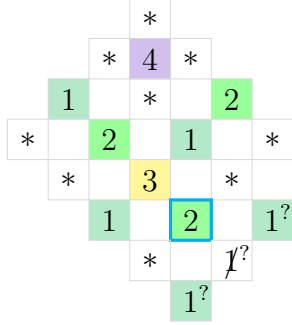
\begin{center}
\begin{tikzpicture}
\def\3{\ccolor{yellow!40!white}{$3$}}
\def\s{++(1,0)  node[cell] {$*$}}
\def\n{\ccolor{white}{\makebox[0mm][c]{$1^{?\!\!\!}$}\makebox[0mm][c]{$/$}}}
\def\l{\ccolor{blue!30!green!30!white}{$1^{?\!\!\!}$}}
\draw (0,6) \0\0\0\s\0;
\draw (0,5) \0\0\s\4\s\0;
\draw (0,4) \0\1\0\s\0\2;
\draw (0,3) \s\0\2\0\1\0\s;
\draw (0,2) \0\s\0\3\0\s;
\draw (0,1) \0\0\1\0\2\0\l; \draw (5,1) node [cell, very thick, draw=cyan] {};
\draw (0,0) \0\0\0\s\0\n;
\draw (0,-1)\0\0\0\0\l\0;
\end{tikzpicture}
\end{center}
But then we have a contradiction with~\eqref{eq:3243},
because there is no room to place a $\cll{4}$-node at distance $3$ from
the $\cll{2}$-node $[1,-1]$ (taking into account that the $\cll{4}$-nodes
are of type~$0{:}0$).
Hence, the color of $[2,-2]$ is $\cll{1}$.
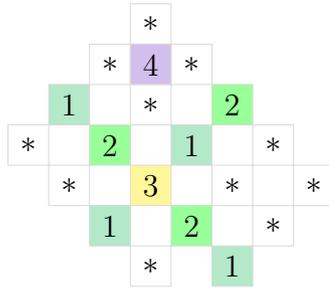
\begin{center}
\begin{tikzpicture}
\def\3{\ccolor{yellow!40!white}{$3$}}
\def\s{++(1,0)  node[cell] {$*$}}
\def\n{\ccolor{white}{\makebox[0mm][c]{$1^{?\!\!\!}$}\makebox[0mm][c]{$/$}}}
\def\l{\ccolor{blue!30!green!30!white}{$1^{?\!\!\!}$}}
\draw (0,6) \0\0\0\s\0;
\draw (0,5) \0\0\s\4\s\0;
\draw (0,4) \0\1\0\s\0\2;
\draw (0,3) \s\0\2\0\1\0\s;
\draw (0,2) \0\s\0\3\0\s\0\s;
\draw (0,1) \0\0\1\0\2\0\s; ;
\draw (0,0) \0\0\0\s\0\1;
\end{tikzpicture}
\end{center}
We see that $d^2(\cll{1},\cll{3})=2$,
$d^2(\cll{2},\cll{3})=2$, and  $d^2(\cll{2},\cll{1})=4$.
Applying the last two  equalities to $[2,2]$
and noting that a $\cll{3}$-node cannot be neighbor to a type $0{:}0$ node,
we find that there is only one way to color
the nodes $[2,4]$, $[3,3]$, and $[4,2]$:
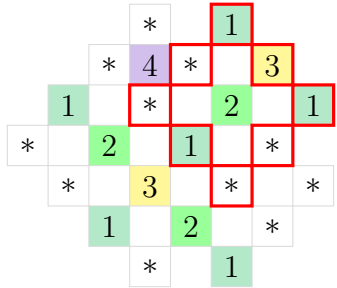
\begin{center}
\begin{tikzpicture}
\def\3{\ccolor{yellow!40!white}{$3$}}
\def\s{++(1,0)  node[cell] {$*$}}
\draw (0,6) \0\0\0\s\0\1;
\draw (0,5) \0\0\s\4\s\0\3;
\draw (0,4) \0\1\0\s\0\2\0\1;
\draw (0,3) \s\0\2\0\1\0\s;
\draw (0,2) \0\s\0\3\0\s\0\s;
\draw (0,1) \0\0\1\0\2\0\s; ;
\draw (0,0) \0\0\0\s\0\1;
\draw [mrk/.style={cell, very thick, draw=red!40}] (8,4) node [mrk] {}
(7,3) node [mrk] {}
(6,2) node [mrk] {}
(7,5) node [mrk] {}
(5,3) node [mrk] {}
(6,6) node [mrk] {}
(5,5) node [mrk] {}
(4,4) node [mrk] {};
\end{tikzpicture}
\end{center}
Directly from the condition of case \ref{case6}, we find two more $\cll{2}$-nodes:
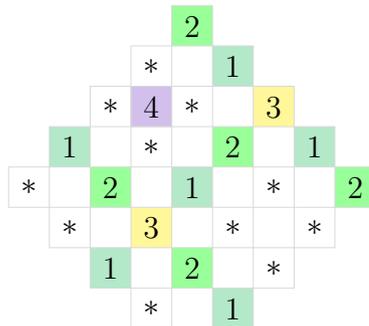
\begin{center}
\begin{tikzpicture}
\def\3{\ccolor{yellow!40!white}{$3$}}
\def\s{++(1,0)  node[cell] {$*$}}
\draw (0,7) \0\0\0\0\2\0;
\draw (0,6) \0\0\0\s\0\1;
\draw (0,5) \0\0\s\4\s\0\3;
\draw (0,4) \0\1\0\s\0\2\0\1;
\draw (0,3) \s\0\2\0\1\0\s\0\2;
\draw (0,2) \0\s\0\3\0\s\0\s;
\draw (0,1) \0\0\1\0\2\0\s; ;
\draw (0,0) \0\0\0\s\0\1;
\end{tikzpicture}
\end{center}
Now, by analogy, we can recover the placement of all the $\{\cll{1},\cll{2},\cll{3}\}$-nodes.
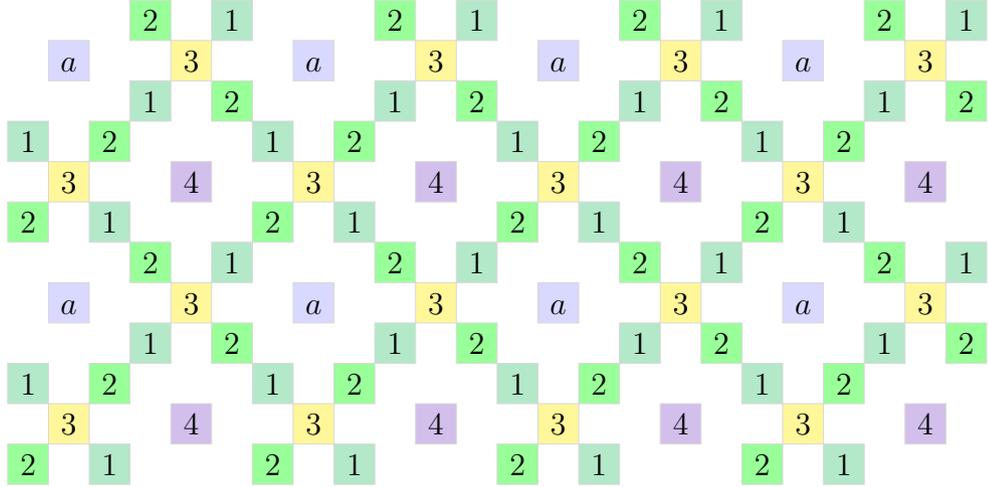
\begin{center}
\begin{tikzpicture}
\def\3{\ccolor{yellow!40!white}{$3$}}
\def\a{\ccolor{blue!15!white}{$a$}}
\def\b{++(1,0)  node[cell] {$b$}}
\def\c{++(1,0)  node[cell] {$c$}}
\def\d{++(1,0)  node[cell] {$d$}}
\def\e{++(1,0)  node[cell] {$e$}}
\draw (0,11)\0\0\0\2\0\1\0\0\0\2\0\1\0\0\0\2\0\1\0\0\0\2\0\1;
\draw (0,10)\0\a\0\0\3\0\0\a\0\0\3\0\0\a\0\0\3\0\0\a\0\0\3\0;
\draw (0,9) \0\0\0\1\0\2\0\0\0\1\0\2\0\0\0\1\0\2\0\0\0\1\0\2;
\draw (0,8) \1\0\2\0\0\0\1\0\2\0\0\0\1\0\2\0\0\0\1\0\2\0\0\0;
\draw (0,7) \0\3\0\0\4\0\0\3\0\0\4\0\0\3\0\0\4\0\0\3\0\0\4\0;
\draw (0,6) \2\0\1\0\0\0\2\0\1\0\0\0\2\0\1\0\0\0\2\0\1\0\0\0;
\draw (0,5) \0\0\0\2\0\1\0\0\0\2\0\1\0\0\0\2\0\1\0\0\0\2\0\1;
\draw (0,4) \0\a\0\0\3\0\0\a\0\0\3\0\0\a\0\0\3\0\0\a\0\0\3\0;
\draw (0,3) \0\0\d\1\0\2\0\0\0\1\0\2\0\0\0\1\0\2\0\0\0\1\0\2;
\draw (0,2) \1\b\2\e\0\0\1\0\2\0\0\0\1\0\2\0\0\0\1\0\2\0\0\0;
\draw (0,1) \c\3\c\0\4\0\0\3\0\0\4\0\0\3\0\0\4\0\0\3\0\0\4\0;
\draw (0,0) \2\b\1\0\0\0\2\0\1\0\0\0\2\0\1\0\0\0\2\0\1\0\0\0;
\end{tikzpicture}
\end{center}
Moreover, since $\dd^3(\cll{1},\cll{4})=3$ or $\dd^3(\cll{1},\cll{4})=6$, we see that there are
at most two colors of odd type $0{:}0$ nodes
and these nodes are colored
periodically with periods $[6,0]$ and $[0,6]$.
Applying the definition of a perfect coloring to the neighborhoods of $\cll{1}$- and $\cll{2}$-nodes,
we conclude that the opposite neighbors of a $\cll{3}$-node must have the same color. We separate five subcases, according to
the equalities between the colors of the marked nodes in the last picture:
\begin{itemize}
 \item[\rm (i)]  $\cll a = \cll 4$, $\cll b = \cll c$, $\cll d = \cll e$;
 \item[\rm (ii)]  $\cll a = \cll 4$, $\cll b = \cll c$, $\cll d \ne \cll e$;
 \item[\rm (iii)] $\cll a = \cll 4$, $\cll b \ne \cll c$, $\cll d = \cll e$;
 \item[\rm (iv)] $\cll a = \cll 4$, $\cll b \ne \cll c$, $\cll d \ne \cll e$;
 \item[\rm (v)]  $\cll a \ne \cll 4$  (it follow that $\cll b \ne \cll c$ and $\cll d \ne \cll e$).
\end{itemize}

Leaving the further arguments as an exercise, we formulate the result.

\begin{proposition}
In subcase \ref{case6b}, the coloring $\FfF$ is equivalent
to one of the following five colorings:
\begin{itemize}
 \item[\rm (i)] the coloring shown at Fig.~\ref{FF},
 with merged colors \cll6 and \cll7;
 \item[\rm (ii)] the coloring shown at Fig.~\ref{FF};
 \item[\rm (iii)] the coloring shown at Fig.~\ref{GG},
 with merged colors \cll7 and \cll8;
 \item[\rm (iv)] the coloring shown at Fig.~\ref{GG};
 \item[\rm (v)] the coloring shown at Fig.~\ref{HH}.
\end{itemize}
\end{proposition}

\subsubsection[The coloring F contains fragment 21..12..21]{The coloring $\FfF$ contains the following fragment:}
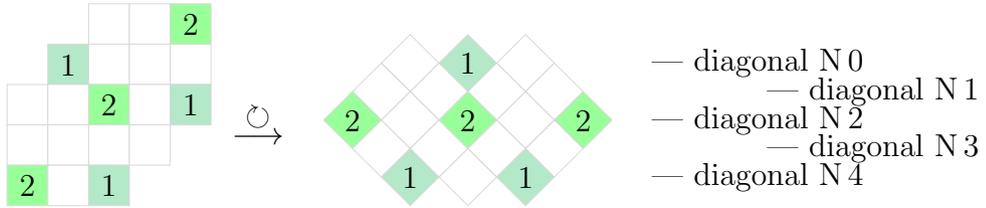
\label{case6c}
$$
\begin{tikzpicture}
\draw (2, 4)     \o\o\2;
\draw (1, 3)   \1\o\o\o;
\draw (0, 2) \o\o\2\o\1;
\draw (0, 1) \o\o\o\o  ;
\draw (0, 0) \2\o\1    ;
\end{tikzpicture}
\ \ \raisebox{2em}{$\stackrel{\displaystyle\circlearrowright}{\longrightarrow    }$}\
\begin{tikzpicture}
\def\1{\dcolor{\colONE}{$1$}}
\def\2{\dcolor{\colTWO}{$2$}}
\def\o{\dcolor{white}{$\,$}}
\draw (1\DK, 4\DK)       \o\1\o   ++(2\DK,0) node[anchor=west] {--- diagonal $D_0$};
\draw (0.5\DK, 3.5\DK)  \o\o\o\o ++(3.5\DK, 0) node[anchor=west] {--- diagonal $D_1$};
\draw (0\DK, 3\DK)     \2\o\2\o\2 ++(1\DK,0) node[anchor=west] {--- diagonal $D_2$};
\draw (0.5\DK, 2.5\DK)  \o\o\o\o ++(3.5\DK, 0) node[anchor=west] {--- diagonal $D_3$};
\draw (1\DK, 2\DK)       \1\o\1   ++(2\DK,0) node[anchor=west] {--- diagonal $D_4$};
\end{tikzpicture}
$$
In this subsection, we will draw the grid $45^\circ$ rotated
and give indexed names to R-diagonals as illustrated in the figure above.

By the hypothesis of case~\ref{case6}, all neighbors
of a $\{\cll{1},\cll{2}\}$-node are of type $1{:}1$.
Subsequently applying this condition to the nodes
of the diagonals $D_1$ and $D_3$, we uniquely reconstruct
all $\{\cll{1},\cll{2}\}$-nodes
in the diagonals $D_0$, $D_2$, and $D_4$:
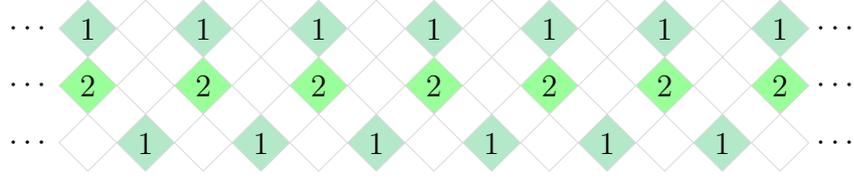
\begin{center}\begin{tikzpicture}
\def\1{\dcolor{\colONE}{$1$}}
\def\2{\dcolor{\colTWO}{$2$}}
\def\o{\dcolor{white}{$\,$}}
\def\t{ ++(1\DK, 0) node {$\cdots$}}
\draw (0\DK, 4\DK) \t\1\o\1\o\1\o\1\o\1\o\1\o\1\t ;
\draw (0\DK, 3\DK) \t\2\o\2\o\2\o\2\o\2\o\2\o\2\t ;
\draw (0\DK, 2\DK) \t\o\1\o\1\o\1\o\1\o\1\o\1\o\t ;
\end{tikzpicture}\end{center}

\begin{claim}
  The diagonal $D_0$ is binary.
\end{claim}
\begin{proof}
  Assume the contrary. Then $D_0$ contains
  two nodes of different colors,
  say~$\cll{3}$ and~$\cll{4}$,
  at distance~$4$ from each other.
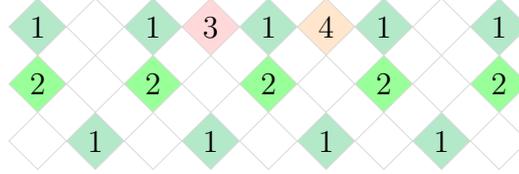
\begin{center}\begin{tikzpicture}
\def\1{\dcolor{\colONE}{$1$}}
\def\2{\dcolor{\colTWO}{$2$}}
\def\3{\dcolor{red!15!white}{$3$}}
\def\4{\dcolor{orange!20!white}{$4$}}
\def\o{\dcolor{white}{$\,$}}
\def\d{\dcolor{white}{$\cdots$}}
\draw (0\DK, 4\DK) \1\o\1\3\1\4\1\o\1 ;
\draw (0\DK, 3\DK) \2\o\2\o\2\o\2\o\2 ;
\draw (0\DK, 2\DK) \o\1\o\1\o\1\o\1\o ;
\end{tikzpicture}\end{center}
  Considering the central $\cll{2}$-node
  of the diagonal~$D_2$ in the diagram above, we see that
  $d^2(\cll{2},\cll{i})$
  is odd for $\cll{i}=\cll{3},\cll{4}$
  and even for any other $\cll{i}$.
  Since $\cll{1}$ and $\cll{2}$ are twins,
  the same holds for $d^2(\cll{1},\cll{i})$.
  Considering the other $\cll{2}$-nodes
  of $D_2$, we see that
  the colors of the diagonal~$D_0$ alternate like
  $\cll{1}$, $\cll{3}$, $\cll{1}$, $\cll{4}$, $\cll{1}$, $\cll{3}$, $\cll{1}$, $\cll{4}$, \ldots
  (we will talk about a $\dll{1}{3}\dll{1}{4}$-diagonal
  in this case).
  Considering the $\cll{1}$-nodes
  of the diagonal~$D_4$ (recall that $d^2(\cll{1},\cll{3})$ and $d^2(\cll{1},\cll{3})$ are odd), we see that the diagonal~$D_6$
  is a $\dll{2}{3}\dll{2}{4}$-diagonal,
  and the situation is, up to the reflection,
  as in the following figure:
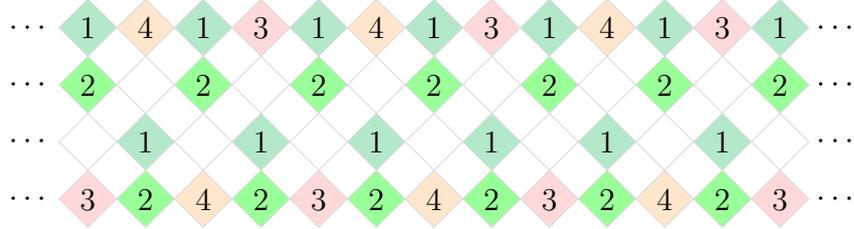
\begin{center}\begin{tikzpicture}
\def\1{\dcolor{\colONE}{$1$}}
\def\2{\dcolor{\colTWO}{$2$}}
\def\3{\dcolor{red!15!white}{$3$}}
\def\4{\dcolor{orange!20!white}{$4$}}
\def\o{\dcolor{white}{$\,$}}
\def\t{ ++(1\DK, 0) node {$\cdots$}}
\draw (0\DK, 4\DK) \t\1\4\1\3\1\4\1\3\1\4\1\3\1\t ;
\draw (0\DK, 3\DK) \t\2\o\2\o\2\o\2\o\2\o\2\o\2\t ;
\draw (0\DK, 2\DK) \t\o\1\o\1\o\1\o\1\o\1\o\1\o\t ;
\draw (0\DK, 1\DK) \t\3\2\4\2\3\2\4\2\3\2\4\2\3\t ;
\end{tikzpicture}\end{center}
As we see from the diagonal~$D_0$, $d^2(\cll{1},\cll{3})$
and $d^2(\cll{1},\cll{4})$ are not less than~$2$.
Since they are known to be odd, at least one of these two values is $3$.
Let, w.l.o.g., $d^2(\cll{1},\cll{3})=3$.
Then $d^2(\cll{2},\cll{3})=3$ too.
Considering these values for the bold $\cll{1}$-node
$  \bar u+[3,1]$
and $\cll{2}$-node $  \bar u+[1,1]$
from the diagonals~$D_2$ and~$D_4$, see the diagram
\begin{center}\begin{tikzpicture}
\def\1{\dcolor{\colONE}{$1$}}
\def\2{\dcolor{\colTWO}{$2$}}
\def\9{\dcolor{\colONE}{$\mathbf{1}$}}
\def\8{\dcolor{\colTWO}{$\mathbf{2}$}}
\def\3{\dcolor{red!15!white}{$3$}}
\def\4{\dcolor{orange!20!white}{$4$}}
\def\a{\dcolor{white}{$\bar{u}$}}
\def\b{\dcolor{white}{$\bar{v}$}}
\def\B{\dcolor{white}{$\bar{w}$}}
\def\d{\dcolor{white}{$\cdot$}}
\def\o{\dcolor{white}{$\,$}}
\def\t{ ++(1\DK, 0) node {$\cdots$}}
\draw (0\DK, 4\DK) \t\1\4\1\3\1\4\1\3\1\4\1\3\1\t ;
\draw (0\DK, 3\DK) \t\2\o\2\a\8\o\2\o\2\d\2\o\2\t ;
\draw (0\DK, 2\DK) \t\B\1\o\1\o\9\b\1\o\1\o\1\d\t ;
\draw (0\DK, 1\DK) \t\3\2\4\2\3\2\4\2\3\2\4\2\3\t ;
\draw [mrk/.style={dell, very thick, draw=red!30!white}]
(5\DK, 2\DK) node [mrk] {}
(5\DK, 3\DK) node [mrk] {}
(5\DK, 4\DK) node [mrk] {}
(6\DK, 4\DK) node [mrk] {}
(7\DK, 2\DK) node [mrk] {}
(7\DK, 4\DK) node [mrk] {};
\draw [mrk/.style={dell, very thick, draw=blue!30!white}]
(6\DK, 1\DK) node [mrk] {}
(6\DK, 2\DK) node [mrk, draw=purple!30!white] {}
(6\DK, 3\DK) node [mrk] {}
(7\DK, 1\DK) node [mrk] {}
(7\DK, 3\DK) node [mrk, draw=purple!30!white] {}
(8\DK, 1\DK) node [mrk] {}
(8\DK, 2\DK) node [mrk] {}
(8\DK, 3\DK) node [mrk] {};
\end{tikzpicture}\end{center}
we find the following:
\begin{claim}\label{cl:uvw}
A node~$\bar u$ from~$D_2$
has color~$\cll{3}$ if and only if $\bar v$ is a $\cll{3}$-node too,
where  $\bar v= \bar u+[4,2]$
(similarly, for $\bar w= \bar u-[4,2]$).
\end{claim}
Starting from an arbitrary $\cll{3}$-node in~$D_2$ or~$D_4$
and applying Claim~~\ref{cl:uvw},
we inductively derive that
the $\cll{3}$-nodes of the diagonal~$D_2$,
as well as of~$D_4$, are placed with period $[6,6]$:
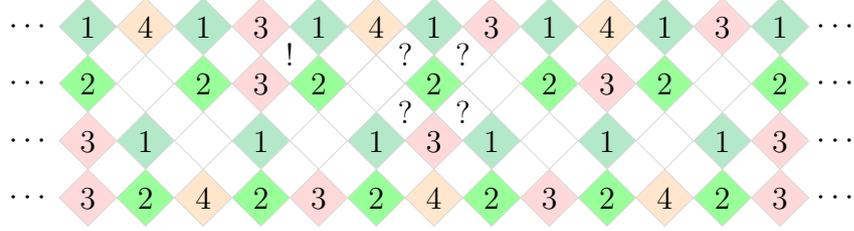
\begin{center}\begin{tikzpicture}
\def\1{\dcolor{\colONE}{$1$}}
\def\2{\dcolor{\colTWO}{$2$}}
\def\8{\dcolor{green!40!white}{$\mathbf 2$}}
\def\3{\dcolor{red!15!white}{$3$}}
\def\4{\dcolor{orange!20!white}{$4$}}
\def\a{\dcolor{white}{$\mathbf{!}$}}
\def\b{\dcolor{white}{$\mathbf{?}$}}
\def\o{\dcolor{white}{$\ $}}
\def\t{ ++(1\DK, 0) node {$\cdots$}}
\draw (0\DK, 4\DK) \t\1\4\1\3\1\4\1\3\1\4\1\3\1\t;
\draw (4.5\DK, 3.5\DK)      \a\o\b\b;
\draw (0\DK, 3\DK) \t\2\o\2\3\2\o\2\o\2\3\2\o\2\t;
\draw (6.5\DK, 2.5\DK)          \b\b;
\draw (0\DK, 2\DK) \t\3\1\o\1\o\1\3\1\o\1\o\1\3\t;
\draw (0\DK, 1\DK) \t\3\2\4\2\3\2\4\2\3\2\4\2\3\t;
\end{tikzpicture}\end{center}
Now we see an obvious contradiction with the definition
of a perfect coloring: there are nodes with the neighbors of colors
$\cll{1}$, $\cll{2}$, $\cll{3}$, $\cll{3}$; but there is
a $\cll{2}$-node that is not adjacent with such a node.
The contradiction proves the statement.
\end{proof}

Thus, there is a binary diagonal.
Well, it is not a $\dll{1}{2}$-diagonal,
but nevertheless, its colors are twin and
hence the situation is described in some other considered case,
\ref{case1}, \ref{case2}, \ref{case3}, or \ref{case6d}.

\subsubsection[The coloring F contains
a 12-diagonal]{The coloring $\FfF$ contains
a $\dll{1}{2}$-diagonal}\label{case6d}

Without loss of generality, we assume that $\FfF$ contains
a $\dll{1}{2}$-R-diagonal
through $[0,0]$.
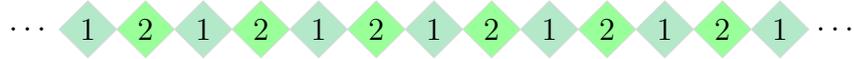
\begin{center}\begin{tikzpicture}
\def\1{\dcolor{\colONE}{$1$}}
\def\2{\dcolor{\colTWO}{$2$}}
\def\t{ ++(1\DK, 0) node {$\cdots$}}
\draw (0\DK, 1\DK) \t\1\2\1\2\1\2\1\2\1\2\1\2\1\t;
\end{tikzpicture}\end{center}
\begin{proposition}\label{p:44}
Either the coloring is periodic with period $[2,2]$ or $[4,4]$,
or the two diagonals at distance $4$ from the given ``main'' $\dll{1}{2}$-diagonal
are $\dll{1}{2}$-diagonals too.
\end{proposition}
\begin{proof}
We first note that if some nodes $[x,x+1]$ and $[x+1,x]$
have colors $\cll{a}$ and $\cll{b}$, then $[x+2,x+3]$ and $[x+3,x+2]$
have, respectively, the colors $\cll{a}$ and $\cll{b}$ or $\cll{b}$ and $\cll{a}$;
the same is true for $[x-2,x-1]$ and $[x-1,x-2]$:
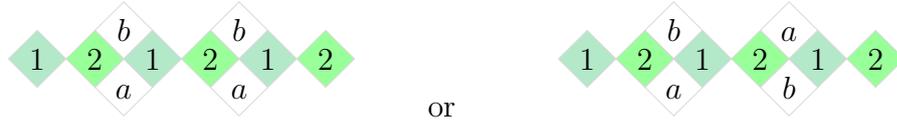
\begin{center}
\def\1{\dcolor{\colONE}{$1$}}
\def\2{\dcolor{\colTWO}{$2$}}
\def\0{ ++(1\DK, 0) }
\def\a{\dcolor{white}{${a}$}}
\def\b{\dcolor{white}{${b}$}}
\begin{tabular}{c}
 \begin{tikzpicture}
\draw (1.5\DK, 0.5\DK) \a\0\a;
\draw (1.5\DK, 1.5\DK) \b\0\b;
\draw (0\DK, 1\DK) \1\2\1\2\1\2;
\end{tikzpicture}
\end{tabular}
\qquad or \qquad
\begin{tabular}{c}
\begin{tikzpicture}
\draw (1.5\DK, 0.5\DK) \a\0\b;
\draw (1.5\DK, 1.5\DK) \b\0\a;
\draw (0\DK, 1\DK) \1\2\1\2\1\2;
\end{tikzpicture}
\end{tabular}
\end{center}
If both diagonals neighbor to the considered diagonal are periodic with period $[4,4]$,
then the coloring is periodic too. Otherwise there is one of the following two fragments:
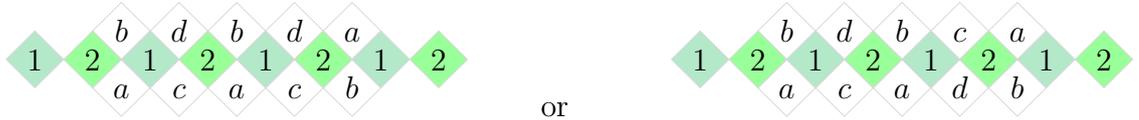
\begin{center}
\def\1{\dcolor{\colONE}{$1$}}
\def\2{\dcolor{\colTWO}{$2$}}
\def\0{ ++(1\DK, 0) }
\def\a{\dcolor{white}{${a}$}}
\def\b{\dcolor{white}{${b}$}}
\def\c{\dcolor{white}{${c}$}}
\def\d{\dcolor{white}{${d}$}}
\begin{tabular}{c}
\begin{tikzpicture}
\draw (1.5\DK, 0.5\DK) \a\c\a\c\b;
\draw (1.5\DK, 1.5\DK) \b\d\b\d\a;
\draw (0\DK, 1\DK) \1\2\1\2\1\2\1\2;
\end{tikzpicture}
\end{tabular}\quad or \quad
\begin{tabular}{c}
\begin{tikzpicture}
\draw (1.5\DK, 0.5\DK) \a\c\a\d\b;
\draw (1.5\DK, 1.5\DK) \b\d\b\c\a;
\draw (0\DK, 1\DK) \1\2\1\2\1\2\1\2;
\end{tikzpicture}
\end{tabular}
\end{center}
where $\cll{a}\neq\cll{b}$ but other equalities are allowed.
We see that in both cases $d^2(\cll{c},\cll{a})\geq 4$ and $d^2(\cll{c},\cll{b})\geq 3$;
hence, there are $\{\cll{a},\cll{b}\}$-nodes at distance $3$ from the main diagonal.
It follows that there are $\{\cll{1},\cll{2}\}$-nodes at distance $4$ from the main diagonal.
By the hypothesis of case~\ref{case6}, such a node must belong to a $\dll{1}{2}$-diagonal.
Similarly, there is an ``opposite'', with respect to the main diagonal, $\dll{1}{2}$-diagonal.
\end{proof}
We separate Case~\ref{case6d} into two subcases.


\paragraph
[The coloring F is not {[4,4]}-periodic]
{The coloring $\FfF$ is not periodic with period $[4,4]$ or $[2,2]$}
\label{case6d1}

By Proposition~\ref{p:44}, the $\dll{1}{2}$-diagonals in $\FfF$ occur
with period $[2,-2]$.
The other even nodes also constitute
R-diagonals
that occur with period $[2,-2]$.
\begin{claim}\label{cl:chain}
  Every even
  R-diagonal
  is colored in such a way that the color of every
  of its nodes ${\bar x}$ uniquely determines the unordered pair of colors of nearest diagonal nodes
  ${\bar x}-[1,1]$ and ${\bar x}+[1,1]$ (in other words, we have a perfect coloring of the diagonal
  considered as a chain graph).
\end{claim}
\begin{proof}
For every two colors $\cll{i}$ and $\cll{j}$ occuring on the considered diagonal, the value
$d^2(\cll{i},\cll{j})$ does not depend on the choice of the initial $\cll{i}$-node ${\bar x}$. Moreover,
it coincides to the number of $\cll{j}$-nodes
among ${\bar x}-[1,1]$ and ${\bar x}+[1,1]$ (indeed, the two R-diagonals at distance $2$ from ${\bar x}$ have no $\cll{j}$-nodes).
\end{proof}

\begin{proposition}\label{p:6d1}
In subcase \ref{case6d1},
one of the following takes place.

{\rm (i)} The hypothesis and the conclusion
of Lemma~\ref{l:even} {(and hence, of Corollary~\ref{c:evenodd})} are satisfied {both for even and odd nodes.}

{\rm (ii)} The coloring is equivalent,
up to shifting  $\dll{1}{2}$- and $\dll{3}{4}$- diagonals,
to the coloring in Fig.~\ref{JJ}.

\end{proposition}
\begin{proof}
We already know that
 $\dll{1}{2}$-R-diagonals
occur at distance $4$ from each other.
Let us focus on an R-diagonal at distance~$2$
from a $\dll{1}{2}$-diagonal. As shown above
(see the proof of Proposition~\ref{p:44}),
there is a fragment like the following:
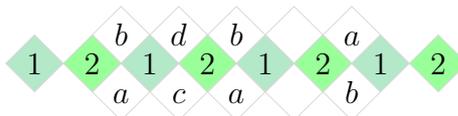
\def\s{\dcolor{black!4}{$*$}}
\def\t{ ++(1\DK, 0) node {$\cdots$}}
\def\o{\dcolor{white}{$\ $}}
\dfc{cD}{2e5c3a}\dfc{cD}{009527}\def\1{\Dcolor{cD!25}{cD!40}{$1$}}
\dfc{cC}{65a370}\dfc{cC}{18b234}\def\2{\Dcolor{green!45!cC!30}{cC!45}{$2$}}
\dfc{cA}{914019}\dfc{cA}{914019}\def\a{\Dcolor{red!8!cA!25}{red!5!cA!45}{$a$}}
\dfc{cB}{845042}\dfc{cB}{915319}\def\b{\Dcolor{green!8!cB!30}{green!5!cB!50}{$b$}}
\dfc{cF}{d36808}
\dfc{cE}{bcaa1f}\dfc{cE}{e7cc00}\def\c{\DDolor{cE!65}{cE!65}{210}{$c$}}
\dfc{cH}{f0ab3a}\dfc{cH}{ffb131}\def\d{\DDolor{cF!40!cH!60}{cF!30!cH!60}{210}{$d$}}
\dfc{cG}{a08c04}\dfc{cG}{ffde00}\def\f{\DDolor{cG!35}{cG!35}{210}{$e$}}
\dfc{cF}{d36808}\dfc{cF}{f97600}\def\e{\DDolor{cF!90!cH!35}{cF!70!cH!35}{210}{$f$}}
\begin{center}
\begin{tikzpicture}
\draw (1.5\DK, 0.5\DK) \a\c\a\o\b;
\draw (1.5\DK, 1.5\DK) \b\d\b\o\a;
\draw (0\DK, 1\DK) \1\2\1\2\1\2\1\2;
\end{tikzpicture}
\end{center}
Now consider the neighbors of the $\cll{a}$-nodes.
If the fragment expands like
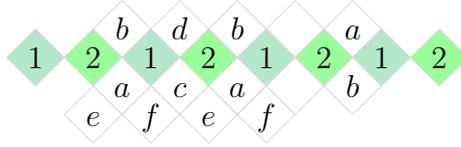
\def\g{\DDolor{black!10}{black!4}{210}{$e$}}
\def\h{\DDolor{black!15}{black!4}{210}{$f$}}
\begin{center}
\begin{tikzpicture}
\draw (1.5\DK, 1.5\DK)  \b\d\b\o\a;
\draw (0\DK, 1\DK)   \1\2\1\2\1\2\1\2;
\draw (1.5\DK, 0.5\DK)  \a\c\a\o\b;
\draw (1\DK, 0\DK)     \g\h\g\h;
\end{tikzpicture}
\end{center}
for some (maybe equal) $\cll{e}$ and $\cll{f}$,
then by Claim~\ref{cl:chain} there is an $\dll{e}{f}$-diagonal.
It follows immediately that all even R-diagonals are either
$\dll{1}{2}$ or $\dll{e}{f}$.
Then all colors of nodes of the same parity are mutually twin,
and p.\,(i) of the proposition statement takes place.

It remains to consider the following case,
where $\cll{f}\ne\cll{e}$:
$$
\begin{tikzpicture}
\draw (1.5\DK, 1.5\DK)  \b\d\b\o\a;
\draw (0\DK, 1\DK)   \1\2\1\2\1\2\1\2;
\draw (1.5\DK, 0.5\DK)  \a\c\a\s\b;
\draw (1\DK, 0\DK)     \e\f\f\e;
\end{tikzpicture}
$$
Since the $\cll{c}$-nodes have no $\cll{f}$-neighbors,
the node marked by $*$ at the figure above
is a $\cll{d}$-node:
$$
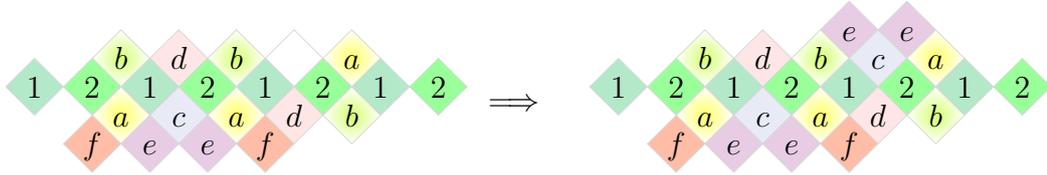
\begin{tikztab}
\phantom{\draw (4\DK, 2\DK) \f;}
\draw (4\DK, 2\DK)           \dfant;
\draw (1.5\DK, 1.5\DK)  \b\d\b\o\a;
\draw (0\DK, 1\DK)   \1\2\1\2\1\2\1\2;
\draw (1.5\DK, 0.5\DK)  \a\c\a\d\b;
\draw (1\DK, 0\DK)     \e\f\f\e;
\end{tikztab}
\ \ \Longrightarrow\ \
\begin{tikztab}
\draw (4\DK, 2\DK)           \f\f;
\draw (1.5\DK, 1.5\DK)  \b\d\b\c\a;
\draw (0\DK, 1\DK)   \1\2\1\2\1\2\1\2;
\draw (1.5\DK, 0.5\DK)  \a\c\a\d\b;
\draw (1\DK, 0\DK)     \e\f\f\e;
\end{tikztab}
$$
Since $d^2(\cll{e},\cll{f})=2$, we get
$$
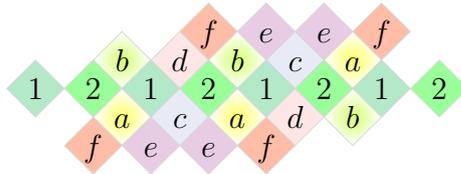
\begin{tikzpicture}
\draw (3\DK, 2\DK)         \e\f\f\e;
\draw (1.5\DK, 1.5\DK)  \b\d\b\c\a;
\draw (0\DK, 1\DK)   \1\2\1\2\1\2\1\2;
\draw (1.5\DK, 0.5\DK)  \a\c\a\d\b;
\draw (1\DK, 0\DK)     \e\f\f\e;
\end{tikzpicture}
$$
Now we see that a $\cll{b}$-node
has neighbors of colors $\cll{e}$ and $\cll{f}$.
By Claim~\ref{cl:chain}, a $\dll{e}{e}\dll{f}{f}$-diagonal
is uniquely reconstructed:
$$
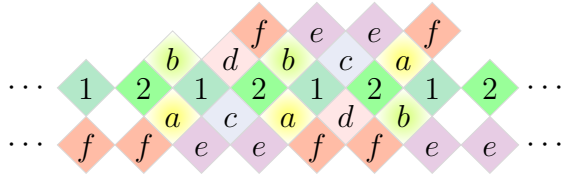
\begin{tikzpicture}
\draw (3\DK, 2\DK)         \e\f\f\e;
\draw (1.5\DK, 1.5\DK)  \b\d\b\c\a;
\draw (-1\DK, 1\DK)\t\1\2\1\2\1\2\1\2\t;
\draw (1.5\DK, 0.5\DK)  \a\c\a\d\b;
\draw (-1\DK, 0\DK)\t\e\e\f\f\e\e\f\f\t;
\end{tikzpicture}
$$
It is easy to see now that the whole coloring, up to equivalence
and shifting binary diagonals, is as in p.\,(ii) of the proposition statement.
\end{proof}

\paragraph
[The coloring F is {[4,4]}-periodic]
{The coloring $\FfF$ is periodic with period $[4,4]$}
\label{case6d2}

In this subcase, $\FfF$ has a $\dll{1}{2}$-R-diagonal and is periodic
with period~$[4,4]$.
We further divide this subcase depending on the colors of nodes
at distance~$1$ and~$2$ from the $\dll{1}{2}$-diagonal.

\noindent\subparagraph%
[There is ab-diagonal neighbor to 12-diagonal]
{There is an $\dll{a}{b}$-diagonal neighbor to a $\dll{1}{2}$-diagonal}
\label{case6d2i}
In this subcase, the conclusion of Theorem~\ref{th:main}
holds by Proposition~\ref{p:twobin}.

\subparagraph%
[There is ab-diagonal at distance 2 from 12-diagonal]
{There is an $\dll{a}{b}$-diagonal at distance $2$ from a $\dll{1}{2}$-diagonal}
\label{case6d2ii}
Let us consider the diagonal between the $\dll{a}{b}$-diagonal and the $\dll{1}{2}$-diagonal.
If it is binary or one-color, then subcase~\ref{case6d2i} takes place.
Otherwise, it contains two nodes at distance~$4$ colored with different colors,
say~$\cll c$ and~$\cll d$:

\definecolor{yell}{rgb}{0.98,0.98,0.45}
\def\1{\dcolor{\colONE}{$1$}}
\def\2{\dcolor{\colTWO}{$2$}}
\def\0{ ++(1\DK, 0) }
\def\o{\dcolor{white}{$\ $}}
$$
\def\a{\dcolor{orange!20!white}{$a$}}
\def\b{\dcolor{red!25!orange!20!white}{$b$}}
\def\A{\dcolor{orange!20!white}{$\underline a$}}
\def\B{\dcolor{red!25!orange!20!white}{$\underline b$}}
\def\c{\Dcolor{green!15!yell}{white}{$c$}}
\def\d{\Dcolor{yell}{white}{$d$}}
\def\t{ ++(1\DK, 0) node {$\cdots$}}
\begin{tikztab}
\draw (1\DK, 2\DK)     \dfant;
\draw (0.5\DK, 1.5\DK)\t\b\A\B\a\t;
\draw (1\DK, 1\DK)     \0\c\0\d;
\draw (0.5\DK, 0.5\DK)\t\1\2\1\2\t;
\end{tikztab}
\Longrightarrow
\begin{tikztab}
\draw (1\DK, 2\DK)     \0\d\0\c;
\draw (0.5\DK, 1.5\DK)\t\b\a\b\a\t;
\draw (1\DK, 1\DK)     \0\c\0\d;
\draw (0.5\DK, 0.5\DK)\t\1\2\1\2\t;
\end{tikztab}
$$
Since $\cll a$ and $\cll b$ are either the same or twins,
we find that there are $\{\cll c,\cll d\}$-nodes at distance $3$ from
the $\dll{1}{2}$-diagonal.
Hence, there are $\{\cll 1,\cll 2\}$-nodes at distance $4$ from
the $\dll{1}{2}$-diagonal; of course,
{they belong to a $\dll{1}{2}$-diagonal
because neighbors of $\{\cll1,\cll2\}$ nodes are of type~$1{:}1$.}
Similarly, there is an $\dll{a}{b}$-diagonal at distance~$4$ from
the considered $\dll{a}{b}$-diagonal.
Now we can see that all odd colors are mutually twin and we get the situation
described in {Proposition~\ref{p:6d1}(i)},
with the only nonessential difference that now there is the period~$[4,4]$.

In the following subcases we automatically assume that the R-diagonals
at distance~$2$ from every $\dll{1}{2}$-diagonal are not one-color or binary.

\subparagraph
[There is aaab-diagonal neighbor to 12-diagonal]
{There is an $\dll{a}{a}\dll{a}{b}$-diagonal
neighbor to a $\dll{1}{2}$-diagonal, $\cll{a}\ne\cll{b}$}\label{case6d2iii}
\def\0{ ++(1\DK, 0) }
\def\1{\dcolor{\colONE}{$1$}}
\def\2{\dcolor{\colTWO}{$2$}}
\def\a{\Dcolor{green!15!yell}{white}{$a$}}
\def\b{\Dcolor{yell}{white}{$b$}}
\def\t{ ++(1\DK, 0) node {$\cdots$}}
\def\c{\Dcolor{red!15!yell!80!white  }{white}{$c$}}
\def\s{\dcolor{white}{$*$}}
\def\o{\dcolor{white}{}}
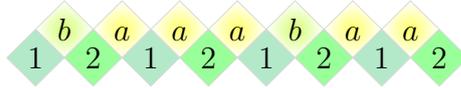
$$
\begin{tikzpicture}
\draw (0.5\DK, 1.5\DK)\0\b\a\a\a\b\a\a;
\draw (0\DK, 1\DK)   \t\1\2\1\2\1\2\1\2\t;
\end{tikzpicture}
$$
It is easy to find that the diagonal next to the $\dll{a}{a}\dll{a}{b}$-diagonal
(at distance $2$ from the  $\dll{1}{2}$-diagonal) is one-color or binary.

\subparagraph
[There is aabc-diagonal neighbor to 12-diagonal]
{There is an $\dll{a}{a}\dll{b}{c}$-diagonal
neighbor to a $\dll{1}{2}$-diagonal, $\cll{a}\ne\cll{b}\ne\cll{c}\ne\cll{a}$}
\label{case6d2iv}

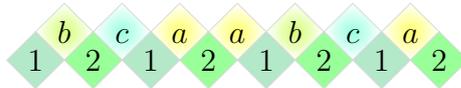
\def\a{\Dcolor{cyan!55!blue!20!white}{white}{$a$}}
\def\d{\dcolor{green!35!black!19!white}{$d$}}
\def\f{\dcolor{orange!40!red!15!white}{$f$}}
\definecolor{oell}{rgb}{0.95,0.92,0.55}
\def\e{\dcolor{cyan!30!blue!16!white}{$e$}}
$$
\begin{tikzpicture}
\draw (0.5\DK, 1.5\DK)\0\b\c\a\a\b\c\a;
\draw (0\DK, 1\DK)   \t\1\2\1\2\1\2\1\2\t;
\end{tikzpicture}
$$
We have:
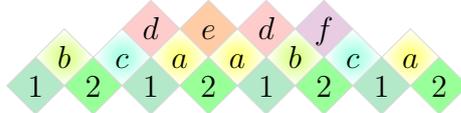
$$
\begin{tikzpicture}
\draw (0\DK, 2\DK)   \t\d\f\d\e\d\f\d\e\t;
\draw (0.5\DK, 1.5\DK)\0\b\c\a\a\b\c\a;
\draw (0\DK, 1\DK)   \t\1\2\1\2\1\2\1\2\t;
\end{tikzpicture}
$$
where $\cll{f}\ne\cll{e}$ (otherwise, we obtain a $\dll{d}{e}$-diagonal
and return to subcase~\ref{case6d2ii}).
We see that a $\cll d$-node has neighbors of all three colors
$\cll a$, $\cll b$, and $\cll c$.
This means that there are two $\{\cll 1,\cll 2\}$-diagonals at distance $2$
from a $\cll d$-node:
$$
\begin{tikzpicture}
\draw (0\DK, 3\DK)   \t\1\2\1\2\1\2\1\2\t;
\draw (0\DK, 2\DK)   \0\d\f\d\e\d\f\d\e;
\draw (0.5\DK, 1.5\DK)\0\b\c\a\a\b\c\a;
\draw (0\DK, 1\DK)   \t\1\2\1\2\1\2\1\2\t;
\end{tikzpicture}
$$
{Since
$d^2(\cll d,\cll f)=2\ne d^2(\cll e,\cll f)$,
we see that $\cll d$ is different from~$\cll e$.
Similarly, since $d^2(\cll d,\cll e)=2\ne d^2(\cll f,\cll e)$,
we have $\cll d \ne \cll f$.}
Since $\{\cll b,\cll c\}$-nodes do not have $\cll e$-neighbors,
every $\cll e$-node has four $\cll a$-neighbors
(looking to the neighborhoods of $\{\cll 1,\cll 2\}$-nodes,
we find that there are only three odd colors).
Then the whole coloring can be reconstructed up to shifting $\dll 12$-diagonals:
$$
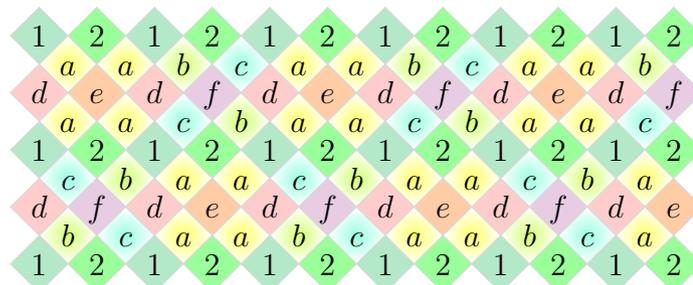
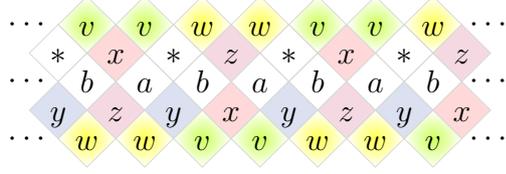
\begin{tikzpicture}
\draw (0\DK, 5\DK)   \1\2\1\2\1\2\1\2\1\2\1\2;
\draw (0.5\DK, 4.5\DK)\a\a\b\c\a\a\b\c\a\a\b;
\draw (0\DK, 4\DK)   \d\e\d\f\d\e\d\f\d\e\d\f;
\draw (0.5\DK, 3.5\DK)\a\a\c\b\a\a\c\b\a\a\c;
\draw (0\DK, 3\DK)   \1\2\1\2\1\2\1\2\1\2\1\2;
\draw (0.5\DK, 2.5\DK)\c\b\a\a\c\b\a\a\c\b\a;
\draw (0\DK, 2\DK)   \d\f\d\e\d\f\d\e\d\f\d\e;
\draw (0.5\DK, 1.5\DK)\b\c\a\a\b\c\a\a\b\c\a;
\draw (0\DK, 1\DK)   \1\2\1\2\1\2\1\2\1\2\1\2;
\end{tikzpicture}
$$
We get the coloring
shown in Fig.~\ref{AA}
with the following color correspondence:
$$
\begin{array}{r||c|c|c|c|c|c|c|c}
 \FfF: & \cll 1 & \cll 2 & \cll{a} & \cll{b}  & \cll{c} & \cll{d} & \cll{e}  & \cll{f}\\ \hline
 \mbox{Fig.~\ref{AA}}:
 & \cll 4 & \cll 5 & \cll{7} \cup \cll{8} & \cll 2  & \cll 3 & \cll 6 & \cll 9  & \cll 1 \\
\end{array}
$$

\begin{remark}
The colors $\cll b$ and $\cll c$ are twin,
and we considered colorings with
such two twin colors in Case~\ref{case5}, which also leads
to the coloring in Fig.~\ref{AA} in one of the subcases.
\end{remark}

\subparagraph
[There is aabb-diagonal
neighbor to 12-diagonal]
{There is an $\dll{a}{a}\dll{b}{b}$-diagonal
neighbor to a $\dll{1}{2}$-diagonal, $\cll{a}\ne\cll{b}$}
\label{case6d2v}

$$
\def\1{\dcolor{\colONE}{$1$}}
\def\2{\dcolor{\colTWO}{$2$}}
\def\0{ ++(1\DK, 0) }
\def\o{\dcolor{white}{$\ $}}
\def\a{\Dcolor{yellow!60!white         }{white}{$a$}}
\def\b{\Dcolor{green!30!yellow!55!white}{white}{$b$}}
\def\c{\dcolor{cyan!30!blue!8!white}{$c$}}
\def\d{\dcolor{red!10!white}{$d$}}
\def\f{\dcolor{purple!70!blue!20!white}{$e$}}
\def\e{\dcolor{red!50!orange!35!white}{$f$}}
\def\s{\dcolor{white}{$*$}}
\def\t{ ++(1\DK, 0) node {$\cdots$}}
\begin{tikzpicture}
\draw (0.5\DK, 1.5\DK)\b\b\a\a\b\b\a;
\draw (0\DK, 1\DK)   \1\2\1\2\1\2\1\2;
\phantom{\draw (1.5\DK, 0.5\DK)  \a;}
\end{tikzpicture}
\ \ \raisebox{1.5em}{$\Longrightarrow$}\ \
\begin{tikzpicture}
\draw (0.5\DK, 1.5\DK)\b\b\a\a\b\b\a;
\draw (0\DK, 1\DK)   \1\2\1\2\1\2\1\2;
\draw (0.5\DK, 0.5\DK)\a\a\b\b\a\a\b;
\end{tikzpicture}
$$
It is immediately follows that the $\dll{1}{2}$-diagonal is surrounded by two
$\dll{a}{a}\dll{b}{b}$-diagonals, as in the figure.
Next, a diagonal, which is not binary or one-color,
at distance~$2$ from the $\dll{1}{2}$-diagonal must be a
a $\dll cd\dll ed$-diagonal
for some distinct $\cll c$, $\cll d$, and $\cll e$.
\vspace{0.3em}

\begin{proposition}\label{p:aabb}
In subcase~\ref{case6d2v}, the coloring has
the following structure, up to renaming colors.
Every $\dll{1}{2}$-diagonal is surrounded by two
$\dll{3}{3}\dll{4}{4}$-diagonals.
Each of the $\dll{3}{3}\dll{4}{4}$-diagonals is followed
by a $\dll{5}{6}\dll{7}{6}$-diagonal.
Then, a  $\dll{8}{8}\dll{9}{9}$-diagonal can follow;
then a $\DDll{10}{11}\DDll{12}{11}$-diagonal, and so on,
until we meet one of the following fragments:


\def\krotovspec{\hskip-13pt}

\begin{enumerate}[itemsep=0pt]
\def\t{ ++(1\DK, 0) node {$\cdots$}}
\dfc{cI}{f0f936}
\dfc{cJ}{fcc807}\def\j{\Dcolor{cJ}{white}{$j$}}
\dfc{cK}{fea602}\def\k{\Dcolor{cK}{white}{$k$}}
\dfc{cL}{ec6f17}\def\l{\dcolor{cL!50}{$l$}}
\dfc{cM}{bb3509}\def\m{\dcolor{cM!40}{$m$}}
\dfc{cN}{721b10}\def\n{\dcolor{cN!30}{$n$}}
\def\s{\dcolor{cN!50!cM!40}{$m$}}
 \item[\rm(a)]\krotovspec
\begin{tikztab}
\draw (0.5\DK, 1.5\DK)\t\k\k\j\j\k\k\j\t;
\draw (1\DK, 1\DK)     \m\n\m\n\m\n\m\n;
\draw (0.5\DK, 0.5\DK)\t\j\j\k\k\j\j\k\t;
\end{tikztab}
\!\!\!or\!\!\!\!\!
\begin{tikztab}
\draw (0.5\DK, 1.5\DK)\t\k\k\j\j\k\k\j\t;
\draw (1\DK, 1\DK)     \s\s\s\s\s\s\s\s;
\draw (0.5\DK, 0.5\DK)\t\j\j\k\k\j\j\k\t;
\end{tikztab}
 \item[\rm(b)]\krotovspec
\begin{tikztab}
\draw (0.5\DK, 1.5\DK)\t\k\k\j\j\k\k\j\t;
\draw (1\DK, 1\DK)     \m\n\m\l\m\n\m\l;
\draw (0.5\DK, 0.5\DK)\t\k\k\j\j\k\k\j\t;
\end{tikztab}
 \def\i{\dcolor{cI!65}{$i$}}
 \def\j{\dcolor{cJ!60}{$j$}}
 \def\k{\dcolor{red!10!cK!45}{$k$}}
 \def\m{\Dcolor{cM!50}{white}{$m$}}
 \def\n{\Dcolor{cN!50}{white}{$n$}}
 \def\s{\Dcolor{cN!50!cM!50}{white}{$m$}}
 \item[\rm(c)]\krotovspec
\begin{tikztab}
\draw (0.5\DK, 1.5\DK)\t\i\j\k\j\i\j\k\t;
\draw (1\DK, 1\DK)     \m\n\m\n\m\n\m\n;
\draw (0.5\DK, 0.5\DK)\t\k\j\i\j\k\j\i\t;
\end{tikztab}
\!\!\!or\!\!\!\!\!\!
\begin{tikztab}
\draw (0.5\DK, 1.5\DK)\t\i\j\k\j\i\j\k\t;
\draw (1\DK, 1\DK)     \s\s\s\s\s\s\s\s;
\draw (0.5\DK, 0.5\DK)\t\k\j\i\j\k\j\i\t;
\end{tikztab}
 \item[\rm(d)]\krotovspec
\begin{tikztab}
\draw (0.5\DK, 1.5\DK)\t\i\j\k\j\i\j\k\t;
\draw (1\DK, 1\DK)     \m\m\n\n\m\m\n\n;
\draw (0.5\DK, 0.5\DK)\t\i\j\k\j\i\j\k\t;
\end{tikztab}
\end{enumerate}
After that, the same diagonals follow in the reverse order, until a $\dll 12$-diagonal;
then the situation repeats.
\end{proposition}

The corresponding colorings are shown in Fig.~\ref{MM}--\ref{PP}.
\begin{proof}
{%
We use induction on~$i$,
the distance to the $\dll12$-diagonal.
We assume that every $\dll12$-diagonal
is surrounded by two \dll33\dll44-diagonals.
Each of the \dll33\dll44-diagonals is followed
by a \dll56\dll76-diagonal. Then, a \dll88\dll99-diagonal and a \DDll{10}{11}\DDll{12}{11}-diagonal, and so
on, until we have colored a diagonal at distance~$i$ away from the initial \dll12-diagonal. We also
assume that the colors have not been exhausted,
that is, we have not yet crossed the diagonal
where the pattern folds and the colors start to repeat (those diagonals correspond, as will see,
to cases (a)--(d) in the statement of the proposition).
In what follows we will show what the
possibilities are for the diagonal at distance~$i+1$.
The arguments above
the proposition provide the induction base,
and now we can assume $i\ge 2$.}
Consider two cases according to the parity of~$i$.


(A) $i$ is even. We consider two neighbor $\dll{v}{v}\dll{w}{w}$- and $\dll{x}{y}\dll{z}{y}$-diagonals at distance $i-1$ and $i$ from
a \cll1\cll2-diagonal, respectively.
Considering the neighbors of $\cll y$-nodes, we see that the diagonal next to the
$\dll{x}{y}\dll{z}{y}$-diagonal must be an $\dll{a}{b}$- or $\dll{a}{a}\dll{b}{b}$-diagonal
for some (maybe equal) $\cll a$, $\cll b$.

At first, assume that it is an $\dll{a}{b}$-diagonal.
 \dfc{cA}{b8c58b}\def\j{\Dcolor{green!5!cA}{white}{$v$}}
 \dfc{cB}{a19d6e}\def\k{\Dcolor{cB!80}{white}{$w$}}
 \dfc{cI}{f0f936}\def\l{\dcolor{cI!65}{$x$}}
 \dfc{cJ}{fcc807}\def\m{\dcolor{cJ!60}{$y$}}
 \dfc{cK}{fea602}\def\n{\dcolor{cK!55}{$z$}}
 \dfc{cM}{bb3509}\def\a{\Dcolor{cM!50}{white}{$a$}}
 \dfc{cN}{721b10}\def\b{\Dcolor{cN!50}{white}{$b$}}
\def\t{ ++(1\DK, 0) node {$\cdots$}}
\begin{center}
\begin{tikzpicture}
\draw (0.5\DK, 1.5\DK)\t\b\a\b\a\b\a\b\t;
\draw (1\DK, 1\DK)     \m\n\m\l\m\n\m\l;
\draw (0.5\DK, 0.5\DK)\t\k\k\j\j\k\k\j\t;
\end{tikzpicture}
\end{center}

The colors $\cll a$ and $\cll b$ are the same or twins, and considering the neighborhoods
of the corresponding nodes, we find positions of some $\{\cll x,\cll z\}$-nodes;
then we uniquely color the neighborhoods of these nodes:
\begin{center}
\def\o{\dcolor{cJ!8}{$*$}}
\begin{tikzpicture}
\draw (0.5\DK, 2.5\DK)\t\j\j\k\k\j\j\k\t;
\draw (1\DK, 2\DK)     \o\l\o\n\o\l\o\n;
\draw (0.5\DK, 1.5\DK)\t\b\a\b\a\b\a\b\t;
\draw (1\DK, 1\DK)     \m\n\m\l\m\n\m\l;
\draw (0.5\DK, 0.5\DK)\t\k\k\j\j\k\k\j\t;
\end{tikzpicture}
\end{center}
We note that $\cll{a},\cll{b}\not\in\{\cll{v},\cll{w}\}$
(indeed, each $\{\cll{v},\cll{w}\}$-node
has two non-$\{\cll{x},\cll{y},\cll{z}\}$ neighbors).
Now we see that the nodes marked by $*$ in the figure must be colored with
$\cll y$ (as we see from the four previous diagonals, $\cll y$ has no twins).
So, we obtain fragment (c).

Now consider the second variant:
\begin{center}
\def\a{\dcolor{black!3}{$a$}}
\def\b{\dcolor{black!3}{$b$}}
\begin{tikzpicture}
\draw (0.5\DK, 1.5\DK)\t\b\b\a\a\b\b\a\t;
\draw (1\DK, 1\DK)     \m\n\m\l\m\n\m\l;
\draw (0.5\DK, 0.5\DK)\t\k\k\j\j\k\k\j\t;
\end{tikzpicture}
\end{center}
where $\cll a \ne \cll b$. It is not difficult to see that
either $\cll a=\cll v$ and $\cll b=\cll w$
or $\cll a$ and $\cll b$
are ``new'' colors (that is, they do not occur between the
$\dll{x}{y}\dll{z}{y}$-diagonal and the nearest $\dll{1}{2}$-diagonal).
{%
Indeed, if $\cll a$ is an old color, then
$\cll a = \cll v$ because all other old colors have no
$\cll x$-neighbors.
Similarly, if $\cll b$ is old, then
$\cll b = \cll w$.
If $\cll a$ is new, then we see $d^2(\cll w, \cll a)=1$
while $d^2(\cll b, \cll a)\ge 2$;
so, $\cll b\ne \cll w$ and hence $\cll b$ is new too.
Similarly, if $\cll b$ is new, then $\cll a$ is new.
}

In the first case, we have fragment (b);
in the second, we move to the next step and consider
the $\dll{x}{y}\dll{z}{y}$- and $\dll{a}{a}\dll{b}{b}$-diagonals,
see p.\,(B) of the proof.

(B) $i$ is odd. We consider two neighbor $\dll{u}{v}\dll{w}{v}$-
and $\dll{x}{x}\dll{y}{y}$-diagonals at distance $i-1$ and $i$ from
a \cll1\cll2-diagonal, respectively.
Considering the neighbors of $\cll y$-nodes,
we see that the diagonal next to the
$\dll{x}{x}\dll{y}{y}$-diagonal must be an $\dll{a}{b}\dll{c}{b}$-diagonal
for some $\cll a$, $\cll b$, $\cll c$.
\def\t{ ++(1\DK, 0) node {$\cdots$}}
\dfc{cL}{ec6f17}
\def\a{\dcolor{black!50!cL!4}{$a$}}
\def\b{\dcolor{black!50!cM!4}{$b$}}
\def\c{\dcolor{black!50!cN!4}{$c$}}
\dfc{cJ}{fcc807}\def\l{\Dcolor{cJ}{white}{$x$}}
\dfc{cK}{fea602}\def\m{\Dcolor{cK}{white}{$y$}}
\dfc{ccL}{cabb70}\def\u{\dcolor{yellow!10!ccL!70}{$u$}}
\dfc{ccM}{b8c58b}\def\v{\dcolor{ccM!70}{$v$}}
\dfc{ccN}{a19d6e}\def\w{\dcolor{ccN!60}{$w$}}
\def\0{ ++(1\DK, 0) }
\def\s{\Dcolor{black!6}{white}{$*$}}
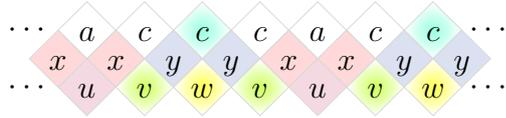
\begin{center}
\begin{tikzpicture}
\draw (0.5\DK, 1.5\DK)\t\a\b\c\b\a\b\c\t;
\draw (1\DK, 1\DK)     \l\l\m\m\l\l\m\m;
\draw (0.5\DK, 0.5\DK)\t\u\v\w\v\u\v\w\t;
\end{tikzpicture}
\end{center}
At first, assume $\cll a=\cll c$,
i.e., we have an $\dll ab$-diagonal, where
$\cll a$ and $\cll b$ are the same color or twins.
As every $\{\cll a,\cll b\}$-node
has two $\cll x$- and two $\cll y$-neighbors,
the next diagonal is uniquely colored.
We get fragment (a).

Then, assume $\cll c=\cll b$:
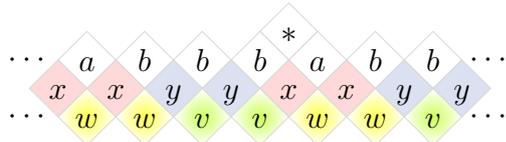
\begin{center}
\begin{tikzpicture}
\draw (1\DK, 2\DK)     \0\0\0\0\s\0\0\0;
\draw (0.5\DK, 1.5\DK)\t\a\b\b\b\a\b\b\t;
\draw (1\DK, 1\DK)     \l\l\m\m\l\l\m\m;
\draw (0.5\DK, 0.5\DK)\t\u\v\w\v\u\v\w\t;
\end{tikzpicture}
\end{center}
Considering the neighborhoods of $\cll b$-nodes,
we find that the node marked by $*$ has the color $\cll y$.
Since the $\cll y$-nodes have neighbors of colors
$\cll v$, $\cll w$, $\cll b$ only, we conclude that $\cll a=\cll b$;
hence, $\cll a=\cll c$, which case was considered above.

The last possibility is $\cll a\ne\cll b\ne\cll c\ne\cll a$.
If $\{\cll a,\cll b,\cll c\} \in \{\cll u,\cll v,\cll w\}$, then
the color of a $\{\cll a,\cll b,\cll c\}$-node is uniquely determined
by the occurence of $\cll x$- or $\cll y$-nodes in its neighborhood.
Namely, $\cll a = \cll u$, $\cll b = \cll v$, $\cll c = \cll w$.
We get fragment (d).
Otherwise, $\cll a$,  $\cll b$, and $\cll c$ are new colors,
and we continue considering the $\dll{x}{x}\dll{y}{y}$
and $\dll{a}{b}\dll{c}{b}$--diagonals
with p.\,(A) of the current proof.

As the number of colors is finite,
we have to meet one of the fragments (a)--(d) at some step.
\end{proof}

\subparagraph
[There is abcb-diagonal neighbor to 12-diagonal]
{There is an $\dll{a}{b}\dll{c}{b}$-diagonal
neighbor to a $\dll{1}{2}$-diagonal, $\cll{a}\ne\cll{b}\ne\cll{c}\ne\cll{a}$}
\label{case6d2vi}
{
\def\1{\dcolor{\colONE}{$1$}}
\def\2{\dcolor{\colTWO}{$2$}}
\def\0{ ++(1\DK, 0) }
\def\o{\dfant}
\def\a{\Dcolor{yellow!60!white}{white}{$a$}}
\def\b{\Dcolor{green!30!yellow!55!white}{white}{$b$}}
\def\c{\dcolor{cyan!30!blue!8!white}{$c$}}
\def\d{\dcolor{red!10!white}{$d$}}
\def\f{\dcolor{purple!70!blue!20!white}{$e$}}
\def\e{\dcolor{red!50!orange!35!white}{$f$}}
\def\s{\dcolor{black!5}{$*$}}
\def\t{ ++(1\DK, 0) node {$\cdots$}}
%
\dfc{cO}{b7d7b3}\def\1{\dcolor{cO!90}{$1$}}
\dfc{cT}{adbe64}\def\2{\dcolor{cT!80}{$2$}}
\dfc{cA}{f4ccdc}\def\a{\Dcolor{cA}{cA!50}{$a$}}
\dfc{cB}{dd98ab}\def\b{\Dcolor{cB!85}{cB!45}{$b$}}
\dfc{cC}{dea78c}\def\c{\Dcolor{cC}{cC!50}{$c$}}

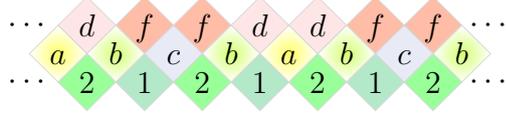
\dfc{cD}{d5e3f8}\def\d{\dcolor{cD}{$d$}}
\dfc{cF}{a5bce6}\def\e{\dcolor{cF!75}{$e$}}

\noindent
$
\begin{tikztab}
\draw (0\DK, 2\DK)   \o;
\draw (0.5\DK, 1.5\DK)\a\b\c\b\a\b\c\b;
\draw (0\DK, 1\DK)   \t\2\1\2\1\2\1\2\t;
\end{tikztab}
\!\!\!\!\Longrightarrow\!\!\!\!\!
\begin{tikztab}
\draw (0\DK, 2\DK)   \t\d\e\e\d\d\e\e\t;
\draw (0.5\DK, 1.5\DK)\a\b\c\b\a\b\c\b;
\draw (0\DK, 1\DK)   \t\2\1\2\1\2\1\2\t;
\end{tikztab}
$

Considering the neighborhoods of the $\cll b$-nodes,
we find that the next diagonal is a $\dll dd\dll ee$-diagonal
for some different $\cll d$ and $\cll e$
(recall that we avoid the possibility for this diagonal to be binary or one-color,
considered in subcase~\ref{case6d2ii}).
We see that every $\{\cll 1,\cll 2\}$-node has neighbors of
colors $\cll a$, $\cll b$, and $\cll c$. This allows us to reconstruct
the colors of the following nodes:
$$
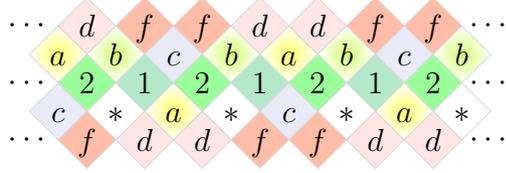
\begin{tikzpicture}
\draw (0\DK, 2\DK)   \t\d\e\e\d\d\e\e\t;
\draw (0.5\DK, 1.5\DK)\a\b\c\b\a\b\c\b;
\draw (0\DK, 1\DK)   \t\2\1\2\1\2\1\2\t;
\draw (0.5\DK, 0.5\DK)\c\s\a\s\c\s\a\s;
\draw (0\DK, 0\DK)   \t\e\d\d\e\e\d\d\t;
\end{tikzpicture}
$$
If the nodes marked by $*$ have a color different from $\cll b$,
then there are $\cll b$-nodes at distance~$3$ from the $\dll 12$-diagonal,
and, hence, there are $\{\cll 1,\cll 2\}$-nodes at distance~$4$
from the $\dll 12$-diagonal. We finally obtain the situation of
Proposition~\ref{p:6d1}(ii), i.e. the coloring is equivalent,
up to shifting of $\dll 12$-diagonals, to the coloring of Fig.~\ref{JJ}.

Otherwise, we have the situation as follows,
which can be solved similarly to subcase~\ref{case6d2v}:
$$
\begin{tikzpicture}
\draw (0\DK, 2\DK)   \t\d\e\e\d\d\e\e\t;
\draw (0.5\DK, 1.5\DK)\a\b\c\b\a\b\c\b;
\draw (0\DK, 1\DK)   \t\2\1\2\1\2\1\2\t;
\draw (0.5\DK, 0.5\DK)\c\b\a\b\c\b\a\b;
\draw (0\DK, 0\DK)   \t\e\d\d\e\e\d\d\t;
\end{tikzpicture}
$$
}

\begin{proposition}\label{p:abcb}
In the subcase \ref{case6d2vi}, the coloring
either is described in Proposition~\ref{p:6d1}
or has the following structure, up to renaming colors:
Every $\dll{1}{2}$-diagonal is surrounded by two
$\dll{3}{4}\dll{5}{4}$-diagonals.
Each of the $\dll{3}{4}\dll{5}{4}$-diagonals is followed
by a $\dll{6}{6}\dll{7}{7}$-diagonal.
Then, a  $\dll{8}{9}\Ddll{10}{9}$-diagonal can follow;
then a $\DDll{11}{11}\DDll{12}{12}$-diagonal, and so on,
until we meet one of the  fragments {\rm (a)--(d)} of
Proposition~\ref{p:aabb}.
After that, the same diagonals follow in the reverse order, until a $\dll 12$-diagonal;
then the situation repeats.
\end{proposition}
The corresponding colorings are shown in Fig.~\ref{OO}, \ref{SS}, \ref{QQ}, and~\ref{RR} (the special case $n=7$ corresponds to in Fig.~\ref{JJ} with merged colors~\cll3 and~\cll4).

\subparagraph
[There is abcd-diagonal neighbor to 12-diagonal]
{There is an $\dll{a}{b}\dll{c}{d}$-diagonal
neighbor to a $\dll{1}{2}$-diagonal; $\cll{a}$--$\cll{d}$ are pairwise distinct}
\label{case6d2vii}

\smallskip\noindent
$
\def\1{\dcolor{\colONE}{$1$}}
\def\2{\dcolor{\colTWO}{$2$}}
\def\o{\dcolor{white}{$\ $}}
\def\a{\Dcolor{yellow!60!white}{white}{$a$}}
\def\b{\Dcolor{green!30!yellow!55!white}{white}{$b$}}
\def\c{\dcolor{cyan!30!blue!8!white}{$c$}}
\def\d{\dcolor{red!10!white}{$d$}}
\def\e{\dcolor{purple!60!blue!15!white}{$a$}}
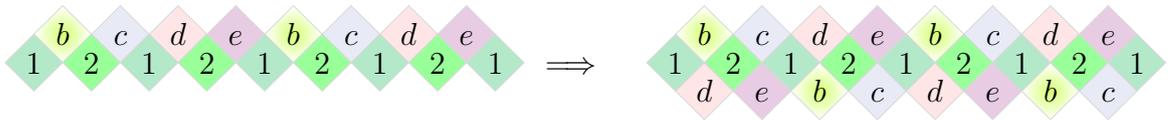
\begin{tikzpicture}
\draw (0.5\DK, 1.5\DK)\b\c\d\e\b\c\d\e;
\draw (0\DK, 1\DK)   \1\2\1\2\1\2\1\2\1;
\phantom{\draw (0.5\DK, 0.5\DK)\o;}
\end{tikzpicture}
\ \ \raisebox{1.5em}{$\Longrightarrow$}\ \
\begin{tikzpicture}
\draw (0.5\DK, 1.5\DK)\b\c\d\e\b\c\d\e;
\draw (0\DK, 1\DK)   \1\2\1\2\1\2\1\2\1;
\draw (0.5\DK, 0.5\DK)\d\e\b\c\d\e\b\c;
\end{tikzpicture}
$

\noindent
This subcase is considered similarly
to the previous two subcases, resulting in:
\begin{proposition}\label{p:abcd}
In the subcase~\ref{case6d2vii}, the coloring has
the following structure, up to renaming colors.
Every $\dll{1}{2}$-diagonal is surrounded by two
$\dll{3}{4}\dll{5}{6}$-diagonals.
Then, every $\dll{a}{b}\dll{c}{d}$-diagonal is followed
by an $\dll{e}{f}\dll{g}{h}$-diagonal,
where $h=g+1=f+2=e+3=d+4=c+5=b+6=a+7\equiv 2 \bmod 4$,
until we meet one of the following fragments:
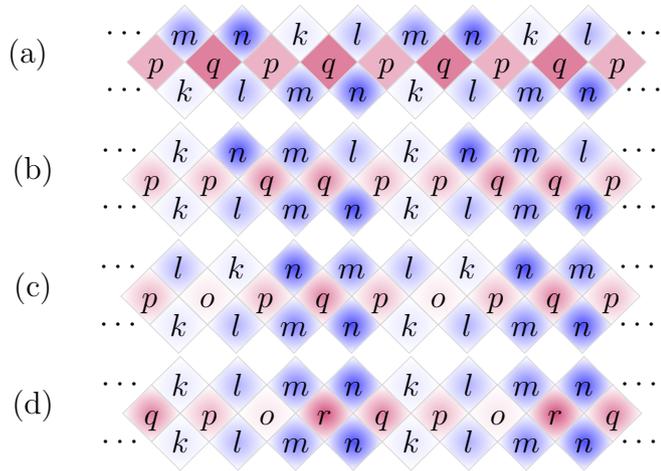
\begin{center}
\def\1{\dcolor{\colONE}{$1$}}
\def\2{\dcolor{\colTWO}{$2$}}
\def\0{ ++(1\DK, 0) }
\def\i{\Dcolor{blue!15!white  }{white}{$k$}}
\def\j{\Dcolor{blue!25!white  }{white}{$l$}}
\def\k{\Dcolor{blue!35!white  }{white}{$m$}}
\def\l{\Dcolor{blue!45!white  }{white}{$n$}}
\def\p{\Dcolor{purple!15!white}{white}{$o$}}
\def\n{\Dcolor{purple!30!white}{white}{$p$}}
\def\m{\Dcolor{purple!40!white}{purple!5!white}{$q$}}
\def\q{\Dcolor{purple!50!white}{purple!10!white}{$r$}}
\def\P{\dcolor{purple!15!white}{$p$}}
\def\Q{\dcolor{purple!25!white}{$q$}}
\def\t{ ++(1\DK, 0) node {$\cdots$}}
\rm{(a)}
\raisebox{-0.5\height}{
\begin{tikzpicture}
\draw (0.5\DK, 1.5\DK)\t\k\l\i\j\k\l\i\j\t;
\draw (1\DK, 1\DK)     \P\Q\P\Q\P\Q\P\Q\P;
\draw (0.5\DK, 0.5\DK)\t\i\j\k\l\i\j\k\l\t;
\end{tikzpicture}}
\\[3pt] \rm{(b)}
\raisebox{-0.5\height}{\begin{tikzpicture}
\draw (0.5\DK, 1.5\DK)\t\i\l\k\j\i\l\k\j\t;
\draw (1\DK, 1\DK)     \n\n\m\m\n\n\m\m\n;
\draw (0.5\DK, 0.5\DK)\t\i\j\k\l\i\j\k\l\t;
\end{tikzpicture}}
\\[3pt] \rm{(c)}
\raisebox{-0.5\height}{\begin{tikzpicture}
\draw (0.5\DK, 1.5\DK)\t\j\i\l\k\j\i\l\k\t;
\draw (1\DK, 1\DK)     \n\p\n\m\n\p\n\m\n;
\draw (0.5\DK, 0.5\DK)\t\i\j\k\l\i\j\k\l\t;
\end{tikzpicture}}
\\[3pt] \rm{(d)}
\raisebox{-0.5\height}{\begin{tikzpicture}
\draw (0.5\DK, 1.5\DK)\t\i\j\k\l\i\j\k\l\t;
\draw (1\DK, 1\DK)     \m\n\p\q\m\n\p\q\m;
\draw (0.5\DK, 0.5\DK)\t\i\j\k\l\i\j\k\l\t;
\end{tikzpicture}}
\end{center}
{\rm (}where all symbols denote different colors except maybe \cll{p} and \cll{q} in {\rm(a))}. After that, the same diagonals follow in the reverse order, until a $\dll 12$-diagonal;
then the situation repeats.
\end{proposition}

The corresponding colorings are shown in Fig.~\ref{TT}, \ref{UU}, \ref{VV}, and~\ref{WW}.

\section{Proof of Theorem~\ref{th:01}}\label{s:p2}

We start with an obvious claim,
then prove two lemmas,
considering two special cases,
then give a proof of Theorem~\ref{th:01},
which consists of crucial Proposition~\ref{p:01} and a short concluding part.

\begin{claim}\label{cl:1122}
For a perfect coloring of $G(\ZZ^2)$
with quotient $\{0,1\}$-matrix with zero diagonal and without equal rows,
 any two nodes with difference in
 $\{\,[ 0,\pm 1]$, $[\pm 1, 0]$, $[\pm 1,\pm 1]$, $[0,\pm 2]$, $[\pm 2,0]$, $[\pm 2,\pm 2]\,\}$
 are colored with different colors.
\end{claim}
\begin{proof}
The only non-trivial case is with the difference of form $[\pm 2,\pm 2]$.
If, w.l.o.g., $\FfF\Lft 0,0\Rgt =\FfF\Lft 2,2\Rgt =\cll{1}$ and $\FfF\Lft 1,1\Rgt =\cll{2}$, then
each of $[0,1]$, $[0,1]$, $[1,2]$, $[2,1]$ is adjacent to a $\cll{1}$-node and to a $\cll{2}$-node.
Since the colors of $[0,1]$, $[0,1]$, $[1,2]$, $[2,1]$ are pairwise different,
we see that the $1$st and $2$nd rows of the quotient matrix
coincide, contradicting the hypothesis of the claim.
\end{proof}

\def\O{\ccolor{white}{\ }}
\def\o{ ++(1, 0) }
\def\0{\ccolor{gray!20!white}{$0$}}
\def\1{\ccolor{purple!10!white}{$1$}}
\def\2{\ccolor{green!17!white}{$2$}}
\def\3{\ccolor{cyan!15!white}{$3$}}
\def\4{\ccolor{red!10!white}{$4$}}
\def\5{\ccolor{yellow!30!white}{$5$}}
\def\6{\ccolor{blue!10!white}{$6$}}
\def\7{\ccolor{gray!10!white}{$7$}}
\def\A{\ccolor{white}{$A$}}
\def\B{\ccolor{white}{$B$}}
\def\C{\ccolor{white}{$C$}}
\def\D{\ccolor{white}{$D$}}
\def\E{\ccolor{white}{$E$}}
\def\F{\ccolor{white}{$F$}}
\def\G{\ccolor{white}{$G$}}
\def\H{\ccolor{white}{$H$}}
\def\q{\ccolor{white}{$?$}}
\def\s{++(1,0) node {$*$}}

\begin{lemma}\label{l:+}
If a bipartite perfect coloring $\FfF$ of $G(\ZZ^2)$ with quotient $\{0,1\}$-matrix contains the fragment
$$
 
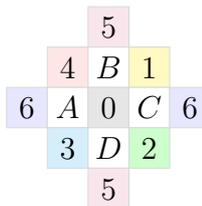
\begin{tikzpicture}
\draw (0,4) \o\o\o\5\o\o;
\draw (0,3) \o\o\4\B\1\o;
\draw (0,2) \o\6\A\0\C\6;
\draw (0,1) \o\o\3\D\2\o;
\draw (0,0) \o\o\o\5\o\o;
\end{tikzpicture}
$$
then $\FfF$, up to equivalence, is one of the following
two colorings, with periods, respectively, $[4,0]$ and $[0,4]$,
$[4,0]$ and $[0,8]$.
$$
\raisebox{2em}{}
 \begin{tikzpicture}
\draw (0,4) \7\G\5\H\7\G\5\H\7;
\draw (0,3) \E\4\B\1\E\4\B\1\E;
\draw (0,2) \6\A\0\C\6\A\0\C\6;
\draw (0,1) \F\3\D\2\F\3\D\2\F;
\draw (0,0) \7\G\5\H\7\G\5\H\7;
\draw (0,-1)\E\4\B\1\E\4\B\1\E;
\draw (0,-2)\6\A\0\C\6\A\0\C\6;
\draw (0,-3)\F\3\D\2\F\3\D\2\F;
\draw (0,-4)\7\G\5\H\7\G\5\H\7;
\end{tikzpicture}\quad
\raisebox{2em}{}
 \begin{tikzpicture}
\draw (0,4) \7\H\5\G\7\H\5\G\7;
\draw (0,3) \E\4\B\1\E\4\B\1\E;
\draw (0,2) \6\A\0\C\6\A\0\C\6;
\draw (0,1) \F\3\D\2\F\3\D\2\F;
\draw (0,0) \7\G\5\H\7\G\5\H\7;
\draw (0,-1)\E\1\B\4\E\1\B\4\E;
\draw (0,-2)\6\C\0\A\6\C\0\A\6;
\draw (0,-3)\F\2\D\3\F\2\D\3\F;
\draw (0,-4)\7\H\5\G\7\H\5\G\7;
\end{tikzpicture}
$$
\end{lemma}
\begin{proof}
 The only \cll{1}-node in the fragment cannot be adjacent to an \cll{A}- or \cll{D}-node;
 so, its right neighbor has a color,
 say \cll{E}, different from \cll{A}, \cll{B}, \cll{C}, \cll{D}.
 The same can be said about the right neighbor of the \cll{2}-node, denote its color by \cll{F}.
 $$
 \begin{tikzpicture}
\draw (0,4) \o\o\o\5\o\o;
\draw (0,3) \o\o\4\B\1\E;
\draw (0,2) \o\6\A\0\C\6;
\draw (0,1) \o\o\3\D\2\F;
\draw (0,0) \o\o\o\5\o\o;
\end{tikzpicture}
$$
The fourth neighbor of the right \cll{6}-node has color \cll{A}.
In its turn, this \cll{A}-node has a \cll{0}-neighbor,
and it can be only the right neighbor.
$$
 \begin{tikzpicture}
\draw (0,4) \o\o\o\5\o\o;
\draw (0,3) \o\o\4\B\1\E;
\draw (0,2) \o\6\A\0\C\6\A\0;
\draw (0,1) \o\o\3\D\2\F;
\draw (0,0) \o\o\o\5\o\o;
\end{tikzpicture}
$$
There are only two possibilities for the neighborhood of the right \cll{A}-node.
In each case, the neighborhood of the right \cll{0}-node is uniquely colored:
$$
\raisebox{2em}{(a) }
 \begin{tikzpicture}
\draw (0,4) \o\o\o\5\o\o\o\5;
\draw (0,3) \o\o\4\B\1\E\4\B\1;
\draw (0,2) \o\6\A\0\C\6\A\0\C\6;
\draw (0,1) \o\o\3\D\2\F\3\D\2;
\draw (0,0) \o\o\o\5\o\o\o\5;
\end{tikzpicture}
\raisebox{2em}{\quad or\quad (b) }
\begin{tikzpicture}
\draw (0,4) \o\o\o\5\o\o\o\5;
\draw (0,3) \o\o\4\B\1\E\3\D\2;
\draw (0,2) \o\6\A\0\C\6\A\0\C\6;
\draw (0,1) \o\o\3\D\2\F\4\B\1;
\draw (0,0) \o\o\o\5\o\o\o\5;
\end{tikzpicture}
$$
The same situation takes place in the vertical direction
$$
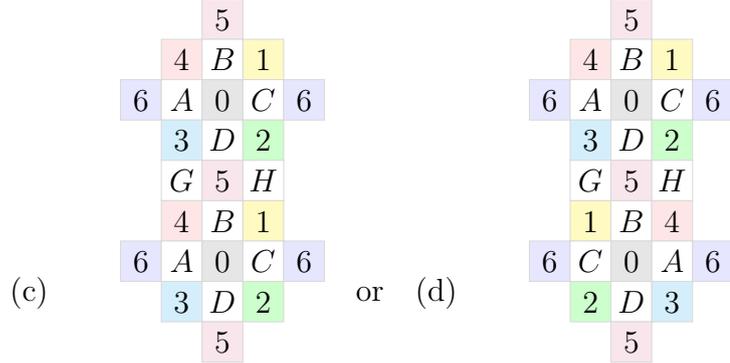
\raisebox{2em}{(c) }
 \begin{tikzpicture}
\draw (0,4) \o\o\o\5\o\o;
\draw (0,3) \o\o\4\B\1\o;
\draw (0,2) \o\6\A\0\C\6;
\draw (0,1) \o\o\3\D\2\o;
\draw (0,0) \o\o\G\5\H\o;
\draw (0,-1)\o\o\4\B\1\o;
\draw (0,-2)\o\6\A\0\C\6;
\draw (0,-3)\o\o\3\D\2\o;
\draw (0,-4)\o\o\o\5\o\o;
\end{tikzpicture}
\raisebox{2em}{\quad or\quad (d) }
\begin{tikzpicture}
\draw (0,4) \o\o\o\5\o\o;
\draw (0,3) \o\o\4\B\1\o;
\draw (0,2) \o\6\A\0\C\6;
\draw (0,1) \o\o\3\D\2\o;
\draw (0,0) \o\o\G\5\H\o;
\draw (0,-1)\o\o\1\B\4\o;
\draw (0,-2)\o\6\C\0\A\6;
\draw (0,-3)\o\o\2\D\3\o;
\draw (0,-4)\o\o\o\5\o\o;
\end{tikzpicture}
$$
Each of (a), (b) can be combined with each of (c), (d),
see e.g. (a)+(c) and (a)+(d) in the first two pictures below, except (b)+(d), see
{%
the contradiction at
the third picture (the color of the node marked with
``?'' must be~$\cll F$ because of the $\cll 4$-node at the left;
on the other hand, it cannot be~$\cll F$ because the $\cll 7$-node
above already has an $\cll F$-neighbor).
}
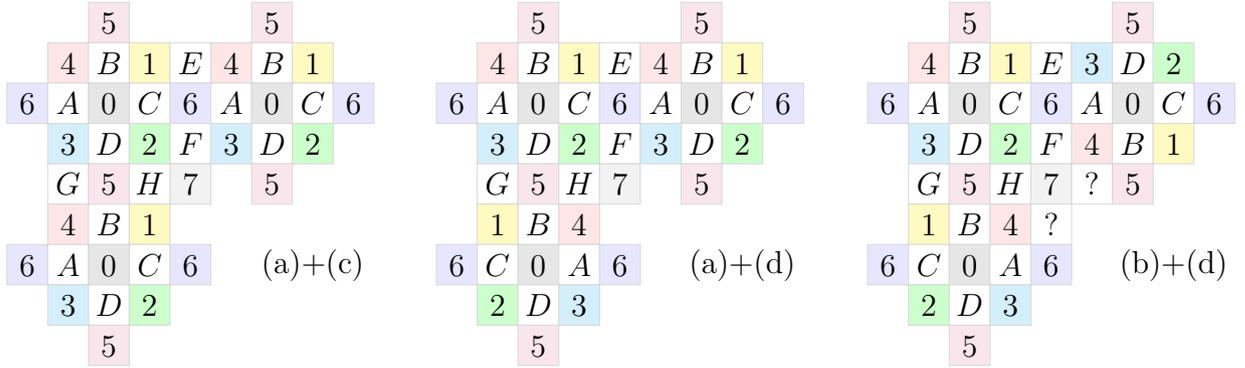
$$\!\!\!\!\!\!
 \begin{tikzpicture}
\draw (0,4) \o\o\5\o\o\o\5;
\draw (0,3) \o\4\B\1\E\4\B\1;
\draw (0,2) \6\A\0\C\6\A\0\C\6;
\draw (0,1) \o\3\D\2\F\3\D\2;
\draw (0,0) \o\G\5\H\7\o\5;
\draw (0,-1)\o\4\B\1\o;
\draw (0,-2)\6\A\0\C\6 ++(3,0) node {(a)+(c)};
\draw (0,-3)\o\3\D\2\o;
\draw (0,-4)\o\o\5\o\o;
\end{tikzpicture}
\!\!\!
 \begin{tikzpicture}
\draw (0,4) \o\o\o\5\o\o\o\5;
\draw (0,3) \o\o\4\B\1\E\4\B\1;
\draw (0,2) \o\6\A\0\C\6\A\0\C\6;
\draw (0,1) \o\o\3\D\2\F\3\D\2;
\draw (0,0) \o\o\G\5\H\7\o\5;
\draw (0,-1)\o\o\1\B\4\o;
\draw (0,-2)\o\6\C\0\A\6 ++(3,0) node {(a)+(d)};
\draw (0,-3)\o\o\2\D\3\o;
\draw (0,-4)\o\o\o\5\o\o;
\end{tikzpicture}
\!\!\!
 \begin{tikzpicture}
\draw (0,4) \o\o\o\5\o\o\o\5;
\draw (0,3) \o\o\4\B\1\E\3\D\2;
\draw (0,2) \o\6\A\0\C\6\A\0\C\6;
\draw (0,1) \o\o\3\D\2\F\4\B\1;
\draw (0,0) \o\o\G\5\H\7\o\5;
\draw (0,-1)\o\o\1\B\4\q;
\draw (0,-2)\o\6\C\0\A\6 ++(3,0) node {(b)+(d)};
\draw (0,-3)\o\o\2\D\3\o;
\draw (0,-4)\o\o\o\5\o\o;
\end{tikzpicture}
$$
In  cases (a)+(c) and (a)+(d), it is easy to see that the quotient matrix and the
whole coloring are uniquely reconstructed.
Cases (a)+(d) and (b)+(c) result in equivalent colorings.
\end{proof}

Next, we consider the case where two nodes with difference in
$\{ [ \pm 3, \pm 1 ], [ \pm 1, \pm 3 ] \}$ (w.l.o.g., $[3,1]$)
have the same color.

\begin{lemma} \label{l:/}
 Let $\FfF$ be a perfect coloring with quotient $\{0,1\}$-matrix
 with zero diagonal and without equal rows.
 If $\FfF({\bar x})=\FfF( {\bar x}+[3,1]  )$ for some ${\bar x}\in\ZZ^2$, then either
 \begin{enumerate}
  \item[\rm(a)]
 $\FfF$ has $3 \beta$ colors,
 $\beta\in \{3,4,5,\ldots\}$,
 periods $[6,0]$ and $[0,\beta]$,
 and satisfies $\FfF(i,j)\equiv \FfF(i+3,-j+1)$, or
 \item[\rm(b)]
  $\FfF$ has $\alpha$ colors,
 $ \alpha \in \{5,7,8,9,\ldots\}$, and periods $[\alpha,0]$ and $[3,1]$.
 \end{enumerate}
In both cases, $\FfF$ is an orbit coloring.
\end{lemma}
\begin{proof}
Without loss of generality, assume  $\FfF\Lft 0,0\Rgt =\FfF\Lft 3,1\Rgt =\cll{1}$.
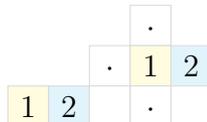
\def\O{\ccolor{white}{\ }}
\def\o{ ++(1, 0) }
\def\oo{ ++(0.5, 0) }
\def\1{\ccolor{yellow!15!white}{$1$}}
\def\2{\ccolor{cyan!10!white}{$2$}}
\def\3{\ccolor{green!10!white}{$3$}}
\def\4{\ccolor{red!10!white}{$4$}}
\def\5{\ccolor{magenta!12!white}{$5$}}
\def\6{\ccolor{blue!10!white}{$6$}}
\def\x{\ccolor{white}{$X$}}
\def\y{\ccolor{white}{$Y$}}
\def\z{\ccolor{white}{$Z$}}
\def\A{\ccolor{white}{$A$}}
\def\B{\ccolor{white}{$B$}}
\def\C{\ccolor{white}{$C$}}
\def\D{\ccolor{white}{$D$}}
\def\E{\ccolor{white}{$E$}}
\def\F{\ccolor{white}{$F$}}
\def\s{\ccolor{white}{$\star$}}
\def\q{\ccolor{white}{$?$}}
\def\nA{\ccolor{white}{$\cdot$}}
$$
\begin{tikzpicture}
\draw (0,2) \o\o\o\o\nA\o\o\o;
\draw (0,1) \o\o\o\nA\1\2\o\o;
\draw (0,0) \o\1\2\O\nA\o\o\o;
\end{tikzpicture}
$$
Denote by \cll{2} the color of $[1,0]$. By Claim~\ref{cl:1122}, the \cll{2}-neighbor
of $[3,1]$ is not $[2,1]$, $[3,0]$, or $[3,2]$. Hence, it is $[4,1]$.
Similarly, the colors of $[-1,0]$ and $[2,1]$ coincide.
By induction,
we have
\begin{equation}\label{eq:i0i1}
 {\FfF}\Lft i,0\Rgt  = {\FfF}\Lft i+3,1\Rgt  \qquad \mbox{for all } i \in \ZZ.
\end{equation}

Denote by $\alpha$ the first positive integer such that
$\FfF\Lft i,0\Rgt =\FfF\Lft i+\alpha,0\Rgt $ for some $i$.
From $\FfF\Lft i,0\Rgt =\FfF\Lft i+3,1\Rgt $ and Claim~\ref{cl:1122},
we see $\alpha\ge 5$. We consider two cases.

(a) \emph{$\alpha=6$}.
Without loss of generality, $ {\FfF}\Lft 0,0\Rgt  = {\FfF}\Lft 6,0\Rgt  $.
$$
 \begin{tikzpicture}
\draw (0,1) \o\o\o\o\1\A\2\B\3\C;
\draw (0,0) \o\1\A\2\B\3\C\1\o;
\end{tikzpicture}
$$
Similarly to \eqref{eq:i0i1}, we have
$$
 {\FfF}\Lft i,1\Rgt  = {\FfF}\Lft i+3,0\Rgt  \qquad \mbox{for all } i\in\ZZ,
$$
and we know the coloring of the two rows
$$
 \begin{tikzpicture}
\draw (0,1) \t\2\B\3\C\1\A\2\B\3\C\1\A\t;
\draw (0,0) \t\C\1\A\2\B\3\C\1\A\2\B\3\t;
\end{tikzpicture}
$$
\begin{claim}\label{cl:123abc}
There are no $\{\cll{1},\cll{2},\cll{3},
\cll{A},\cll{B},\cll{C}\}$-nodes
in the row $[i,2]$, $i \in \ZZ$.
\end{claim}
W.l.o.g.,
we will show this for \cll{1}-nodes.
If $i$ is even or $i\equiv 3\bmod 6$, then
the claim
follows from Claim~\ref{cl:1122}.
Seeking a contradiction
in the remaining case,
we assume without loss of generality
that $\FfF\Lft 1,2\Rgt =\cll{1}$.
$$
 \begin{tikzpicture}
\def\xX{\ccolor{white}{$\bar x$}}
\draw (0,2) \o\o\o\1\O\3\xX\q;
\draw (0,1) \o\2\B\3\C\1\A\2\B;
\draw (0,0) \o\C\1\A\2\B\3\C\1;
\end{tikzpicture}
$$
Then $[3,2]$ is a \cll{3}-node,
$\bar x=[4, 2]$ is an $\{\cll{A},\cll{B},\cll{C}\}$-node,
and
$[5,2]$ is an $\{\cll{1},\cll{2},\cll{3}\}$-node.
However, $\FfF\Lft 5,2\Rgt  \in \{\cll{2},\cll{3}\}$ contradicts Claim~\ref{cl:1122},
while $\FfF\Lft 5,2\Rgt =\cll{1}$ contradicts the fact that the $\cll{1}$-node $[3,1]$
is not adjacent to any $\cll{2}$-node.
The contradiction obtained proves Claim~\ref{cl:123abc}.

It immediately follows that the row $[i,2]$, $i\in \ZZ$,
is colored periodically with period $[6,0]$
with new colors. We divide this situation into two subcases, depending on the number of different colors in this row.

(a') The row $[i,2]$ is colored with less than~$6$ colors.
Then it contains
two nodes of the same color at distance~$3$ from each other.
W.l.o.g., $\FfF\Lft 0,2\Rgt =\FfF\Lft 3,2\Rgt =\cll{X}$;
denote $\cll{Y}=\FfF\Lft 2,2\Rgt $, $\cll{Z}=\FfF\Lft 1,2\Rgt $.
$$
\begin{tikzpicture}
\draw (0,2) \o\x\z\y\x;
\draw (0,1) \2\B\3\C\1\A;
\draw (0,0) \C\1\A\2\B\3\o;
\end{tikzpicture}
\raisebox{1em}{$\Rightarrow$}
\begin{tikzpicture}
\draw (0,3) \o\1\O\O\B;
\draw (0,2) \o\x\z\y\x;
\draw (0,1) \2\B\3\C\1\A ;
\draw (0,0) \C\1\A\2\B\3\o ;
\end{tikzpicture}
\raisebox{1em}{$\Rightarrow$}
\begin{tikzpicture}
\draw (0,3) \o\1\A\2\B;
\draw (0,2) \o\x\z\y\x;
\draw (0,1) \2\B\3\C\1\A ;
\draw (0,0) \C\1\A\2\B\3 ;
\end{tikzpicture}
$$
The right \cll{X}-node at the picture must have a \cll{B}-neighbor,
which has three $\{\cll{1},\cll{2},\cll{3}\}$-neighbors.
The node $[2,3]$ abobe the \cll{Y}-node can only be a \cll{2}-node.
Similarly, $[1,3]$ is an \cll{A}-node.
Now (see the last picture above)
the quotient matrix is completely determined,
and it is easy to see that the coloring is uniquely reconstructed and
meets assertion (a) of the lemma with $\beta=3$.

(a'') The row $[i,2]$ is colored with exactly $6$ colors.
$$
 \begin{tikzpicture}
\draw (0,2) \t\D\6\E\4\F\5\D\6\E\4\F\5\t;
\draw (0,1) \t\2\B\3\C\1\A\2\B\3\C\1\A\t;
\draw (0,0) \t\C\1\A\2\B\3\C\1\A\2\B\3\t;
\end{tikzpicture}
$$
It is easy to see that the row $[i,3]$, $i\in \ZZ$, has no
$\{\cll{1},\cll{2},\cll{3},\cll{A},\cll{B},\cll{C}\}$-nodes
(for example, a \cll{1}-node cannot have a $\{\cll{4},\cll{5},\cll{6},\cll{D},\cll{E}\}$-neighbor,
and an \cll{F}-node cannot have two \cll{1}-neighbors).
So, there are two possibilities for that row.

 {1.} It can have a $\{\cll{4},\cll{5},\cll{6},\cll{D},\cll{E},\cll{F}\}$-node, w.l.o.g., a $\cll 4$-node.
According to Claim~\ref{cl:1122}, that $\cll 4$-node
can only be above a $\cll{D}$-node, and by arguments similar as above
we have $  {\FfF}\Lft i,3\Rgt  = {\FfF}\Lft i-3,2\Rgt  $ for all $i \in \ZZ$.
$$
 \begin{tikzpicture}
\draw (0,3) \t\4\F\5\D\6\E\4\F\5\D\6\E\t;
\draw (0,2) \t\D\6\E\4\F\5\D\6\E\4\F\5\t;
\draw (0,1) \t\2\B\3\C\1\A\2\B\3\C\1\A\t;
\draw (0,0) \t\C\1\A\2\B\3\C\1\A\2\B\3\t;
\end{tikzpicture}
$$
From these four rows, we know the quotient matrix, and the coloring is uniquely reconstructed and satisfies
the statement of the lemma with $\beta=4$.

{2.}  In the other subcase, the row $[i,3]$, $i\in \ZZ$, is colored with new colors
so that $\FfF\Lft i,3\Rgt =\FfF\Lft i+6,3\Rgt $.
If the number of those new colors is less than~$6$,
the arguments similar to (a') show that this number is~$3$
and that assertion (a) of the lemma holds with $\beta=5$.
Otherwise, we have six new colors.

In its turn, the next row $[i,4]$, $i\in \ZZ$, either satisfies
$  {\FfF}\Lft i,4\Rgt  = {\FfF}\Lft i-3,3\Rgt $, $i \in \ZZ$,
or colored with new three of six colors.

In the last case, we apply the same argument to the row $[i,4]$, then $[i,5]$, and so on,
until the colors are exhausted after some row $[i,\gamma]$
in the sense that the number of remaining colors is less than~$6$.
If at least one new color remains, then similarly
to~(a')
the next row $[i,\gamma+1]$ is colored with period $[3,0]$ and assertion~(a)
of the lemma holds with  $\beta=2\gamma+1$.
Otherwise,
the only possibility
is
$  {\FfF}\Lft i,\gamma+1\Rgt  = {\FfF}\Lft i-3,\gamma\Rgt $, $i \in \ZZ$.
From the $\gamma+2$ rows $[i,0]$, \ldots, $[i,\gamma+1]$,
the quotient matrix and the coloring is uniquely reconstructed and satisfies
assertion~(a) of the lemma with $\beta=2\gamma$.

(b) \emph{$\alpha= 5$ or $\alpha\ge 7$}.

The colors of $[0,0]$, $[1,0]$, \ldots, $[\alpha-1,0]$ are pairwise distinct, by the definition of
$\alpha$. The node $[\alpha,0]$ (the right \cll{1}-node in the picture below) is colored with \cll{1}.

\def\1{\Ccolor{white}{red!15!white           }{$1$}}
\def\2{\Ccolor{white}{orange!20!white        }{$2$}}
\def\3{\Ccolor{white}{yellow!30!white        }{$3$}}
\def\4{\Ccolor{white}{green!15!white         }{$4$}}
\def\5{\Ccolor{white}{cyan!20!white          }{$5$}}
\def\E{\Ccolor{white}{blue!10!white          }{$l$}}
\def\L{\Ccolor{white}{blue!50!purple!10!white}{$m$}}
\def\F{\Ccolor{white}{purple!10!white}{$n$}}
$$
 \begin{tikzpicture}
\draw (0,1) \o\o\o\o\1\2\3\4\5\oo\t\oo\o\o\o\E;
\draw (0,0) \o\1\2\3\4\5\O\O\o\oo\t\oo\E\L\F\1\s\o;
\end{tikzpicture}
$$
We see that the neighborhood of a \cll{1}-node has colors
\cll{2}, \cll{4}, \cll{l}, \cll{n}. The color of $[\alpha+1,0]$
(marked by $\star$) cannot be \cll{l} or \cll{n}.
And it cannot be \cll{4}, because
$\FfF\Lft 3,0\Rgt =\FfF\Lft \alpha+1,0\Rgt $ contradicts the definition of $\alpha$.
So, it is \cll{2}. Analogously, the color of $[2,1]$, and hence, of $[-1,0]$, is \cll{n}.
Using similar arguments, we can prove by induction that
$\FfF\Lft i,0\Rgt =\FfF\Lft i+\alpha,0\Rgt $ for any $ i \in \ZZ $, i.e., the row $[i,0]$ (and hence, $[i,1]$)
is colored with $\alpha$ colors in a cyclic way.

Next, we see that a \cll{4}-node has a
\cll{1}-neighbor.
Also we know that $\FfF\Lft 6,0\Rgt \ne \cll{1}$ because $\alpha\ne 6$.
Considering the \cll{4}-node $[6,1]$, we find that the color of $[6,2]$ is \cll{1}:
$$
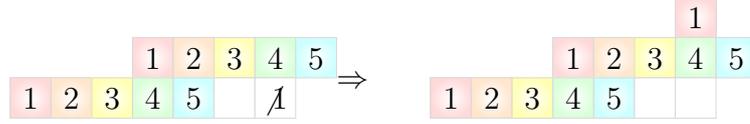
\def\nE{\ccolor{white}{$\not1$}}
 \begin{tikzpicture}
\draw (0,1) \o\o\o\o\1\2\3\4\5;
\draw (0,0) \o\1\2\3\4\5\O\nE\o;
\end{tikzpicture}\ \ \ \
\raisebox{1em}{$\Rightarrow$}
 \begin{tikzpicture}
\draw (0,2) \o\o\o\o\o\o\o\1;
\draw (0,1) \o\o\o\o\1\2\3\4\5 ;
\draw (0,0) \o\1\2\3\4\5\O\O\o ;
\end{tikzpicture}
$$
Therefore, $\FfF\Lft i,2\Rgt  = \FfF\Lft i-3,1\Rgt  = \FfF\Lft i-6,0\Rgt $ for every $ i \in \ZZ $.
Similarly, $\FfF$ is periodic with period $[3,1]$.
\end{proof}

Given a perfect coloring $\FfF$ of $G( \ZZ^2 )$,
we say that two subsets $M$ and $N$ of $\ZZ^2$ are colored \emph{similarly}
if $\varphi(M)=(N)$ and $\FfF(\varphi(\cdot))\equiv\FfF(\cdot)$ for some automorphism
$\varphi$ of $ G( \ZZ^2 )$.

\begin{proposition}\label{p:01}
 Let $\FfF$ be a perfect coloring of $G(\ZZ^2)$ with quotient $\{0,1\}$-matrix with zero diagonal and without equal rows. Then every two radius-$2$ balls centered in nodes of the same color
 are colored similarly.
\end{proposition}

\begin{proof}
Consider a radius-$2$ ball and denote colors of its nodes as follows:
 \def\O{\ccolor{white}{\ }}
\def\o{ ++(1, 0) }
\def\0{\ccolor{gray!10!white}{$0$}}
\def\1{\ccolor{gray!10!white}{$a$}}
\def\2{\ccolor{gray!10!white}{$b$}}
\def\3{\ccolor{gray!10!white}{$c$}}
\def\4{\ccolor{gray!10!white}{$d$}}
\def\5{\ccolor{gray!10!white}{$e$}}
\def\6{\ccolor{gray!10!white}{$f$}}
\def\7{\ccolor{gray!10!white}{$g$}}
\def\8{\ccolor{gray!10!white}{$h$}}
\def\A{\ccolor{white}{$A$}}
\def\B{\ccolor{white}{$B$}}
\def\C{\ccolor{white}{$C$}}
\def\D{\ccolor{white}{$D$}}
\def\F{\ccolor{white}{$\star$}}
\def\D{\ccolor{white}{$D$}}
\def\s{++(1,0) node {$*$}}
$$
 \begin{tikzpicture}
\draw (0,4) \o\o\o\5\o\o;
\draw (0,3) \o\o\4\B\1\o;
\draw (0,2) \o\8\A\0\C\6;
\draw (0,1) \o\o\3\D\2\o;
\draw (0,0) \o\o\o\7\o\o;
\end{tikzpicture}
$$
We consider different cases depending on the equalities that can
take place between the colors \cll{a}--\cll{f}
{%
(independently of equalities that can occur between the colors of
 nodes of different parities,
for example, $\cll a=\cll D$ or $\cll e = \cll C$).
}

{%
By Claim~\ref{cl:1122},
two nodes of the same color cannot have difference in
 $\{\,[\pm 1,\pm 1]$, $[0,\pm 2]$, $[\pm 2,0]$, $[\pm 2,\pm 2]\,\}$.

If there are two nodes of the same color with difference
in
 $\{\,[\pm 1,\pm 3]$, $[\pm 3,\pm 1]\,\}$,
  then the statement holds by Lemma~\ref{l:/}.

So, we can assume that all colors \cll{a}--\cll{f} are mutually different
with maybe one or two exceptions from
$\cll{e}=\cll{g}$, $\cll{f}=\cll{h}$.

Now, if $\cll{f}\ne\cll{h}$ (the case $\cll{e}\ne\cll{g}$ is similar),
then the statement is obvious: for any \cll{0}-node,
the \cll{A}- and \cll{C}-neighbors at distance~$1$ from it
must be opposite, and the colors of the nodes at distance~$2$ from it are uniquely determined
by the four $\{\cll{A},\cll{B},\cll{C},\cll{D}\}$-neighbors.
}

 {Consider the last remaining case:}
 assume that the colors \cll{a}--\cll{f} are pairwise distinct with two exceptions
$\cll{e}=\cll{g}$ and $\cll{f}=\cll{h}$.
If the coloring $\FfF$ is bipartite,
then it is one of two colorings in the statement
of Lemma~\ref{l:+}.
If $\FfF$ is not bipartite, then the corresponding
bipartite coloring $\overline{\FfF}$ (each color is subdivided into
two subcolors in accordance with the parity of the node)
satisfies the hypothesis of Lemma~\ref{l:+} and hence
is one of the two colorings in the conclusion of Lemma~\ref{l:+}.
The original coloring $\FfF$ is obtained from $\overline{\FfF}$
by identifying each color from \cll{0}--\cll{7}
with one of \cll{A}--\cll{H}.

Consider the right coloring
(in further discussion, we refer the picture in Lemma~\ref{l:+}).
Clearly, \cll{0} cannot be identified {with one of its
neighbor colors
\cll{A}, \cll{B}, \cll{C}, \cll{D}}. Next,
\cll{0} cannot be identified with
\cll{G} or \cll{H}
(indeed, every $\{\cll{G},\cll{H}\}$-node has~$7$
 different colors at distance~$2$ from it,
while a \cll{0}-node has only~$6$).
We assume it is identified with \cll{F}
(the case of \cll{E} is similar).

{Consider the left coloring.
Again, \cll{0} cannot be identified with
\cll{A}, \cll{B}, \cll{C}, or \cll{D}.
The remaining four cases (\cll{E}, \cll{F}, \cll{G}, \cll{H})
are similar, and we again assume that
\cll{0} is identified with \cll{F}.}

{The rest of the proof is common for the both colorings.
For every $\cll i$ from \cll1 to \cll6, it holds
$d^2(\cll 0,\cll i)>0$, while for $\cll{i}=\cll 7$
this is not true: $d^2(\cll 0,\cll 7)=0$.
Similarly, $\cll B$ is a distinguished color with respect to
$\cll F$: $d^2(\cll F,\cll B)=0$, while
$d^2(\cll F,\cll j)>0$ for any other color~\cll{j}
of the same parity.
It follows that}
\cll{7} is identified with~\cll{B}.
Now, let us take a look to the neighborhood
of an \cll{F}-node.
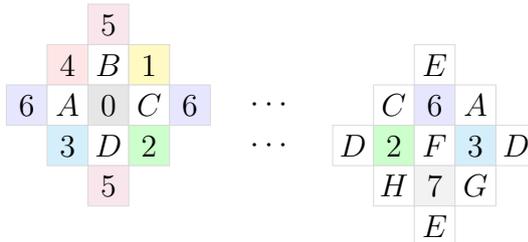
\def\O{\ccolor{white}{\ }}
\def\o{ ++(1, 0) }
\def\0{\ccolor{gray!20!white}{$0$}}
\def\0{\ccolor{gray!20!white}{$0$}}
\def\1{\ccolor{purple!10!white}{$1$}}
\def\2{\ccolor{green!20!white}{$2$}}
\def\3{\ccolor{cyan!15!white}{$3$}}
\def\4{\ccolor{red!10!white}{$4$}}
\def\5{\ccolor{yellow!30!white}{$5$}}
\def\6{\ccolor{blue!10!white}{$6$}}
\def\7{\ccolor{gray!10!white}{$7$}}
\def\A{\ccolor{white}{$A$}}
\def\B{\ccolor{white}{$B$}}
\def\C{\ccolor{white}{$C$}}
\def\D{\ccolor{white}{$D$}}
\def\E{\ccolor{white}{$E$}}
\def\F{\ccolor{white}{$F$}}
\def\G{\ccolor{white}{$G$}}
\def\H{\ccolor{white}{$H$}}
\def\q{\ccolor{white}{$?$}}
\def\u{\ccolor{white}{$=$}}
\def\R{ ++(1, 0) node {\raisebox{-0.3em}{\makebox[0pt]{$\stackrel{\text{up to reflection}}\longleftrightarrow$}}}}
\def\t{ ++(1, 0) node {$\cdots$}}
\def\u{ ++(1, 0) node {$=$}}
$$
\begin{tikzpicture}
\draw (0,5) \o\o\o\5\o\o\o\o\o\o\o\R\o\o;
\draw (0,4) \o\o\4\B\1\o\o\o\o\o\o\E\o\o;
\draw (0,3) \o\6\A\0\C\6\o\t\o\o\C\6\A\o;
\draw (0,2) \o\o\3\D\2\o\o\t\o\D\2\F\3\D;
\draw (0,1) \o\o\o\5\o\o\o\o\o\o\H\7\G\o;
\draw (0,0) \o\o\o\R\o\o\o\o\o\o\o\E\o\o;
\end{tikzpicture}
$$
The neighbors of an \cll{F}-node have colors
\cll{2}, \cll{6}, \cll{3}, \cll{7},
and they should be identified with the neighbor
colors of a \cll{0}-node, i.e., with
\cll{A}, \cll{B}, \cll{C}, \cll{D}, in some order.
We already know that \cll{7}
and \cll{B} are identified.
There is only one possibility for the other three colors,
taking into account that
there is no two neighbor nodes
of the same color. Therefore,

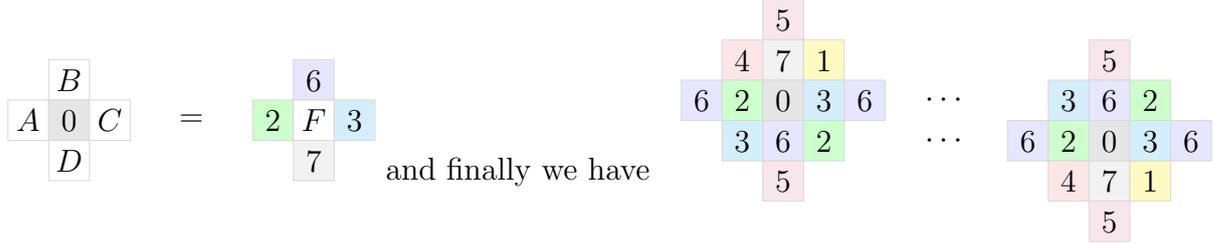
$$
\def\oo{ ++(0.5,0) }
\begin{array}{c}
\begin{tikzpicture}
\draw (0,4) \o\B\o\o    \o\o\7;
\draw (0,3) \A\0\C\oo\u\oo\2\F\3;
\draw (0,2) \o\D\o\o    \o\o\6;
\end{tikzpicture}
\end{array}
\mbox{, and finally we have }
\begin{array}{c}
\begin{tikzpicture}
\draw (0,5) \o\o\5\o\o\o    \o\o\o\o\o\o;
\draw (0,4) \o\4\7\1\o\o    \o\o\o\5\o\o;
\draw (0,3) \6\2\0\3\6\oo\t\oo\o\3\6\2\o;
\draw (0,2) \o\3\6\2\o\oo\t\oo\6\2\0\3\6;
\draw (0,1) \o\o\5\o\o\o    \o\o\4\7\1\o;
\draw (0,0) \o\o\o\o\o\o    \o\o\o\5\o\o;
\end{tikzpicture}
\end{array}
$$
So, we have convinced that
the conclusion of the proposition
holds for two \cll{0}-nodes
not only of the same parity
(which is immediate from Lemma~\ref{l:+}),
but also
of different parities.

\end{proof}

\begin{proof}[Proof of Theorem~\ref{th:01}]
The case when the coloring is bipartite and has twin colors is solved by the characterization in Theorem~\ref{th:main}
and Corollary~\ref{c:bipartite01}.

If the coloring, $\FfF$, is not bipartite and has two equal rows
 in the quotient matrix, then the corresponding bipartite coloring $\overline{\FfF}$
 has twin colors and the quotient is a $\{0,1\}$-matrix.
 As in the paragraph above, $\overline{\FfF}$ has a binary diagonal, and hence $\FfF$ has a binary diagonal too.

 Assume we have a coloring~$\FfF$ satisfying the hypothesis
 of the theorem without two equal rows
 in the quotient matrix.
 {%
 We need to proof that the coloring is orbit,
 which, by the definition, is equivalent to the following assertion.
\begin{claim}\label{cl:ll}
For any two nodes ${\bar x}$ and ${\bar y}$ of the same color,
there is an automorphism~$\varphi$ of~$G(\ZZ^2)$ such that
$\varphi({\bar x}) = {\bar y} $ and for all $\bar u \in \ZZ^2$
 \begin{equation}\label{eq:au1}
 \FfF(\varphi({\bar u}))=\FfF({\bar u}).
 \end{equation}
\end{claim}
We prove the claim by induction on the distance~$l$
between~$\bar x$ and~$\bar u$.
By Proposition~\ref{p:01}, it holds
if $d(\bar x,\bar u)\le 2$, with some automorphism~$\varphi$.
Suppose \eqref{eq:au1} holds with the same~$\varphi$
for every~$\bar u$ at distance less than~$l$
from~$\bar x$, where $l\ge 3$.

Consider~$\bar u$ at distance~$l$ from~$\bar x$.
Consider a node~$\bar x'$ between~$\bar x$ and~$\bar u$ such that
$d(\bar x,\bar x')=l-2$ and $d(\bar x',\bar u)=2$.
By the induction hypothesis, $\FfF(\varphi({\bar x'}))=\FfF({\bar x'})$ and, moreover
\begin{equation}\label{eq:Fw}
\FfF(\varphi({\bar u'}))=\FfF({\bar u'})\quad
\mbox{for all~$\bar u'$ at distance~$1$ from~$\bar x'$,}
\end{equation}
because all such~$\bar u'$ satisfy
$d(\bar x,\bar u')\le d(\bar x,\bar x')+d(\bar x',\bar u')=l-1$.
Now, consider the nodes~$\bar x'$ and $\bar y' = \varphi(\bar x')$.
They have the same color, and by Proposition~\ref{p:01}
there is an automorphism $\psi$ of~$G(\ZZ^2)$ such that
$\psi({\bar x'}) = {\bar y'} $ and for all $\bar u \in \ZZ^2$
 \begin{equation}\label{eq:Fu}
 \FfF(\psi({\bar u'}))=\FfF({\bar u'}) \quad
 \mbox{for all~$\bar u'$ at distance at most~$2$ from~$\bar x'$.}
 \end{equation}
Now we see $\varphi(\bar x')=\psi(\bar x')$.
Moreover, from~\eqref{eq:Fw} and~\eqref{eq:Fu} we find that
both~$\varphi$ and~$\psi$ preserve the color of each neighbor of~$\bar x'$.
Since the neighbors of~$\bar x'$ have pairwise different colors,
it follows that the values of~$\varphi$ and~$\psi$ coincide on the
neighborhood of~$\bar x'$.
Since an automorphism of the grid~$G(\ZZ^2)$ is uniquely determined by its action
on the neighborhood of a node, we have $\psi\equiv \varphi$.
Then, substituting $\bar u'=\bar u$ in~\eqref{eq:Fu},
we obtain~\eqref{eq:au1},
which completes the induction step and the proof of the theorem.
}%
\end{proof}

\section{Non-orbit perfect colorings}\label{s:no}
Here, we provide examples of non-orbit
perfect colorings of $G(\ZZ^2)$,
answering some natural questions.
At first, we observe that in the colorings shown
in Fig.~\ref{MM}--\ref{PP} and~\ref{TT}--\ref{WW}
merging all groups of twin colors results
in non-orbit bipartite perfect colorings
without twin colors (for example in Fig.~\ref{MM}, a~\cll{1}-node
has two \cll{4}-neighbors
in opposite directions,
which is not true for the two \cll{4}-neighbors of a~\cll{2}-node;
after merging, such nodes become of the same color
but cannot be in the same orbit),
providing an infinite number of such examples.
One can ask if there are another examples (in particular,
that cannot be obtained by merging twin colors).
Up to now, we know a finite number of examples
of non-orbit perfect colorings of $G(\ZZ^2)$
without a binary or single-color
diagonal. Here we list the numbers of
such colorings up to $9$ colors,
according to the catalogue~\cite{Kro:9colors}:
2-5,
2-7(1),
3-6,
3-8(1),
3-9(1),
3-14,
3-15,
3-17(2),
3-17(3),
3-20(1),
4-7,
\underline{4-13},
\underline{4-14(2)},
4-16,
4-18,
4-29,
4-32(1),
4-32(2),
4-36,
4-37,
4-38(1),
4-41(1),
5-29,
5-36(1),
5-36(2),
5-42*,
5-43(1)*,
5-45,
\underline{6-22},
\underline{6-23},
\underline{6-36(1)},
\underline{6-58(2)},
6-67(1),
6-80(2)*,
6-81*,
6-83(1),
6-89(3)*,
\underline{8-32},
\underline{8-45},
\underline{8-49},
\underline{8-122},
\underline{8-123(2)},
8-150(1)*,
\underline{9-127}; the underlined numbers correspond to bipartite colorings,
the symbol * indicates that the quotient matrix is a $\{0,1\}$-matrix.
{Each bipartite coloring from those examples
is
either of form~$\overline\GgG$ for
some non-bipartite perfect colorings~$\GgG$,}
or obtained from Theorem~\ref{th:main}
by merging (not necessarily all) twin colors.
Below, we show two non-bipartite examples.

\begin{itemize}

\def\9{\ccolor{red!15!white            }{$9$}}
\def\8{\ccolor{blue!80!yellow!08!white }{$8$}}
\def\7{\ccolor{purple!70!blue!25!white }{$7$}}
\def\6{\ccolor{red!80!yellow!08!white  }{$6$}}
\def\5{\ccolor{blue!30!white           }{$5$}}
\def\4{\ccolor{red!30!white            }{$4$}}
\def\3{\ccolor{yellow!50!white}{$3$}}
\def\2{\ccolor{green!40!white          }{$2$}}
\def\1{\ccolor{blue!30!green!30!white  }{$1$}}
\def\A{\ccolor{blue!15!white           }{$10$}}

\item 3-17(2,3): 
Examples of non-orbit perfect colorings with twin colors.
Note that the quotient matrix (the same for the two colorings) has no equal rows,
and hence the corresponding bipartite
perfect colorings have no twin colors;
so, the hypothetical future
characterization of non-bipartite perfect colorings of $G(\ZZ^2)$
with twin colors
cannot be derived straightforwardly from Theorem~\ref{th:main}.
$$
\dfc{cA}{6eafa7}\dfc{cB}{485d60}\dfc{cC}{ffdb97}
\def\1{\ccolor{cA!60}{$1$}}
\def\2{\ccolor{cB!60}{$2$}}
\def\3{\ccolor{cC!70}{$3$}}
\begin{tikzpicture}
\draw (0,-0) \1\3\3\2\1\3\3\2\1\3\3\2 ;
\draw (0,-1) \1\3\3\2\1\3\3\2\1\3\3\2 ;
\draw (0,-2) \2\2\1\1\2\2\1\1\2\2\1\1 ;
\draw (0,-3) \3\1\2\3\3\1\2\3\3\1\2\3 ;
\draw (0,-4) \3\1\2\3\3\1\2\3\3\1\2\3 ;
\draw (0,-5) \2\2\1\1\2\2\1\1\2\2\1\1 ;
\draw (0,-6) \1\3\3\2\1\3\3\2\1\3\3\2 ;
\draw (0,-7) \1\3\3\2\1\3\3\2\1\3\3\2 ;
\end{tikzpicture}
~
\begin{tikzpicture}
\draw (0,-0)\1\3\3\2\1\3\3\2\1\3\3\2 ;
\draw (0,-1)\1\2\1\2\1\2\1\2\1\2\1\2 ;
\draw (0,-2)\3\2\1\3\3\2\1\3\3\2\1\3 ;
\draw (0,-3)\3\1\2\3\3\1\2\3\3\1\2\3 ;
\draw (0,-4)\2\1\2\1\2\1\2\1\2\1\2\1 ;
\draw (0,-5)\2\3\3\1\2\3\3\1\2\3\3\1 ;
\draw (0,-6)\1\3\3\2\1\3\3\2\1\3\3\2 ;
\draw (0,-7)\1\2\1\2\1\2\1\2\1\2\1\2 ;
\end{tikzpicture}
$$

\item 8-150(1,2):
An example of two colorings with the same $\{0,1\}$-matrix,
the first colorings is not orbit, the second one is orbit.
$$
\dfc{cAA}{57a2e6}
\dfc{cBB}{9cb9cb}
 \dfc{cLL}{b09970}
 \dfc{cMM}{636f61}
 \dfc{cNN}{fbe8c8}
 \dfc{cOO}{cfdfee}
\dfc{cXX}{26344f}
\dfc{cYY}{284b71}
\dfc{cAA}{d2e6eb}
\dfc{cBB}{b6d5e8}
 \dfc{cLL}{f8e8c2}
 \dfc{cMM}{a7a0a6}
 \dfc{cNN}{a78b7a}
 \dfc{cOO}{8b898f}
\dfc{cXX}{828a8e}
\dfc{cYY}{68727a}
\def\3{\ccolor{cAA}{$3$}}
\def\6{\ccolor{cBB}{$6$}}
 \def\1{\ccolor{cLL}{$1$}}
 \def\2{\ccolor{cNN!80}{$2$}}
 \def\7{\ccolor{cMM!95}{$7$}}
 \def\8{\ccolor{cOO!50}{$8$}}
\def\4{\ccolor{cXX}{$4$}}
\def\5{\ccolor{cYY}{$5$}}
\definecolor{cAA}{HTML}{303a31}
\definecolor{cBB}{HTML}{394344}
\definecolor{cLL}{HTML}{cad2e2}
\definecolor{cMM}{HTML}{dce3f3}
\definecolor{cNN}{HTML}{b9c0d0}
\definecolor{cOO}{HTML}{a4acb9}
\definecolor{cXX}{HTML}{6c5545}
\definecolor{cYY}{HTML}{816e72}
\def\3{\ccolor{cAA!85}{$3$}}
\def\6{\ccolor{cBB!70}{$6$}}
%
\definecolor{cAA}{HTML}{e6ecf8}\def\3{\ccolor{cAA!70}{$1$}}
\definecolor{cBB}{HTML}{f3edf3}\def\6{\ccolor{cBB!50}{$2$}}
\definecolor{cLL}{HTML}{cad2e2}\def\1{\ccolor{cLL}{$3$}}
\definecolor{cMM}{HTML}{dce3f3}\def\7{\ccolor{cMM}{$4$}}
\definecolor{cNN}{HTML}{b9c0d0}\def\2{\ccolor{cNN}{$5$}}
\definecolor{cOO}{HTML}{a4acb9}\def\8{\ccolor{cOO}{$6$}}
\definecolor{cXX}{HTML}{6c5545}\def\4{\ccolor{cXX!80}{$7$}}
\definecolor{cYY}{HTML}{303a31}\def\5{\ccolor{cYY!60}{$8$}}
\definecolor{cAA}{HTML}{c8e1e7}\def\3{\ccolor{cAA}{$1$}}
\definecolor{cBB}{HTML}{b4dbe5}\def\6{\ccolor{cBB}{$2$}}
\definecolor{cLL}{HTML}{3e5157}\def\1{\ccolor{cLL!85}{$3$}}
\definecolor{cMM}{HTML}{919193}\def\7{\ccolor{cMM}{$4$}}
\definecolor{cNN}{HTML}{4f4b4e}\def\2{\ccolor{cNN!80}{$5$}}
\definecolor{cOO}{HTML}{8f887f}\def\8{\ccolor{cOO}{$6$}}
\definecolor{cXX}{HTML}{dbdbdf}\def\4{\ccolor{cXX}{$7$}}
\definecolor{cYY}{HTML}{a8a4c1}\def\5{\ccolor{cYY!80}{$8$}}
\definecolor{cAA}{HTML}{c8e1e7}\def\3{\ccolor{cAA!80}{$1$}}
\definecolor{cBB}{HTML}{b4dbe5}\def\6{\ccolor{cBB!80}{$2$}}
\definecolor{cNN}{HTML}{707384}\def\1{\ccolor{cNN!90!white}{$5$}}
\definecolor{cOO}{HTML}{4e6166}\def\7{\ccolor{cOO!95!white!90!white}{$6$}}
\definecolor{cLL}{HTML}{5e828b}\def\2{\ccolor{cLL!80!white!90!white}{$3$}}
\definecolor{cMM}{HTML}{524b4b}\def\8{\ccolor{cMM!77!white!80!white!90!white}{$4$}}
\definecolor{cXX}{HTML}{dbdbdf}\def\4{\ccolor{cXX!90!white}{$7$}}
\definecolor{cYY}{HTML}{a8a4c1}\def\5{\ccolor{cYY!80!white!90!white}{$8$}}
\begin{tikzpicture}
\draw (0,-0) \1\6\6\7\1\6\6\7\1\6\6\7;
\draw (0,-1) \5\7\1\5\5\7\1\5\5\7\1\5;
\draw (0,-2) \4\2\8\4\4\2\8\4\4\2\8\4;
\draw (0,-3) \8\3\3\2\8\3\3\2\8\3\3\2;
\draw (0,-4) \1\6\6\7\1\6\6\7\1\6\6\7;
\draw (0,-5) \5\7\1\5\5\7\1\5\5\7\1\5;
\draw (0,-6) \4\2\8\4\4\2\8\4\4\2\8\4;
\draw (0,-7) \8\3\3\2\8\3\3\2\8\3\3\2;
\draw (0,-8) \1\6\6\7\1\6\6\7\1\6\6\7;
\draw (0,-9) \5\7\1\5\5\7\1\5\5\7\1\5;
\end{tikzpicture}
~~
\begin{tikzpicture}
\draw (0,-0) \1\7\2\8\1\7\2\8\1\7\2\8;
\draw (0,-1) \5\5\4\4\5\5\4\4\5\5\4\4;
\draw (0,-2) \7\1\8\2\7\1\8\2\7\1\8\2;
\draw (0,-3) \6\6\3\3\6\6\3\3\6\6\3\3;
\draw (0,-4) \1\7\2\8\1\7\2\8\1\7\2\8;
\draw (0,-5) \5\5\4\4\5\5\4\4\5\5\4\4;
\draw (0,-6) \7\1\8\2\7\1\8\2\7\1\8\2;
\draw (0,-7) \6\6\3\3\6\6\3\3\6\6\3\3;
\draw (0,-8) \1\7\2\8\1\7\2\8\1\7\2\8;
\draw (0,-9) \5\5\4\4\5\5\4\4\5\5\4\4;
\end{tikzpicture}
$$
\end{itemize}
The number of colors (nine) for which the perfect colorings of $G(\ZZ^2)$
are characterized is relatively small to make convincing conjectures,
but clearly, the problem of characterization of all such colorings
includes the following question:

{\bf Question.} Is the number of non-orbit perfect colorings of $G(\ZZ^2)$
without
a binary or single-color diagonal finite?

\appendix
\section{Appendix. Wallpaper groups and orbit colorings}
The concept of orbit coloring plays a fundamental role
in our study.
Here, we briefly mention the classification
of subgroups of the automorphism group of
$G(\ZZ^2)$ of finite index (i.e., with finite number of orbits),
mainly focused on subgroups
whose orbit partitions
fit the hypothesis and conclusion of Theorem~\ref{th:01}.
The set of nodes of
$G(\ZZ^2)$ can be naturally treated as a set points
of the Euclidean plain,
forming a square lattice.
Every automorphism of the square grid is continued
to an isometry of the Euclidean plain.
A group of isometries of the Euclidean plain
whose translation
subgroup is additively spanned
by two linearly independent vectors
(any group of automorphisms of~$G(\ZZ^2)$
of finite index satisfies this condition)
is called a \emph{wallpaper group}. It is known that the
wallpaper groups are divided into $17$ families
(see, e.g., \cite[Ch.\,2]{CBGS:Symmetries}),
where the groups in the same family
are conjugate to each other in the group of affine transformations
of the plane.
Of the $17$ families, only~$12$ include subgroups of the
automorphism group of $G(\ZZ^2)$
(the groups from the remaining $5$ families contain a
$120^{\circ}$-rotation, and hence cannot be such a subgroup).
The groups are parameterized with
three vectors $\bar u$, $\bar v$, and~$\bar o$.
The vectors~$\bar u$ and~$\bar v$ are periods
that generate the translation subgroup;
$\bar o$ is some point of the plain that defines the positions
of the rotation centers and/or reflection (glide reflection) axes
(some groups do not depend on $\bar o$
or depend only on one coordinate component of it).

There are three kinds of non-translation elements
in wallpaper groups:
rotations; reflections; glide reflections
(a glide reflection is the composition
of a reflection {and} a translation along the reflection axe).

The group of all automorphisms of $G(\ZZ^2)$
is a wallpaper group of the type called \emph{p4m}.
All its finite-index subgroups
are listed in \cite{SBZ:78},
as well as the wallpaper subgroups
of each of the $17$ wallpaper groups.

It is easy to see that if a group of automorphisms of~$G(\ZZ^2)$ contains
a rotation or a reflection, then there is an orbit of nodes with the minimum
distance~$1$ or~$2$ between different elements. This means that the quotient
matrix contains a non-zero element on the main diagonal or
an element larger than~$1$ anywhere. Hence, such orbit coloring cannot
be a covering of a simple graph.
Below, we describe the two families (named \emph{p1} and \emph{pg})
of wallpaper groups that do not have
reflections and rotations.
The description of each group is followed (after a bullet)
by the restrictions on the parameters that are necessary and sufficient for the group to be
a subgroup of the automorphism group of~$G(\ZZ^2)$.

\begin{itemize}
 \item[\bf\em p1.] The group does not have non-translation elements.
 \item For a \emph{p1}-group to consist of automorphisms of~$G(\ZZ^2)$,
 the translations must be integer.
 For the orbit coloring of the nodes of~$G(\ZZ^2)$ to be a covering,
 all translations must be different from $[0,\pm1]$, $[0,\pm2]$,
 $[\pm1,\pm1]$.
 The orbit coloring
 is bipartite if and only if all translations are even.

 \item[\bf\em pg.] The translation subgroup
 is generated by two perpendicular vectors~$\bar u$ and~$\bar v$;
 the non-translation elements are
 glide reflections with axes along~$\bar u$ through
 the points $\bar o+\frac j2\bar v$, $j\in\ZZ$.
 \item
 There are two possibilities for a \emph{pg}-group to consist of automorphisms of $G(\ZZ^2)$.
 \begin{itemize}
 \item[(a)] Diagonal glide reflections:
 $\bar u = [s,s]$, $\bar v = [t,-t]$ or
 $\bar u = [s,-s]$, $\bar v = [t,t]$,
 $s,t\in\ZZ^+$.
 For the orbit coloring of the nodes of $G(\ZZ^2)$ to be a covering,
 $s$ must be at least~$3$ and $t$ at least~$2$.
 There are additional restrictions:
 \begin{itemize}
           \item[(a0)] if $s$ is even, then $\bar o$ is either integer
           or has two non-integer coordinate components
           (and hence the reflection axes contain integer points);
           the orbit coloring is bipartite in this case;
           \item[(a1)] if $s$ is odd, then $\bar o$ has exactly
           one integer component
           (and hence the reflection axes do not contain integer points);
           the orbit coloring is not bipartite.
 \end{itemize}
 \item[(b)] Horizontal or vertical reflections:
            $\bar u = [0,2s]$, $\bar v = [t,0]$ or
            $\bar u = [2s,0]$, $\bar v = [0,t]$, $s,t\in\ZZ^+$.
           Case (b) is divided into three subcases:
 \begin{itemize}
           \item[(b0)] all reflection axes
           contain integer points (even~$t$);
           \item[(b1)] there are reflection
           axes both containing and not containing
           integer points (odd~$t$);
           \item[(b2)] none of reflection axes
           contains an integer point (even~$t$).
 \end{itemize}
           Only subcase (b0) with even~$s$ and subcase (b2)
           with odd~$s$ correspond to bipartite orbit colorings.
           For the orbit coloring of the nodes of~$G(\ZZ^2)$ to be a covering,
           $t$ must be at least~$3$ and $s$ at least $2$ for~(b0) and~(b1)
           and at least~$1$ for~(b2).
 \end{itemize}
 \end{itemize}


\section{Appendix. Perfect colorings with twin colors}
Here we describe $23$ types of perfect colorings that, according to Theorem~\ref{th:main}, exhaust the bipartite perfect of $G(\ZZ^2)$
with twin colors. Each type is represented with the coloring
shown in one of Fig.~\ref{AA}--\ref{WW}, and the other colorings
of the same type are obtained by applying automorphisms of $G(\ZZ^2)$,
renaming colors, shifting binary diagonals, and/or merging
groups of twin colors.

Fig.~\ref{KK}--\ref{WW} show infinite parameterized series,
where $n=m+1=l+2=\ldots=c+11$ is the number of colors.
The arrows show the direction in which the colors increase
with increment $2$ (Fig.~\ref{KK}--\ref{LL}), $5$ (\ref{MM}--\ref{SS}),
or $4$ (\ref{TT}--\ref{WW}); the number of steps depends on $n$.
Each of Fig.~\ref{KK}--\ref{WW} is accompanied with
an analytic description of the corresponding coloring.

To simplify some descriptions,
we introduce  three colorings
with infinite number of colors.
The coloring $\FfF_{23}$, $\FfF_{32}$, and $\FfF_{44}$
(partially illustrated in Fig.~\ref{MM}--\ref{PP},
 Fig.~\ref{QQ}--\ref{SS}, and  Fig.~\ref{TT}--\ref{WW}, respectively)
 are defined by the period $[-4,4]$ for $\FfF_{23}$, $\FfF_{32}$
 and $[4,4]$ for $\FfF_{44}$
and the following equalities for each
\def\bgg{\big}
$\alpha \in \{1,2,3,\ldots\}$,
 $\gamma \in \{0,1,2,3\}$:
$$
\FfF_{23}\Lft 0,0\Rgt =\FfF_{23}\Lft -2,2\Rgt =  \cll{1},
\qquad
\FfF_{23}\Lft -1,1\Rgt =\FfF_{23}\Lft -3,3\Rgt =  \cll{2},
$$
$$
\FfF_{23}\big( \pm [\alpha + \gamma - 1,\alpha - \gamma] \,\big)
=
\framebox{$5\alpha - \frac12 - |\gamma - \frac32|$},
$$
$$
\FfF_{23}\big( \pm [\alpha + \gamma + 1,\alpha - \gamma - 1] \,\big)
=
\framebox{$5\alpha + |2 - \gamma|$}.
$$
$$
\FfF_{32}\Lft 0,0\Rgt =\FfF_{32}\Lft -2,2\Rgt = \cll{1},
\qquad
\FfF_{32}\Lft -1,1\Rgt =\FfF_{32}\Lft -3,3\Rgt =\cll{2},
$$
$$
\FfF_{32}\big(\, [{\textstyle\frac12},-{\textstyle\frac12}]
\pm [ \alpha+\beta-{\textstyle\frac32}, \alpha-\beta+{\textstyle\frac12} ] \,\big)
= \framebox{$5\alpha-|\gamma-2|$},
$$
$$
\FfF_{32}\big(\, [{\textstyle\frac12},-{\textstyle\frac12}]
\pm [ \alpha+\beta-{\textstyle\frac12}, \alpha-\beta+{\textstyle\frac12} ] \,\big)
= \framebox{$5\alpha+\frac52-|\gamma-\frac32|$}.
$$
$$
\FfF_{44}\Lft 0,0\Rgt =\FfF_{44}\Lft 2,2\Rgt =\cll{1},
\qquad
\FfF_{44}\Lft 1,1\Rgt =\FfF_{44}\Lft 3,3\Rgt =\cll{2},
$$
$$
\FfF_{44}\Lft \alpha + \gamma,\gamma\Rgt =
\FfF_{44}\Lft 2 + \gamma,\alpha + 2 + \gamma\Rgt =
\framebox{$4\alpha+\gamma-1$}.
$$

\def\ccolor#1#2{ ++(1,0) node [cell,fill=#1] {\raisebox{-0.25em}[0.2em][0em]{\makebox[0mm][c]{#2}}} }
\begin{figure}[H]
\scalebox{1.10}{
\begin{tikzpicture}
\definecolor{Ca}{HTML}{b85f06}
\definecolor{cA}{HTML}{fdf520}\def\1{\Ccolor{Ca}{cA}{$1$}}
\definecolor{cB}{HTML}{92b76c}\def\2{\ccolor{cB}{$2$}}
\definecolor{cC}{HTML}{56721e}\def\3{\ccolor{cC!80}{$3$}}
\definecolor{cD}{HTML}{76a2d5}\def\4{\Ccolor{cyan!10!cD!80}{white}{$4$}}
\definecolor{cE}{HTML}{a6b8d6}\def\5{\Ccolor{red!10!cE!90}{white}{$5$}}
\definecolor{cF}{HTML}{6e90bd}\def\6{\Ccolor{cF!80}{white}{$6$}}
\definecolor{cG}{HTML}{75abe8}\def\7{\ccolor{cG}{$7$}}
\definecolor{cH}{HTML}{4280d0}\def\8{\ccolor{cH!90}{$8$}}
\definecolor{Ci}{HTML}{b4ceed}
\definecolor{cI}{HTML}{ffffef}\def\9{\Ccolor{cI}{Ci!50!cG}{$9$}}
\draw (0,-0) \9\8\4\7\9\7\4\8\9\8\4\7\9\7\4\8 ;
\draw (0,-1) \7\5\3\6\8\5\2\6\7\5\3\6\8\5\2\6 ;
\draw (0,-2) \4\2\1\2\4\3\1\3\4\2\1\2\4\3\1\3 ;
\draw (0,-3) \8\6\3\5\7\6\2\5\8\6\3\5\7\6\2\5 ;
\draw (0,-4) \9\7\4\8\9\8\4\7\9\7\4\8\9\8\4\7 ;
\draw (0,-5) \8\5\2\6\7\5\3\6\8\5\2\6\7\5\3\6 ;
\draw (0,-6) \4\3\1\3\4\2\1\2\4\3\1\3\4\2\1\2 ;
\draw (0,-7) \7\6\2\5\8\6\3\5\7\6\2\5\8\6\3\5 ;
\draw (0,-8) \9\8\4\7\9\7\4\8\9\8\4\7\9\7\4\8 ;
\draw (0,-9) \7\5\3\6\8\5\2\6\7\5\3\6\8\5\2\6 ;
\draw (0,-10)\4\2\1\2\4\3\1\3\4\2\1\2\4\3\1\3 ;
\draw (0,-11) \7\6\2\5\8\6\3\5\7\6\2\5\8\6\3\5 ;
\draw [diag] (0.5,-9.5)--(2.5,-11.5);
\draw [diag] (0.5,-5.5)--(6.5,-11.5);
\draw [diag] (0.5,-1.5)--(10.5,-11.5);
\draw [diag] (2.5, 0.5)--(14.5,-11.5);
\draw [diag] (6.5, 0.5)--(16.5,-9.5);
\draw [diag] (10.5, 0.5)--(16.5,-5.5);
\draw [diag] (14.5, 0.5)--(16.5,-1.5);
\draw [diag] (0.5,-2.5)--(3.5, 0.5);
\draw [diag] (0.5,-6.5)--(7.5, 0.5);
\draw [diag] (0.5,-10.5)--(11.5, 0.5);
\draw [diag] (3.5,-11.5)--(15.5, 0.5);
\draw [diag] (7.5,-11.5)--(16.5,-2.5);
\draw [diag] (11.5,-11.5)--(16.5,-6.5);
\draw [diag] (15.5,-11.5)--(16.5,-10.5);
\end{tikzpicture}%
} 
\hfill
\begin{tikzpicture}
\def\1{\ccolor{white}{$1$}}
\def\2{\ccolor{white}{$2$}}
\def\4{\ccolor{white}{$4$}}
\def\O{\ccolor{white}{\ }}
\def\0{\ccolor{white}{$0$}}
\def\o{ ++(1, 0) }
\draw (0,-0) \O\2\2\o\o\o\o\o\o ;
\draw (0,-1) \1\O\O\1\1\1\o\o\o ;
\draw (0,-2) \1\O\O\1\1\1\o\o\o ;
\draw (0,-3) \o\1\1\O\O\O\1\1\o ;
\draw (0,-4) \o\1\1\O\O\O\1\1\o ;
\draw (0,-5) \o\1\1\O\O\O\1\1\o ;
\draw (0,-6) \o\o\o\1\1\1\O\O\1 ;
\draw (0,-7) \o\o\o\1\1\1\O\O\1 ;
\draw (0,-8) \o\o\o\o\o\o\2\2\O ;
\draw (0.5,0.5)rectangle(9.5,-8.5);
\end{tikzpicture}
\refstepcounter{figure}
\caption*{Fig.~\thefigure\ (Sunflower field).\hfill Twins: $\cll2,\cll3$; $\cll4,\cll5,\cll6$; $\cll7,\cll8$.\hfill Periods: $[8,0]$, $[4,4]$.}\label{AA}
\end{figure}

\begin{figure}[H]
\scalebox{1.10}{
\begin{tikzpicture}
\dfc{Ca}{3d090a}
\dfc{cA}{8e0603}\def\1{\Ccolor{Ca}{cA}{$1$}}
\dfc{cBB}{e88f8f}\dfc{cCC}{fe6a35}
\dfc{cB}{f96146}\def\2{\Ccolor{cB}{cBB}{$2$}}
\dfc{cC}{d93903}\def\3{\Ccolor{cC}{cCC}{$3$}}
\dfc{cE}{93b2cf}\def\5{\ccolor{cE}{$5$}}
\dfc{cG}{6497cd}\def\7{\ccolor{cG}{$7$}}
\dfc{cII}{6b6ca0}
\dfc{cI}{c2c6cf}\def\9{\Ccolor{cI}{cI!80!cE}{$9$}}
\dfc{cD}{86958e}\def\4{\ccolor{cD}{$4$}}
\dfc{cF}{869148}\def\6{\ccolor{cF}{$6$}}
\dfc{cH}{6b7924}\def\8{\ccolor{cH}{$8$}}
\dfc{cJ}{c2b67c}
\dfc{cJJ}{c2b67c}\def\A{\ccolor{cJJ!60!cF}{$10$}}
\draw (0,-0)  \9\7\4\8\A\8\4\7\9\7\4\8\A\8\4\7 ;
\draw (0,-1)  \7\5\3\6\8\6\2\5\7\5\3\6\8\6\2\5 ;
\draw (0,-2)  \4\2\1\2\4\3\1\3\4\2\1\2\4\3\1\3 ;
\draw (0,-3)  \8\6\3\5\7\5\2\6\8\6\3\5\7\5\2\6 ;
\draw (0,-4)  \A\8\4\7\9\7\4\8\A\8\4\7\9\7\4\8 ;
\draw (0,-5)  \8\6\2\5\7\5\3\6\8\6\2\5\7\5\3\6 ;
\draw (0,-6)  \4\3\1\3\4\2\1\2\4\3\1\3\4\2\1\2 ;
\draw (0,-7)  \7\5\2\6\8\6\3\5\7\5\2\6\8\6\3\5 ;
\draw (0,-8)  \9\7\4\8\A\8\4\7\9\7\4\8\A\8\4\7 ;
\draw (0,-9)  \7\5\3\6\8\6\2\5\7\5\3\6\8\6\2\5 ;
\draw (0,-10) \4\2\1\2\4\3\1\3\4\2\1\2\4\3\1\3 ;
\draw (0,-11) \7\5\2\6\8\6\3\5\7\5\2\6\8\6\3\5 ;
\end{tikzpicture}%
}
\hspace{-5cm}
\hfill%
\begin{tikzpicture}
\def\1{\ccolor{white}{$1$}}
\def\2{\ccolor{white}{$2$}}
\def\4{\ccolor{white}{$4$}}
\def\O{\ccolor{white}{\ }}
\def\0{\ccolor{white}{$0$}}
\def\o{ ++(1, 0) }
\draw (0,-0) \O\2\2\o\o\o\o\o\o\o ;
\draw (0,-1) \1\O\O\1\1\1\o\o\o\o ;
\draw (0,-2) \1\O\O\1\1\1\o\o\o\o ;
\draw (0,-3) \o\1\1\O\O\O\1\1\o\o ;
\draw (0,-4) \o\1\1\O\O\O\2\0\o\o ;
\draw (0,-5) \o\1\1\O\O\O\0\2\o\o ;
\draw (0,-6) \o\o\o\1\2\0\O\O\1\0 ;
\draw (0,-7) \o\o\o\1\0\2\O\O\0\1 ;
\draw (0,-8) \o\o\o\o\o\o\4\0\O\O ;
\draw (0,-9) \o\o\o\o\o\o\0\4\O\O ;
\draw (0.5,0.5)rectangle(10.5,-9.5);
\end{tikzpicture}
\refstepcounter{figure}
\caption*{Fig.~\thefigure\ (Poppy field).\hfill Twins: $\cll2,\cll3$.\hfill Periods: $[8,0]$, $[4,4]$.}
\label{BB}
\end{figure}

\begin{figure}[H]
$
\begin{tikzpicture}
\def\cccolor#1#2{\Ccolor{#1}{#1!70}{#2}}
\dfc{cA}{657c9a}\def\1{\cccolor{cA}{$1$}}
\dfc{cA}{f1f0ed}\def\1{\cccolor{cA}{$1$}}
\dfc{cB}{a1acbd}\def\2{\cccolor{cB!90!white}{$2$}}
\dfc{cC}{a4c5e6}\def\3{\ccolor{cC}{$3$}}
\dfc{cD}{019366}\def\4{\ccolor{cD!85}{$4$}}
\dfc{cE}{6299dc}\def\5{\ccolor{cE!85}{$5$}}
\dfc{cF}{1f6f65}\def\6{\ccolor{cF!80}{$6$}}
\dfc{cG}{679019}\def\7{\cccolor{cG}{$7$}}
\dfc{cH}{516e29}\def\8{\cccolor{cH}{$8$}}
\draw (0,-0) \1\7\2\8\1\7\2\8\1\7\2\8\1\7\2\8\1\7;
\draw (0,-1) \5\3\5\3\5\3\5\3\5\3\5\3\5\3\5\3\5\3;
\draw (0,-2) \2\8\1\7\2\8\1\7\2\8\1\7\2\8\1\7\2\8;
\draw (0,-3) \6\4\6\4\6\4\6\4\6\4\6\4\6\4\6\4\6\4;
\draw (0,-4) \1\7\2\8\1\7\2\8\1\7\2\8\1\7\2\8\1\7;
\draw (0,-5) \5\3\5\3\5\3\5\3\5\3\5\3\5\3\5\3\5\3;
\draw (0,-6) \2\8\1\7\2\8\1\7\2\8\1\7\2\8\1\7\2\8;
\draw (0,-7) \6\4\6\4\6\4\6\4\6\4\6\4\6\4\6\4\6\4;
\draw (0,-8) \1\7\2\8\1\7\2\8\1\7\2\8\1\7\2\8\1\7;
\draw (0,-9) \5\3\5\3\5\3\5\3\5\3\5\3\5\3\5\3\5\3;
\end{tikzpicture}%
\hfill
\begin{tikzpicture}
\def\1{\ccolor{white}{$1$}}
\def\2{\ccolor{white}{$2$}}
\def\4{\ccolor{white}{$4$}}
\def\O{\ccolor{white}{\ }}
\def\0{\ccolor{white}{$0$}}
\def\o{ ++(1, 0) }
\draw (0,-0) \O\O\O\O\1\1\1\1 ;
\draw (0,-1) \O\O\O\O\1\1\1\1 ;
\draw (0,-2) \O\O\O\O\2\0\1\1 ;
\draw (0,-3) \O\O\O\O\0\2\1\1 ;
\draw (0,-4) \1\1\2\0\O\O\O\O ;
\draw (0,-5) \1\1\0\2\O\O\O\O ;
\draw (0,-6) \1\1\1\1\O\O\O\O ;
\draw (0,-7) \1\1\1\1\O\O\O\O ;
\draw (0.5,0.5) rectangle (8.5,-7.5);
\end{tikzpicture}
$
\refstepcounter{figure}
\caption*{Fig.~\thefigure\ (Mountain lake).\hfill Twins: $\cll1,\cll2$; $\cll7,\cll8$.\hfill Periods: $[4,0]$, $[0,4]$.\\
          The quotient matrix is similar to that in Fig.~\ref{JJ}.}
\label{CC}
\end{figure}

\begin{figure}[H]
$
\begin{tikzpicture}
\def\cccolor#1#2{\Ccolor{#1}{#1!70}{#2}}
\def\CColor#1#2#3{\Ccolor{#2}{#1}{#3}}
\dfc{o}{ffffff}
\dfc{cAA}{e6d1bc}\dfc{cBB}{dfdbda}
\dfc{cA}{bc9c86}\def\1{\Ccolor{cA}{cAA!70}{$1$}}
\dfc{cB}{a9a3a7}\def\2{\Ccolor{cB}{cBB!70}{$2$}}
\dfc{cC}{d6935b}\def\3{\ccolor{cC!85}{$3$}}
\dfc{cF}{bd8056}\def\6{\ccolor{cF!85}{$6$}}
\dfc{cD}{3e8661}\def\4{\ccolor{green!2!cD!95}{$4$}}
\dfc{cE}{3d8789}\def\5{\ccolor{blue!2!cE!95}{$5$}}
\dfc{cG}{a99675}\def\7{\ccolor{cG!70}{$7$}}
\dfc{cH}{768672}\def\8{\ccolor{cH!60}{$8$}}
\dfc{cI}{8db49e}\def\9{\ccolor{cI}{$9$}}
\draw (0,-0) \1\7\2\8\1\7\2\8\1\7\2\8\1\7\2\8\1\7;
\draw (0,-1) \6\3\6\3\6\3\6\3\6\3\6\3\6\3\6\3\6\3;
\draw (0,-2) \2\8\1\7\2\8\1\7\2\8\1\7\2\8\1\7\2\8;
\draw (0,-3) \9\5\9\4\9\5\9\4\9\5\9\4\9\5\9\4\9\5;
\draw (0,-4) \1\8\2\7\1\8\2\7\1\8\2\7\1\8\2\7\1\8;
\draw (0,-5) \6\3\6\3\6\3\6\3\6\3\6\3\6\3\6\3\6\3;
\draw (0,-6) \2\7\1\8\2\7\1\8\2\7\1\8\2\7\1\8\2\7;
\draw (0,-7) \9\4\9\5\9\4\9\5\9\4\9\5\9\4\9\5\9\4;
\draw (0,-8) \1\7\2\8\1\7\2\8\1\7\2\8\1\7\2\8\1\7;
\draw (0,-9) \6\3\6\3\6\3\6\3\6\3\6\3\6\3\6\3\6\3;
\end{tikzpicture}%
\hfill
\begin{tikzpicture}
\def\1{\ccolor{white}{$1$}}
\def\2{\ccolor{white}{$2$}}
\def\4{\ccolor{white}{$4$}}
\def\O{\ccolor{white}{\ }}
\def\0{\ccolor{white}{$0$}}
\def\o{ ++(1, 0) }
\draw (0,-0) \O\O\O\O\O\1\1\1\1 ;
\draw (0,-1) \O\O\O\O\O\1\1\1\1 ;
\draw (0,-2) \O\O\O\O\O\2\1\1\0 ;
\draw (0,-3) \O\O\O\O\O\0\2\0\2 ;
\draw (0,-4) \O\O\O\O\O\0\0\2\2 ;
\draw (0,-5) \1\1\2\0\0\O\O\O\O ;
\draw (0,-6) \1\1\1\1\0\O\O\O\O ;
\draw (0,-7) \1\1\1\0\1\O\O\O\O ;
\draw (0,-8) \1\1\0\1\1\O\O\O\O ;
\draw (0.5,0.5)rectangle(9.5,-8.5);
\end{tikzpicture}
$
\refstepcounter{figure}
\caption*{Fig.~\thefigure\ (Beach).\hfill Twins: $\cll1,\cll2$.\hfill Periods: $[4,0]$, $[0,8]$. }
\label{DD}
\end{figure}

\begin{figure}[H]
$
\begin{tikzpicture}
\def\cccolor#1#2{\Ccolor{white}{#1}{#2}}
\dfc{o}{ffffff}
\dfc{cAA}{e6d1bc}\dfc{cBB}{dfdbda}
\dfc{cB}{dbf5fb}\def\2{\ccolor{cB!80}{$2$}}
\dfc{cA}{eee9ca}\def\1{\Ccolor{yellow!20}{cB!60}{$1$}}
\dfc{cC}{f9cd24}\def\3{\ccolor{cC!90!white}{$3$}}
\dfc{cF}{ac967f}\def\6{\ccolor{cF}{$6$}}
\dfc{cD}{98a142}\def\4{\ccolor{cD}{$4$}}
\dfc{cE}{b995e2}\def\5{\ccolor{cE}{$5$}}
\dfc{cG}{ffffff}\def\7{\ccolor{cC!90!white!50!cF}{$7$}}
\dfc{cH}{ffffff}\def\8{\ccolor{cC!90!white!50!cD}{$8$}}
\dfc{cI}{ffffff}\def\9{\ccolor{cE!50!cD}{$9$}}
\dfc{cJ}{ffffff}\def\A{\ccolor{cE!50!cF}{$10$}}
\draw (0, 1) \6\1\4\2\6\1\4\2\6\1\4\2\6\1\4\2\6\1;
\draw (0,-0) \A\5\9\5\A\5\9\5\A\5\9\5\A\5\9\5\A\5;
\draw (0,-1) \6\2\4\1\6\2\4\1\6\2\4\1\6\2\4\1\6\2;
\draw (0,-2) \7\3\8\3\7\3\8\3\7\3\8\3\7\3\8\3\7\3;
\draw (0,-3) \6\1\4\2\6\1\4\2\6\1\4\2\6\1\4\2\6\1;
\draw (0,-4) \A\5\9\5\A\5\9\5\A\5\9\5\A\5\9\5\A\5;
\draw (0,-5) \6\2\4\1\6\2\4\1\6\2\4\1\6\2\4\1\6\2;
\draw (0,-6) \7\3\8\3\7\3\8\3\7\3\8\3\7\3\8\3\7\3;
\draw (0,-7) \6\1\4\2\6\1\4\2\6\1\4\2\6\1\4\2\6\1;
\draw (0,-8) \A\5\9\5\A\5\9\5\A\5\9\5\A\5\9\5\A\5;
\end{tikzpicture}%
\hfill
\begin{tikzpicture}
\def\1{\ccolor{white}{$1$}}
\def\2{\ccolor{white}{$2$}}
\def\4{\ccolor{white}{$4$}}
\def\O{\ccolor{white}{\ }}
\def\0{\ccolor{white}{$0$}}
\def\o{ ++(1, 0) }
\draw (0,-0) \O\O\1\1\1\1\o\o\o\o ;
\draw (0,-1) \O\O\1\1\1\1\o\o\o\o ;
\draw (0,-2) \1\1\O\O\O\O\1\1\0\0 ;
\draw (0,-3) \1\1\O\O\O\O\0\1\1\0 ;
\draw (0,-4) \1\1\O\O\O\O\0\0\1\1 ;
\draw (0,-5) \1\1\O\O\O\O\1\0\0\1 ;
\draw (0,-6) \o\o\2\0\0\2\O\O\O\O ;
\draw (0,-7) \o\o\2\2\0\0\O\O\O\O ;
\draw (0,-8) \o\o\0\2\2\0\O\O\O\O ;
\draw (0,-9) \o\o\0\0\2\2\O\O\O\O ;
\draw (0.5,0.5)rectangle(10.5,-9.5);
\end{tikzpicture}
$
\refstepcounter{figure}
\caption*{Fig.~\thefigure\ (Spring). \hfill Twins: $\cll1,\cll2$.\hfill Periods: $[4,0]$, $[0,4]$.\\
         The quotient matrix is similar to that in Fig.~\ref{WW}($n=10$).
         }
         \label{EE}
\end{figure}

\begin{figure}[H]
\begin{tikzpicture}
\definecolor{cHH}{HTML}{eaa86e}\definecolor{cA}{HTML}{613414}
\definecolor{ch}{HTML}{c47d2f}\def\8{\CColor{cH}{cHH!50!cH}{210}{$8$}}
\definecolor{cH}{HTML}{c45c34}\def\8{\CColor{cH!70!ch!70!cA}{cH!30!cHH}{210}{$8$}}
\def\3{\CColor{cA!99}{cA!80!cHH}{210}{$1$}}
\definecolor{cE}{HTML}{cce1f9}\def\7{\Ccolor{white}{cE}{$5$}}
\definecolor{cF}{HTML}{6c774c}\def\5{\Ccolor{cF!99}{cF!90!white!70}{$6$}}
\definecolor{cG}{HTML}{3b492f}\def\6{\Ccolor{cG!90}{cG!80!white!70}{$7$}}
\definecolor{cDD}{HTML}{b0dff6}
\definecolor{cD}{HTML}{a9c5ea}\def\4{\CColor{cD}{cDD!70!cD}{180}{$4$}}
\definecolor{cB}{HTML}{a5a6a0}
\definecolor{cC}{HTML}{e3ded8}\def\2{\Ccolor{cC!30!cB}{cC!50!cB!50}{$3$}}
\definecolor{cB}{HTML}{a5a6a0}\def\1{\Ccolor{cC!00!cB}{cC!10!cB!50}{$2$}}
\draw (0,11) \5\4\6\2\8\1\5\4\6\2\8\1\5\4\6\2;
\draw (0,10) \4\7\4\8\3\8\4\7\4\8\3\8\4\7\4\8;
\draw (0, 9) \6\4\5\1\8\2\6\4\5\1\8\2\6\4\5\1;
\draw (0, 8) \1\8\2\6\4\5\1\8\2\6\4\5\1\8\2\6;
\draw (0, 7) \8\3\8\4\7\4\8\3\8\4\7\4\8\3\8\4;
\draw (0, 6) \2\8\1\5\4\6\2\8\1\5\4\6\2\8\1\5;
\draw (0, 5) \5\4\6\2\8\1\5\4\6\2\8\1\5\4\6\2;
\draw (0, 4) \4\7\4\8\3\8\4\7\4\8\3\8\4\7\4\8;
\draw (0, 3) \6\4\5\1\8\2\6\4\5\1\8\2\6\4\5\1;
\draw (0, 2) \1\8\2\6\4\5\1\8\2\6\4\5\1\8\2\6;
\draw (0, 1) \8\3\8\4\7\4\8\3\8\4\7\4\8\3\8\4;
\end{tikzpicture}
\hfill
\begin{tikzpicture}
\def\1{\ccolor{white}{$1$}}
\def\2{\ccolor{white}{$2$}}
\def\4{\ccolor{white}{$4$}}
\def\O{\ccolor{white}{\ }}
\def\0{\ccolor{white}{$0$}}
\def\o{ ++(1, 0) }
\draw (0,-0) \O\O\O\O\0\0\0\4 ;
\draw (0,-1) \O\O\O\O\0\1\1\2 ;
\draw (0,-2) \O\O\O\O\0\1\1\2 ;
\draw (0,-3) \O\O\O\O\1\1\1\1 ;
\draw (0,-4) \0\0\0\4\O\O\O\O ;
\draw (0,-5) \0\1\1\2\O\O\O\O ;
\draw (0,-6) \0\1\1\2\O\O\O\O ;
\draw (0,-7) \1\1\1\1\O\O\O\O ;
\draw (0.5,0.5)rectangle(8.5,-7.5);
\end{tikzpicture}
\refstepcounter{figure}
\caption*{Fig.~\thefigure\ (Street basketball).\hfill Twins: $\cll2,\cll3$; $\cll6,\cll7$.\hfill Periods: $[6,0]$, $[0,6]$.}
\label{FF}
\end{figure}
\vspace{-5mm}

\begin{figure}[H]
\begin{tikzpicture}
\dfc{cA}{d5c5b5}\def\1{\CColor{cA}{cA!30!white}{210}{$1$}}
\dfc{cI}{fc9702}\dfc{cII}{e6b47e}\def\9{\CColor{cI}{cII}{210}{$9$}}
\dfc{cJJ}{586faf}\dfc{cJ}{24285d}\def\A{\CColor{cJ}{cJJ}{210}{$10$}}
\dfc{cb}{518ad4}
\dfc{cB}{80c5ff}\def\2{\Ccolor{cb!20!cB}{cB!50}{$2$}}
\dfc{cC}{c0e2fe}\def\3{\Ccolor{cb!20!cB!30!cC}{cC!50}{$3$}}
\dfc{cG}{4aa1bd}\def\7{\Ccolor{cB!30!cC!20}{cb!20!cB!30!cC}{$7$}}
\dfc{cH}{93c7d2}\def\8{\Ccolor{cB!70!cC!20}{cb!20!cB!70!cC}{$8$}}
\dfc{cD}{e4d5c5}\def\4{\ccolor{cG}{$4$}}
\dfc{cE}{af9076}\def\5{\ccolor{cH!40!cG}{$5$}}
\dfc{cF}{486d7c}\def\6{\ccolor{cF}{$6$}}
\draw (0,11) \8\5\7\3\A\2\8\5\7\3\A\2\8\5\7\3;
\draw (0,10) \4\6\4\9\1\9\4\6\4\9\1\9\4\6\4\9;
\draw (0, 9) \7\5\8\2\A\3\7\5\8\2\A\3\7\5\8\2;
\draw (0, 8) \2\A\3\7\5\8\2\A\3\7\5\8\2\A\3\7;
\draw (0, 7) \9\1\9\4\6\4\9\1\9\4\6\4\9\1\9\4;
\draw (0, 6) \3\A\2\8\5\7\3\A\2\8\5\7\3\A\2\8;
\draw (0, 5) \8\5\7\3\A\2\8\5\7\3\A\2\8\5\7\3;
\draw (0, 4) \4\6\4\9\1\9\4\6\4\9\1\9\4\6\4\9;
\draw (0, 3) \7\5\8\2\A\3\7\5\8\2\A\3\7\5\8\2;
\draw (0, 2) \2\A\3\7\5\8\2\A\3\7\5\8\2\A\3\7;
\draw (0, 1) \9\1\9\4\6\4\9\1\9\4\6\4\9\1\9\4;
\end{tikzpicture}
\hfill
\begin{tikzpicture}
\def\1{\ccolor{white}{$1$}}
\def\2{\ccolor{white}{$2$}}
\def\4{\ccolor{white}{$4$}}
\def\O{\ccolor{white}{\ }}
\def\0{\ccolor{white}{$0$}}
\def\o{ ++(1, 0) }
\draw (0,-0) \O\O\O\O\O\0\0\0\2\2 ;
\draw (0,-1) \O\O\O\O\O\0\1\1\1\1 ;
\draw (0,-2) \O\O\O\O\O\0\1\1\1\1 ;
\draw (0,-3) \O\O\O\O\O\1\1\1\1\0 ;
\draw (0,-4) \O\O\O\O\O\1\1\1\0\1 ;
\draw (0,-5) \0\0\0\2\2\O\O\O\O\O ;
\draw (0,-6) \0\1\1\1\1\O\O\O\O\O ;
\draw (0,-7) \0\1\1\1\1\O\O\O\O\O ;
\draw (0,-8) \1\1\1\1\0\O\O\O\O\O ;
\draw (0,-9) \1\1\1\0\1\O\O\O\O\O ;
\draw (0.5,0.5)rectangle(10.5,-9.5);
\end{tikzpicture}
\refstepcounter{figure}
\caption*{Fig.~\thefigure\ (Beach volleyball).\hfill Twins: $\cll2,\cll3$; $\cll7,\cll8$.\hfill Periods: $[6,0]$, $[0,6]$.}
\label{GG}
\end{figure}
\vspace{-5mm}

\begin{figure}[H]
\begin{tikzpicture}
\dfc{cH}{7a9393}
\dfc{cK}{76a0c0}
\dfc{cKK}{fdfcfc}
\dfc{cHH}{c6cfd8}
\def\8{\CColor{black!20!cK!90}{cKK}{210}{$8$}}
\def\B{\CColor{cH!60}{red!2!white}{210}{$11$}}
\dfc{cA}{1d242d}\def\1{\CColor{cA}{cA!80!white}{210}{$\color{cH}1$}}
\dfc{cI}{6fb6ca}\dfc{cII}{ffffff}\def\6{\Ccolor{cII}{cI}{$6$}}
\dfc{cD}{3f6f17}\dfc{cDD}{a1cc43}\def\4{\ccolor{cD!99}{$4$}}
\dfc{cG}{798738}\def\7{\ccolor{cD!70!cDD}{$7$}}
\dfc{cF}{dd4920}\def\6{\ccolor{cF}{$6$}}
\dfc{cX}{f7e500}\def\9{\ccolor{cX}{$9$}}
\dfc{cY}{5a783c}\def\5{\ccolor{cD!20!cDD!90}{$5$}}
\dfc{cJ}{749351}\def\A{\ccolor{cD!45!cDD}{$10$}}
\dfc{cB}{fa846d}\def\2{\Ccolor{cD}{cDD!70!cD}{$2$}}
\dfc{cB}{fa846d}\def\3{\Ccolor{cD!90}{cDD!60!cD!70}{$3$}}
\draw (0,11) \A\5\A\3\8\2\A\5\A\3\8\2\A\5\A\3;
\draw (0,10) \5\9\5\B\1\B\5\9\5\B\1\B\5\9\5\B;
\draw (0, 9) \A\5\A\2\8\3\A\5\A\2\8\3\A\5\A\2;
\draw (0, 8) \2\B\3\7\4\7\2\B\3\7\4\7\2\B\3\7;
\draw (0, 7) \8\1\8\4\6\4\8\1\8\4\6\4\8\1\8\4;
\draw (0, 6) \3\B\2\7\4\7\3\B\2\7\4\7\3\B\2\7;
\draw (0, 5) \A\5\A\3\8\2\A\5\A\3\8\2\A\5\A\3;
\draw (0, 4) \5\9\5\B\1\B\5\9\5\B\1\B\5\9\5\B;
\draw (0, 3) \A\5\A\2\8\3\A\5\A\2\8\3\A\5\A\2;
\draw (0, 2) \2\B\3\7\4\7\2\B\3\7\4\7\2\B\3\7;
\draw (0, 1) \8\1\8\4\6\4\8\1\8\4\6\4\8\1\8\4;
\end{tikzpicture}
\hfill
\begin{tikzpicture}
\def\1{\ccolor{white}{$1$}}
\def\2{\ccolor{white}{$2$}}
\def\4{\ccolor{white}{$4$}}
\def\O{\ccolor{white}{\ }}
\def\0{\ccolor{white}{$0$}}
\def\o{ ++(1, 0) }
\draw (0,-0) \O\O\O\O\O\0\0\2\0\0\2 ;
\draw (0,-1) \O\O\O\O\O\0\1\1\0\1\1 ;
\draw (0,-2) \O\O\O\O\O\0\1\1\0\1\1 ;
\draw (0,-3) \O\O\O\O\O\1\2\1\0\0\0 ;
\draw (0,-4) \O\O\O\O\O\0\0\0\1\2\1 ;
\draw (0,-5) \0\0\0\4\0\O\O\O\O\O\O ;
\draw (0,-6) \0\1\1\2\0\O\O\O\O\O\O ;
\draw (0,-7) \1\1\1\1\0\O\O\O\O\O\O ;
\draw (0,-8) \0\0\0\0\4\O\O\O\O\O\O ;
\draw (0,-9) \0\1\1\0\1\O\O\O\O\O\O ;
\draw (0,-10)\1\1\1\0\1\O\O\O\O\O\O ;
\draw (0.5,0.5)rectangle(11.5,-10.5);
\end{tikzpicture}
\refstepcounter{figure}
\caption*{Fig.~\thefigure\ (Football).\hfill Twins: $\cll2,\cll3$.\hfill Periods: $[6,0]$, $[0,6]$.} \label{HH}
\end{figure}

\begin{figure}[H]
\def\0{ ++(1\DK, 0) }
\def\dColor#1#2{\Dcolor{#1!75}{#1!5}{#2}}
\def\dcOlor#1#2{\dcolor{#1!75}{#2}}
\def\dfc#1#2{\definecolor{#1}{HTML}{#2}}
 \dfc{cE}{f1c502}\def\c{\dcOlor{cE}{$5$}}
 \dfc{cF}{f19521}\def\f{\dcOlor{cF!65!cE}{$8$}}
 \dfc{cG}{aa7946}\def\e{\dcOlor{cG!75!cF}{$7$}}
 \dfc{cH}{e46601}\def\d{\dcOlor{cH!75!white}{$6$}}
\dfc{cA}{4f987d}\def\1{\dColor{cA}{$1$}}
\dfc{cB}{62884b}\def\2{\dColor{cB}{$2$}}
\dfc{cD}{5e953f}\def\b{\dColor{cD}{$4$}}
\dfc{cC}{998914}\def\a{\dColor{cC!75!cD}{$3$}}
\def\t{ ++(1\DK, 0) node {$\ $}}
\scalebox{1.2}{
\begin{tikzpicture}
\begin{scope}
\clip [](0.05\DK,-1.97\DK) rectangle (11.95\DK,5.45\DK);
\draw (-0\DK, 5\DK)      \1\2\1\2\1\2\1\2\1\2\1;
\draw (-0.5\DK, 4.5\DK) \e\f\e\f\e\f\e\f\e\f\e\f;
\draw (-0\DK, 4\DK)      \a\b\a\b\a\b\a\b\a\b\a;
\draw (-0.5\DK, 3.5\DK) \c\d\c\d\c\d\c\d\c\d\c\d;
\draw (-0\DK, 3\DK)      \1\2\1\2\1\2\1\2\1\2\1;
\draw (-0.5\DK, 2.5\DK) \e\f\e\f\e\f\e\f\e\f\e\f;
\draw (-0\DK, 2\DK)      \a\b\a\b\a\b\a\b\a\b\a;
\draw (-0.5\DK, 1.5\DK) \c\d\c\d\c\d\c\d\c\d\c\d;
\draw (-0\DK, 1\DK)      \1\2\1\2\1\2\1\2\1\2\1;
\draw (-0.5\DK, 0.5\DK) \e\f\e\f\e\f\e\f\e\f\e\f;
\draw (-0\DK, 0\DK)      \a\b\a\b\a\b\a\b\a\b\a;
\draw (-0.5\DK, -.5\DK) \c\d\c\d\c\d\c\d\c\d\c\d;
\draw (-0\DK,-1\DK)      \1\2\1\2\1\2\1\2\1\2\1;
\draw (-0.5\DK,-1.5\DK) \e\f\e\f\e\f\e\f\e\f\e\f;
\draw [draw=black!40!white, thin, step=\DK, xshift=0.5\DK, yshift=-1\DK,
      dash pattern=on 0\DK off 0.28\DK on 0.44\DK off 0.28\DK]
(-0.5\DK,-1.5\DK) grid (11.5\DK,6.5\DK);
 \draw [draw=black!40!white, thin, step=\DK, xshift=1\DK, yshift=-0.5\DK,
       dash pattern=on 0\DK off 0.28\DK on 0.44\DK off 0.28\DK]
 (-1.5\DK,-1.5\DK) grid (10.97\DK,5.97\DK);
\end{scope}
\end{tikzpicture}
}
\hfill
\begin{tikzpicture}
\def\1{\ccolor{white}{1}}
\def\2{\ccolor{white}{2}}
\def\O{\ccolor{white}{\ }}
\def\0{\ccolor{white}{0}}
\def\o{ ++(1, 0) }
\draw (0,-0)  \O\O\O\O\1\1\1\1 ;
\draw (0,-1)  \O\O\O\O\1\1\1\1 ;
\draw (0,-2)  \O\O\O\O\1\1\1\1 ;
\draw (0,-3)  \O\O\O\O\1\1\1\1 ;
\draw (0,-4)  \1\1\1\1\O\O\O\O ;
\draw (0,-5)  \1\1\1\1\O\O\O\O ;
\draw (0,-6)  \1\1\1\1\O\O\O\O ;
\draw (0,-7)  \1\1\1\1\O\O\O\O ;
\draw (0.5,0.5)rectangle(8.5,-7.5);
\end{tikzpicture}
\refstepcounter{figure}
\caption*{Fig.~\thefigure\ (Fall).\hfill Twins: $\cll1,\cll2,\cll3,\cll4$;
                $\cll5,\cll6,\cll7,\cll8$.\hfill Periods: $[2,2]$, $[2,-2]$.}
                \label{II}
\end{figure}

\begin{figure}[H]
\dfc{cA}{914019}\def\1{\Dcolor{cA!75}{cA!10}{$3$}}
\dfc{cB}{845042}\def\2{\Dcolor{cB!75}{cB!10}{$4$}}
\dfc{cD}{2e5c3a}\def\3{\Dcolor{cD!65}{cD!10}{$1$}}
\dfc{cC}{65a370}\def\4{\Dcolor{cC!75}{cC!10}{$2$}}
\dfc{cF}{d36808}
\dfc{cE}{bcaa1f}\def\5{\DDolor{cE!90}{cE!50}{210}{$7$}}
\dfc{cH}{f0ab3a}\def\6{\DDolor{cF!40!cH!90}{cF!30!cH!50}{210}{$8$}}
\dfc{cG}{a08c04}\def\7{\DDolor{cG}{cE!70}{210}{$5$}}
\dfc{cF}{d36808}\def\8{\DDolor{cF!90!cH}{cF!70!cH!50}{210}{$6$}}
\scalebox{1.2}{
\begin{tikzpicture}
\begin{scope}
\clip [] (0.52\DK,-1.5\DK) rectangle (12.5\DK,6\DK);
 \draw (0.5\DK, 5.5\DK)   \5\4\6\4\5\4\6\4\5\4\6;
 \draw (0\DK, 5\DK)      \7\7\8\8\7\7\8\8\7\7\8\8;
 \draw (0.5\DK, 4.5\DK)   \5\3\6\3\5\3\6\3\5\3\6;
 \draw (0\DK, 4\DK)      \1\2\1\2\1\2\1\2\1\2\1\2;
 \draw (0.5\DK, 3.5\DK)   \6\4\5\4\6\4\5\4\6\4\5;
 \draw (0\DK, 3\DK)      \8\8\7\7\8\8\7\7\8\8\7\7;
 \draw (0.5\DK, 2.5\DK)   \6\3\5\3\6\3\5\3\6\3\5;
 \draw (0\DK, 2\DK)      \1\2\1\2\1\2\1\2\1\2\1\2;
 \draw (0.5\DK, 1.5\DK)   \5\4\6\4\5\4\6\4\5\4\6;
 \draw (0\DK, 1\DK)      \7\7\8\8\7\7\8\8\7\7\8\8;
 \draw (0.5\DK, 0.5\DK)   \5\3\6\3\5\3\6\3\5\3\6;
 \draw (0\DK, 0\DK)      \1\2\1\2\1\2\1\2\1\2\1\2;
 \draw (0.5\DK,-0.5\DK)   \6\4\5\4\6\4\5\4\6\4\5;
 \draw (0\DK,-1\DK)      \8\8\7\7\8\8\7\7\8\8\7\7;
\draw [draw=black!50!white, thin, step=2\DK, xshift=.5\DK, yshift=0\DK,
      dash pattern=on 0\DK off 0.28\DK on 0.44\DK off 0.28\DK]
(-0.5\DK,-0.5\DK) grid (11.98\DK,5.98\DK);
\end{scope}
\end{tikzpicture}
}
\hfill
\begin{tikzpicture}
\def\1{\ccolor{white}{1}}
\def\2{\ccolor{white}{2}}
\def\O{\ccolor{white}{\ }}
\def\0{\ccolor{white}{0}}
\def\o{ ++(1, 0) }
\draw (0,-0)  \O\O\1\1\1\1 ;
\draw (0,-1)  \O\O\1\1\1\1 ;
\draw (0,-2)  \1\1\O\O\O\O\1\1 ;
\draw (0,-3)  \1\1\O\O\O\O\1\1 ;
\draw (0,-4)  \1\1\O\O\O\O\2\0 ;
\draw (0,-5)  \1\1\O\O\O\O\0\2 ;
\draw (0,-6)  \o\o\1\1\2\0\O\O ;
\draw (0,-7)  \o\o\1\1\0\2\O\O ;
\draw (0.5,0.5)rectangle(8.5,-7.5);
\end{tikzpicture}
\refstepcounter{figure}
\caption*{Fig.~\thefigure\ (Apples \& oranges).\hfill Twins: $\cll1,\cll2$; $\cll3,\cll4$.\hfill Periods: $[4,0]$, $[0,4]$.\\
          The quotient matrix is similar to that in Fig.~\ref{CC}.
}\label{JJ}
\end{figure}

\begin{figure}[H]
\begin{tikzpicture}[yscale=-1]
\dfc{cA}{09899b}\def\1{\ccolor{blue!3!cA!65}{$1$}}
\dfc{cB}{1da8b4}\def\2{\ccolor{green!2!cB!65}{$2$}}
\dfc{cE}{bc966e}\def\5{\ccolor{cE!75}{$5$}}
\dfc{cF}{c4aa93}\def\6{\ccolor{cF!75}{$6$}}
\dfc{cC}{98c9ba}\def\3{\Ccolor{cE!50!cC!70}{cE!50!cC!10}{$3$}}
\dfc{cD}{75a49d}\def\4{\Ccolor{cD!70}{cD!10}{$4$}}
\dfc{cK}{599ae1}\def\k{\ccolor{red!5!cK!55}{$k$}}
\dfc{cL}{1c7cd2}\def\l{\ccolor{cL!40}{$l$}}
\dfc{cM}{8c9ead}\def\m{\Ccolor{cM!70}{cM!10}{$m$}}
\dfc{cN}{82a7cb}\def\n{\Ccolor{cN!70}{cN!10}{$n$}}
\draw (3, 3)       \5                        ++(2.5,0)  \k\m\1\3\5                        ++(2.5,0)  \k\m\1\3\5     ;
\draw (3, 2)       \4\6                      ++(2.5,0)    \l\n\2\4\6                      ++(2.5,0)    \l\n\2\4\6   ;
\draw (3, 1)       \1\3\5                    ++(2.5,0)      \k\m\1\3\5                    ++(2.5,0)      \k\m\1\3\5 ;
\draw (3, 0)       \n\2\4\6                  ++(2.5,0)        \l\n\2\4\6                  ++(2.5,0)        \l\n\2\4 ;
\draw (3,-1)       \k\m\1\3\5                ++(2.5,0)          \k\m\1\3\5                ++(2.5,0)          \k\m\1 ;
\draw (4,-2)         \l\n\2\4\6              ++(2.5,0)            \l\n\2\4\6              ++(2.5,0)            \l\n ;
\draw (5,-3)           \k\m\1\3\5            ++(2.5,0)              \k\m\1\3\5            ++(2.5,0)              \k ;
\draw (2.5,-4)        \6  ++(2.5,0)   \l\n\2\4\6          ++(2.5,0)                \l\n\2\4\6          ++(2.5,0)                 ;
\draw (2.5,-5)      \3\5   ++(2.5,0)    \k\m\1\3\5        ++(2.5,0)                  \k\m\1\3\5        ++(2.5,0)                 ;
\draw (2.5,-6)      \2\4\6   ++(2.5,0)    \l\n\2\4\6      ++(2.5,0)                    \l\n\2\4\6      ++(2.5,0)                 ;
\draw (2.5,-7)      \m\1\3\5  ++(2.5,0)     \k\m\1\3\5    ++(2.5,0)                      \k\m\1\3\5    ++(2.5,0)    \k\m\1\3\5   ;
\draw (2.5,-8)      \l\n\2\4\6   ++(2.5,0)    \l\n\2\4\6  ++(2.5,0)                        \l\n\2\4\6  ++(2.5,0)      \l\n\2\4\6 ;
\draw [diag, dash phase=0.5\DK]
(3.5,-4 ) -- +(5,-5)
++(0,-1) --  +(4,-4)
++(0,-1) --  +(3,-3)
++(0,-1) --  +(2,-2)
++(0,-1) --  +(1,-1)
  (3,0) -- +(9,-9)
++(0,1) -- +(10,-10)
++(0,1) -- +(11,-11)
++(0,1) -- +(12,-12)
++(0,1) -- +(13,-13)
++(3.5,0) -- +(13,-13)
++(1,0) -- +(13,-13)
++(1,0) -- +(13,-13)
++(1,0) -- +(13,-13)
++(1,0) -- +(13,-13)
++(3.5,0) -- +(13,-13)
++(  1,0) -- +(13,-13)
++(  1,0) -- +(13,-13)
++(  1,0) -- +(13,-13)
++(  1,0) -- +(13,-13);
 \draw [->, dashed] ( 7.5,-7 ) -- node [below,near start] {{$\scriptscriptstyle +2$}} +(1.5,0);
 \draw [->, dashed] ( 3.5,-3 ) -- node [below,near start] {{$\scriptscriptstyle +2$}} +(1.5,0);
 \draw [->, dashed] (15,-7 ) -- node [below,near start] {{$\scriptscriptstyle +2$}} +(1.5,0);
 \draw [->, dashed] (22.5,-7 ) -- +(1.5,0);
 \draw [->, dashed] (30,-7 ) -- +(1.5,0);
 \draw [->, dashed] (11,-3 ) -- node [below,near start] {{$\scriptscriptstyle +2$}} +(1.5,0);
 \draw [->, dashed] ( 7, 1 ) -- node [below,near start] {{$\scriptscriptstyle +2$}} +(1.5,0);
 \draw [->, dashed] (18.5,-3 ) -- node [below,near start] {{$\scriptscriptstyle +2$}} +(1.5,0);
 \draw [->, dashed] (14.5, 1 ) -- node [below,near start] {{$\scriptscriptstyle +2$}} +(1.5,0);
\end{tikzpicture}
\hspace{-10cm}\hfill
\scalebox{0.78}{
\begin{tikzpicture}
\def\1{\ccolor{white}{1}}
\def\0{\ccolor{white}{0}}
\def\O{\cocolor{white}{\ }}
\def\d{\cocolor{white}{\ddots}}
\def\t{\cocolor{white}{\cdots}}
\def\v{\cocolor{white}{\vdots}}
\def\o{ ++(1, 0) }
\fill [fill=white] (0.5,0.5)rectangle(13.5,-12.5);
\draw (0,-0)  \O\O\1\1\o\o\o\o\t\o\o\1\1 ;
\draw (0,-1)  \O\O\1\1\o\o\o\o\t\o\o\1\1 ;
\draw (0,-2)  \1\1\O\O\1\1\o\o\t ;
\draw (0,-3)  \1\1\O\O\1\1\o\o\t ;
\draw (0,-4)  \o\o\1\1\O\O\1\1\t ;
\draw (0,-5)  \o\o\1\1\O\O\1\1\O\v\v\v\v ;
\draw (0,-6)  \o\o\o\o\d\d\d\d\d ;
\draw (0,-7)  \v\v\v\v\O\1\1\O\O\1\1 ;
\draw (0,-8)  \o\o\o\o\t\1\1\O\O\1\1 ;
\draw (0,-9)  \o\o\o\o\t\o\o\1\1\O\O\1\1 ;
\draw (0,-10) \o\o\o\o\t\o\o\1\1\O\O\1\1 ;
\draw (0,-11) \1\1\o\o\t\o\o\o\o\1\1\O\O ;
\draw (0,-12) \1\1\o\o\t\o\o\o\o\1\1\O\O ;
\draw (0.5,0.5)rectangle(13.5,-12.5);
\end{tikzpicture}
}
\caption{$n=4\mbox{(\ref{II})},8\mbox{(\ref{II})},12,16,20,\ldots$.\hfill Twins: $\cll1,\cll2$; $\cll3,\cll4$; $\cll5,\cll6$, \ldots.\hfill Periods: $[2,2]$, $\big[\frac n2,0\big]$.\\[0.3em]
         \mbox{}\hfill
          $\FfF\Lft \alpha+\beta,\beta\Rgt =\framebox{$2\alpha+\beta+1$},
         \quad \alpha = 0,1,\ldots,\frac n2-1,\ \beta=0,1.$\hfill\mbox{}
         }
         \label{KK}
\end{figure}

\begin{figure}[H]
\def\o{\ccolor{white}{$\ $}}
\mbox{}$\!\!\!\!$
\makebox[1cm][l]{
\begin{tikzpicture}[yscale=-1]
\def\1{\ccolor{blue!40!green!30!white}{$1$}}
\def\2{\ccolor{black!30!green!95!yellow!30!white}{$2$}}
\def\3{\Ccolor{yellow!40!white}{white}{$3$}}
\def\4{\Ccolor{orange!30!white}{white}{$4$}}
\def\5{\ccolor{blue!20!white}{$5$}}
\def\6{\ccolor{blue!50!purple!20!white}{$6$}}
\def\k{\Ccolor{black!7!white}{white}{$k$}}
\def\l{\Ccolor{black!12!white}{white}{$l$}}
\def\m{\ccolor{black!4!white}{$m$}}
\def\n{\ccolor{black!7!white}{$n$}}
\begin{scope}
\clip [](4.7,3.5)-- ++(12,-12)-- ++(-6.4,0) -- ++(-8,8) -- ++(0,4)--cycle;
\draw (2, 3)           \3\5\o ;
\draw (2, 2)           \2\4\6\o ;
\draw (2, 1)           \3\1\3\5\o ;
\draw (2, 0)           \6\4\2\4\6\o ;
\draw (2,-1)           \o\5\3\1\3\5\o ;
\draw (3,-2)             \o\6\4\2\4\6\o ;
\draw (4,-3)               \o\5\3\1\3\5\o ;
\draw (5,-4)                 \o\6\4\2\4\6\o ;
\draw (6,-5)                   \o\5\3\1\3\5\o ;
\draw (7,-6)                     \o\6\4\2\4\6\o ;
\draw (8,-7)                       \o\5\3\1\3\5\o ;
\draw (9,-8)                         \o\6\4\2\4\6\o\o ;
\end{scope}
\begin{scope}
\clip [](10.3,3.5)-- ++(6.4,0)-- ++(12,-12)-- ++(-6.4,0)--cycle;
\draw (10, 3)   \o\5\3\1\3\5\o ;
\draw (11, 2)     \o\6\4\2\4\6\o ;
\draw (12, 1)       \o\5\3\1\3\5\o ;
\draw (13, 0)         \o\6\4\2\4\6\o ;
\draw (14,-1)           \o\5\3\1\3\5\o ;
\draw (15,-2)             \o\6\4\2\4\6\o ;
\draw (16,-3)               \o\5\3\1\3\5\o ;
\draw (17,-4)                 \o\6\4\2\4\6 ;
\draw (18,-5)                   \o\5\3\1\3 ;
\draw (19,-6)                     \o\6\4\2\4\6\o ;
\draw (20,-7)                       \o\5\3\1\3\5\o ;
\draw (21,-8)                         \o\6\4\2\4\6\o ;
\end{scope}
\begin{scope}
  \clip []( 9,-8.5)-- ++(-3,0)-- ++(-4,4)-- ++(0,3)--cycle;
\draw ( 2,-2)  \o ;
\draw ( 2,-3)  \k\o ;
\draw ( 2,-4)  \n\l\o ;
\draw ( 2,-5)  \k\m\k\o ;
\draw ( 2,-6)  \o\l\n\l\o ;
\draw ( 3,-7)    \o\k\m\k\o ;
\draw ( 4,-8)      \o\l\n\l\o ;
\end{scope}\draw(2,-8)\6;\begin{scope}
\clip [](6,3.5)-- ++(3,0)-- ++(12,-12)-- ++(-3,0)--cycle;
\draw ( 5, 3)   \o\k\m\k\o ;
\draw ( 6, 2)     \o\l\n\l\o ;
\draw ( 7, 1)       \o\k\m\k\o ;
\draw ( 8, 0)         \o\l\n\l\o ;
\draw ( 9,-1)           \o\k\m\k\o ;
\draw (10,-2)             \o\l\n\l\o ;
\draw (11,-3)               \o\k\m\k\o ;
\draw (12,-4)                 \o\l\n\l\o ;
\draw (13,-5)                   \o\k\m\k\o ;
\draw (14,-6)                     \o\l\n\l\o ;
\draw (15,-7)                       \o\k\m\k\o ;
\draw (16,-8)                         \o\l\n\l\o ;
\end{scope}
\begin{scope}
\clip [](18,3.5)-- ++(3,0)-- ++(12,-12)-- ++(-3,0)--cycle;
\draw (12+ 5, 3)   \o\k\m\k\o ;
\draw (12+ 6, 2)     \o\l\n\l\o ;
\draw (12+ 7, 1)       \o\k\m\k\o ;
\draw (12+ 8, 0)         \o\l\n\l\o ;
\draw (12+ 9,-1)           \o\k\m\k\o ;
\draw (12+10,-2)             \o\l\n\l\o ;
\draw (12+11,-3)               \o\k\m\k\o ;
\draw (12+12,-4)                 \o\l\n\l\o ;
\draw (12+13,-5)                   \o\k\m\k\o ;
\draw (12+14,-6)                     \o\l\n\l\o ;
\draw (12+15,-7)                       \o\k\m\k ;
\draw (12+16,-8)                         \o\l\n ;
\end{scope}
\draw [diag, dash phase=0.5\DK]
  (2,1) -- +(10,-10)
++(0,1) -- +(11,-11)
++(0,1) -- +(12,-12)
++(0,1) -- +(13,-13)
++(1,0) -- +(13,-13) ++(-1,-11) -- +(2,-2) ++(0,3) -- +(5,-5) ++(0,1) -- +(6,-6) ++(0,1) -- +(7,-7)
  (6,4) -- +(13,-13)
++(1,0) -- +(13,-13) ++(1,0) -- +(13,-13)
++(3,0) -- +(13,-13)
++(1,0) -- +(13,-13)
++(1,0) -- +(13,-13)
++(1,0) -- +(13,-13)
++(1,0) -- +(13,-13)
++(3,0) -- +(13,-13)
++(1,0) -- +(13,-13)
++(1,0) -- +(12,-12);

\draw [->, dashed] (19, 1) -- node [below, very near start] {{$\scriptscriptstyle +2$}} +(1,0);
\draw [->, dashed] (27,-7) -- node [below, very near start] {{$\scriptscriptstyle +2$}} +(1,0);
\draw [<-, dashed] (20,-7) -- node [above, very near end] {{$\scriptscriptstyle +2$}} +(1,0);
\draw [<-, dashed] (16,-3) -- node [above, very near end] {{$\scriptscriptstyle +2$}} +(1,0);
\draw [<-, dashed] (12, 1) -- node [above, very near end] {{$\scriptscriptstyle +2$}} +(1,0);
\draw [->, dashed] ( 7, 1) -- node [below, very near start] {{$\scriptscriptstyle +2$}} +(1,0);
\draw [->, dashed] (11,-3) -- node [below, very near start] {{$\scriptscriptstyle +2$}} +(1,0);
\draw [->, dashed] (15,-7) -- node [below, very near start] {{$\scriptscriptstyle +2$}} +(1,0);
\draw [->, dashed] ( 3,-7) -- node [below, very near start] {{$\scriptscriptstyle +2$}} +(1,0);
\draw [<-, dashed] (8.0,-7) -- node [above, very near end] {{$\scriptscriptstyle +2$}} +(1,0);
\draw [<-, dashed] (4.0,-3) -- node [above, very near end] {{$\scriptscriptstyle +2$}} +(1,0);
\end{tikzpicture}
}
\hspace{-16cm}\hfill
\scalebox{0.805}{\begin{tikzpicture}
\def\2{\ccolor{white}{2}}
\def\1{\ccolor{white}{1}}
\def\0{\ccolor{white}{0}}
\def\O{\cocolor{white}{\ }}
\def\d{\cocolor{white}{\ddots}}
\def\t{\cocolor{white}{\cdots}}
\def\v{\cocolor{white}{\vdots}}
\def\o{ ++(1, 0) }
\fill [fill=white] (0.25,0.6)rectangle(11.5,-10);
\draw (0,-0)  \O\O\2\2 ;
\draw (0,-1)  \O\O\2\2 ;
\draw (0,-2)  \1\1\O\O\1\1 ;
\draw (0,-3)  \1\1\O\O\1\1 ;
\draw (0,-4)  \o\o\1\1\O\O\1\1 ;
\draw (0,-5)  \o\o\1\1\O\O\1\1 ;
\draw (0,-6)  \o\o\o\o\d\d\d\d\d ;
\draw (0,-7)  \o\o\o\o\o\1\1\O\O\1\1 ;
\draw (0,-8)  \o\o\o\o\o\1\1\O\O\1\1 ;
\draw (0,-9)  \o\o\o\o\o\o\o\2\2\O\O ;
\draw (0,-10) \o\o\o\o\o\o\o\2\2\O\O ;
\draw (0.5,0.5)rectangle(11.5,-10.5);
\draw (0,-11) \o\o\o\o ;
\end{tikzpicture}}
\caption{$n=4\mbox{(\ref{II})},6\mbox{(\ref{II})},8,10,12,\ldots$.\hfill Twins: $\cll1,\cll2$; $\cll3,\cll4$; $\cll5,\cll6$; \ldots.\hfill Periods: $[2,2]$, $[n-2,0]$.\\[0.3em]
         \mbox{}\hfill\mbox{}
         $\FfF\Lft \pm\alpha+\beta,\beta\Rgt =\framebox{$2\alpha+\beta+1$},
         \quad \alpha = 0,1,\ldots,\frac n2-1,\ \beta=0,1.$
         \mbox{}\hfill\mbox{}
         } \label{LL}
\end{figure}

\def\drdgnl#1{\draw [draw=black!50!white, thin,
       dash pattern=on 0.22\DK off 0.56\DK on 0.22\DK off 0]
       (0.5\DK,#1\DK)--(10.5\DK,#1\DK);}
\def\drdgnla#1{\draw [draw=black!50!white, thin,
       dash pattern=on 0\DK off 0.28\DK on 0.44\DK off 0.28\DK]
       (0.5\DK,#1\DK)--(10.5\DK,#1\DK);}

\begin{figure}[H]
\def\0{ ++(1\DK, 0) }
\def\o{\dcolor{white}{$\ $}}
\def\1{\dcolor{yellow!40!white}{$1$}}
\def\2{\dcolor{yellow!80!white}{$2$}}
\def\3{\Dcolor{green!30!yellow!45!white}{white}{$3$}}
\def\4{\Dcolor{green!30!yellow!80!white}{white}{$4$}}
\def\5{\Dcolor{green!80!yellow!30!white}{white}{$5$}}
\def\6{\Dcolor{green!80!yellow!55!white}{white}{$6$}}
\def\7{\Dcolor{green!80!yellow!85!white}{white}{$7$}}
\def\8{\Dcolor{green!60!cyan!40!white  }{white}{$8$}}
\def\9{\Dcolor{green!60!cyan!80!white  }{white}{$9$}}
\def\h{\Dcolor{blue!80!cyan!10!white   }{white}{$h$}}
\def\i{\Dcolor{blue!80!cyan!30!white   }{white}{$i$}}
\def\j{\Dcolor{blue!80!cyan!50!white   }{white}{$j$}}
\def\k{\Dcolor{blue!80!red!20!white    }{white}{$k$}}
\def\l{\Dcolor{blue!80!red!40!white    }{white}{$l$}}
\def\m{\dcolor{blue!30!red!15!white}{$m$}}
\def\n{\dcolor{blue!30!red!30!white}{$n$}}
\def\t{ ++(1\DK, 0) node {$\cdots$}}
\begin{tikzpicture}[rotate=90]
\draw (0.5\DK, 1.5\DK)\4\4\3\3\4\4\3\3\4;
\draw (0\DK, 1\DK)   \1\2\1\2\1\2\1\2\1\2;
\drdgnl{1}
\draw (0.5\DK, 0.5\DK)\3\3\4\4\3\3\4\4\3;
\draw (0\DK, 0\DK)   \6\5\6\7\6\5\6\7\6\5;
\draw (0.5\DK, -0.5\DK)\8\8\9\9\8\8\9\9\8;
\draw (0\DK, -2\DK)     \i\h\i\j\i\h\i\j\i\h;
\draw (0.5\DK, -2.5\DK)\k\k\l\l\k\k\l\l\k;
\draw (0\DK, -3\DK)     \m\n\m\n\m\n\m\n\m\n;
\drdgnl{-3}
\draw (0.5\DK, -3.5\DK)\l\l\k\k\l\l\k\k\l;
\draw (0\DK, -4\DK)     \i\j\i\h\i\j\i\h\i\j;
\draw (0.5\DK, -5.5\DK)\9\9\8\8\9\9\8\8\9;
\draw (0\DK, -6\DK)   \6\7\6\5\6\7\6\5\6\7;
\draw (0.5\DK, -6.5\DK)\4\4\3\3\4\4\3\3\4;
\draw (0\DK, -7\DK)   \1\2\1\2\1\2\1\2\1\2;
\drdgnl{-7}
\draw (0.5\DK,-7.5\DK)\3\3\4\4\3\3\4\4\3;
\draw (0\DK, -8\DK)   \6\5\6\7\6\5\6\7\6\5;
\draw (0.5\DK, -8.5\DK)\8\8\9\9\8\8\9\9\8;
\draw [<-,dashed] (0.0\DK+1.5\DK, -0.7\DK+3.5\DK) -- +(0\DK, -0.6\DK);
\draw [<-,dashed] (1.5\DK+1.5\DK, -1.2\DK+3.5\DK) -- +(0\DK, -0.6\DK);
\draw [<-,dashed] (3.0\DK+1.5\DK, -0.7\DK+3.5\DK) -- node [above, very near end] {$\scriptscriptstyle+5$}+(0\DK, -0.6\DK);
\draw [<-,dashed] (4.5\DK+1.5\DK, -1.2\DK+3.5\DK) -- +(0\DK, -0.6\DK);
\draw [<-,dashed] (6.0\DK+1.5\DK, -0.7\DK+3.5\DK) -- +(0\DK, -0.6\DK);
\draw [<-,dashed] (7.5\DK+1.5\DK, -1.2\DK+3.5\DK) -- node [above, very near start] {$\scriptscriptstyle+5$}+(0\DK, -0.6\DK);
\draw [->,dashed] (1.5\DK, -1.2\DK) -- (1.5\DK, -1.8\DK);
\draw [->,dashed] (3.0\DK, -0.7\DK) -- (3.0\DK, -1.3\DK);
\draw [->,dashed] (4.5\DK, -1.2\DK) -- node [above, very near start] {$\scriptscriptstyle+5$} +(0\DK, -0.6\DK);
\draw [->,dashed] (6.0\DK, -0.7\DK) -- +(0\DK, -0.6\DK);
\draw [->,dashed] (7.5\DK, -1.2\DK) -- +(0\DK, -0.6\DK);
\draw [->,dashed] (9.0\DK, -0.7\DK) -- node [above, very near end] {$\scriptscriptstyle+5$} +(0\DK, -0.6\DK) ;
\draw [<-,dashed] (0.0\DK+1.5\DK, -0.7\DK-3.5\DK) -- +(0\DK, -0.6\DK);
\draw [<-,dashed] (1.5\DK+1.5\DK, -1.2\DK-3.5\DK) -- +(0\DK, -0.6\DK);
\draw [<-,dashed] (3.0\DK+1.5\DK, -0.7\DK-3.5\DK) -- node [above, very near end] {$\scriptscriptstyle+5$} +(0\DK, -0.6\DK);
\draw [<-,dashed] (4.5\DK+1.5\DK, -1.2\DK-3.5\DK) -- +(0\DK, -0.6\DK);
\draw [<-,dashed] (6.0\DK+1.5\DK, -0.7\DK-3.5\DK) -- +(0\DK, -0.6\DK);
\draw [<-,dashed] (7.5\DK+1.5\DK, -1.2\DK-3.5\DK) -- node [above, very near start] {$\scriptscriptstyle+5$} +(0\DK, -0.6\DK);
\draw [->,dashed] (1.5\DK, -1.2\DK-8\DK) -- +(0\DK, -0.6\DK);
\draw [->,dashed] (3.0\DK, -0.7\DK-8\DK) -- +(0\DK, -0.6\DK);
\draw [->,dashed] (4.5\DK, -1.2\DK-8\DK) --  node [above, very near start] {$\scriptscriptstyle+5$} +(0\DK, -0.6\DK);
\draw [->,dashed] (6.0\DK, -0.7\DK-8\DK) -- +(0\DK, -0.6\DK);
\draw [->,dashed] (7.5\DK, -1.2\DK-8\DK) -- +(0\DK, -0.6\DK);
\draw [->,dashed] (9.0\DK, -0.7\DK-8\DK) -- node [above, very near end] {$\scriptscriptstyle+5$} +(0\DK, -0.6\DK);
\end{tikzpicture}%
\hspace{-10cm}\hfill%
\raisebox{1em}{\scalebox{0.67}
{\begin{tikzpicture}
\def\4{\ccolor{white}{$\boldsymbol{4}$}}
\def\2{\ccolor{white}{$\boldsymbol{2}$}}
\def\1{\ccolor{white}{$\boldsymbol{1}$}}
\def\0{\ccolor{white}{0}}
\def\O{\cocolor{white}{\ }}
\def\d{\cocolor{white}{\ddots}}
\def\t{\cocolor{white}{\cdots}}
\def\v{\cocolor{white}{\vdots}}
\def\o{ ++(1, 0) }
\draw (0,-0) \O\O\2\2 ;
\draw (0,-1) \O\O\2\2 ;
\draw (0,-2) \1\1\O\O\1\1\0 ;
\draw (0,-3) \1\1\O\O\0\1\1 ;
\draw (0,-4) \o\o\2\0\O\O\O\2\0 ;
\draw (0,-5) \o\o\1\1\O\O\O\1\1 ;
\draw (0,-6) \o\o\0\2\O\O\O\0\1 ;
\draw (0,-7) \o\o\o\o\1\1\0\O\O\1\1\0 ;
\draw (0,-8) \o\o\o\o\0\1\1\O\O\0\1\1 ;
\draw (0,-9) \o\o\o\o\o\o\o\2\0\d\d\d\d ;
\draw (0,-10) \o\o\o\o\o\o\o\1\1\d\O\O\O\2\0 ;
\draw (0,-11) \o\o\o\o\o\o\o\0\2\d\O\O\O\1\1 ;
\draw (0,-12) \o\o\o\o\o\o\o\o\o\d\O\O\O\0\2 ;
\draw (0,-13) \o\o\o\o\o\o\o\o\o\o\1\1\0\O\O\1\1 ;
\draw (0,-14) \o\o\o\o\o\o\o\o\o\o\0\1\1\O\O\1\1 ;
\draw (0,-15) \o\o\o\o\o\o\o\o\o\o\o\o\o\2\2\O\O ;
\draw (0,-16) \o\o\o\o\o\o\o\o\o\o\o\o\o\2\2\O\O ;
\draw (0.5,0.5)rectangle(17.5,-16.5);
\end{tikzpicture}
}}
\caption{$n=6\mbox{(\ref{II})},11,16,21,26,\ldots$.\hfill Twins: $\cll1,\cll2$; $\ccll{m},\cll{n}$.\hfill Periods: $[-4,4]$,
         $\big[\frac{2n-2}5,\frac{2n-2}5\big]$.\\[0.3em]
         \mbox{}\hfill\mbox{}
         $\FfF\Lft x,y\Rgt =\FfF_{23}\Lft x,y\Rgt,\ |x+y|<\frac{2n-2}5,$
         \mbox{}\hfill\mbox{}
         \\[0.3em]
         \mbox{}\hfill\mbox{}
         $\FfF\Lft \frac{n-1}5{-}2z{-}\beta,\
                \frac{n-1}5{+}2z{+}\beta \Rgt
         =\framebox{$n{-}1{+}\beta$},
         \
         z\in\ZZ,\beta\in\{0,1\}.$
         \mbox{}\hfill\mbox{}
         } \label{MM}
\end{figure}

\begin{figure}[H]
%
\def\0{ ++(1\DK, 0) }
\def\o{\dcolor{white}{$\ $}}
\def\1{\dcolor{yellow!40!white}{$1$}}
\def\2{\dcolor{yellow!80!white}{$2$}}
\def\3{\Dcolor{green!30!yellow!45!white}{white}{$3$}}
\def\4{\Dcolor{green!30!yellow!80!white}{white}{$4$}}
\def\5{\Dcolor{green!80!yellow!30!white}{white}{$5$}}
\def\6{\Dcolor{green!80!yellow!55!white}{white}{$6$}}
\def\7{\Dcolor{green!80!yellow!85!white}{white}{$7$}}
\def\8{\Dcolor{green!60!cyan!40!white  }{white}{$8$}}
\def\9{\Dcolor{green!60!cyan!80!white  }{white}{$9$}}
\def\h{\Dcolor{blue!80!cyan!10!white   }{white}{$g$}}
\def\i{\Dcolor{blue!80!cyan!30!white   }{white}{$h$}}
\def\j{\Dcolor{blue!80!cyan!50!white   }{white}{$i$}}
\def\k{\Dcolor{blue!80!red!20!white    }{white}{$j$}}
\def\l{\Dcolor{blue!80!red!40!white    }{white}{$k$}}
\def\m{\Dcolor{blue!30!red!15!white    }{white}{$l$}}
\def\n{\Dcolor{blue!30!red!30!white    }{white}{$m$}}
\def\o{\Dcolor{blue!30!red!45!white    }{white}{$n$}}
\def\t{ ++(1\DK, 0) node {$\cdots$}}
\begin{tikzpicture}[rotate=90]
\draw (0.5\DK, 1.5\DK)\4\4\3\3\4\4\3\3\4;
\draw (0\DK, 1\DK)   \1\2\1\2\1\2\1\2\1\2;
\drdgnl{1}
\draw (0.5\DK, 0.5\DK)\3\3\4\4\3\3\4\4\3;
\draw (0\DK, 0\DK)   \6\5\6\7\6\5\6\7\6\5;
\draw (0.5\DK, -0.5\DK)\8\8\9\9\8\8\9\9\8;
\draw (0\DK, -2\DK)     \i\h\i\j\i\h\i\j\i\h;
\draw (0.5\DK, -2.5\DK)\k\k\l\l\k\k\l\l\k;
\draw (0\DK, -3\DK)     \n\m\n\o\n\m\n\o\n\m;
\draw (0.5\DK, -3.5\DK)\k\k\l\l\k\k\l\l\k;
\draw (0\DK, -4\DK)     \i\h\i\j\i\h\i\j\i\h;
\draw (0.5\DK, -5.5\DK)\8\8\9\9\8\8\9\9\8;
\draw (0\DK, -6\DK)   \6\5\6\7\6\5\6\7\6\5;
\draw (0.5\DK, -6.5\DK)\3\3\4\4\3\3\4\4\3;
\draw (0\DK, -7\DK)   \1\2\1\2\1\2\1\2\1\2;
\draw (0.5\DK,-7.5\DK)\4\4\3\3\4\4\3\3\4;
\draw (0\DK, -8\DK)   \6\7\6\5\6\7\6\5\6\7;
\draw (0.5\DK, -8.5\DK)\9\9\8\8\9\9\8\8\9;
\draw [<-,dashed] (0.0\DK+1.5\DK, -0.7\DK+3.5\DK) --  +(0\DK, -0.6\DK);
\draw [<-,dashed] (1.5\DK+1.5\DK, -1.2\DK+3.5\DK) --  +(0\DK, -0.6\DK);
\draw [<-,dashed] (3.0\DK+1.5\DK, -0.7\DK+3.5\DK) -- node [above, very near end] {$\scriptscriptstyle+5$}  +(0\DK, -0.6\DK);
\draw [<-,dashed] (4.5\DK+1.5\DK, -1.2\DK+3.5\DK) -- +(0\DK, -0.6\DK);
\draw [<-,dashed] (6.0\DK+1.5\DK, -0.7\DK+3.5\DK) --   +(0\DK, -0.6\DK);
\draw [<-,dashed] (7.5\DK+1.5\DK, -1.2\DK+3.5\DK) -- node [above, very near start] {$\scriptscriptstyle+5$}  +(0\DK, -0.6\DK);
\draw [->,dashed] (1.5\DK, -1.2\DK) --  +(0\DK, -0.6\DK);
\draw [->,dashed] (3.0\DK, -0.7\DK) --  +(0\DK, -0.6\DK);
\draw [->,dashed] (4.5\DK, -1.2\DK) -- node [above, very near start] {$\scriptscriptstyle+5$}  +(0\DK, -0.6\DK);
\draw [->,dashed] (6.0\DK, -0.7\DK) --  +(0\DK, -0.6\DK);
\draw [->,dashed] (7.5\DK, -1.2\DK) -- +(0\DK, -0.6\DK);
\draw [->,dashed] (9.0\DK, -0.7\DK) --   node [above, very near end] {$\scriptscriptstyle+5$}  +(0\DK, -0.6\DK);
\draw [<-,dashed] (0.0\DK+1.5\DK, -0.7\DK-3.5\DK) --  +(0\DK, -0.6\DK);
\draw [<-,dashed] (1.5\DK+1.5\DK, -1.2\DK-3.5\DK) --  +(0\DK, -0.6\DK);
\draw [<-,dashed] (3.0\DK+1.5\DK, -0.7\DK-3.5\DK) -- node [above, very near end] {$\scriptscriptstyle+5$}  +(0\DK, -0.6\DK);
\draw [<-,dashed] (4.5\DK+1.5\DK, -1.2\DK-3.5\DK) --  +(0\DK, -0.6\DK);
\draw [<-,dashed] (6.0\DK+1.5\DK, -0.7\DK-3.5\DK) --   +(0\DK, -0.6\DK);
\draw [<-,dashed] (7.5\DK+1.5\DK, -1.2\DK-3.5\DK) -- node [above, very near start] {$\scriptscriptstyle+5$} +(0\DK, -0.6\DK);
\draw [->,dashed] (1.5\DK, -1.2\DK-8\DK) --  +(0\DK, -0.6\DK);
\draw [->,dashed] (3.0\DK, -0.7\DK-8\DK) --  +(0\DK, -0.6\DK);
\draw [->,dashed] (4.5\DK, -1.2\DK-8\DK) -- node [above, very near start] {$\scriptscriptstyle+5$} +(0\DK, -0.6\DK);
\draw [->,dashed] (6.0\DK, -0.7\DK-8\DK) --  +(0\DK, -0.6\DK);
\draw [->,dashed] (7.5\DK, -1.2\DK-8\DK) --  +(0\DK, -0.6\DK);
\draw [->,dashed] (9.0\DK, -0.7\DK-8\DK) -- node [above, very near end] {$\scriptscriptstyle+5$}    +(0\DK, -0.6\DK);
\drdgnl{-7}
\end{tikzpicture}%
\hspace{-10cm}\hfill%
\raisebox{1em}{\scalebox{0.63}
{\begin{tikzpicture}
\def\4{\ccolor{white}{$\boldsymbol{4}$}}
\def\2{\ccolor{white}{$\boldsymbol{2}$}}
\def\1{\ccolor{white}{$\boldsymbol{1}$}}
\def\0{\ccolor{white}{0}}
\def\O{\cocolor{white}{\ }}
\def\d{\cocolor{white}{\ddots}}
\def\t{\cocolor{white}{\cdots}}
\def\v{\cocolor{white}{\vdots}}
\def\o{ ++(1, 0) }
\draw (0,-0) \O\O\2\2 ;
\draw (0,-1) \O\O\2\2 ;
\draw (0,-2) \1\1\O\O\1\1\0 ;
\draw (0,-3) \1\1\O\O\0\1\1 ;
\draw (0,-4) \o\o\2\0\O\O\O\2\0 ;
\draw (0,-5) \o\o\1\1\O\O\O\1\1 ;
\draw (0,-6) \o\o\0\2\O\O\O\0\1 ;
\draw (0,-7) \o\o\o\o\1\1\0\O\O\1\1\0 ;
\draw (0,-8) \o\o\o\o\0\1\1\O\O\0\1\1 ;
\draw (0,-9) \o\o\o\o\o\o\o\2\0\d\d\d\d ;
\draw (0,-10) \o\o\o\o\o\o\o\1\1\d\O\O\O\2\0 ;
\draw (0,-11) \o\o\o\o\o\o\o\0\2\d\O\O\O\1\1 ;
\draw (0,-12) \o\o\o\o\o\o\o\o\o\d\O\O\O\0\2 ;
\draw (0,-13) \o\o\o\o\o\o\o\o\o\o\1\1\0\O\O\1\1\0 ;
\draw (0,-14) \o\o\o\o\o\o\o\o\o\o\0\1\1\O\O\0\1\1 ;
\draw (0,-15) \o\o\o\o\o\o\o\o\o\o\o\o\o\4\0\O\O\O ;
\draw (0,-16) \o\o\o\o\o\o\o\o\o\o\o\o\o\2\2\O\O\O ;
\draw (0,-17) \o\o\o\o\o\o\o\o\o\o\o\o\o\0\4\O\O\O ;
\draw (0.5,0.5)rectangle(18.5,-17.5);
\end{tikzpicture}
}}
\caption{$n=7\mbox{(\ref{AA})},12,17,22,27,\ldots$.\hfill Twins: $\cll1,\cll2$.\hfill Periods: $[-4,4]$, $\big[\frac{2n-4}5-2,\frac{2n-4}5+2\big]$.
         \\[0.3ex]
         \mbox{}\hfill\mbox{}
         $\FfF\Lft x,y\Rgt =\FfF_{23}\Lft x,y\Rgt,\ |x+y|\le\frac{2n-4}5.$
         \mbox{}\hfill\mbox{}
         } \label{NN}
\end{figure}

\begin{figure}[H]
\def\0{ ++(1\DK, 0) }
\def\o{\dcolor{white}{$\ $}}
\def\1{\dcolor{yellow!40!white}{$1$}}
\def\2{\dcolor{yellow!80!white}{$2$}}
\def\3{\Dcolor{green!30!yellow!45!white}{white}{$3$}}
\def\4{\Dcolor{green!30!yellow!80!white}{white}{$4$}}
\def\5{\Dcolor{green!80!yellow!30!white}{white}{$5$}}
\def\6{\Dcolor{green!80!yellow!55!white}{white}{$6$}}
\def\7{\Dcolor{green!80!yellow!85!white}{white}{$7$}}
\def\8{\Dcolor{green!60!cyan!40!white  }{white}{$8$}}
\def\9{\Dcolor{green!60!cyan!80!white  }{white}{$9$}}
\def\h{\Dcolor{blue!80!cyan!20!white   }{white}{$h$}}
\def\i{\Dcolor{blue!80!cyan!40!white   }{white}{$i$}}
\def\j{\Dcolor{blue!80!red!10!white    }{white}{$j$}}
\def\k{\Dcolor{blue!80!red!30!white    }{white}{$k$}}
\def\l{\Dcolor{blue!80!red!50!white    }{white}{$l$}}
\def\m{\dcolor{blue!30!red!15!white}{$m$}}
\def\n{\dcolor{blue!30!red!30!white}{$n$}}
\def\t{ ++(1\DK, 0) node {$\cdots$}}
\begin{tikzpicture}[rotate=90]
\draw (0.5\DK, 1.5\DK)\4\4\3\3\4\4\3\3\4;
\draw (0\DK, 1\DK)   \1\2\1\2\1\2\1\2\1\2;
\drdgnl{1}
\draw (0.5\DK, 0.5\DK)\3\3\4\4\3\3\4\4\3;
\draw (0\DK, 0\DK)   \6\5\6\7\6\5\6\7\6\5;
\draw (0.5\DK, -0.5\DK)\8\8\9\9\8\8\9\9\8;
\draw (0.5\DK, -2.5\DK)  \h\h\i\i\h\h\i\i\h;
\draw (0\DK,         -3\DK)\k\j\k\l\k\j\k\l\k\j;
\draw (0.5\DK, -3.5\DK)  \m\n\m\n\m\n\m\n\m;
\drdgnla{-3.5}
\draw (0\DK,         -4\DK)\k\l\k\j\k\l\k\j\k\l;
\draw (0.5\DK, -4.5\DK)  \i\i\h\h\i\i\h\h\i;
\draw (0.5\DK, -6.5\DK)\9\9\8\8\9\9\8\8\9;
\draw (0\DK, -7\DK)   \6\7\6\5\6\7\6\5\6\7;
\draw (0.5\DK, -7.5\DK)\4\4\3\3\4\4\3\3\4;
\draw (0\DK, -8\DK)   \1\2\1\2\1\2\1\2\1\2;
\drdgnl{-8}
\draw (0.5\DK,-8.5\DK)\3\3\4\4\3\3\4\4\3;
\draw (0\DK, -9\DK)   \6\5\6\7\6\5\6\7\6\5;
\draw (0.5\DK, -9.5\DK)\8\8\9\9\8\8\9\9\8;
\draw [<-,dashed] (1.5\DK, -1\DK+3.50\DK) -- node [above] {$                    $} +(0\DK, -0.6\DK);
\draw [<-,dashed] (3.0\DK, -1\DK+3.50\DK) -- node [above] {$                    $} +(0\DK, -0.6\DK);
\draw [<-,dashed] (4.5\DK, -1\DK+3.50\DK) -- node [above] {$                    $} +(0\DK, -0.6\DK);
\draw [<-,dashed] (6.0\DK, -1\DK+3.50\DK) -- node [above] {$                    $} +(0\DK, -0.6\DK);
\draw [<-,dashed] (7.5\DK, -1\DK+3.50\DK) -- node [above] {$                    $} +(0\DK, -0.6\DK);
\draw [<-,dashed] (9.0\DK, -1\DK+3.50\DK) -- node [above] {$                    $} +(0\DK, -0.6\DK);
\draw [->,dashed] (1.5\DK, -1\DK-0.25\DK) -- node [above] {$                    $} +(0\DK, -0.6\DK);
\draw [->,dashed] (3.0\DK, -1\DK-0.25\DK) -- node [above] {$                    $} +(0\DK, -0.6\DK);
\draw [->,dashed] (4.5\DK, -1\DK-0.25\DK) -- node [above] {$\scriptscriptstyle+5$} +(0\DK, -0.6\DK);
\draw [->,dashed] (6.0\DK, -1\DK-0.25\DK) -- node [above] {$                    $} +(0\DK, -0.6\DK);
\draw [->,dashed] (7.5\DK, -1\DK-0.25\DK) -- node [above] {$                    $} +(0\DK, -0.6\DK);
\draw [->,dashed] (9.0\DK, -1\DK-0.25\DK) -- node [above] {$\scriptscriptstyle+5$} +(0\DK, -0.6\DK);
\draw [<-,dashed] (1.5\DK, -1\DK-4.25\DK) -- node [above] {$                    $} +(0\DK, -0.6\DK);
\draw [<-,dashed] (3.0\DK, -1\DK-4.25\DK) -- node [above] {$                    $} +(0\DK, -0.6\DK);
\draw [<-,dashed] (4.5\DK, -1\DK-4.25\DK) -- node [above] {$\scriptscriptstyle+5$} +(0\DK, -0.6\DK);
\draw [<-,dashed] (6.0\DK, -1\DK-4.25\DK) -- node [above] {$                    $} +(0\DK, -0.6\DK);
\draw [<-,dashed] (7.5\DK, -1\DK-4.25\DK) -- node [above] {$                    $} +(0\DK, -0.6\DK);
\draw [<-,dashed] (9.0\DK, -1\DK-4.25\DK) -- node [above] {$\scriptscriptstyle+5$} +(0\DK, -0.6\DK);
\draw [->,dashed] (1.5\DK, -1\DK-9.25\DK) -- node [above] {$                    $} +(0\DK, -0.6\DK);
\draw [->,dashed] (3.0\DK, -1\DK-9.25\DK) -- node [above] {$                    $} +(0\DK, -0.6\DK);
\draw [->,dashed] (4.5\DK, -1\DK-9.25\DK) -- node [above] {$                    $} +(0\DK, -0.6\DK);
\draw [->,dashed] (6.0\DK, -1\DK-9.25\DK) -- node [above] {$                    $} +(0\DK, -0.6\DK);
\draw [->,dashed] (7.5\DK, -1\DK-9.25\DK) -- node [above] {$                    $} +(0\DK, -0.6\DK);
\draw [->,dashed] (9.0\DK, -1\DK-9.25\DK) -- node [above] {$                    $} +(0\DK, -0.6\DK);
\end{tikzpicture}%
\hspace{-10cm}\hfill%
\raisebox{1em}{\scalebox{0.60}
{\begin{tikzpicture}
\def\4{\ccolor{white}{$\boldsymbol{4}$}}
\def\2{\ccolor{white}{$\boldsymbol{2}$}}
\def\1{\ccolor{white}{$\boldsymbol{1}$}}
\def\0{\ccolor{white}{0}}
\def\O{\cocolor{white}{\ }}
\def\d{\cocolor{white}{\ddots}}
\def\t{\cocolor{white}{\cdots}}
\def\v{\cocolor{white}{\vdots}}
\def\o{ ++(1, 0) }
\draw (0,-0) \O\O\2\2 ;
\draw (0,-1) \O\O\2\2 ;
\draw (0,-2) \1\1\O\O\1\1\0 ;
\draw (0,-3) \1\1\O\O\0\1\1 ;
\draw (0,-4) \o\o\2\0\O\O\O\2\0 ;
\draw (0,-5) \o\o\1\1\O\O\O\1\1 ;
\draw (0,-6) \o\o\0\2\O\O\O\0\1 ;
\draw (0,-7) \o\o\o\o\1\1\0\O\O\1\1\0 ;
\draw (0,-8) \o\o\o\o\0\1\1\O\O\0\1\1 ;
\draw (0,-9) \o\o\o\o\o\o\o\2\0\d\d\d\d ;
\draw (0,-10) \o\o\o\o\o\o\o\1\1\d\O\O\1\1\0 ;
\draw (0,-11) \o\o\o\o\o\o\o\0\2\d\O\O\0\1\1 ;
\draw (0,-12) \o\o\o\o\o\o\o\o\o\d\2\0\O\O\O\1\1 ;
\draw (0,-13) \o\o\o\o\o\o\o\o\o\o\1\1\O\O\O\1\1 ;
\draw (0,-14) \o\o\o\o\o\o\o\o\o\o\0\2\O\O\O\1\1 ;
\draw (0,-15) \o\o\o\o\o\o\o\o\o\o\o\o\1\2\1\O\O ;
\draw (0,-16) \o\o\o\o\o\o\o\o\o\o\o\o\1\2\1\O\O ;
\draw (0.5,0.5)rectangle(17.5,-16.5);
\end{tikzpicture}
}}
\caption{$n=4\eqref{II},9,14,19,24,\ldots$.\hfill Twins: $\cll1,\cll2$; $\ccll{m},\cll{n}$.\hfill Periods: $[-4,4]$, $\big[\frac{2n-3}5,\frac{2n-3}5\big]$.
         \\[0.3em]
         \mbox{}\hfill\mbox{}
         $\FfF\Lft x,y\Rgt =\FfF_{23}\Lft x,y\Rgt,\ |x+y|<\frac{2n-3}5,$
         \mbox{}\hfill\mbox{}
         \\[0.3em]
         \mbox{}\hfill\mbox{}
         $\FfF\Lft \frac{n-4}5{-}2z{-}\beta,\
                \frac{n+1}5{+}2z{+}\beta \Rgt
         =\framebox{$n{-}1{+}\beta$},
         \
         z\in\ZZ,\beta\in\{0,1\}.$
         \mbox{}\hfill\mbox{}
         } \label{OO}
\end{figure}

\begin{figure}[H]
\def\0{ ++(1\DK, 0) }
\def\o{\dcolor{white}{$\ $}}
\def\1{\dcolor{yellow!40!white}{$1$}}
\def\2{\dcolor{yellow!80!white}{$2$}}
\def\3{\Dcolor{green!30!yellow!45!white}{white}{$3$}}
\def\4{\Dcolor{green!30!yellow!80!white}{white}{$4$}}
\def\5{\Dcolor{green!80!yellow!30!white}{white}{$5$}}
\def\6{\Dcolor{green!80!yellow!55!white}{white}{$6$}}
\def\7{\Dcolor{green!80!yellow!85!white}{white}{$7$}}
\def\8{\Dcolor{green!60!cyan!40!white  }{white}{$8$}}
\def\9{\Dcolor{green!60!cyan!80!white  }{white}{$9$}}
\def\h{\Dcolor{blue!80!cyan!20!white   }{white}{$h$}}
\def\i{\Dcolor{blue!80!cyan!40!white   }{white}{$i$}}
\def\j{\Dcolor{blue!80!red!10!white    }{white}{$j$}}
\def\k{\Dcolor{blue!80!red!30!white    }{white}{$k$}}
\def\l{\Dcolor{blue!80!red!50!white    }{white}{$l$}}
\def\m{\Dcolor{blue!30!red!20!white    }{white}{$m$}}
\def\n{\Dcolor{blue!30!red!45!white    }{white}{$n$}}
\def\t{ ++(1\DK, 0) node {$\cdots$}}
\begin{tikzpicture}[rotate=90]
\draw (0.5\DK, 1.5\DK)\4\4\3\3\4\4\3\3\4;
\draw (0\DK, 1\DK)   \1\2\1\2\1\2\1\2\1\2;\drdgnl{1}
\draw (0.5\DK, 0.5\DK)\3\3\4\4\3\3\4\4\3;
\draw (0\DK, 0\DK)   \6\5\6\7\6\5\6\7\6\5;
\draw (0.5\DK, -0.5\DK)\8\8\9\9\8\8\9\9\8;
\draw (0.5\DK, -2.5\DK)\h\h\i\i\h\h\i\i\h;
\draw (0\DK,         -3\DK)\k\j\k\l\k\j\k\l\k\j;
\draw (0.5\DK, -3.5\DK)\m\m\n\n\m\m\n\n\m;
\draw (0\DK,         -4\DK)\k\j\k\l\k\j\k\l\k\j;
\draw (0.5\DK, -4.5\DK)\h\h\i\i\h\h\i\i\h;
\draw (0.5\DK, -6.5\DK)\8\8\9\9\8\8\9\9\8;
\draw (0\DK, -7\DK)   \6\5\6\7\6\5\6\7\6\5;
\draw (0.5\DK, -7.5\DK)\3\3\4\4\3\3\4\4\3;
\draw (0\DK, -8\DK)   \1\2\1\2\1\2\1\2\1\2;\drdgnl{-8}
\draw (0.5\DK,-8.5\DK)\4\4\3\3\4\4\3\3\4;
\draw (0\DK, -9\DK)   \6\7\6\5\6\7\6\5\6\7;
\draw (0.5\DK, -9.5\DK)\9\9\8\8\9\9\8\8\9;
\draw [<-,dashed] (1.5\DK, -1\DK+3.50\DK) -- node [above] {$                    $} +(0\DK, -0.6\DK);
\draw [<-,dashed] (3.0\DK, -1\DK+3.50\DK) -- node [above] {$                    $} +(0\DK, -0.6\DK);
\draw [<-,dashed] (4.5\DK, -1\DK+3.50\DK) -- node [above] {$                    $} +(0\DK, -0.6\DK);
\draw [<-,dashed] (6.0\DK, -1\DK+3.50\DK) -- node [above] {$                    $} +(0\DK, -0.6\DK);
\draw [<-,dashed] (7.5\DK, -1\DK+3.50\DK) -- node [above] {$                    $} +(0\DK, -0.6\DK);
\draw [<-,dashed] (9.0\DK, -1\DK+3.50\DK) -- node [above] {$                    $} +(0\DK, -0.6\DK);
\draw [->,dashed] (1.5\DK, -1\DK-0.25\DK) -- node [above] {$                    $} +(0\DK, -0.6\DK);
\draw [->,dashed] (3.0\DK, -1\DK-0.25\DK) -- node [above] {$                    $} +(0\DK, -0.6\DK);
\draw [->,dashed] (4.5\DK, -1\DK-0.25\DK) -- node [above] {$\scriptscriptstyle+5$} +(0\DK, -0.6\DK);
\draw [->,dashed] (6.0\DK, -1\DK-0.25\DK) -- node [above] {$                    $} +(0\DK, -0.6\DK);
\draw [->,dashed] (7.5\DK, -1\DK-0.25\DK) -- node [above] {$                    $} +(0\DK, -0.6\DK);
\draw [->,dashed] (9.0\DK, -1\DK-0.25\DK) -- node [above] {$\scriptscriptstyle+5$} +(0\DK, -0.6\DK);
\draw [<-,dashed] (1.5\DK, -1\DK-4.25\DK) -- node [above] {$                    $} +(0\DK, -0.6\DK);
\draw [<-,dashed] (3.0\DK, -1\DK-4.25\DK) -- node [above] {$                    $} +(0\DK, -0.6\DK);
\draw [<-,dashed] (4.5\DK, -1\DK-4.25\DK) -- node [above] {$\scriptscriptstyle+5$} +(0\DK, -0.6\DK);
\draw [<-,dashed] (6.0\DK, -1\DK-4.25\DK) -- node [above] {$                    $} +(0\DK, -0.6\DK);
\draw [<-,dashed] (7.5\DK, -1\DK-4.25\DK) -- node [above] {$                    $} +(0\DK, -0.6\DK);
\draw [<-,dashed] (9.0\DK, -1\DK-4.25\DK) -- node [above] {$\scriptscriptstyle+5$} +(0\DK, -0.6\DK);
\draw [->,dashed] (1.5\DK, -1\DK-9.25\DK) -- node [above] {$                    $} +(0\DK, -0.6\DK);
\draw [->,dashed] (3.0\DK, -1\DK-9.25\DK) -- node [above] {$                    $} +(0\DK, -0.6\DK);
\draw [->,dashed] (4.5\DK, -1\DK-9.25\DK) -- node [above] {$                    $} +(0\DK, -0.6\DK);
\draw [->,dashed] (6.0\DK, -1\DK-9.25\DK) -- node [above] {$                    $} +(0\DK, -0.6\DK);
\draw [->,dashed] (7.5\DK, -1\DK-9.25\DK) -- node [above] {$                    $} +(0\DK, -0.6\DK);
\draw [->,dashed] (9.0\DK, -1\DK-9.25\DK) -- node [above] {$                    $} +(0\DK, -0.6\DK);
\end{tikzpicture}%
\hspace{-10cm}\hfill%
\raisebox{1em}{\scalebox{0.60}
{\begin{tikzpicture}
\def\4{\ccolor{white}{$\boldsymbol{4}$}}
\def\2{\ccolor{white}{$\boldsymbol{2}$}}
\def\1{\ccolor{white}{$\boldsymbol{1}$}}
\def\0{\ccolor{white}{0}}
\def\O{\cocolor{white}{\ }}
\def\d{\cocolor{white}{\ddots}}
\def\t{\cocolor{white}{\cdots}}
\def\v{\cocolor{white}{\vdots}}
\def\o{ ++(1, 0) }
\draw (0,-0) \O\O\2\2 ;
\draw (0,-1) \O\O\2\2 ;
\draw (0,-2) \1\1\O\O\1\1\0 ;
\draw (0,-3) \1\1\O\O\0\1\1 ;
\draw (0,-4) \o\o\2\0\O\O\O\2\0 ;
\draw (0,-5) \o\o\1\1\O\O\O\1\1 ;
\draw (0,-6) \o\o\0\2\O\O\O\0\1 ;
\draw (0,-7) \o\o\o\o\1\1\0\O\O\1\1\0 ;
\draw (0,-8) \o\o\o\o\0\1\1\O\O\0\1\1 ;
\draw (0,-9) \o\o\o\o\o\o\o\2\0\d\d\d\d ;
\draw (0,-10) \o\o\o\o\o\o\o\1\1\d\O\O\1\1\0 ;
\draw (0,-11) \o\o\o\o\o\o\o\0\2\d\O\O\0\1\1 ;
\draw (0,-12) \o\o\o\o\o\o\o\o\o\d\2\0\O\O\O\2\0 ;
\draw (0,-13) \o\o\o\o\o\o\o\o\o\o\1\1\O\O\O\1\1 ;
\draw (0,-14) \o\o\o\o\o\o\o\o\o\o\0\2\O\O\O\0\2 ;
\draw (0,-15) \o\o\o\o\o\o\o\o\o\o\o\o\2\2\0\O\O ;
\draw (0,-16) \o\o\o\o\o\o\o\o\o\o\o\o\0\2\2\O\O ;
\draw (0.5,0.5)rectangle(17.5,-16.5);
\end{tikzpicture}
}}
\caption{$n=4\eqref{II},9,14,19,24,\ldots$.\hfill Twins: $\cll1,\cll2$.\hfill Periods: $[-4,4]$, $\big[\frac{2n-3}5-2,\frac{2n-3}5+2\big]$.
         \\[0.3em]
         \mbox{}\hfill\mbox{}
         $\FfF\Lft x,y\Rgt =\FfF_{23}\Lft x,y\Rgt,\ |x+y|\le\frac{2n-3}5.
         $
         \mbox{}\hfill\mbox{}
         } \label{PP}
\end{figure}

\begin{figure}[H]
\def\0{ ++(1\DK, 0) }
\def\o{\dcolor{white}{$\ $}}
\def\1{\dcolor{yellow!40!white}{$1$}}
\def\2{\dcolor{yellow!80!white}{$2$}}
\def\3{\Dcolor{green!30!yellow!40!white}{white}{$3$}}
\def\4{\Dcolor{green!30!yellow!70!white}{white}{$4$}}
\def\5{\Dcolor{green!30!yellow!99!white}{white}{$5$}}
\def\6{\Dcolor{green!80!yellow!40!white}{white}{$6$}}
\def\7{\Dcolor{green!80!yellow!80!white}{white}{$7$}}
\def\8{\Dcolor{green!60!cyan!30!white  }{white}{$8$}}
\def\9{\Dcolor{green!60!cyan!55!white  }{white}{$9$}}
\def\a{\Dcolor{green!60!cyan!80!white  }{white}{$10$}}
\def\h{\Dcolor{blue!80!cyan!20!white   }{white}{$h$}}
\def\i{\Dcolor{blue!80!cyan!40!white   }{white}{$i$}}
\def\j{\Dcolor{blue!80!red!10!white    }{white}{$j$}}
\def\k{\Dcolor{blue!80!red!30!white    }{white}{$k$}}
\def\l{\Dcolor{blue!80!red!50!white    }{white}{$l$}}
\def\m{\dcolor{blue!30!red!15!white}{$m$}}
\def\n{\dcolor{blue!30!red!30!white}{$n$}}
\def\t{ ++(1\DK, 0) node {$\cdots$}}
\begin{tikzpicture}[rotate=90]
\draw (0\DK, 1.5\DK)   \4\5\4\3\4\5\4\3\4\5;
\draw (0.5\DK, 1\DK)    \1\2\1\2\1\2\1\2\1;\drdgnla{1}
\draw (0\DK, 0.5\DK)   \4\3\4\5\4\3\4\5\4\3;
\draw (0.5\DK, 0\DK)    \6\6\7\7\6\6\7\7\6;
\draw (0\DK, -0.5\DK)  \9\8\9\a\9\8\9\a\9\8;
\draw (0.5\DK, -2.5\DK)  \h\h\i\i\h\h\i\i\h;
\draw (0\DK,         -3\DK)\k\j\k\l\k\j\k\l\k\j;
\draw (0.5\DK, -3.5\DK)  \m\n\m\n\m\n\m\n\m;\drdgnla{-3.5}
\draw (0\DK,         -4\DK)\k\l\k\j\k\l\k\j\k\l;
\draw (0.5\DK, -4.5\DK)  \i\i\h\h\i\i\h\h\i;
\draw (0\DK, -6.5\DK)   \9\a\9\8\9\a\9\8\9\a;
\draw (0.5\DK, -7\DK)    \7\7\6\6\7\7\6\6\7;
\draw (0\DK, -7.5\DK)   \4\5\4\3\4\5\4\3\4\5;
\draw (0.5\DK, -8\DK)    \1\2\1\2\1\2\1\2\1;\drdgnla{-8}
\draw (0\DK,-8.5\DK)    \4\3\4\5\4\3\4\5\4\3;
\draw (0.5\DK, -9\DK)    \6\6\7\7\6\6\7\7\6;
\draw (0\DK, -9.5\DK)   \9\8\9\a\9\8\9\a\9\8;
\draw [<-,dashed] (0.0\DK+1.5\DK, -1.2\DK+3.75\DK) -- node [above ] {$                    $} +(0\DK, -0.6\DK);
\draw [<-,dashed] (3.0\DK+1.5\DK, -1.2\DK+3.75\DK) -- node [above ] {$                    $} +(0\DK, -0.6\DK);
\draw [<-,dashed] (6.0\DK+1.5\DK, -1.2\DK+3.75\DK) -- node [above ] {$\scriptscriptstyle+5$} +(0\DK, -0.6\DK);
\draw [->,dashed] (1.5\DK, -0.7\DK-0.25\DK) -- node [above, very near end  ] {$                    $} +(0\DK, -0.6\DK);
\draw [->,dashed] (3.0\DK, -1.2\DK-0.25\DK) -- node [above, very near start] {$                    $} +(0\DK, -0.6\DK);
\draw [->,dashed] (4.5\DK, -0.7\DK-0.25\DK) -- node [above, very near end  ] {$\scriptscriptstyle+5$} +(0\DK, -0.6\DK);
\draw [->,dashed] (6.0\DK, -1.2\DK-0.25\DK) -- node [above, very near start] {$                    $} +(0\DK, -0.6\DK);
\draw [->,dashed] (7.5\DK, -0.7\DK-0.25\DK) -- node [above, very near end  ] {$                    $} +(0\DK, -0.6\DK);
\draw [->,dashed] (9.0\DK, -1.2\DK-0.25\DK) -- node [above, very near start] {$\scriptscriptstyle+5$} +(0\DK, -0.6\DK);
\draw [<-,dashed] (0.0\DK+1.5\DK, -1.2\DK-4.25\DK) -- node [above, very near start] {$                    $} +(0\DK, -0.6\DK);
\draw [<-,dashed] (1.5\DK+1.5\DK, -0.7\DK-4.25\DK) -- node [above, very near end  ] {$                    $} +(0\DK, -0.6\DK);
\draw [<-,dashed] (3.0\DK+1.5\DK, -1.2\DK-4.25\DK) -- node [above, very near start] {$\scriptscriptstyle+5$} +(0\DK, -0.6\DK);
\draw [<-,dashed] (4.5\DK+1.5\DK, -0.7\DK-4.25\DK) -- node [above, very near end  ] {$                    $} +(0\DK, -0.6\DK);
\draw [<-,dashed] (6.0\DK+1.5\DK, -1.2\DK-4.25\DK) -- node [above, very near start] {$                    $} +(0\DK, -0.6\DK);
\draw [<-,dashed] (7.5\DK+1.5\DK, -0.7\DK-4.25\DK) -- node [above, very near end  ] {$\scriptscriptstyle+5$} +(0\DK, -0.6\DK);
\draw [->,dashed] (1.5\DK, -0.7\DK-9.25\DK) -- node [above ] {$                    $} +(0\DK, -0.6\DK);
\draw [->,dashed] (4.5\DK, -0.7\DK-9.25\DK) -- node [above ] {$                    $} +(0\DK, -0.6\DK);
\draw [->,dashed] (7.5\DK, -0.7\DK-9.25\DK) -- node [above ] {$\scriptscriptstyle+5$} +(0\DK, -0.6\DK);
\end{tikzpicture}%
\hspace{-10cm}\hfill%
\raisebox{1em}{\scalebox{0.59}
{\begin{tikzpicture}
\def\4{\ccolor{white}{$\boldsymbol{4}$}}
\def\2{\ccolor{white}{$\boldsymbol{2}$}}
\def\1{\ccolor{white}{$\boldsymbol{1}$}}
\def\0{\ccolor{white}{0}}
\def\O{\cocolor{white}{\ }}
\def\d{\cocolor{white}{\ddots}}
\def\t{\cocolor{white}{\cdots}}
\def\v{\cocolor{white}{\vdots}}
\def\o{ ++(1, 0) }
\draw (0,-0) \O\O\1\2\1 ;
\draw (0,-1) \O\O\1\2\1 ;
\draw (0,-2) \1\1\O\O\O\2\0 ;
\draw (0,-3) \1\1\O\O\O\1\1 ;
\draw (0,-4) \1\1\O\O\O\0\2 ;
\draw (0,-5) \o\o\1\1\0\O\O\1\1\0 ;
\draw (0,-6) \o\o\0\1\1\O\O\0\1\1 ;
\draw (0,-7) \o\o\o\o\o\2\0\O\O\O\2\0 ;
\draw (0,-8) \o\o\o\o\o\1\1\O\O\O\1\1 ;
\draw (0,-9) \o\o\o\o\o\0\2\O\O\O\0\2 ;
\draw (0,-10)\o\o\o\o\o\o\o\1\1\0\d\d\d\d ;
\draw (0,-11)\o\o\o\o\o\o\o\0\1\1\d\O\O\1\1\0 ;
\draw (0,-12)\o\o\o\o\o\o\o\o\o\o\d\O\O\0\1\1 ;
\draw (0,-13)\o\o\o\o\o\o\o\o\o\o\d\2\0\O\O\O\1\1 ;
\draw (0,-14)\o\o\o\o\o\o\o\o\o\o\o\1\1\O\O\O\1\1 ;
\draw (0,-15)\o\o\o\o\o\o\o\o\o\o\o\0\2\O\O\O\1\1 ;
\draw (0,-16)\o\o\o\o\o\o\o\o\o\o\o\o\o\1\2\1\O\O ;
\draw (0,-17)\o\o\o\o\o\o\o\o\o\o\o\o\o\1\2\1\O\O ;
\draw (0.5,0.5)rectangle(18.5,-17.5);
\end{tikzpicture}
}}
\caption{$n=7\eqref{II},12,17,22,27,\ldots$.\hfill Twins: $\cll1,\cll2$; $\ccll{m},\cll{n}$.\hfill Periods: $[-4,4]$, $\big[\frac{2n-4}5,\frac{2n-4}5\big]$.\\[0.3em]
         \mbox{}\hfill\mbox{}
         $\FfF\Lft x,y\Rgt =\FfF_{32}\Lft x,y\Rgt,\ |x+y|<\frac{2n-4}5,$
         \mbox{}\hfill\mbox{}
         \\[0.3em]
         \mbox{}\hfill\mbox{}
         $\FfF\big([\frac{n-2}5{-}2z{-}\beta,\
                \frac{n-2}5{+}2z{+}\beta ]\big)
         =\framebox{$n{-}1{+}\beta$},
         \
         z\in\ZZ,\beta\in\{0,1\}.$
         \mbox{}\hfill\mbox{}
         } \label{QQ}
\end{figure}

\begin{figure}[H]
\def\0{ ++(1\DK, 0) }
\def\o{\dcolor{white}{$\ $}}
\def\1{\dcolor{yellow!40!white}{$1$}}
\def\2{\dcolor{yellow!80!white}{$2$}}
\def\3{\Dcolor{green!30!yellow!40!white}{white}{$3$}}
\def\4{\Dcolor{green!30!yellow!70!white}{white}{$4$}}
\def\5{\Dcolor{green!30!yellow!99!white}{white}{$5$}}
\def\6{\Dcolor{green!80!yellow!40!white}{white}{$6$}}
\def\7{\Dcolor{green!80!yellow!80!white}{white}{$7$}}
\def\8{\Dcolor{green!60!cyan!30!white  }{white}{$8$}}
\def\9{\Dcolor{green!60!cyan!55!white  }{white}{$9$}}
\def\a{\Dcolor{green!60!cyan!80!white  }{white}{$10$}}
\def\h{\Dcolor{blue!80!cyan!20!white   }{white}{$h$}}
\def\i{\Dcolor{blue!80!cyan!40!white   }{white}{$i$}}
\def\j{\Dcolor{blue!80!red!10!white    }{white}{$j$}}
\def\k{\Dcolor{blue!80!red!30!white    }{white}{$k$}}
\def\l{\Dcolor{blue!80!red!50!white    }{white}{$l$}}
\def\m{\Dcolor{blue!30!red!20!white    }{white}{$m$}}
\def\n{\Dcolor{blue!30!red!40!white    }{white}{$n$}}
\def\t{ ++(1\DK, 0) node {$\cdots$}}
\begin{tikzpicture}[rotate=90]
\draw (0\DK, 1.5\DK)   \4\5\4\3\4\5\4\3\4\5;
\draw (0.5\DK, 1\DK)    \1\2\1\2\1\2\1\2\1;\drdgnla{1}
\draw (0\DK, 0.5\DK)   \4\3\4\5\4\3\4\5\4\3;
\draw (0.5\DK, 0\DK)    \6\6\7\7\6\6\7\7\6;
\draw (0\DK, -0.5\DK)  \9\8\9\a\9\8\9\a\9\8;
\draw (0.5\DK, -2.5\DK)  \h\h\i\i\h\h\i\i\h;
\draw (0\DK,         -3\DK)\k\j\k\l\k\j\k\l\k\j;
\draw (0.5\DK, -3.5\DK)  \m\m\n\n\m\m\n\n\m;
\draw (0\DK,         -4\DK)\k\j\k\l\k\j\k\l\k\j;
\draw (0.5\DK, -4.5\DK)  \h\h\i\i\h\h\i\i\h;
\draw (0\DK, -6.5\DK)   \9\8\9\a\9\8\9\a\9\8;
\draw (0.5\DK, -7\DK)    \6\6\7\7\6\6\7\7\6;
\draw (0\DK, -7.5\DK)   \4\3\4\5\4\3\4\5\4\3;
\draw (0.5\DK, -8\DK)    \1\2\1\2\1\2\1\2\1;\drdgnla{-8}
\draw (0\DK,-8.5\DK)    \4\5\4\3\4\5\4\3\4\5;
\draw (0.5\DK, -9\DK)    \7\7\6\6\7\7\6\6\7;
\draw (0\DK, -9.5\DK)   \9\a\9\8\9\a\9\8\9\a;
\draw [<-,dashed] (0.0\DK+1.5\DK, -1.2\DK+3.75\DK) -- node [above ] {$                    $} +(0\DK, -0.6\DK);
\draw [<-,dashed] (3.0\DK+1.5\DK, -1.2\DK+3.75\DK) -- node [above ] {$                    $} +(0\DK, -0.6\DK);
\draw [<-,dashed] (6.0\DK+1.5\DK, -1.2\DK+3.75\DK) -- node [above ] {$\scriptscriptstyle+5$} +(0\DK, -0.6\DK);
\draw [->,dashed] (1.5\DK, -0.7\DK-0.25\DK) -- node [above, very near end  ] {$                    $} +(0\DK, -0.6\DK);
\draw [->,dashed] (3.0\DK, -1.2\DK-0.25\DK) -- node [above, very near start] {$                    $} +(0\DK, -0.6\DK);
\draw [->,dashed] (4.5\DK, -0.7\DK-0.25\DK) -- node [above, very near end  ] {$\scriptscriptstyle+5$} +(0\DK, -0.6\DK);
\draw [->,dashed] (6.0\DK, -1.2\DK-0.25\DK) -- node [above, very near start] {$                    $} +(0\DK, -0.6\DK);
\draw [->,dashed] (7.5\DK, -0.7\DK-0.25\DK) -- node [above, very near end  ] {$                    $} +(0\DK, -0.6\DK);
\draw [->,dashed] (9.0\DK, -1.2\DK-0.25\DK) -- node [above, very near start] {$\scriptscriptstyle+5$} +(0\DK, -0.6\DK);
\draw [<-,dashed] (0.0\DK+1.5\DK, -1.2\DK-4.25\DK) -- node [above, very near start] {$                    $} +(0\DK, -0.6\DK);
\draw [<-,dashed] (1.5\DK+1.5\DK, -0.7\DK-4.25\DK) -- node [above, very near end  ] {$                    $} +(0\DK, -0.6\DK);
\draw [<-,dashed] (3.0\DK+1.5\DK, -1.2\DK-4.25\DK) -- node [above, very near start] {$\scriptscriptstyle+5$} +(0\DK, -0.6\DK);
\draw [<-,dashed] (4.5\DK+1.5\DK, -0.7\DK-4.25\DK) -- node [above, very near end  ] {$                    $} +(0\DK, -0.6\DK);
\draw [<-,dashed] (6.0\DK+1.5\DK, -1.2\DK-4.25\DK) -- node [above, very near start] {$                    $} +(0\DK, -0.6\DK);
\draw [<-,dashed] (7.5\DK+1.5\DK, -0.7\DK-4.25\DK) -- node [above, very near end  ] {$\scriptscriptstyle+5$} +(0\DK, -0.6\DK);
\draw [->,dashed] (1.5\DK, -0.7\DK-9.25\DK) -- node [above ] {$                    $} +(0\DK, -0.6\DK);
\draw [->,dashed] (4.5\DK, -0.7\DK-9.25\DK) -- node [above ] {$                    $} +(0\DK, -0.6\DK);
\draw [->,dashed] (7.5\DK, -0.7\DK-9.25\DK) -- node [above ] {$\scriptscriptstyle+5$} +(0\DK, -0.6\DK);
\end{tikzpicture}%
\hspace{-10cm}\hfill%
\raisebox{1em}{\scalebox{0.59}
{\begin{tikzpicture}
\def\4{\ccolor{white}{$\boldsymbol{4}$}}
\def\2{\ccolor{white}{$\boldsymbol{2}$}}
\def\1{\ccolor{white}{$\boldsymbol{1}$}}
\def\0{\ccolor{white}{0}}
\def\O{\cocolor{white}{\ }}
\def\d{\cocolor{white}{\ddots}}
\def\t{\cocolor{white}{\cdots}}
\def\v{\cocolor{white}{\vdots}}
\def\o{ ++(1, 0) }
\draw (0,-0) \O\O\1\2\1 ;
\draw (0,-1) \O\O\1\2\1 ;
\draw (0,-2) \1\1\O\O\O\2\0 ;
\draw (0,-3) \1\1\O\O\O\1\1 ;
\draw (0,-4) \1\1\O\O\O\0\2 ;
\draw (0,-5) \o\o\1\1\0\O\O\1\1\0 ;
\draw (0,-6) \o\o\0\1\1\O\O\0\1\1 ;
\draw (0,-7) \o\o\o\o\o\2\0\O\O\O\2\0 ;
\draw (0,-8) \o\o\o\o\o\1\1\O\O\O\1\1 ;
\draw (0,-9) \o\o\o\o\o\0\2\O\O\O\0\2 ;
\draw (0,-10)\o\o\o\o\o\o\o\1\1\0\d\d\d\d ;
\draw (0,-11)\o\o\o\o\o\o\o\0\1\1\d\O\O\1\1\0 ;
\draw (0,-12)\o\o\o\o\o\o\o\o\o\o\d\O\O\0\1\1 ;
\draw (0,-13)\o\o\o\o\o\o\o\o\o\o\d\2\0\O\O\O\2\0 ;
\draw (0,-14)\o\o\o\o\o\o\o\o\o\o\o\1\1\O\O\O\1\1 ;
\draw (0,-15)\o\o\o\o\o\o\o\o\o\o\o\0\2\O\O\O\0\2 ;
\draw (0,-16)\o\o\o\o\o\o\o\o\o\o\o\o\o\2\2\0\O\O ;
\draw (0,-17)\o\o\o\o\o\o\o\o\o\o\o\o\o\0\2\2\O\O ;
\draw (0.5,0.5)rectangle(18.5,-17.5);
\end{tikzpicture}
}}
\caption{$n=7\eqref{JJ},12,17,22,27,\ldots$.\hfill Twins: $\cll1,\cll2$.\hfill Periods: $[-4,4]$, $\big[\frac{2n-4}5-2,\frac{2n-4}5+2\big]$.
         \\[0.3em]
         \mbox{}\hfill\mbox{}
         $\FfF\Lft x,y\Rgt =\FfF_{32}\Lft x,y\Rgt,\ |x+y|\le \frac{2n-4}5.$
         \mbox{}\hfill\mbox{}
         } \label{RR}
\end{figure}

\begin{figure}[H]
\def\0{ ++(1\DK, 0) }
\def\1{\dcolor{yellow!40!white}{$1$}}
\def\2{\dcolor{yellow!80!white}{$2$}}
\def\3{\Dcolor{green!30!yellow!40!white}{white}{$3$}}
\def\4{\Dcolor{green!30!yellow!70!white}{white}{$4$}}
\def\5{\Dcolor{green!30!yellow!99!white}{white}{$5$}}
\def\6{\Dcolor{green!80!yellow!40!white}{white}{$6$}}
\def\7{\Dcolor{green!80!yellow!80!white}{white}{$7$}}
\def\8{\Dcolor{green!60!cyan!30!white  }{white}{$8$}}
\def\9{\Dcolor{green!60!cyan!55!white  }{white}{$9$}}
\def\a{\Dcolor{green!60!cyan!80!white  }{white}{$10$}}
\def\h{\Dcolor{blue!80!cyan!10!white   }{white}{$g$}}
\def\i{\Dcolor{blue!80!cyan!30!white   }{white}{$h$}}
\def\j{\Dcolor{blue!80!cyan!50!white   }{white}{$i$}}
\def\k{\Dcolor{blue!80!red!20!white    }{white}{$j$}}
\def\l{\Dcolor{blue!80!red!40!white    }{white}{$k$}}
\def\m{\Dcolor{blue!30!red!15!white    }{white}{$l$}}
\def\n{\Dcolor{blue!30!red!30!white    }{white}{$m$}}
\def\o{\Dcolor{blue!30!red!45!white    }{white}{$n$}}
\def\t{ ++(1\DK, 0) node {$\cdots$}}
\begin{tikzpicture}[rotate=90]
\draw (0\DK, 1.5\DK)   \4\5\4\3\4\5\4\3\4\5;
\draw (0.5\DK, 1\DK)    \1\2\1\2\1\2\1\2\1;\drdgnla{1}
\draw (0\DK, 0.5\DK)   \4\3\4\5\4\3\4\5\4\3;
\draw (0.5\DK, 0\DK)    \6\6\7\7\6\6\7\7\6;
\draw (0\DK, -0.5\DK)  \9\8\9\a\9\8\9\a\9\8;
\draw (0\DK, -2.5\DK) \i\h\i\j\i\h\i\j\i\h;
\draw (0.5\DK, -3\DK)  \k\k\l\l\k\k\l\l\k;
\draw (0\DK, -3.5\DK) \n\m\n\o\n\m\n\o\n\m;
\draw (0.5\DK, -4\DK)  \k\k\l\l\k\k\l\l\k;
\draw (0\DK, -4.5\DK) \i\h\i\j\i\h\i\j\i\h;
\draw (0\DK, -6.5\DK)   \9\8\9\a\9\8\9\a\9\8;
\draw (0.5\DK, -7\DK)    \6\6\7\7\6\6\7\7\6;
\draw (0\DK, -7.5\DK)   \4\3\4\5\4\3\4\5\4\3;
\draw (0.5\DK, -8\DK)    \1\2\1\2\1\2\1\2\1;\drdgnla{-8}
\draw (0\DK,-8.5\DK)    \4\5\4\3\4\5\4\3\4\5;
\draw (0.5\DK, -9\DK)    \7\7\6\6\7\7\6\6\7;
\draw (0\DK, -9.5\DK)   \9\a\9\8\9\a\9\8\9\a;
\draw [<-,dashed] (1.5\DK, -1\DK+3.50\DK) -- node [above] {$                    $} +(0\DK, -0.6\DK);
\draw [<-,dashed] (3.0\DK, -1\DK+3.50\DK) -- node [above] {$                    $} +(0\DK, -0.6\DK);
\draw [<-,dashed] (4.5\DK, -1\DK+3.50\DK) -- node [above] {$                    $} +(0\DK, -0.6\DK);
\draw [<-,dashed] (6.0\DK, -1\DK+3.50\DK) -- node [above] {$                    $} +(0\DK, -0.6\DK);
\draw [<-,dashed] (7.5\DK, -1\DK+3.50\DK) -- node [above] {$                    $} +(0\DK, -0.6\DK);
\draw [<-,dashed] (9.0\DK, -1\DK+3.50\DK) -- node [above] {$                    $} +(0\DK, -0.6\DK);
\draw [->,dashed] (1.5\DK, -1\DK-0.25\DK) -- node [above] {$                    $} +(0\DK, -0.6\DK);
\draw [->,dashed] (3.0\DK, -1\DK-0.25\DK) -- node [above] {$                    $} +(0\DK, -0.6\DK);
\draw [->,dashed] (4.5\DK, -1\DK-0.25\DK) -- node [above] {$\scriptscriptstyle+5$} +(0\DK, -0.6\DK);
\draw [->,dashed] (6.0\DK, -1\DK-0.25\DK) -- node [above] {$                    $} +(0\DK, -0.6\DK);
\draw [->,dashed] (7.5\DK, -1\DK-0.25\DK) -- node [above] {$                    $} +(0\DK, -0.6\DK);
\draw [->,dashed] (9.0\DK, -1\DK-0.25\DK) -- node [above] {$\scriptscriptstyle+5$} +(0\DK, -0.6\DK);
\draw [<-,dashed] (1.5\DK, -1\DK-4.25\DK) -- node [above] {$                    $} +(0\DK, -0.6\DK);
\draw [<-,dashed] (3.0\DK, -1\DK-4.25\DK) -- node [above] {$                    $} +(0\DK, -0.6\DK);
\draw [<-,dashed] (4.5\DK, -1\DK-4.25\DK) -- node [above] {$\scriptscriptstyle+5$} +(0\DK, -0.6\DK);
\draw [<-,dashed] (6.0\DK, -1\DK-4.25\DK) -- node [above] {$                    $} +(0\DK, -0.6\DK);
\draw [<-,dashed] (7.5\DK, -1\DK-4.25\DK) -- node [above] {$                    $} +(0\DK, -0.6\DK);
\draw [<-,dashed] (9.0\DK, -1\DK-4.25\DK) -- node [above] {$\scriptscriptstyle+5$} +(0\DK, -0.6\DK);
\draw [->,dashed] (1.5\DK, -1\DK-9.25\DK) -- node [above] {$                    $} +(0\DK, -0.6\DK);
\draw [->,dashed] (3.0\DK, -1\DK-9.25\DK) -- node [above] {$                    $} +(0\DK, -0.6\DK);
\draw [->,dashed] (4.5\DK, -1\DK-9.25\DK) -- node [above] {$                    $} +(0\DK, -0.6\DK);
\draw [->,dashed] (6.0\DK, -1\DK-9.25\DK) -- node [above] {$                    $} +(0\DK, -0.6\DK);
\draw [->,dashed] (7.5\DK, -1\DK-9.25\DK) -- node [above] {$                    $} +(0\DK, -0.6\DK);
\draw [->,dashed] (9.0\DK, -1\DK-9.25\DK) -- node [above] {$                    $} +(0\DK, -0.6\DK);
\end{tikzpicture}%
\hspace{-10cm}\hfill%
\raisebox{2em}{\scalebox{0.55}
{\begin{tikzpicture}
\def\4{\ccolor{white}{$\boldsymbol{4}$}}
\def\2{\ccolor{white}{$\boldsymbol{2}$}}
\def\1{\ccolor{white}{$\boldsymbol{1}$}}
\def\0{\ccolor{white}{0}}
\def\O{\cocolor{white}{\ }}
\def\d{\cocolor{white}{\ddots}}
\def\t{\cocolor{white}{\cdots}}
\def\v{\cocolor{white}{\vdots}}
\def\o{ ++(1, 0) }
\fill [fill=white] (0.25,0.55)rectangle(18.5,-18.5);
\draw (0,-0) \O\O\1\2\1 ;
\draw (0,-1) \O\O\1\2\1 ;
\draw (0,-2) \1\1\O\O\O\2\0 ;
\draw (0,-3) \1\1\O\O\O\1\1 ;
\draw (0,-4) \1\1\O\O\O\0\2 ;
\draw (0,-5) \o\o\1\1\0\O\O\1\1\0 ;
\draw (0,-6) \o\o\0\1\1\O\O\0\1\1 ;
\draw (0,-7) \o\o\o\o\o\2\0\O\O\O\2\0 ;
\draw (0,-8) \o\o\o\o\o\1\1\O\O\O\1\1 ;
\draw (0,-9) \o\o\o\o\o\0\2\O\O\O\0\2 ;
\draw (0,-10)\o\o\o\o\o\o\o\1\1\0\d\d\d\d ;
\draw (0,-11)\o\o\o\o\o\o\o\0\1\1\d\O\O\O\2\0 ;
\draw (0,-12)\o\o\o\o\o\o\o\o\o\o\d\O\O\O\1\1 ;
\draw (0,-13)\o\o\o\o\o\o\o\o\o\o\d\O\O\O\0\2 ;
\draw (0,-14)\o\o\o\o\o\o\o\o\o\o\o\1\1\0\O\O\1\1\0 ;
\draw (0,-15)\o\o\o\o\o\o\o\o\o\o\o\0\1\1\O\O\0\1\1 ;
\draw (0,-16)\o\o\o\o\o\o\o\o\o\o\o\o\o\o\4\0\O\O\O ;
\draw (0,-17)\o\o\o\o\o\o\o\o\o\o\o\o\o\o\2\2\O\O\O ;
\draw (0,-18)\o\o\o\o\o\o\o\o\o\o\o\o\o\o\0\4\O\O\O ;
\draw (0.5,0.5)rectangle(19.5,-18.5);
\end{tikzpicture}
}}
\caption{$n=5\eqref{II},10,15,20,25,\ldots$.\hfill Twins: $\cll1,\cll2$.\hfill Periods: $[-4,4]$, $\big[\frac{2n-5}5-2,\frac{2n-5}5+2\big]$.
         \\[0.3em]
         \mbox{}\hfill\mbox{}
         $\FfF\Lft x,y\Rgt =\FfF_{32}\Lft x,y\Rgt,\ |x+y|\le \frac{2n-5}5.$
         \mbox{}\hfill\mbox{}
         } \label{SS}
\end{figure}

\def\S{\cocolor{white}{\ }}
\def\o{\ccolor{white}{$\ $}}
\def\1{\ccolor{yellow!40!white}{$1$}}
\def\2{\ccolor{black!5!yellow!60!white}{$2$}}
\def\3{\Ccolor{green!10!white}{green!05!white}{$3$}}
\def\4{\Ccolor{green!30!white}{green!10!white}{$4$}}
\def\5{\Ccolor{green!50!white}{green!15!white}{$5$}}
\def\6{\Ccolor{green!70!white}{green!20!white}{$6$}}
\def\7{\Ccolor{cyan!10!white }{cyan!05!white }{$7$}}
\def\8{\Ccolor{cyan!30!white }{cyan!10!white }{$8$}}
\def\9{\Ccolor{cyan!50!white }{cyan!15!white }{$9$}}
\def\a{\Ccolor{cyan!70!white }{cyan!20!white }{$10$}}
\def\T{ ++(1\DK, 0) node {$\vdots$}}
\begin{figure}[H]
 \hskip-7pt
\begin{tikzpicture}[yscale=-1] 
\def\g{\Ccolor{blue!10!white  }{blue!03!white  }{$e$}}
\def\h{\Ccolor{blue!25!white  }{blue!06!white  }{$f$}}
\def\i{\Ccolor{blue!40!white  }{blue!09!white  }{$g$}}
\def\j{\Ccolor{blue!55!white  }{blue!12!white  }{$h$}}
\def\k{\Ccolor{purple!10!white}{purple!04!white}{$i$}}
\def\l{\Ccolor{purple!25!white}{purple!08!white}{$j$}}
\def\m{\Ccolor{purple!40!white}{purple!12!white}{$k$}}
\def\n{\Ccolor{purple!55!white}{purple!16!white}{$l$}}
\def\p{\ccolor{orange!30!white}{$m$}}
\def\q{\ccolor{yellow!30!orange!50!white}{$n$}}
\begin{scope}
\clip [](9.2,-8.5)-- ++(-5.7,0) -- ++(0,5.7)--cycle;
\draw (3,-8.5)   \g\k\p\n\g\o ;
\draw (3,-7.5)   \n\q\m\j\o ;
\draw (3,-6.5)   \p\l\i\o ;
\draw (3,-5.5)   \k\h\o ;
\draw (3,-4.5)   \g\o ;
\draw (3,-3.5)   \o  (3,-2.5)   \o ;
\end{scope}
\begin{scope}
\clip(20.8,3.5)-- ++(1.2,0)-- ++(0,-1.2)--cycle;\draw(20,2.5)\o(21,3.5)\j;
\end{scope}
\begin{scope}
\clip [](16.7,-8.5)-- ++(-6.4,0)-- ++(-6.8,6.8)-- ++(0,5.2)-- ++(1.2,0)--cycle;
\draw (9,-8)   \o\a\3\2\6\a\o\o ;
\draw (8,-7)   \o\9\6\1\5\9\o ;
\draw (7,-6)   \o\8\5\2\4\8\o ;
\draw (6,-5)   \o\7\4\1\3\7\o ;
\draw (5,-4)   \o\a\3\2\6\a\o ;
\draw (4,-3)   \o\9\6\1\5\9\o ;
\draw (3,-2)   \o\8\5\2\4\8\o ;
\draw (3,-1)   \7\4\1\3\7\o ;
\draw (3, 0)   \3\2\6\a\o ;
\draw (3, 1)   \1\5\9\o ;
\draw (3, 2)   \4\8\o ;
\draw (3, 3)   \7\o ;
\end{scope}
\begin{scope}
\clip [](5.8,3.5)-- ++(6.4,0)-- ++(12,-12)-- ++(-6.4,0)--cycle;
\draw (16.5,-8)   \o\j\n\q\m;
\draw (15.5,-7)   \o\i\m\p\l\i ;
\draw (14.5,-6)   \o\h\l\q\k\h\o ;
\draw (13.5,-5)   \o\g\k\p\n\g\o ;
\draw (12.5,-4)   \o\j\n\q\m\j\o ;
\draw (11.5,-3)   \o\i\m\p\l\i\o ;
\draw (10.5,-2)   \o\h\l\q\k\h\o ;
\draw ( 9.5,-1)   \o\g\k\p\n\g\o ;
\draw ( 8.5, 0)   \o\j\n\q\m\j\o ;
\draw ( 7.5, 1)   \o\i\m\p\l\i\o ;
\draw ( 6.5, 2)   \o\h\l\q\k\h\o ;
\draw ( 5.5, 3)   \o\g\k\p\n\g\o ;
\end{scope}
\begin{scope}
\clip []( 13.3, 3.5) -- ++(6.5,0)-- ++(2.5,-2.5)-- ++(0,-6.5)--cycle;
\draw ( 20.5,-5.5) \o                 ;
\draw ( 20.5,-4.5) \o                 ;
\draw ( 19.5,-3.5) \o               \9;
\draw ( 18.5,-2.5) \o             \8\5;
\draw ( 17.5,-1.5) \o           \7\4\1;
\draw ( 16.5,-0.5) \o         \a\3\2\6;
\draw ( 15.5, 0.5) \o       \9\6\1\5\9;
\draw ( 14.5, 1.5) \o     \8\5\2\4\8\o ;
\draw ( 13.5, 2.5) \o   \7\4\1\3\7\o ;
\draw ( 12.5, 3.5) \o \a\3\2\6\a\o  ;
\end{scope}
\draw [diag, dash phase=0.5\DK] 
(3,-5.5) -- +(3,-3)
(3,2) -- +(11,-11) ++(5.5,2) -- +(13,-13)
(16.5,3.5) -- +(6,-6);
\draw [->, dashed] (9,-7 ) -- node [left,near start] {{$\scriptscriptstyle +4$}}  +(0,-1.5);
\draw [->, dashed] (5,-3 ) -- node [left,near start] {{$\scriptscriptstyle +4$}}  +(0,-1.5);
\draw [->, dashed] (14,-6 ) -- node [below,near start] {{$\scriptscriptstyle +4$}}  +(1.5,0);
\draw [->, dashed] (10,-2 ) -- node [below,near start] {{$\scriptscriptstyle +4$}}  +(1.5,0);
\draw [->, dashed] ( 6, 2 ) -- node [below,near start] {{$\scriptscriptstyle +4$}}  +(1.5,0);
\draw [->, dashed] (21.4, 1.5 ) --  +(.5,0);
\draw [->, dashed] (12.5, 3.8 ) --  +(0,-0.85);
\draw [->, dashed] (16.5, 0.5 ) -- node [left,near start] {{$\scriptscriptstyle +4$}}  +(0,-1.5);
\draw [->, dashed] (20.5,-3.5 ) --  node [left,near start] {{$\scriptscriptstyle +4$}}   +(0,-1.5);
\end{tikzpicture}%
\hspace{-5cm}\hfill%
\raisebox{1em}{\scalebox{0.62}
{\begin{tikzpicture}
\def\4{\ccolor{white}{$\boldsymbol{4}$}}
\def\2{\ccolor{white}{$\boldsymbol{2}$}}
\def\1{\ccolor{white}{$\boldsymbol{1}$}}
\def\0{\ccolor{white}{0}}
\def\O{\cocolor{white}{\ }}
\def\d{\cocolor{white}{\ddots}}
\def\t{\cocolor{white}{\cdots}}
\def\v{\cocolor{white}{\vdots}}
\def\o{ ++(1, 0) }
\fill[fill=white] (0.25,0.75)rectangle(16.75,-17);
\draw (0, -0) \O\O\1\1\1\1 ;
\draw (0, -1) \O\O\1\1\1\1 ;
\draw (0, -2) \1\1\O\O\O\O\1\0\0\1 ;
\draw (0, -3) \1\1\O\O\O\O\1\1\0\0 ;
\draw (0, -4) \1\1\O\O\O\O\0\1\1\0 ;
\draw (0, -5) \1\1\O\O\O\O\0\0\1\1 ;
\draw (0, -6) \o\o\1\1\0\0\O\O\O\O\1\0\0\1 ;
\draw (0, -7) \o\o\0\1\1\0\O\O\O\O\1\1\0\0 ;
\draw (0, -8) \o\o\0\0\1\1\O\O\O\O\0\1\1\0 ;
\draw (0, -9) \o\o\1\0\0\1\O\O\O\O\0\0\1\1 ;
\draw (0,-10) \o\o\o\o\o\o\d\d\d\d\d\d\d\d\d;
\draw (0,-11) \o\o\o\o\o\o\o\1\1\0\0\O\O\O\O\1\1 ;
\draw (0,-12) \o\o\o\o\o\o\o\0\1\1\0\O\O\O\O\1\1 ;
\draw (0,-13) \o\o\o\o\o\o\o\0\0\1\1\O\O\O\O\1\1 ;
\draw (0,-14) \o\o\o\o\o\o\o\1\0\0\1\O\O\O\O\1\1 ;
\draw (0,-15) \o\o\o\o\o\o\o\o\o\o\o\1\1\1\1\O\O ;
\draw (0,-16) \o\o\o\o\o\o\o\o\o\o\o\1\1\1\1\O\O ;
\draw (0.5,0.5)rectangle(17.5,-16.5);
\end{tikzpicture}
}} 
\caption{$n=4\eqref{II},8\eqref{II},12,16,20,\ldots$.\hfill
         Twins: $\cll1,\cll2$; $\ccll{m},\cll{n}$.\hfill
         Periods: $[4,4]$, $\big[\frac{n}4,-\frac{n}{4}\big]$.\\[0.3em]
         \mbox{}\hfill\mbox{}
         $\FfF\Lft x,y\Rgt =\FfF_{44}\Lft x,y\Rgt,\ |x-y|<\frac{n}4,$
         \mbox{}\hfill\mbox{}
         \\[0.2em]
         \mbox{}\hfill\mbox{}
         $\FfF\big([ \frac{n}4{+}2z{+}\beta,\
                2z{+}\beta ]\big)
         =\framebox{$n{-}1{+}\beta$},
         \
         z\in\ZZ,\beta\in\{0,1\}.$
         \mbox{}\hfill\mbox{}
         } \label{TT}
\end{figure}

\def\Rt#1{ ++(1\DK, 0) node {$\stackrel{\scriptscriptstyle+#1}\rightarrow$}}
\def\Lt#1{ ++(1\DK, 0) node {$\stackrel{\scriptscriptstyle+#1}\leftarrow$}}
\begin{figure}[H]
 \hskip-11pt
\makebox[1em][l]{
\begin{tikzpicture}[yscale=-1] 
\def\gg{\Ccolor{blue!10!white  }{blue!03!white  }{\ }}
\def\nn{\Ccolor{purple!55!white}{purple!16!white}{\ }}
\def\qq{\Ccolor{orange!60!white}{orange!20!white}{\ }}
\def\mm{\Ccolor{purple!40!white}{purple!12!white}{\ }}
\def\jj{\Ccolor{blue!55!white  }{blue!12!white  }{\ }}
\def\g{\Ccolor{blue!10!white  }{blue!03!white  }{$e$}}
\def\h{\Ccolor{blue!25!white  }{blue!06!white  }{$f$}}
\def\i{\Ccolor{blue!40!white  }{blue!09!white  }{$g$}}
\def\j{\Ccolor{blue!55!white  }{blue!12!white  }{$h$}}
\def\k{\Ccolor{purple!10!white}{purple!04!white}{$i$}}
\def\l{\Ccolor{purple!25!white}{purple!08!white}{$j$}}
\def\m{\Ccolor{purple!40!white}{purple!12!white}{$k$}}
\def\n{\Ccolor{purple!55!white}{purple!16!white}{$l$}}
\def\p{\Ccolor{orange!30!white}{orange!10!white}{$m$}}
\def\q{\Ccolor{orange!60!white}{orange!20!white}{$n$}}
\begin{scope}
\clip [](20.8,3.5)-- ++(3.0,0)-- ++(0.0,-3.0)-- cycle;
\draw (22, 0.5) \o        ;
\draw (22, 1.5) \o     ;
\draw (21, 2.5) \o   \j;
\draw (20, 3.5) \o \g\n;
\end{scope}
\begin{scope}
\clip [](9.2,-8.5)-- ++(-5.8,0)--  ++(0,5.8)--cycle;
\draw (3,-8.5)   \j\m\q\n\g\o ; 
\draw (3,-7.5)   \n\q\m\j\o ;
\draw (3,-6.5)   \p\l\i\o ;
\draw (3,-5.5)   \k\h\o ;
\draw (3,-4.5)   \g\o ;
\draw (3,-3.5)   \o ;
\end{scope}
\begin{scope}
\clip [](16.7,-8.5)-- ++(-6.4,0)-- ++(-7,7)-- ++(0,5.4) -- ++(1,0)--cycle;
\draw (3 ,3)           \7\o ;
\draw (3 ,2)         \4\8\o ;
\draw (3, 1)       \1\5\9\o ;
\draw (3, 0)     \3\2\6\a\o ;
\draw (3,-1)   \7\4\1\3\7\o ;
\draw (3,-2)   \o\8\5\2\4\8\o ;
\draw (4,-3)   \o\9\6\1\5\9\o ;
\draw (5,-4)   \o\a\3\2\6\a\o ;
\draw (6,-5)   \o\7\4\1\3\7\o ;
\draw (7,-6)   \o\8\5\2\4\8\o ;
\draw (8,-7)   \o\9\6\1\5\9\o ;
\draw (9,-8)   \o\a\3\2\6\a\o ;
\end{scope}
\begin{scope}
\clip [](5.8,3.5)-- ++(6.4,0)-- ++(11.3,-11.3)-- ++(0,-0.7)-- ++(-5.7,0)--cycle;
\draw ( 5.5, 3)   \o\g\k\p\n\j\o ;
\draw ( 6.5, 2)   \o\h\l\q\m\i\o ;
\draw ( 7.5, 1)   \o\i\m\q\l\h\o ;
\draw ( 8.5, 0)   \o\j\n\p\k\g\o ;
\draw ( 9.5,-1)   \o\g\k\p\n\j\o ;
\draw (10.5,-2)   \o\h\l\q\m\i\o ;
\draw (11.5,-3)   \o\i\m\q\l\h\o ;
\draw (12.5,-4)   \o\j\n\p\k\g\o ;
\draw (13.5,-5)   \o\g\k\p\n\j\o ;
\draw (14.5,-6)   \o\h\l\q\m\i\o ;
\draw (15.5,-7)   \o\i\m\q\l\h\o ;
\draw (16.5,-8)   \o\j\n\p\k\g \o ;
\end{scope}
\begin{scope}
\clip [](13.3,3.5)-- ++(6.4,0)-- ++(3.9,-3.9)-- ++(0,-6.4)--cycle;
\draw ( 22,-7) \o                     ;
\draw ( 22,-6) \o                   ;
\draw ( 21,-5) \o                 \a;
\draw ( 20,-4) \o               \7\3;
\draw ( 19,-3) \o             \8\4\2;
\draw ( 18,-2) \o           \9\5\1\6;
\draw ( 17,-1) \o         \a\6\2\3\a;
\draw ( 16, 0) \o       \7\3\1\4\7\o ;
\draw ( 15, 1) \o     \8\4\2\5\8\o ;
\draw ( 14, 2) \o   \9\5\1\6\9\o ;
\draw ( 13, 3) \o \a\6\2\3\a\o ;
\end{scope}
\draw [diag, dash phase=0.5\DK] (3,2) -- +(11,-11);
\draw [diag, dash phase=0.5\DK] (16,4) -- +(8,-8);
\draw [->, dashed] (9,-7 ) -- node [left,near start] {{$\scriptscriptstyle +4$}}  +(0,-1.5);
\draw [->, dashed] (5,-3 ) -- node [left,near start] {{$\scriptscriptstyle +4$}}  +(0,-1.5);
\draw [->, dashed] (14,-6 ) -- node [below,near start] {{$\scriptscriptstyle +4$}}  +(1.5,0);
\draw [->, dashed] (10,-2 ) -- node [below,near start] {{$\scriptscriptstyle +4$}}  +(1.5,0);
\draw [->, dashed] ( 6, 2 ) -- node [below,near start] {{$\scriptscriptstyle +4$}}  +(1.5,0);
\draw [<-, dashed] (14.5, 1 ) -- node [above,near end] {{$\scriptscriptstyle +4$}}  +(1.5,0);
\draw [<-, dashed] (18.5,-3 ) -- node [above,near end] {{$\scriptscriptstyle +4$}}  +(1.5,0);
\draw [<-, dashed] (22.5,-7 ) -- node [above,near end] {{$ $}}  +(1.0,0);
\draw [->, dashed] (21,2 ) -- node [right,near start] {{$\scriptscriptstyle +4$}}  +(0,1.5);
\end{tikzpicture}%
}
\hspace{-5cm}\hfill%
\raisebox{1.em}{\scalebox{0.545}
{\begin{tikzpicture}
\def\4{\ccolor{white}{$\boldsymbol{4}$}}
\def\2{\ccolor{white}{$\boldsymbol{2}$}}
\def\1{\ccolor{white}{$\boldsymbol{1}$}}
\def\0{\ccolor{white}{0}}
\def\O{\cocolor{white}{\ }}
\def\d{\cocolor{white}{\ddots}}
\def\t{\cocolor{white}{\cdots}}
\def\v{\cocolor{white}{\vdots}}
\def\o{ ++(1, 0) }
\fill[fill=white] (0.2,0.8)rectangle(17.5,-17.1);
\draw (0, -0) \O\O\1\1\1\1 ;
\draw (0, -1) \O\O\1\1\1\1 ;
\draw (0, -2) \1\1\O\O\O\O\1\0\0\1 ;
\draw (0, -3) \1\1\O\O\O\O\1\1\0\0 ;
\draw (0, -4) \1\1\O\O\O\O\0\1\1\0 ;
\draw (0, -5) \1\1\O\O\O\O\0\0\1\1 ;
\draw (0, -6) \o\o\1\1\0\0\O\O\O\O\1\0\0\1 ;
\draw (0, -7) \o\o\0\1\1\0\O\O\O\O\1\1\0\0 ;
\draw (0, -8) \o\o\0\0\1\1\O\O\O\O\0\1\1\0 ;
\draw (0, -9) \o\o\1\0\0\1\O\O\O\O\0\0\1\1 ;
\draw (0,-10) \o\o\o\o\o\o\d\d\d\d\d\d\d\d\d;
\draw (0,-11) \o\o\o\o\o\o\o\1\1\0\0\O\O\O\O\2\0 ;
\draw (0,-12) \o\o\o\o\o\o\o\0\1\1\0\O\O\O\O\1\1 ;
\draw (0,-13) \o\o\o\o\o\o\o\0\0\1\1\O\O\O\O\0\2 ;
\draw (0,-14) \o\o\o\o\o\o\o\1\0\0\1\O\O\O\O\1\1 ;
\draw (0,-15) \o\o\o\o\o\o\o\o\o\o\o\2\1\0\1\O\O ;
\draw (0,-16) \o\o\o\o\o\o\o\o\o\o\o\0\1\2\1\O\O ;
\draw (0.5,0.5)rectangle(17.5,-16.5);
\end{tikzpicture}
}} 
\caption{$n=4\eqref{II},8\eqref{JJ},12,16,20,\ldots$.\hfill Twins: $\cll1,\cll2$.\hfill Periods: $[4,4]$, $\big[\frac{n}2,-\frac{n}{2}\big]$.\\[0.3ex]
         \mbox{}\hfill
         $\FfF\Lft x,y\Rgt =\FfF\Lft \frac{n}2{+}3{-}x,3{-}y\Rgt =\FfF_{44}\Lft x,y\Rgt,\ |x-y|< \frac{n}4,$
         \hfill\mbox{}
         \\[0.3ex]
         \mbox{}\hfill
         \mbox{%
         $
         \FfF\big([  \frac{n}{4}{+}4z{+}2{+}\gamma,4z{+}2{+}\gamma ]\big)
         =
         \FfF\big([  4z{+}\gamma,\frac{n}{4}{+}4z{+}\gamma ]\big)
         =
         \framebox{$n{-}\frac32{+}|\gamma{-}\frac32|$},\ z\in\ZZ, \gamma\in\{0,1,2,3\}.
         $%
         }
         \hfill\mbox{}
        } \label{UU}
\end{figure}

\begin{figure}[H]
  \hskip-11pt
\makebox[1em][l]{
\begin{tikzpicture}[yscale=-1] 
\def\g{\Ccolor{blue!10!white  }{blue!03!white  }{$d$}}
\def\h{\Ccolor{blue!25!white  }{blue!06!white  }{$e$}}
\def\i{\Ccolor{blue!40!white  }{blue!09!white  }{$f$}}
\def\j{\Ccolor{blue!55!white  }{blue!12!white  }{$g$}}
\def\k{\Ccolor{purple!10!white}{purple!04!white}{$h$}}
\def\l{\Ccolor{purple!25!white}{purple!08!white}{$i$}}
\def\m{\Ccolor{purple!40!white}{purple!12!white}{$j$}}
\def\n{\Ccolor{purple!55!white}{purple!16!white}{$k$}}
\def\0{\Ccolor{orange!20!white}{orange!10!white}{$l$}}
\def\p{\Ccolor{orange!50!white}{orange!20!white}{$m$}}
\def\q{\Ccolor{orange!80!white}{orange!30!white}{$n$}}
\begin{scope}
\clip [](9.2,-8.5)-- ++(-4.9,0) -- ++(0,4.9)--cycle;
\draw (4,-8.5)   \n\q\n\g\o ;
\draw (4,-7.5)   \p\m\j\o ;
\draw (4,-6.5)   \l\i\o ;
\draw (4,-5.5)   \h\o ;
\draw (4,-4.5)   \o ;
\draw (4,-3.5)   \o ;
\end{scope}
\begin{scope}
\clip [](16.7,-8.5)-- ++(-6.4,0)-- ++(-6,6)-- ++(0,6.1) -- ++(0.3,0)--cycle;
\draw ( 4,3)             \o ;
\draw ( 4,2)           \8\o ;
\draw (4, 1)         \5\9\o ;
\draw (4, 0)       \2\6\a\o ;
\draw (4,-1)     \4\1\3\7\o ;
\draw (4,-2)   \8\5\2\4\8\o ;
\draw (4,-3)   \o\9\6\1\5\9\o ;
\draw (5,-4)   \o\a\3\2\6\a\o ;
\draw (6,-5)   \o\7\4\1\3\7\o ;
\draw (7,-6)   \o\8\5\2\4\8\o ;
\draw (8,-7)   \o\9\6\1\5\9\o ;
\draw (9,-8)   \o\a\3\2\6\a\o ;
\end{scope}
\begin{scope}
\clip [](5.8,3.5)-- ++(6.4,0)-- ++(11.3,-11.3)-- ++(0,-0.7)-- ++(-5.7,0)--cycle;
\draw ( 5.5, 3)   \o\g\k\0\k\g\o ;
\draw ( 6.5, 2)   \o\h\l\p\n\j\o ;
\draw ( 7.5, 1)   \o\i\m\q\m\i\o ;
\draw ( 8.5, 0)   \o\j\n\p\l\h\o ;
\draw ( 9.5,-1)   \o\g\k\0\k\g\o ;
\draw (10.5,-2)   \o\h\l\p\n\j\o ;
\draw (11.5,-3)   \o\i\m\q\m\i\o ;
\draw (12.5,-4)   \o\j\n\p\l\h\o ;
\draw (13.5,-5)   \o\g\k\0\k\g\o ;
\draw (14.5,-6)   \o\h\l\p\n\j\o ;
\draw (15.5,-7)   \o\i\m\q\m\i\o ;
\draw (16.5,-8)   \o\j\n\p\l\h\o ;
\end{scope}
\begin{scope}
\clip [](13.3,3.5)-- ++(6.4,0)-- ++(13,-13)-- ++(-6.4,0)--cycle;
\draw ( 21,-5) \o                  \7;
\draw ( 20,-4) \o                \8\4;
\draw ( 19,-3) \o              \9\5\1;
\draw ( 18,-2) \o            \a\6\2\3;
\draw ( 17,-1) \o          \7\3\1\4\7;
\draw ( 16, 0) \o        \8\4\2\5\8\o ;
\draw ( 15, 1) \o      \9\5\1\6\9\o ;
\draw ( 14, 2) \o    \a\6\2\3\a\o ;
\draw ( 13, 3) \o  \7\3\1\4\7\o ;
\end{scope}
\begin{scope}
\clip [](20.8,3.5)-- ++(6.0,0)-- ++(0.0,-6.0)-- cycle;
\draw (22, 1.5) \o     ;
\draw (21, 2.5) \o   \g;
\draw (20, 3.5) \o \h\k;
\end{scope}
\draw [diag, dash phase=0.5\DK] (4,1) -- +(10,-10);
\draw [diag, dash phase=0.5\DK] (16,4) -- +(8,-8);
\draw [->, dashed] (6,-4 ) -- node [left,near start] {{$\scriptscriptstyle +4$}}  +(0,-1.5);
\draw [->, dashed] (10,-8 ) --  +(0,-.8);
\draw [->, dashed] (15,-7 ) -- node [below,near start] {{$\scriptscriptstyle +4$}}  +(1.5,0);
\draw [->, dashed] (11,-3 ) -- node [below,near start] {{$\scriptscriptstyle +4$}}  +(1.5,0);
\draw [->, dashed] ( 7, 1 ) -- node [below,near start] {{$\scriptscriptstyle +4$}}  +(1.5,0);
\draw [<-, dashed] (14.5, 1 ) -- node [above,near end] {{$\scriptscriptstyle +4$}}  +(1.5,0);
\draw [<-, dashed] (18.5,-3 ) -- node [above,near end] {{$\scriptscriptstyle +4$}}  +(1.5,0);
\draw [<-, dashed] (22.5,-7 ) --  +(1.,0);
\draw [->, dashed] (21,2 ) -- node [right,near start] {{$\scriptscriptstyle +4$}}  +(0,1.5);
\end{tikzpicture}
}
\hfill%
\raisebox{1.5em}{\scalebox{0.55}
{\begin{tikzpicture}
\def\4{\ccolor{white}{$\boldsymbol{4}$}}
\def\2{\ccolor{white}{$\boldsymbol{2}$}}
\def\1{\ccolor{white}{$\boldsymbol{1}$}}
\def\0{\ccolor{white}{0}}
\def\O{\cocolor{white}{\ }}
\def\d{\cocolor{white}{\ddots}}
\def\t{\cocolor{white}{\cdots}}
\def\v{\cocolor{white}{\vdots}}
\def\o{ ++(1, 0) }
\fill[fill=white] (0.25,0.75)rectangle(18.0,-18.0);
\draw (0,-0 ) \O\O\1\1\1\1 ;
\draw (0,-1 ) \O\O\1\1\1\1 ;
\draw (0,-2 ) \1\1\O\O\O\O\1\0\0\1 ;
\draw (0,-3 ) \1\1\O\O\O\O\1\1\0\0 ;
\draw (0,-4 ) \1\1\O\O\O\O\0\1\1\0 ;
\draw (0,-5 ) \1\1\O\O\O\O\0\0\1\1 ;
\draw (0,-6 ) \o\o\1\1\0\0\O\O\O\O\1\0\0\1 ;
\draw (0,-7 ) \o\o\0\1\1\0\O\O\O\O\1\1\0\0 ;
\draw (0,-8 ) \o\o\0\0\1\1\O\O\O\O\0\1\1\0 ;
\draw (0,-9 ) \o\o\1\0\0\1\O\O\O\O\0\0\1\1 ;
\draw (0,-10) \o\o\o\o\o\o\d\d\d\d\d\d\d\d\d;
\draw (0,-11) \o\o\o\o\o\o\o\1\1\0\0\O\O\O\O\1\1\0 ;
\draw (0,-12) \o\o\o\o\o\o\o\0\1\1\0\O\O\O\O\1\1\0 ;
\draw (0,-13) \o\o\o\o\o\o\o\0\0\1\1\O\O\O\O\0\1\1 ;
\draw (0,-14) \o\o\o\o\o\o\o\1\0\0\1\O\O\O\O\0\1\1 ;
\draw (0,-15) \o\o\o\o\o\o\o\o\o\o\o\2\2\0\0\O\O\O ;
\draw (0,-16) \o\o\o\o\o\o\o\o\o\o\o\1\1\1\1\O\O\O ;
\draw (0,-17) \o\o\o\o\o\o\o\o\o\o\o\0\0\2\2\O\O\O ;
\draw (0.5,0.5)rectangle(18.5,-17.5);
\end{tikzpicture}
}} 
\caption{$n=5\eqref{II},9\eqref{AA},13,17,21,\ldots$.\hfill Twins: $\cll1,\cll2$.\hfill Periods: $[4,4]$, $\big[\frac {n-1}2,-\frac {n-1}2\big]$.\\[0.3em]\mbox{} \hfill
         $\FfF\Lft x,y\Rgt =\FfF\Lft \frac{n-1}2{-}x,-y\Rgt =\FfF_{44}\Lft x,y\Rgt,\ |x-y|< \frac{n-1}4,$
         \hfill \mbox{} \\[0.3em]
         \mbox{} \hfill
         $
         \FfF \big([ \frac{n-1}4{+}4z{+}\gamma, 4z{+}\gamma ]\big) =
         \FfF \big([ 4z{+}2{+}\gamma, \frac{n-1}4{+}4z{+}2{+}\gamma ]\big) =
         \framebox{$n{-}|2-\gamma|$}$,
         \hfill\mbox{}
         \\[0.3em] \mbox{}
         \hfill
         $
         z\in\ZZ, \gamma\in\{0,1,2,3\}.
         $
         \hfill \mbox{}
         } \label{VV}
\end{figure}

\begin{figure}[H]
\hskip-11pt
\makebox[1em][l]{
\begin{tikzpicture}[yscale=-1] 
\def\g{\Ccolor{blue!10!white  }{blue!03!white  }{$c$}}
\def\h{\Ccolor{blue!25!white  }{blue!06!white  }{$d$}}
\def\i{\Ccolor{blue!40!white  }{blue!09!white  }{$e$}}
\def\j{\Ccolor{blue!55!white  }{blue!12!white  }{$f$}}
\def\k{\Ccolor{purple!10!white}{purple!04!white}{$g$}}
\def\l{\Ccolor{purple!25!white}{purple!08!white}{$h$}}
\def\m{\Ccolor{purple!40!white}{purple!12!white}{$i$}}
\def\n{\Ccolor{purple!55!white}{purple!16!white}{$j$}}
\def\0{\Ccolor{orange!15!white}{orange!10!white}{$k$}}
\def\p{\Ccolor{orange!40!white}{orange!20!white}{$l$}}
\def\q{\Ccolor{orange!65!white}{orange!30!white}{$m$}}
\def\r{\Ccolor{orange!90!white}{orange!40!white}{$n$}}
\begin{scope}
\clip [](9.2,-8.5)-- ++(-5.9,0)-- ++(0,5.9)--cycle;
\draw (3,-8.5)   \i\m\q\n\g\o ;
\draw (3,-7.5)   \l\p\m\j\o ;
\draw (3,-6.5)   \0\l\i\o ;
\draw (3,-5.5)   \k\h\o ;
\draw (3,-4.5)   \g\o ;
\draw (3,-3.5)   \o (3,-2.5) \o ;
\end{scope}
\begin{scope}
\clip [](16.7,-8.5)-- ++(-6.4,0)-- ++(-6.8,6.8)-- ++(0,5.2)-- ++(1.2,0)--cycle;
\draw (9,-8)   \o\a\3\2\6\a\o\o ;
\draw (8,-7)   \o\9\6\1\5\9\o ;
\draw (7,-6)   \o\8\5\2\4\8\o ;
\draw (6,-5)   \o\7\4\1\3\7\o ;
\draw (5,-4)   \o\a\3\2\6\a\o ;
\draw (4,-3)   \o\9\6\1\5\9\o ;
\draw (3,-2)   \o\8\5\2\4\8\o ;
\draw (3,-1)   \7\4\1\3\7\o ;
\draw (3, 0)   \3\2\6\a\o ;
\draw (3, 1)   \1\5\9\o ;
\draw (3, 2)   \4\8\o ;
\draw (3, 3)   \7 ;
\end{scope}
\begin{scope}
\clip [](5.8,3.5)-- ++(6.4,0)-- ++(10,-10)--++(0,-2)-- ++(-4.4,0)--cycle;
\draw (16.5,-8)   \o\j\n\r\k ;
\draw (15.5,-7)   \o\i\m\q\n\g ;
\draw (14.5,-6)   \o\h\l\p\m\j\o ;
\draw (13.5,-5)   \o\g\k\0\l\i\o ;
\draw (12.5,-4)   \o\j\n\r\k\h\o ;
\draw (11.5,-3)   \o\i\m\q\n\g\o ;
\draw (10.5,-2)   \o\h\l\p\m\j\o ;
\draw ( 9.5,-1)   \o\g\k\0\l\i\o ;
\draw ( 8.5, 0)   \o\j\n\r\k\h\o ;
\draw ( 7.5, 1)   \o\i\m\q\n\g\o ;
\draw ( 6.5, 2)   \o\h\l\p\m\j\o ;
\draw ( 5.5, 3)   \o\g\k\0\l\i\o ;
\end{scope}

\begin{scope}
\clip []( 13.3, 3.5) -- ++(6.5,0)-- ++(5,-5)-- ++(0,-6.5)--cycle;
\draw ( 20.5,-5.5) \o                 ;
\draw ( 20.5,-4.5) \o                 ;
\draw ( 19.5,-3.5) \o               \7;
\draw ( 18.5,-2.5) \o             \a\3;
\draw ( 17.5,-1.5) \o           \9\6\1;
\draw ( 16.5,-0.5) \o         \8\5\2\4;
\draw ( 15.5, 0.5) \o       \7\4\1\3\7;
\draw ( 14.5, 1.5) \o     \a\3\2\6\a\o ;
\draw ( 13.5, 2.5) \o   \9\6\1\5\9\o ;
\draw ( 12.5, 3.5) \o \8\5\2\4\8\o  ;
\end{scope}
\begin{scope}
\clip []( 20.8, 3.5) -- ++(1.2,0)-- ++(0,-1.2)--cycle;
\draw ( 20, 2.5) \o   ;
\draw ( 21, 3.5)   \h ;
\end{scope}
\draw [diag, dash phase=0.5\DK] (3,2) -- +(11,-11);
\draw [diag, dash phase=0.5\DK] (16.5,3.5) -- +(6,-6);
\draw [->, dashed] (9,-7 ) -- node [left,near start] {{$\scriptscriptstyle +4$}}  +(0,-1.5);
\draw [->, dashed] (5,-3 ) -- node [left,near start] {{$\scriptscriptstyle +4$}}  +(0,-1.5);
\draw [->, dashed] (14,-6 ) -- node [below,near start] {{$\scriptscriptstyle +4$}}  +(1.5,0);
\draw [->, dashed] (10,-2 ) -- node [below,near start] {{$\scriptscriptstyle +4$}}  +(1.5,0);
\draw [->, dashed] ( 6, 2 ) -- node [below,near start] {{$\scriptscriptstyle +4$}}  +(1.5,0);
\draw [->, dashed] (14.5, 2.5 ) -- node [left,near start] {{$\scriptscriptstyle +4$}}  +(0,-1.5);
\draw [->, dashed] (18.5,-1.5 ) --  node [left,near start] {{$\scriptscriptstyle +4$}}   +(0,-1.5);
\draw [->, dashed] ( 19.5, 3.5 ) -- node [above,near end] {{$\scriptscriptstyle +4$}}  +(1.5,0);
\end{tikzpicture}%
}
\hfill%
\raisebox{1em}{\scalebox{0.55}
{\begin{tikzpicture}
\def\4{\ccolor{white}{$\boldsymbol{4}$}}
\def\2{\ccolor{white}{$\boldsymbol{2}$}}
\def\1{\ccolor{white}{$\boldsymbol{1}$}}
\def\0{\ccolor{white}{0}}
\def\O{\cocolor{white}{\ }}
\def\d{\cocolor{white}{\ddots}}
\def\t{\cocolor{white}{\cdots}}
\def\v{\cocolor{white}{\vdots}}
\def\o{ ++(1, 0) }
\fill[fill=white] (0.25,0.75)rectangle(19,-19);
\draw (0,-0 ) \O\O\1\1\1\1 ;
\draw (0,-1 ) \O\O\1\1\1\1 ;
\draw (0,-2 ) \1\1\O\O\O\O\1\0\0\1 ;
\draw (0,-3 ) \1\1\O\O\O\O\1\1\0\0 ;
\draw (0,-4 ) \1\1\O\O\O\O\0\1\1\0 ;
\draw (0,-5 ) \1\1\O\O\O\O\0\0\1\1 ;
\draw (0,-6 ) \o\o\1\1\0\0\O\O\O\O\1\0\0\1 ;
\draw (0,-7 ) \o\o\0\1\1\0\O\O\O\O\1\1\0\0 ;
\draw (0,-8 ) \o\o\0\0\1\1\O\O\O\O\0\1\1\0 ;
\draw (0,-9 ) \o\o\1\0\0\1\O\O\O\O\0\0\1\1 ;
\draw (0,-10) \o\o\o\o\o\o\d\d\d\d\d\d\d\d\d;
\draw (0,-11) \o\o\o\o\o\o\o\1\1\0\0\O\O\O\O\1\0\0\1 ;
\draw (0,-12) \o\o\o\o\o\o\o\0\1\1\0\O\O\O\O\1\1\0\0 ;
\draw (0,-13) \o\o\o\o\o\o\o\0\0\1\1\O\O\O\O\0\1\1\0 ;
\draw (0,-14) \o\o\o\o\o\o\o\1\0\0\1\O\O\O\O\0\0\1\1 ;
\draw (0,-15) \o\o\o\o\o\o\o\o\o\o\o\2\2\0\0\O\O\O\O ;
\draw (0,-16) \o\o\o\o\o\o\o\o\o\o\o\0\2\2\0\O\O\O\O ;
\draw (0,-17) \o\o\o\o\o\o\o\o\o\o\o\0\0\2\2\O\O\O\O ;
\draw (0,-18) \o\o\o\o\o\o\o\o\o\o\o\2\0\0\2\O\O\O\O ;
\draw (0.5,0.5)rectangle(19.5,-18.5);
\end{tikzpicture}
}} 
\caption{$n=6\eqref{II},10,14,18,\ldots$.\hfill Twins: $\cll1,\cll2$.\hfill Periods: $[4,4]$, $\big[2+\frac{n-2}4,2-\frac{n-2}4\big]$.\\[0.3em]
\mbox{}\hfill
         $\FfF\Lft x,y\Rgt =\FfF_{44}\Lft x,y\Rgt,\ |x-y|\le \frac{n-2}4.$
         \hfill\mbox{}
         }\label{WW}
\end{figure}

\vspace{1.5em}
\subsection*{Acknowledgements}

The author thanks the reviewers
for critically reading the manuscript
and many substantial and constructive suggestions,
which helped to improve and clarify this paper enormously,
and Olga Yakovchenko for help with palettes.

\vspace{2.5em}
\providecommand\href[2]{#2} \providecommand\url[1]{\href{#1}{#1}}
  \def\DOI#1{{\small {DOI}:
  \href{http://dx.doi.org/#1}{#1}}}\def\DOIURL#1#2{{\small{DOI}:
  \href{http://dx.doi.org/#2}{#1}}}

\end{document}